\documentclass[12pt]{article}
\usepackage{graphicx}
\usepackage{amsmath,amsthm,amssymb,enumerate}
\usepackage{euscript,mathrsfs}
\usepackage{color}
\usepackage{dsfont}
\usepackage{multirow}
\usepackage{multicol}
\usepackage{url}
\usepackage[citecolor=blue,colorlinks=true]{hyperref}
\usepackage[left=2cm,right=2cm,top=3.5cm,bottom=3.5cm]{geometry}
\usepackage{color}
\usepackage[framemethod=tikz]{mdframed}
\allowdisplaybreaks

\usepackage{soul}

\catcode`\@=11 \@addtoreset{equation}{section}

\catcode`\@=12

\usepackage{subcaption}

\newtheorem{Theorem}{Theorem}[section]
\newtheorem{Proposition}[Theorem]{Proposition}
\newtheorem{Lemma}[Theorem]{Lemma}
\newtheorem{Corollary}[Theorem]{Corollary}

\theoremstyle{definition}
\newtheorem{Definition}[Theorem]{Definition}

\newtheorem{Remark}[Theorem]{Remark}

\newcommand{\bTheorem}[1]{
	\begin{Theorem} \label{T#1} }
	\newcommand{\eT}{\end{Theorem}}

\newcommand{\bProposition}[1]{
	\begin{Proposition} \label{P#1}}
	\newcommand{\eP}{\end{Proposition}}

\newcommand{\bLemma}[1]{
	\begin{Lemma} \label{L#1} }
	\newcommand{\eL}{\end{Lemma}}

\newcommand{\bCorollary}[1]{
	\begin{Corollary} \label{C#1} }
	\newcommand{\eC}{\end{Corollary}}

\newcommand{\bRemark}[1]{
	\begin{Remark} \label{R#1} }
	\newcommand{\eR}{\end{Remark}}

\newcommand{\bDefinition}[1]{
	\begin{Definition} \label{D#1} }
	\newcommand{\eD}{\end{Definition}}

\newcommand{\Pim}{\Pi_{\mathcal{T}}}

\newcommand{\Ds}{\mathbb{D}_x}

\newcommand{\tvU}{\widetilde{\vU}}

\newcommand{\jump}[1]{\left[ \left[ #1 \right] \right]}

\newcommand{\vrh}{\vr_h}

\newcommand{\vmh}{\vm_h}

\newcommand{\ds}{\,\mathrm{d}S_x}

\newcommand{\bFormula}[1]{
	\begin{equation} \label{#1}}
	\newcommand{\eF}{\end{equation}}

\newcommand{\mesh}{\mathcal{T}_h}

\newcommand{\facesK}{\mathcal{E}(K)}
\newcommand{\vuh}{\vu_h}

\newcommand{\TS}{\Delta t}

\newcommand{\Ov}[1]{\overline{#1}}

\newcommand{\aleq}{\stackrel{<}{\sim}}

\newcommand{\Un}[1]{\underline{#1}}
\newcommand{\vr}{\varrho}

\newcommand{\vt}{\vartheta}
\newcommand{\vu}{\vc{u}}
\newcommand{\vm}{\vc{m}}

\newcommand{\vn}{\vc{n}}

\newcommand{\vc}[1]{{\bf #1}}

\newcommand{\Div}{{\rm div}_x}
\newcommand{\Grad}{\nabla_x}

\newcommand{\dx}{\,{\rm d} {x}}

\newcommand{\dt}{\,{\rm d} t }

\newcommand{\vU}{\vc{U}}

\newcommand{\dxdt}{\dx  \dt}
\newcommand{\intO}[1]{\int_{\Omega} #1 \ \dx}

\newcommand{\intQ}[1]{\int_{{\Omega}} #1 \ \dx}

\newcommand{\intOB}[1]{\int_{\Omega} \left( #1 \right) \ \dx}

\newcommand{\vv}{\vc{v}}

\newcommand{\D}{{\rm d}}
\newcommand{\bD}{\mathbb{D}}
\newcommand{\ep}{\varepsilon}

\newcommand{\R}{\mathbb{R}}

\newcommand{\I}{\mathbb{I}}
\newcommand{\bS}{\mathbb{S}}

\newcommand{\br}{ \nonumber \\ }

\def\softd{{\leavevmode\setbox1=\hbox{d}%
		\hbox to 1.05\wd1{d\kern-0.4ex{\char039}\hss}}}
\definecolor{Cgrey}{rgb}{0.85,0.85,0.85}
\definecolor{Cblue}{rgb}{0.50,0.85,0.85}
\definecolor{Cred}{rgb}{1,0,0}
\definecolor{fancy}{rgb}{0.10,0.85,0.10}
\definecolor{amaranth}{rgb}{0.9, 0.17, 0.31}

\newcommand\Cbox[2]{%
	\newbox\contentbox%
	\newbox\bkgdbox%
	\setbox\contentbox\hbox to \hsize{%
		\vtop{
			\kern\columnsep
			\hbox to \hsize{%
				\kern\columnsep%
				\advance\hsize by -2\columnsep%
				\setlength{\textwidth}{\hsize}%
				\vbox{
					\parskip=\baselineskip
					\parindent=0bp
					#2
				}%
				\kern\columnsep%
			}%
			\kern\columnsep%
		}%
	}%
	\setbox\bkgdbox\vbox{
		\color{#1}
		\hrule width  \wd\contentbox %
		height \ht\contentbox %
		depth  \dp\contentbox
		\color{black}
	}%
	\wd\bkgdbox=0bp%
	\vbox{\hbox to \hsize{\box\bkgdbox\box\contentbox}}%
	\vskip\baselineskip%
}

\mdfdefinestyle{MyFrame}{%
	linecolor=black,
	outerlinewidth=1pt,
	roundcorner=5pt,
	innertopmargin=\baselineskip,
	innerbottommargin=\baselineskip,
	innerrightmargin=10pt,
	innerleftmargin=10pt,
	backgroundcolor=white!20!white}

\newcommand{\pd}{\partial}
\newcommand{\Divh}{{\rm div}_h}
\newcommand{\Gradh}{\nabla_h}

\newcommand{\Laph}{\Delta_h}

\newcommand{\Gradd}{\nabla_{\faces}}
\newcommand{\faces}{\mathcal{E}}

\newcommand{\Divmesh}{{\rm div}_{\mathcal{T}}}

\newcommand{\bI}{\mathbb{I}}
\newcommand{\abs}[1]{{\left| #1 \right|}}
\newcommand{\vth}{\vt_h}
\newcommand{\norm}[1]{\left\lVert#1\right\rVert}
\newcommand{\avs}[1]{ \left\{\hspace{-3pt}\left\{ #1 \right\}\hspace{-3pt} \right\} }
\newcommand{\myangle}[1]{\langle #1\rangle}
\newcommand{\Eb}[1]{E_b\left( #1 \right)}

\newcommand{\facesint}{\faces_{\rm int}}
\newcommand{\facesext}{\faces_{\rm ext}}
\newcommand{\intfacesint}[1]{\int_{\facesint}{ #1 \ds}}

\newcommand{\intfacesext}[1]{\int_{\facesext}{ #1 \ds}}

\newcommand{\intTauTauQB}[1]{\int_{-\tau}^{\tau}\int_{\Omega} \left( #1 \right) \dxdt }
\newcommand{\intTauTauQ}[1]{\int_{-\tau}^{\tau}\int_{\Omega} #1  \dxdt }
\newcommand{\intRQB}[1]{\int_{\R}\int_{\Omega} \left( #1 \right) \dxdt }
\newcommand{\intRQ}[1]{\int_{\R}\int_{\Omega}  #1  \dxdt }

\newcommand{\bbT}{ \mathbb{T}}

\newcommand{\Dhuh}{\mathbb{D}_h(\vuh)}
\newcommand{\vthB}{\vt_{B,h}}
\newcommand{\vtB}{\vt_{B}}
\newcommand{\hvt}{\Theta}
\newcommand{\hvth}{\Theta_h}
\newcommand{\difuh}{\bS_h:\Gradh \vuh }

\newcommand{\vw}{\vc{w}}

\newcommand{\PiQ}{\Pi_Q}

\newcommand{\Up}{{\rm Up}}
\newcommand{\Fup}{F_h^{\alpha}}
\newcommand{\bfphi}{\boldsymbol{\varphi}}
\newcommand{\muh}{h^\alpha}

\allowdisplaybreaks

\newcommand{\sumS}{ \!\! \sum_{\sigma \in \faces} \!\!}
\newcommand{\sumK}{ \!\! \sum_{K \in \mesh} \!\!}
\newcommand{\vWh}{ {\bf W}_h}
\newcommand{\facei}{ {\faces}_i}

\begin{document}


\title{\bf TEMPERATURE-DRIVEN TURBULENCE IN COMPRESSIBLE
FLUID FLOWS}

\author{Eduard Feireisl\thanks{
		The work of E.F.\ was partially supported by the
		Czech Sciences Foundation (GA\v CR), Grant Agreement
		24--11034S. The Institute of Mathematics of the Academy of Sciences of
		the Czech Republic is supported by RVO:67985840.
		E.F.\ is a member of the Ne\v cas Center for Mathematical Modelling.} 
	\and M\' aria Luk\'a\v{c}ov\'a-Medvi\softd ov\'a\thanks{The work of  M.L.-M. was supported by the Gutenberg Research College and by
		the Deutsche Forschungsgemeinschaft (DFG, German Research Foundation) -- project number 233630050 -- TRR 146 and
		project number 525853336 -- SPP 2410 ``Hyperbolic Balance Laws: Complexity, Scales and Randomness".
		She is also grateful to  the  Mainz Institute of Multiscale Modelling  for supporting her research.}
	\and Bangwei She\thanks{ The work of B.S. was supported by National Natural Science Foundation of China under grant No.\ 12571433.} 
	\and Yuhuan Yuan\thanks{ The work of Y.Y. was supported by National Natural Science Foundation of China under grant No.\ 12401527 and 12571433, and Natural Science Foundation of Jiangsu Province under grant No. BK20241364.}
}

\date{}

\maketitle
{\small
\vspace{-0.75cm}
\centerline{$^*$Institute of Mathematics of the Academy of Sciences of the Czech Republic}
\centerline{\v Zitn\' a 25, CZ-115 67 Praha 1, Czech Republic}
\centerline{feireisl@math.cas.cz}

\smallskip
\centerline{$^\dag$Institute of Mathematics, Johannes Gutenberg-University Mainz}
\centerline{Staudingerweg 9, 55 128 Mainz, Germany}
\centerline{RMU Co-Affiliate Technical University Darmstadt, Germany}
\centerline{lukacova@uni-mainz.de}

\smallskip
\centerline{$^\ddag$Academy for Multidisciplinary studies, Capital Normal University}
\centerline{ West 3rd Ring North Road 105, 100048 Beijing, P. R. China}
\centerline{bangweishe@cnu.edu.cn}

\smallskip
\centerline{$^\S$School of Mathematics, Nanjing University of Aeronautics and Astronautics}
\centerline{Jiangjun Avenue No. 29, 211106 Nanjing, P. R. China}
\centerline{yuhuanyuan@nuaa.edu.cn}
}

\vspace{-0.1cm}
\begin{abstract}
We study the long-time behaviour of the temperature-driven compressible flows. We show that numerical solutions of a structure-preserving finite volume method generate a discrete attractor that consists of entire discrete trajectories. Further, we prove  the convergence of discrete attractors to their continuous counterparts. Theoretical results are illustrated by extensive numerical simulations of the well-known Rayleigh--B\'enard problem. The numerical results also indicate the validity of the ergodic hypothesis and imply that a non-zero Reynolds stress persist for long time. Finally, we also observe that any invariant measure is of Gaussian type in sharp contrast with the conjecture proposed by [Glimm et al., SN Applied Sciences~2, 2160 (2020)].

\end{abstract}


{\small
	
	\noindent
	{\bf 2020 Mathematics Subject Classification:}{\\ 
		 76F35, 76F50,  76N06, 65M08, 35M12, 37M25, 37L40 (primary); 
		 76F20, 76N10, 35Q79, 36B41 (secondary)}
	
	\medbreak

	
}

\section{Introduction}
\label{i}

The Rayleigh--B\' enard convection problem is an iconic example of a \emph{turbulent 
behaviour} in fluid dynamics. A layer of fluid is heated from below and subject to a uniform temperature difference as well as the gravitational force acting in the vertical direction. In order to see the thermal effect on the motion, the fluid must be \emph{compressible} changing its volume with temperature. There is an incompressible approximation 
of the problem - the Oberbeck--Boussinesq system. The latter can be rigorously justified as a singular limit for vanishing Mach and Froude numbers, see \cite{BelFeiOsch}, on condition that the temperature as well as the density of the fluid are small perturbations of a constant equilibrium state. Here, we consider the problem in full generality, where the compressibility of the fluid is not negligible and must be taken into account.

The state of a viscous, compressible, and heat conducting fluid at a time $t \in \R$ and a spatial position $x \in \Omega$ can be described by 
three phase variables - the mass density $\vr = \vr(t,x)$, the (macroscopic) velocity $\vu = \vu(t,x)$, and the (absolute) temperature $\vt = \vt(t,x)$. 
Their time-evolution is governed by the \emph{Navier-Stokes-Fourier} system 
of partial differential equations:
\begin{align} 
\partial_t \vr + \Div (\vr \vu) &= 0, \label{i1} \\ 
\partial_t(\vr \vu) + \Div (\vr \vu \otimes \vu) + \Grad p &= 
\Div \mathbb{S} + \vr \Grad G, \label{i2} \\
\partial_t (\vr e) + \Div (\vr e \vu)	
+ \Div \vc{q} &= \mathbb{S} : \Grad \vu - p \Div \vu, 
\label{i3}
\end{align}	
where $\mathbb{S}$ is the \emph{viscous stress} given by Newton's law
\begin{equation} \label{i4}
\mathbb{S} = \mu \left( \Grad \vu + \Grad \vu^t - \frac{2}{d} \Div \vu 
\mathbb{I} \right) + \eta \Div \vu \mathbb{I}, 
\end{equation}
and $\vc{q}$ is the \emph{heat flux} given by Fourier's law 
\begin{equation} \label{i5} 
\vc{q} = - \kappa \Grad \vt.
\end{equation}		
The \emph{internal energy} $e = e(\vr, \vt)$ is related to the \emph{pressure} $p = p(\vr, \vt)$ through Gibbs' equation 
\begin{equation} \label{i7}
\vt Ds(\vr, \vt) = D e(\vr, \vt) + p(\vr, \vt) D \left( \frac{ 1}{\vr} \right), 
\end{equation} 
where $s = s(\vr, \vt)$ is the \emph{entropy}.	 

The fluid is confined to a bounded spatial domain 
\[
\Omega = {\mathbb{T}}^{d-1} \times [-H,H],\ 
\mathbb{T}^{d-1} = \left( [-L,L]|_{\{ -L, L\} } \right)^{d-1},\ d = 2,3, 
\]
meaning spatial periodicity is imposed in the horizontal direction. 
In addition, the velocity satisfies the no-slip boundary conditions 
\begin{equation} \label{i8}
\vu|_{\partial \Omega} = 0, 
\end{equation}	
while the temperature $\vt$
\begin{equation} \label{i9}
\vt|_{\partial \Omega} = \vtB, 
\end{equation}
is prescribed on the horizontal boundary $ x_d = -H, H$. The function $G = G(x)$ represents the gravitational potential, typically 
\begin{equation} \label{i9a}
G(x) = - x_d.
\end{equation}

The aim of the present paper is to study the long-time behaviour of the Navier-Stokes-Fourier system for arbitrarily large solutions out of thermodynamic equilibrium.  Although a chaotic motion in certain regimes has been confirmed by many experiments as well as computational results (see e.g. Castillo, Hoover and Hoover \cite{CaHoHo}, John, Schumacher \cite{JohSch}, Tiwari, Sharma and Verma 
\cite{TiShVe}, and the references therein), a rigorous mathematical analysis has been hampered by a total absence of a relevant existence 
theory for the Navier-Stokes-Fourier system with purely Dirichlet boundary conditions. Indeed
a proper concept of weak solutions as well as their  
global existence have been established only recently in 
\cite{ChauFei}, \cite{FeiNovOpen}. Note that the boundary conditions 
\eqref{i8}, \eqref{i9} make the fluid system \emph{energetically open}, 
thus amenable to a chaotic (turbulent) behaviour.

\subsection{Available analytical results}

As shown in \cite[Theorem 3.1]{FeiSwGw}, the Navier-Stokes-Fourier system \eqref{i1}--\eqref{i5} endowed with 
the Dirichlet boundary conditions \eqref{i8}, \eqref{i9} is 
dissipative in the sense of Lewinson, meaning it admits a bounded absorbing set. Specifically, there exists a universal constant 
$\mathcal{E}_\infty$, that can be determined only in terms of  
\[
\mbox{the total mass}\ 
M_0 = \intO{\vr },\ \mbox{and the boundary data}\ \vtB, 
\]
such that 
\begin{equation} \label{i10}
\limsup_{T \to \infty} \intO{ E(\vr,\vu,\vt) (T, \cdot) } \leq \mathcal{E}_\infty,\ \mbox{where}\  E (\vr, \vu, \vt) = \frac{1}{2} \vr |\vu|^2 + \vr e(\vr, \vt),
\end{equation}  
for \emph{any} (weak) solution defined on a time interval $(t_0, \infty)$. The result holds in the framework of the existence theory developed in \cite{FeiNovOpen} under certain physically grounded restrictions imposed on the constitutive equations and transport coefficients.

To apply the available results of the classical dynamical systems theory, we need continuity in time of the phase variables. To achieve this in the weak solution framework used in \cite{FeiSwGw}, 
we replace the standard phase variables $(\vr, \vu, \vt)$ by the 
so-called entropy-conservative  
variables 
\[
(\vr, \vm, S),\ \mbox{where}\ \vm = \vr \vu,\ S = \vr s(\vr,\vt). 
\]

As shown in \cite{FeiSwGw}, the Navier-Stokes-Fourier system admits a global trajectory attractor containing all \emph{entire solutions}, meaning solutions defined for $t \in (-\infty, \infty)$, 
\begin{align}
\mathcal{A} = \Big\{ (\vr, \vm, S) \ \Big| \ &(\vr, \vm, S) \ \mbox{is a (weak) solution of the Navier-Stokes-Fourier system defined for}\ t \in \R,  \br 
&\intO{ E(\vr,\vu,\vt) (t, \cdot) } \leq \mathcal{E}_\infty \ \mbox{for all}\ t \in (- \infty, \infty)
\Big\}. \label{i12}
\end{align}
The attractor $\mathcal{A}$ is non--empty and compact with respect to 
the metric topology 
\[
(\vr, \vm, S) \in 
C_{\rm loc}(\R; W^{-\ell,2}(\Omega; \R^{d+1})) \times 
D_{\rm loc} (\R; W^{-\ell,2}(\Omega; \R)),  
\]	
where $D_{\rm loc}$ denotes the ``weak Skorokhod space'', 
see  \cite[Appendix A.1]{FeiNovOpen} for the definition of the ``weak'' Skorokhod topology 
$D_{\rm loc}$.	
Moreover, $\mathcal{A}$ is obviously time shift invariant, meaning 
\[
(\vr, \vm, S) \in \mathcal{A} \ \Rightarrow \ (\vr, \vm, S) (\cdot + T) 
\in \mathcal{A}  \ \mbox{for any}\ T \in \R.
\]

The attractor reduces to a single point provided 
\[
\intO{ \vr } = \Ov{\vr} |\Omega|,\ \| G \|_{C^1(\Ov{\Omega})} < \ep , \ 
\| \vtB - \Ov{\vt} \|_{C^{2 + \nu}(\Ov{\Omega})} < \ep, 
\]
where $\ep = \ep(\Ov{\vr}, \Ov{\vt}) > 0$ is small enough,  see \cite{FeiLuSun}. Accordingly, \emph{all} global in time 
weak solutions converge to a single equilibrium. 
More precisely, given positive constant $\Ov{\vr}$, $\Ov{\vt}$, there exists $\ep_0 = \ep_0 (\Ov{\vr}, \Ov{\vt}) > 0$ such that
any weak solution of the Navier-Stokes-Fourier system converges to a stationary state 
$(\vr_s, \vm_s, S_s)$:
\begin{align}\label{conv2stat}
\vr(t, \cdot) &\to \vr_s \ \mbox{in}\ L^q(\Omega),\ 1 \leq q < \frac{5}{3}, \br
\vm(t, \cdot) &\to \vm_s \ \mbox{in}\ L^q(\Omega; \R^d),\ 1 \leq q < \frac{5}{4}, \br
S(t, \cdot) &\to S_s \ \mbox{in}\ L^q(\Omega),\ 1 \leq q < \frac{4}{3},	 
\end{align}	
as $t \to \infty$ whenever $0 < \ep < \ep_0$, see \cite{FeiLuSun}.

\subsection{Objectives addressed in the present paper}

Our main goal is to illustrate the above theoretical results by numerical experiments, and to indicate possible 
conjectures in the situations, where the available analytical techniques fail. 
To this end, we first demonstrate the proximity of the numerical discrete solutions and their continuous limits in the long run. Specifically, we show that 
numerical solutions generate a discrete attractor, see Theorem \ref{AT1}, that consists of discrete trajectories defined 
for all $t \in (-\infty, \infty)$. Our main analytical result then asserts that the discrete attractors 
approach their continuous counterpart $\mathcal{A}$ for vanishing discretization parameters, see Theorem \ref{CT1}.
In particular, we show strong convergence of the numerical solutions to the exact solution which is a result of independent interest.

In numerical experiments, we test validity of the so--called \emph{ergodic hypothesis}, namely convergence of the ergodic averages  
\begin{equation} \label{L1a}
	\lim_{T \to \infty} \frac{1}{T}  \int_{0}^T F\Big((\vr, \vm, S)(t, \cdot) \Big) \dt 	\end{equation} 
for any entire solution of the Navier-Stokes-Fourier system and any bounded Borel function $F$ defined on a suitable \emph{phase space}, cf. \cite{FanFeiHof}. The limit, provided it exists, 
generates an invariant measure characterizing the long time behaviour of the 
system. In sharp contrast with the recently proposed conjecture by Glimm et al. \cite{ChenGli12}, \cite{CheGli}, \cite{GlLaCh}, our numerical 
solutions exhibit the normal (Gaussian) distribution of the invariant measure rather than the uniform distribution proposed in the above references. 

\subsection{Organization of the paper}

We start by introducing a time implicit numerical scheme based on finite volume spatial discretization, see Section \ref{n}. 
Motivated by the strategy proposed by Wang \cite{WangX}, we establish the existence of a discrete analogue of the attractor $\mathcal{A}$, see Section \ref{A}. Next, in Section \ref{C}, we state our main analytical results on  
convergence of discrete attractors to $\mathcal{A}$ locally in time for 
vanishing discretization parameters. The proof of this result is then given in Section \ref{P}. Finally, in Section \ref{L}, we recall the results obtained in \cite{FeiSwGw}. 
We illustrate and complement them by numerical experiments, and formulate hypotheses indicated by simulations.  

\section{Numerical method}
\label{n}
In numerical simulations, we impose the standard constitutive relations:
\begin{equation*}
p=  \vr \vt , \ e = c_v \vt, \ s=c_v \log \vt - \log \vr, \ c_v = \frac{1}{\gamma-1}, \ \gamma >1.
\end{equation*}

Next, we recall the weak formulation for the Navier-Stokes-Fourier system, cf. \cite[Definition 2.1]{FeiSwGw}.

\begin{Definition} \label{Dw1} 
	We say that $(\vr, \vt, \vu)$ is a \emph{weak solution} of the Navier--Stokes--Fourier system \eqref{i1}--\eqref{i7} 
	in the time interval $(t_0, \infty)$, $t_0 \geq -\infty$, 
	with the boundary conditions \eqref{i8}, \eqref{i9}, 	
	 if the following holds:
	
	\begin{itemize}
		
		\item The solution belongs to the {\bf regularity class}: 
		\begin{align}
			\vr, \vt &\in L_{\rm loc}^\infty([t_0, \infty); L^p(\Omega)),\ \vr \vu \in L_{\rm loc}^\infty([t_0, \infty); L^p(\Omega, \R^d))  \ \mbox{for some}\ p > 1, \ \vr \geq 0,\ \vt > 0, \br
			\vu &\in L^2_{\rm loc}([t_0, \infty); W^{1,2}_0 (\Omega; \R^d)), \quad
			(\vt - \vtB) \in L_{\rm loc}^2([t_0, \infty); W^{1,2}_0 (\Omega)).
			\label{w6}
		\end{align}
		
		\item The {\bf equation of continuity} \eqref{i1} is satisfied in the sense of distributions 
		\begin{align} 
			\int_{t_0}^{\infty} \intO{ \left[ \vr \partial_t \varphi + \vr \vu \cdot \Grad \varphi \right] } \dt &= 
			0, 
			\label{w3} \\
						\int_{t_0}^\infty \intO{ \left[ b(\vr) \partial_t \varphi + b(\vr) \vu \cdot \Grad \varphi + \Big( 
								b(\vr) - b'(\vr) \vr \Big) \Div \vu \varphi \right] } \dt &=0
						\label{w4}
		\end{align}
		for any $\varphi \in C^\infty_c((t_0 , \infty) \times \Ov{\Omega} )$ and any $b \in C^1(\R)$, 	$b' \in C_c(\R)$.
		\item The {\bf momentum equation} \eqref{i2} is satisfied in the sense of distributions 
		\begin{align}
			&\int_{t_0}^\infty \intO{ \left[ \vr \vu \cdot \partial_t \bfphi + \vr \vu \otimes \vu : \Grad \bfphi + 
				p \Div \bfphi \right] } \dt  = \int_{t_0}^\infty \intO{ \left[ \mathbb{S} : \Grad \bfphi - \vr \Grad G \cdot \bfphi \right] } \dt  
			\label{w5}
		\end{align}	
		for any $\bfphi \in C^\infty_c((t_0, \infty) \times \Omega; \R^d)$.
		
		\item The internal energy equation \eqref{i3} is replaced by the {\bf entropy inequality}
		\begin{align}
			- \int_{t_0}^\infty \intO{ \left[ \vr s \partial_t \varphi + \vr s \vu \cdot \Grad \varphi + \frac{\vc{q}}{\vt} \cdot 
				\Grad \varphi \right] } \dt  &\geq \int_{t_0}^\infty \intO{ \frac{\varphi}{\vt} \left[ \mathbb{S} : \Ds \vu - 
				\frac{\vc{q} \cdot \Grad \vt }{\vt} \right] } \dt 
			\label{w7} 
		\end{align}
		for any $\varphi \in  C^\infty_c((t_0, \infty) \times \Omega)$, $\varphi \geq 0$; and the {\bf ballistic energy inequality} 
		\begin{align}  
			- \int_{t_0}^\infty \partial_t \psi	&\intO{ \left[ \frac{1}{2} \vr |\vu|^2 + \vr e - \hvt \vr s \right] } \dt + \int_{t_0}^\infty \psi
			\intO{ \frac{\hvt}{\vt}	 \left[ \mathbb{S}: \Ds \vu - \frac{\vc{q} \cdot \Grad \vt }{\vt} \right] } \dt  \br
			&\leq 
			\int_{t_0}^\infty \psi \intO{ \left[ \vr \vu \cdot \Grad G - \vr s \partial_t \Theta - \vr s \vu \cdot \Grad \hvt - \frac{\vc{q}}{\vt} \cdot \Grad \hvt \right] } \dt 
			\label{w8}
		\end{align}
		for any $\psi \in  C^\infty_c(t_0, \infty)$, $\psi \geq 0$, and any {$\hvt \in BC^2([t_0, \infty) \times \Ov{\Omega}),\ \hvt > 0,\ \hvt|_{\pd \Omega} = \vtB$}.
	\end{itemize}
	
\end{Definition}

\begin{Remark}
Similarly to \cite{FeiSwGw}, the above definition does not contain any initial data as they are irrelevant in the subsequent analysis. The initial data can be easily accommodated by considering the test functions compactly supported in $[t_0,\infty)$, and adding the corresponding boundary integrals in the definition, see \cite[Chapter 3]{FeiNovOpen}. 
\end{Remark}

\subsection{Notation}
Before formulating the numerical method, we introduce the necessary notation. 
The domain $\Omega$ is divided into uniform cubes (or squares in in the case $d=2$) of size $h\in(0,1)$, denoted $\mesh$. The symbol $Q_h$ denotes the space of piecewise constant functions on the discrete mesh $\mesh$. 
The set of all faces of $\mesh$ is denoted by $\faces$,  $\facesext = \faces \cap \partial \Omega$ and $\facesint = \faces \setminus \facesext$ stand for the set of all exterior and interior faces, respectively.
We denote 
by $\facei$, $i=1,\dots, d$, the set of all faces that are orthogonal to the canonical basis vector ${\bf e}_i$. 
Moreover, we define the $i^{\rm th}$ dual grid ${\cal D}_i$ as set of all cubes of the same size $h$ with mass centers sitting at the same position as $\sigma \in  \facei$. Let $W_h^{(i)}$ be the space of piecewise constants on ${\cal D}_i$ and $\vWh = \{W_h^{(1)},\cdots,W_h^{(d)}\}$. 
For a generic function $v \in Q_h$ we denote 
\begin{equation*}
\begin{aligned}
& v^{\rm in} = \lim_{\delta \rightarrow 0^+} v(x -\delta \vn ), \  \ 
v^{\rm out} = \lim_{\delta \rightarrow 0^+} v(x +\delta \vn ), \  \ 
\jump{v}= v^{\rm out} - v^{\rm in}, \ \  \avs{v}= (v^{\rm out} + v^{\rm in} )/2
\end{aligned}
\end{equation*}
on any face $\sigma \in \faces$. 
Given a velocity field $\vuh$, the upwind flux at $\sigma \in \faces$  for $r_h \in Q_h$ is defined as
\begin{align*}
\Up [r_h, \vuh]|_\sigma   =r_h^{\rm up} \avs{\vuh }_\sigma \cdot \vn_\sigma , \quad 
r_h^{\rm up} =
\begin{cases}
r_h^{\rm in} & \mbox{if } \ \avs{\vuh }_\sigma \cdot \vn_\sigma \geq 0, \\
r_h^{\rm out} & \mbox{if } \ \avs{\vuh }_\sigma \cdot \vn_\sigma < 0.
\end{cases}
\end{align*}
Further, we introduce the following discrete difference operators for $v\in Q_h, \vv\in Q_h^d$ and $\vw \in \vWh$:
\begin{align*}
& \Gradd v(x)  =  \frac{1}{h}\sumS  \mathds{1}_{D_\sigma}(x) \vn \jump{v}  , 
\quad 
\Gradh v(x)  =  \frac{1}{h}\sumK  \mathds{1}_K{(x)}\sum_{\sigma\in \facesK} \vn  \avs{v},
\\& 
\Divmesh \vw (x)= \frac{1}{h}\sumK \mathds{1}_K{(x)}\sum_{\sigma\in \facesK} \vn \cdot \vw,
\quad  
\Divh \vv(x) = \frac{1}{h}\sumK  \mathds{1}_K{(x)} \sum_{\sigma\in \facesK} \vn \cdot \avs{\vv},
\\
& \Laph v (x)= \frac{1}{h^2}\sumK\mathds{1}_K{(x)}  \sum_{\sigma\in \facesK} \jump{v}, 
\quad \bD_h \vv = (\Gradh \vv + \Gradh^T \vv)/2.
\end{align*}

\subsection{Scheme}
The numerical method we use for solving the Navier-Stokes-Fourier system \eqref{i1}--\eqref{i9} is the fully discrete time implicit proposed in \cite{FeLMShYu:2024}. For the sake of simplicity, we set $\Delta t \approx h, \ \Delta t \in (0,1)$ and $\Grad G \cdot \vn|_{\partial \Omega} = \mbox{const}$; a more general situation can be handled in a similar manner. 
Given $k \in \mathbb Z$, we denote 
\[
D_t v^k_h = \frac{v^k_h - v^{k-1}_h} {\Delta t}. 
\]
It is convenient to formulate the method in both weak and strong forms. 
These formulations are equivalent, and will be alternatively used in different contexts.
 
\begin{Definition}[{{\bf Weak form}}]\label{VFV-weak}
Let $\vthB  \in W_h^{(d)}$ be given. 
For $\vU_h^{k-1}\equiv(\vrh^{k-1},\vuh^{k-1}, \vth^{k-1})\in Q_h^{d+2}$,  $\vU_h^k\equiv(\vrh^k,\vuh^k, \vth^k)\in Q_h^{d+2}$  is defined as a solution to the following system of algebraic equations:
\begin{subequations}\label{VFV}
\begin{align}\label{VFV_D}
&\intQ{ D_t \vrh^k  \phi_h } - \intfacesint{  \Fup (\vrh^{k} ,\vuh^{k} )
\jump{\phi_h}   }  = 0 \hspace{4.5cm} \mbox{ for all }\ \phi_h \in Q_h, \\
 \label{VFV_M}
&\intQ{ D_t  (\vrh^k  \vuh^k ) \cdot \bfphi_h } - \intfacesint{ {\bf F}_h^{\alpha}  (\vrh^{k}  \vuh^{k} ,\vuh^{k} ) \cdot \jump{\bfphi_h}   } + \intQ{ (\bS_h^{k} -p_h^{k} \bI ) : \bD_h \bfphi_h }
\br
&\hspace{4cm}
= \intQ{\vrh^{k} \Grad G \cdot  \bfphi_h} \hspace{0.8cm}
\mbox{ for all } \bfphi_h \in Q_h^d, \quad \avs{\bfphi_h}_{\sigma} = 0, \ \sigma \in \facesext,
\\ \label{VFV_E}
&c_v\intQ{ D_t (\vrh^k  \vth^k ) \phi_h } - c_v\intfacesint{  \Fup (\vrh^{k} \vth^{k} ,\vuh^{k} )\jump{\phi_h} }+\intfacesint{  \frac{\kappa}{ h } \jump{\vth^{k}}  \jump{ \phi_h}  }
\br
&\hspace{1.5cm}
+ 2\intfacesext{  \frac{\kappa}{ h } \left( (\vt_h^{k})^{\rm in} - \vthB \right)  \phi_h^{\rm in}  }
= \intQ{ (\bS_h^{k} -p_h^{k} \bI ) : \Gradh \vuh^{k} \phi_h} \quad \mbox{for all}\ \phi_h \in Q_h,
\end{align}
\end{subequations}
where 
$\Fup (r_h,\vuh)$ is the diffusive upwind flux taken as
\begin{equation*}
\Fup (r_h,\vuh)
={\Up}[r_h, \vuh] - \muh \jump{ r_h }, \quad \alpha >-1,
\end{equation*}
and $\bS_h = 2\mu \bD_h \vuh + \lambda \Divh\vuh \bI, \; \Dhuh = (\Gradh \vu_h+\Gradh^t \vu_h)/2, \; \lambda = \eta -  \frac2{d} \mu$  with the  boundary conditions
\[
\avs{\vth}_{\sigma } = \vthB, \quad \avs{\vuh}_{\sigma } = 0, \quad \sigma \in \facesext.
\]
\end{Definition}

Alternatively, we use the strong formulation of the scheme.

\begin{Definition}[{\bf Strong form}]
The FV scheme \eqref{VFV} can be be rewritten in the following strong form: 
\begin{subequations}\label{VFV-1}
\begin{align}\label{VFV_D-1}
&D_t \vrh^{k}   + \Divmesh ( \Fup (\vrh^{k} ,\vuh^{k} ) \cdot \vn) = 0,   \\
 \label{VFV_M-1}
&D_t  (\vrh^{k}  \vuh^{k} )  + \Divmesh ({\bf F}_h^{\alpha}  (\vrh^k  \vuh^k ,\vuh^k ) \cdot \vn )= \Divh (\bS_h^k -p_h^k \bI ) + \vrh^{k} \Grad G ,   \\\label{VFV_IE-1}
&c_v D_t (\vrh^k  \vth^k )  + c_v \Divmesh ({\bf F}_h^{\alpha}  (\vrh^k \vth^k  ,\vuh^k ) \cdot \vn) - \kappa \Laph \vth^k
= (\bS_h^k -p_h^k \bI ) : \Gradh \vuh^k,
\end{align}
\end{subequations}
equipped with the boundary conditions
\begin{align}
 \Fup (r_h , \vuh )|_{\sigma} = 0, \quad \avs{\vth^k}_{\sigma } = \vthB, \quad \avs{\vuh^k}_{\sigma } = 0 \ \mbox{ and }\ \jump{\bS_h^k -p_h^k \bI}_{\sigma} \cdot \vn = 0, \quad \sigma \in \facesext.
\end{align}
\end{Definition}

\subsubsection{Time interpolation}
\label{inter}

Given an initial time  $(t_0, k_0) \in (\R, \mathbb{Z})$, we can identify the finite volume approximation \eqref{VFV} with a time dependent function as follows. 
\begin{itemize}
\item For any $k \geq 0$, we set
\begin{align*}
&  \vU_h(t_0+k\TS) = \vU_h^{k_0+k}.
\end{align*}

\item For each time subinterval $(t_0+{(k-1)}\TS,t_0+k\TS], k \geq 1, $ we construct two time interpolations:
\begin{equation}\label{time_extension}
\begin{aligned}
& \vU_h(t,\cdot) : =\vU_h^{k_0+k} & \mbox{piecewise constant}; \\ 
& \tvU_h(t,\cdot)=\vU_h^{k_0+k-1} + \frac{\vU_h^{k_0+k}- \vU_h^{k_0+k-1}}{\TS}(t-(t_0+{(k-1)}\TS))  & \mbox{piecewise linear}.
\end{aligned}
\end{equation}

\item In addition,  for the piecewise constant interpolation, we set
\[
D_t \vU_h(t) := \frac{\vU_h^{k_0+k}- \vU_h^{k_0+k-1}}{\TS}
 \ \mbox{ for } \ t\in (t_0+{(k-1)}\TS,t_0+k\TS], \ \ k\geq 1.
\]
\end{itemize}

\subsubsection{Hypothesis}
\label{hyp}

We impose a hypothesis of \emph{boundedness}  of numerical solutions:
\begin{equation}\label{HP}
\mbox{\bf (B)} \quad    0< \Un{\vr} \leq \vrh^k \leq \Ov{\vr}, \
0< \Un{\vt} \leq \vth^k \leq \Ov{\vt}, \
\abs{\vuh^k} \leq \Ov{u}, \  \mbox{ uniformly for all } k,  \mbox{ and } h \to 0,  
\end{equation}
for certain positive constants $\underline{\vr}$, $\Ov{\vr}$, $\underline{\vt}$, 
$\Ov{\vt}$, $\Ov{u}$. 
This means that the approximate sequence of numerical 
solutions remains in a physically admissible range. 

Such an assumption frequently imposed in numerical analysis is indispensable in deriving associated stability estimates, consistency and convergence. Throughout the whole text, we tacitly assume 
Hypothesis {\bf (B)} holds.

\section{Discrete attractors}
\label{A}

The trajectory attractor $\mathcal{A}$, cf.\ \eqref{i12}, for the Navier-Stokes-Fourier system identified in \cite{FeiSwGw} consists of entire bounded solutions defined for $t \in (-\infty, \infty)$. 
In this section, we introduce the concept of entire discrete solutions and show their convergence to their continuous counterparts. 

\begin{Definition}[{\bf Discrete solution}]
We say that 
\[
\widetilde{\vc{U}}_h(t,x) = (\widetilde{\vr}, \widetilde{\vu}, \widetilde{\vt})_h(t,x)  
\]
is a discrete solution of the Navier-Stokes-Fourier system on the time interval $[T, \infty)$ 
if 
\[
\widetilde{\vc{U}}_h (T + k \Delta t, x), \ k = 0, 1, \dots, 
\]
 is the linear interpolation of the numerical solution introduced in \eqref{VFV} and \eqref{time_extension}. 
\end{Definition}

\begin{Definition} [{\bf Entire discrete solution}] \label{AD1}
	
We say that 
\[
\widetilde{\vU}_h = (\widetilde{\vr}, \widetilde{\vu}, \widetilde{\vt})_h \in C(\R; Q_h^{d+2}) 
\]
is \emph{entire discrete solution}  of the Navier-Stokes-Fourier system if there exists $\tau \in \R$ such that $\widetilde{\vU}_h$ is a discrete solution on $[\tau - n \Delta t, \infty)$ for any $n=0,1,\dots$.	
	
\end{Definition}

\medskip
	
As shown in \cite{FeiSwGw}, the set $\mathcal{A}$ of all entire solutions is a trajectory attractor for the Navier-Stokes-Fourier system. The following 
result can be seen as a discrete analogue of this statement. 

\begin{Theorem}[\bf Discrete attractor] \label{AT1}
Let $\{\widetilde{\vU}_h^n\}_{n=1}^\infty$ be a sequence of discrete solutions defined on the time 
intervals $[T_n, \infty)$, $T_n \to - \infty$. Suppose that the associated 
numerical solutions $ \widetilde{\vU}_h^n$ satisfy Hypothesis {\bf (B)} (with the same bounds $\underline{\vr}$, $\Ov{\vr}$, $\Ov{u}$, $\underline{\vt}$, $\Ov{\vt}$) uniformly for $t \in \R$, $n \to \infty$.

\medskip
Then there exists a subsequence $\{\widetilde{\vU}_h^{n_m} \}_{m=1}^\infty$ such that 
\begin{equation} \label{A1}
\widetilde{\vU}_h^{ n_m} \to \widetilde{\vU}_h^\infty \ \mbox{in}\ C_{\rm loc}(\R; Q_h^{d+2})\ \mbox{as}\ m \to \infty,
\end{equation}
where $\widetilde{\vU}_h^\infty$ is an entire discrete solution.	
\end{Theorem}	

\begin{proof}
As $h$ and $\Delta t \approx h$ are fixed, and Hypothesis {\bf (B)} is satisfied, the discrete solutions are uniformly globally Lipschitz in 
$[T_n, \infty)$. Thus the convergence claimed in \eqref{A1} follows from Arzel\` a--Ascoli theorem. 

It remains to observe that the limit $\widetilde{\vU}_h^\infty$ is an entire discrete solutions. This is certainly true if $T_n = - a_n \Delta t$, 
where $a(n) \to \infty$ is a sequence of integers. In the general case, we perform the time shift replacing 
\[
\widetilde{\vU}^n_h (t, \cdot) \ \mbox{by}\ 
\widetilde{\vU}^n_h (t + \tau_n, \cdot) ,\ \tau_n \in [0, \Delta t), 
\] 
thus for the new sequence it holds  $T_n = - a_n \Delta t$. In addition, passing to a subsequence as the case may be, we have 
\[
\tau_n \to \tau \in [0, \Delta t]. 
\]
Finally, we observe that 
\[
\widetilde{\vU}^n_h(\cdot + \tau_n, \cdot)  \to \widetilde{\vU}^\infty_h (\cdot + \tau, \cdot) 
\]
whenever 
\[	
\widetilde{\vU}^n_h \to \widetilde{\vU}^\infty_h
\]	
as all functions are globally Lipschitz.	 
\end{proof} 

\section{Convergence of discrete attractors}
\label{C}

Our main analytical result shows proximity of the approximate and exact entire solutions -- trajectory attractors -- for vanishing discretization parameter $(\Delta t,h)$, 
$\Delta t \approx h$.

The \emph{entire weak solutions} to the Navier-Stokes-Fourier system are defined by setting $t_0=-\infty$ in Definition \ref{Dw1}. 
Accordingly, all test functions are supposed to be compactly supported in the time variable. 

\begin{Theorem} [\bf Attractor convergence] \label{CT1}
Let $\{\widetilde{\vU}_{h} = (\widetilde{\vr_h}, \widetilde{\vu_h}, \widetilde{\vt_h})\}_{h > 0}$ be a family of entire discrete solutions satisfying Hypothesis {\bf (B)} (with the same bounds $\underline{\vr}$, $\Ov{\vr}$, $\Ov{u}$, $\underline{\vt}$, $\Ov{\vt}$) uniformly for $t \in \R$ and $h \to 0$. 

Then there exists a sequence $h_n \searrow 0$ such that 
\begin{equation} \label{A2}
\widetilde{\vU}_{h_n} \to \vU \ \mbox{in} \ L^q_{\rm loc}(\R; L^q(\Omega; \R^{d+2})) \ \mbox{for any}\ 1 \leq q  < \infty, 	
\end{equation} 
where $\vU = (\vr, \vu, \vt)$ is an entire weak solution of the Navier-Stokes-Fourier system \eqref{i1}--\eqref{i3}, with the boundary conditions \eqref{i8}, \eqref{i9} in the 
sense of Definition~\ref{Dw1}.	
\end{Theorem}

The proof of Theorem \ref{CT1} will be presented in the following two sections. 
The fundamental ingredients of the convergence analysis are the stability and consistency of the finite volume method \eqref{VFV} discussed in Section~\ref{stacon}.  Having established these results we complete the proof of Theorem~\ref{CT1} in Section~\ref{P}.

\section{Stability and consistency}
\label{stacon}

Performing a simple time shift, we can assume, without loss of generality, that $k \in \mathbb{Z}$, and $\vU_h(0) = \vU_h^0$ for $h \to 0$. 
The following stability and compatibility results 
can be obtained in the same way as in \cite{FeLMShYu:2024} and \cite{FeLMShYu:2025I}. 
\begin{Lemma}[{{\bf Uniform bounds \cite[Lemma A.4]{FeLMShYu:2024}}}]\label{lm_ub}
Let $\vU_h \equiv (\vrh ,\vuh, \vth)$ be an entire numerical solution obtained by the FV method \eqref{VFV} with $(\TS, h) \in (0,1)^2$, $\TS \approx h$, and $\alpha \in (-1,1)$.
Let Hypothesis {\bf (B)} hold. 

Then we have
\begin{subequations}\label{ap}
\begin{align}\label{ap1}
&\norm{\Gradd \vth}_{L^2((T,T+1) \times \Omega; \R^{d})} + \norm{\Gradh \vuh}_{L^2((T,T+1) \times \Omega; \R^{d\times d})} \leq C ,
\\ \label{ap2}
&(\TS)^{1/2} \norm{D_t \vU_h }_{L^2((T,T+1) \times \Omega; \R^{d+2})}  \leq C ,
\\  \label{ap3}
&\int_{T}^{T+1}\intfacesint{ \left( h^\alpha + \abs{ \avs{\vuh}\cdot \vc{n} } \right) \, \abs{\jump{ \vU_h }}^2  }\dt \leq C.
\end{align}
\end{subequations}
The constant $C$ depends on $\|\vtB\|_{W^{2,\infty}(\Omega)}$ and $\Un{\vr}, \Ov{\vr}, \Un{\vt}, \Ov{\vt}, \Ov{u}$, but it is independent of $T \in \R$ and discretization parameters $(h, \TS)$.
\end{Lemma}

\begin{Lemma}[{\bf Compatibility of discrete gradients \cite{FeLMShYu:2025I}}]\label{compatibility-vel}
Under Hypothesis {\bf (B)}, let $\vU_h \equiv (\vrh ,\vuh, \vth)$ be an entire numerical solution obtained by the FV method \eqref{VFV} with $(\TS, h) \in (0,1)^2$, $\TS \approx h$, and $\alpha \in \left(-1,1 \right)$. 
Denote
\begin{subequations}\label{cf}
\begin{align} \label{cf1}
& \left< e_{\Grad\vu}; \bbT \right> \equiv \int_{T}^{T+1} \intOB{ \vu_{h} \cdot \Div \bbT + \Gradh \vuh : \bbT  } \dt,
\\ \label{cf2}
& \left< e_{\Grad{(|\vu|^2})}; \bfphi  \right> \equiv \int_{T}^{T+1} \intOB{ |\vuh|^2  \Div \bfphi  + \Gradh (|\vuh|^2) \cdot \bfphi   }\dt, \\ \label{cf3}
& \left< e_{\Grad\vt}; \bfphi \right> \equiv 
\int_{T}^{T+1} \intO{ \bigg( \vth \Div \bfphi+  \Gradd\vth \cdot \bfphi  \bigg)}\dt -  \int_{T}^{T+1} \int_{\pd \Omega}{ \vtB \, \bfphi \cdot \vn} \ds\dt,\\ \label{cf4}
& \left< e_{\Grad(\vt^2)}; \bfphi  \right> \equiv 
\int_{T}^{T+1} \intO{ \bigg(  \vth^2 \Div \bfphi +\Gradd(\vth^2) \cdot  \bfphi   \bigg)} \dt -  \int_{T}^{T+1} \int_{\pd \Omega}{ \vtB^2 \, \bfphi  \cdot \vn} \ds\dt,
\end{align}
for $\bbT \in C^1([T,T+1]\times \Ov{\Omega}; \R^{d \times d}_{\rm sym}))$ and
$\bfphi \in C^1([T,T+1]\times \Ov{\Omega}; \R^{d}))$.

Then we have the following compatibility error estimates
\begin{align}\label{cf-e1}
& \abs{\left< e_{\Grad\vu}; \bbT \right>  } \leq C  h^{(1-\alpha)/2},
\quad \abs{ \left< e_{\Grad{(|\vu|^2})}; \bfphi \right> }  \leq C h^{(1-\alpha)/2}, \\
\label{cf-e2}
& \abs{ \left< e_{\Grad\vt}; \bfphi \right> } \leq C h,  \
\quad
\hspace{0.6cm} \abs{ \left< e_{\Grad(\vt^2)}; \bfphi \right> } \leq C h,
\end{align}	
\end{subequations}
where the constant $C$ depends on 
$\norm{\bbT}_{L^2(T,T+1; W^{1,2}(\Omega; \R^{d \times d}_{\rm sym}))}$, $\norm{\bfphi}_{L^2(T,T+1; W^{1,2}(\Omega; \R^{d}))}$, 
$\|\vtB\|_{W^{2,\infty}(\Omega)}$ and $\Un{\vr}, \Ov{\vr}, \Un{\vt}, \Ov{\vt}, \Ov{u}$, but  it is independent of $T \in \R$ and discretization parameters $(h, \TS)$.

Further, it holds 
\begin{align}\label{regularity-sq}
\Gradd (\vth^2) \in L^2((T,T+1)\times \Omega;\R^d), \  \Gradh (\abs{\vuh}^2) \in L^2((T,T+1)\times \Omega;\R^d).
\end{align}
\end{Lemma}

\medskip
We proceed to show the consistency formulations, which is motivated 
by the concept of weak solution introduced in Definition \ref{Dw1}.

\begin{Lemma}[{\bf Consistency}]\label{lem_C}
Under Hypothesis {\bf (B)}, let $\vU_h \equiv (\vrh ,\vuh, \vth)$ be an entire numerical solution obtained by the FV method \eqref{VFV} with $(\TS, h) \in (0,1)^2$, $\TS \approx h$, and $\alpha \in \left(-1,1 \right)$. 
Denote
\begin{subequations}\label{cons}
\begin{align}\label{cons-1}
&\myangle{ e_{\vr}, \phi}  = 
\intRQB{ \vrh \partial_t \phi + \vrh \vuh \cdot \Grad \phi }, 
 \\& \label{cons-2}
 \myangle{ e_{\vm}, \bfphi}  =  
  \intRQB{ \vrh \vuh \cdot \partial_t \bfphi + \vrh \vuh \otimes \vuh : \Grad \bfphi } \br
&\hspace{4.5cm} - \intRQ{ ( \bS_h - p_h \I) : \Grad \bfphi } - \intRQ{\vrh \Grad G \cdot \bfphi} , 
\\ & \nonumber
\myangle{ e_{\vr s}, \varphi}  = 
\intRQB{ \vrh s_h (\pd_t\varphi + \vuh \cdot \Grad \varphi ) -  \frac{\kappa}{\vth} \Gradd \vth \cdot \Grad \varphi }
\\ & \label{cons-3}
\hspace{4.5cm} + \intRQ{ \frac{\varphi}{\vth} \left(\kappa \frac{ \chi_h}{ \vth} \abs{\Gradd \vth }^2 + \difuh  \right)}, 
\\ & \label{cons-4}
\myangle{ e_{B}, \hvt, \psi}  =    \intRQ{ \pd_t \psi \left(\frac{1}{2} \vrh |\vuh |^2 + c_v \vrh \vth - \vrh s_h \hvt \right)} 
+ \intTauTauQ{\psi \vrh \Grad G \cdot \vuh} 
\br&  
\hspace{1.8cm}- \intRQ{ \psi \bigg( \frac{ \kappa\hvt \chi_h}{\vth^2} \; \abs{\Gradd \vth}^2 + \frac{\hvt}{\vth }\difuh \bigg)}
\br &
\hspace{1.8cm}- \intRQ{\psi \bigg( \vrh s_h \partial_t \hvt + 
\vrh s_h \vuh \cdot \Grad \hvt - \frac{\kappa}{\vth} \; \Gradd \vth \cdot \Grad \hvt \bigg)} 
\end{align}
for  $\phi \in C^2_c((-\tau,\tau) \times \Ov{\Omega})$, $\bfphi \in C_c^2((-\tau,\tau) \times \Omega;\mathbb R^d)$, and for 
\begin{align*}
& \varphi \in  C_c^2((-\tau,\tau) \times \Omega;\mathbb R^d) \quad \varphi \geq 0; \\
& \psi \in  C^2_c(-\tau,\tau), \ \psi \geq 0; \quad \Theta \in BC^2(\R \times \Ov{\Omega}),\ \Theta > 0,\ \Theta|_{\partial \Omega} = \vtB. 
\end{align*}   
 
Then the consistency errors can be estimated as follows
\begin{equation}\label{con-e}
\begin{cases}
\ \abs{\myangle{e_{\vr}, \phi} } + \abs{ \myangle{e_{\vm}, \bfphi}  } \leq C\left( \TS +h+h^{(1-\alpha)/2} + h^{(1+\alpha)/2} \right),
\\
\ \myangle{e_{\vr s}, \varphi} = \myangle{e_{\vr s}^1, \varphi}  + \myangle{e_{\vr s}^2, \varphi} ,\quad 
\ \myangle{e_{B}, \hvt, \psi} = \myangle{e_{B}^1, \hvt, \psi}  + \myangle{e_{B}^2, \hvt, \psi} , 
\\
\ \abs{ \myangle{e_{\vr s}^1, \varphi}  } \leq C\left( \TS +h+h^{(1-\alpha)/2} + h^{(1+\alpha)/2} \right), \quad \myangle{e_{\vr s}^2, \varphi} \leq 0 \ \mbox{ for any }  \phi\geq 0,
\\
\ \abs{\myangle{e_{B}^1, \hvt, \psi} } \leq C\left( \TS +h + h^{(1-\alpha)/2}\right), \quad \myangle{e_{B}^2, \hvt, \psi} \geq 0 \ \mbox{ for any }  \hvt > 0, \psi \geq 0.
\end{cases}
\end{equation}
\end{subequations}
Here, the generic constant $C$ depends on 
$W^{2,\infty}((-\tau,\tau)\times \Omega )$-norm of the test functions $\phi, \varphi,  \psi, \Theta$, 
$\tau$, $\|\vtB\|_{W^{2,\infty}(\Omega)}$ and $\Un{\vr}, \Ov{\vr}, \Un{\vt}, \Ov{\vt}, \Ov{u}$, but it is independent of the discretization parameters $(h, \TS)$.
\end{Lemma}

\begin{proof}
The proof is analogous to \cite[Lemma A.7]{FeLMShYu:2024}. Compared with that,  the primary difference lies in two parts: 1) $\vU_h$ is now an entire numerical solution and the test function are compact supported in $(-\tau,\tau)$; and 2) the consistency formulation of ballistic energy inequality is slightly different from \cite[Lemma A.7]{FeLMShYu:2024}.  

Firstly, we show the consistency of continuity and momentum equations, i.e.\ \eqref{cons-1} and \eqref{cons-2}.  
Analogously to \cite[Remark C.2]{LMShYu:2025} or \cite[Lemma A.7]{FeLMShYu:2024}, we rewrite the consistency errors with
\[
\myangle{e_{\vr}, \phi} = \myangle{e_{\vr}, \phi} + \int_{-\tau}^{\tau} \eqref{VFV_D} \dt, \ \phi_h = \PiQ \phi; \quad \myangle{e_{\vm}, \bfphi} = \myangle{e_{\vm}, \bfphi} + \int_{-\tau}^{\tau} \eqref{VFV_M}  \dt, \ \bfphi_h = \PiQ \bfphi,
\]
and then decompose them as follows: 
\begin{align*}
&\myangle{e_{\vr}, \phi}  =E_t(\vrh,\phi) + E_F(\vrh,\phi), \\
&\myangle{e_{\vm}, \bfphi}  =E_t(\vm_h,\bfphi) + E_F(\vm_h,\bfphi) +E_{\vm,\bS}(\bfphi) +E_{\vm,p}(\bfphi),
\end{align*}
with
\begin{align*}
E_t(r_h,\phi)&= 
\intTauTauQ{  r_h \pd_t \phi } + \intTauTauQ{  D_t r_h \phi },\\
E_F(r_h,\phi) &= \intTauTauQ{r_h \vuh \cdot \Grad \phi} - \int_{-\tau}^{\tau}  \intfacesint{ \Fup (r_h,\vuh) \jump{ \phi_h}} \dt,
\br
E_{\vm,\bS}(\bfphi) & = - \intTauTauQ{ \bS_h : \Grad \bfphi} + \intTauTauQ{ \bS_h : \Gradh \bfphi_h }, 
\\ 
E_{\vm,p}(\bfphi) & =  \intTauTauQ{ p_h \Div \bfphi } - \intTauTauQ{ p_h \Divh \bfphi_h}.
\end{align*}
Following the analysis in \cite[Lemma A.7]{FeLMShYu:2024} or \cite[Appendix C]{BLMSY}, we have 
\begin{align*}
\abs{E_F(\vrh,\phi)+E_F(\vm_h,\bfphi) +E_{\vm,\bS}(\bfphi) +E_{\vm,p}(\bfphi)} \aleq h^{(1-\alpha)/2} + h^{(1+\alpha)/2}.
\end{align*}

\noindent Fr $E_t(r_h,\phi)$ it holds
\begin{align*}
E_t(r_h,\phi)&= 
\intTauTauQ{  r_h \pd_t \phi } + \intTauTauQ{  \pd_t \widetilde{r_h} \phi } =  \intTauTauQB{ r_h \pd_t \phi -   \widetilde{r_h} \pd_t \phi } \\
&  =  \intTauTauQ{ \big( r_h -  \widetilde{r_h} \big) \pd_t \phi }.
\end{align*}
Letting $- \tau \in (t_k, t_{k+1}], \tau \in (t_n, t_{n+1}]$,  we decompose $E_t(r_h,\phi)$ into three parts:
\begin{align*}
E_t(r_h,\phi) & = \int_{-\tau}^{t_{k+1}}\intO{ \big( r_h -  \widetilde{r_h} \big) \pd_t \phi } \dt + \int_{t_{k+1}}^{t_{n}}\intO{ \big( r_h -  \widetilde{r_h} \big) \pd_t \phi } \dt +  \int^{\tau}_{t_{n}}\intO{ \big( r_h -  \widetilde{r_h} \big) \pd_t \phi } \dt. 
\end{align*}
Thanks to the interpolations \eqref{time_extension}, applying H\"older inequality we obtain
\begin{align*}
&\abs{\int_{-\tau}^{t_{k+1}}\intO{ \big( r_h -  \widetilde{r_h} \big) \pd_t \phi } \dt}   \aleq \TS \norm{\pd_t^2\phi}_{L^{\infty}((-\tau,t_{k+1})\times \Omega)}\int_{-\tau}^{t_{k+1}}\intO{  (t_{k+1} - t) \abs{D_t r_h} }\dt \\
&\hspace{2cm} \aleq \TS \left( \int_{-\tau}^{t_{k+1}}\intO{   (t_{k+1} - t) ^2 }\dt \right)^{1/2} \left( \int_{-\tau}^{t_{k+1}}\intO{  \abs{D_t r_h} ^2 }\dt \right)^{1/2} \aleq \TS^2,\\
& \abs{\int^{\tau}_{t_{n}}\intO{ \big( r_h -  \widetilde{r_h} \big) \pd_t \phi } \dt} \aleq \TS^2.
\end{align*}
Further, we reformulate 
\begin{align*}
& \int_{t_{k+1}}^{t_{n}} \big( r_h -  \widetilde{r_h} \big) \pd_t \phi  \dt = \sum_{m=k+1}^{n-1}D_t r_h^{m+1} \int_{t_m}^{t_{m+1}}  (t - t_{m+1})    \pd_t \phi  \dt \\
&=\frac1{\TS}\sum_{m=k+1}^{n-1}\left( r_h^{m+1} - r_h^m\right)\int_{0}^{\TS}  (t - \TS)    \pd_t \phi (t+t_m) \dt \\
&=\frac1{\TS}\left( \sum_{m=k+1}^{n-1} r_h^{m+1} \int_{0}^{\TS}  (t - \TS)    \pd_t \phi (t+t_m) \dt - \sum_{m=k+1}^{n-1} r_h^m\int_{0}^{\TS}  (t - \TS)    \pd_t \phi (t+t_m) \dt \right)\\
&=\frac1{\TS}\left( \sum_{m=k+2}^{n} r_h^{m} \int_{0}^{\TS}  (t - \TS)    \pd_t \phi (t+t_{m-1}) \dt - \sum_{m=k+1}^{n-1} r_h^m\int_{0}^{\TS}  (t - \TS)    \pd_t \phi (t+t_m) \dt \right)\\
&=\frac1{\TS}\sum_{m=k+2}^{n-1} r_h^{m} \int_{0}^{\TS}  (t - \TS)    \left[ \pd_t \phi (t+t_{m-1}) - \pd_t \phi (t+t_m)\right] \dt  \\
&\quad +  \frac{r_h^{n}}{\TS}\int_{0}^{\TS}  (t - \TS)    \pd_t \phi (t+t_n) \dt -  \frac{r_h^{k+1}}{\TS}  \int_{0}^{\TS}  (t - \TS)    \pd_t \phi (t+t_{k+1}) \dt. 
\end{align*}
Since
\begin{align*}
&\int_{0}^{\TS}  (t - \TS)   f(t) \dt = \int_{0}^{\TS}  (t - \TS)    \dt  \cdot f(\xi) \quad \mbox{with} \quad \xi \in [0,\TS], \ f \in C(\R)
\end{align*}
we have
\begin{equation}\label{time}
\abs{ \int_{t_{k+1}}^{t_{n}}\intO{ \big( r_h -  \widetilde{r_h} \big) \pd_t \phi } \dt}\aleq \TS
\end{equation}
and finish the proof of the consistency of continuity and momentum equations stated in \eqref{cons-1} and \eqref{cons-2}.

 \medskip
 Secondly, we show the consistency of entropy inequality \eqref{cons-3}. Recalling the entropy balance \cite[(A.7)]{FeLMShYu:2024}, i.e.
 \begin{multline}\label{entbal}
\intQ{ D_t \left(\vrh s_h \right) \varphi_h } -
\intfacesint{ \Up(\vrh s_h , \vuh ) \jump{\varphi_h} }
- \intQ{ \frac{\varphi_h}{\vth } \difuh }  + \intfacesint{ \frac{\kappa}{h} \jump{\vth } \jump{\frac{\varphi_h}{\vth } } }
\\
+ 2 \frac{\kappa}{ h }  \intfacesext{ \frac{ \vt_h^{\rm in} - \vthB }{\vt_h^{\rm in}}  \varphi_h^{\rm in} }
= D_s(\varphi_h) + R_{s}(\varphi_h), \quad 
\mbox{with }D_{s}(\varphi) \geq 0 \mbox{ for any } \varphi \geq 0,
\end{multline}
 we reformulate the consistency error as
 \[
\myangle{e_{\vr s}, \varphi} = \myangle{e_{\vr s}, \varphi} + \int_{-\tau}^{\tau}  \eqref{entbal} \dt =  \myangle{e_{\vr s}^1, \varphi}  + \myangle{e_{\vr s}^2, \varphi} 
\]
with 
 \begin{align*}
&-\myangle{e_{\vr s}^1, \varphi}  =-E_t(\vrh s_h,\varphi) + E_{s,F}(\vrh s_h,\varphi) + E_{s,\Grad\vt}(\varphi)+ E_{s,res}(\varphi) , 
\
-\myangle{e_{\vr s}^2, \varphi}  = \int_{-\tau}^{\tau} D_{s}(\varphi_h) \dt  \geq 0,  
\end{align*}
where
\begin{align*}
E_{s,F}(r_h,\varphi) &= - \intTauTauQ{r_h \vuh \cdot \Grad \varphi} + \int_{-\tau}^{\tau} \intfacesint{ \Up [r_h, \vuh] \jump{ \varphi_h}} \dt,
\\ 
E_{s,\Grad\vt}(\varphi) &= 
-\intTauTauQ{ \frac{\kappa}{\vth} \Gradd \varphi_h  \cdot \Gradd \vth } + \intTauTauQ{  \frac{\kappa}{\vth} \Gradd \vth \cdot \Grad \varphi }
\br
&\quad
 - \intTauTauQ{ \varphi_h \Gradd \vth \cdot \Gradd \left(\frac1{\vth}\right) }
- \intTauTauQ{ \frac{ \kappa \varphi \chi_h}{ \vth^2} \abs{\Gradd \vth }^2},
\\
E_{s,res}(\varphi) & = \int_{-\tau}^{\tau} R_{s}(\varphi_h) \dt,
\end{align*}
with
\begin{align*}
\varphi_h = \begin{cases}
0 & \mbox{if } x \notin \Omega, \\
0 & \mbox{if } x \in K \subset \Omega, K \cap \facesext \neq \emptyset, \\
\PiQ \varphi & \mbox{otherwise}.
\end{cases}\end{align*}
Following the analysis in \cite[Lemma A.7]{FeLMShYu:2024}, we have 
\begin{align*}
\abs{E_{s,F}(\vrh s_h,\varphi) + E_{s,\Grad\vt}(\varphi)+ E_{s,res}(\varphi)} \aleq h^{(1-\alpha)/2} + h^{(1+\alpha)/2}.
\end{align*}
Thanks to
\begin{align*}
&\abs{E_t(\vrh s_h,\varphi)} \aleq \TS, \quad \mbox{see \eqref{time},}
\end{align*}
we finish the proof of the consistency formulation of entropy inequality stated in \eqref{cons-3}.

Finally, we prove the ballistic energy consistency \eqref{cons-4}. Recall the Ballistic energy balance in \cite[(A.9)]{FeLMShYu:2024}, i.e.\
\begin{align}\label{BalBal}
&D_t \intQ{ \left(\frac{1}{2} \vrh |\vuh |^2 + c_v \vrh \vth - \vrh s_h \phi_h \right) }
+ \intQ{ \frac{\phi_h}{\vth }\difuh } - \intfacesint{ \frac{\kappa}{h} \avs{\phi_h} \jump{\vth } \jump{\frac{1}{\vth } } }
\br
& \hspace{1cm}
+2 \frac{\kappa}{ h }  \intfacesext{ \frac{ (\vt_h^{\rm in} - \vthB)^2 }{\vt_h^{\rm in}}   }
- \intQ{\vrh \Grad G \cdot  \vuh}
+ D_s(\phi_h) + D_{\rm E}
\br
&=
-\intQ{ \vrh s_h ( D_t \phi_h + \vuh \cdot \Gradh \phi_h) }
+ \intfacesint{ \frac{\kappa}{h} \jump{\vth } \jump{\phi_h} \avs{ \frac{1}{\vth } } }
+ R_{B} (\phi_h) - R_{s}(\phi_h),
\end{align}
with $D_{s}(\varphi) \geq 0$ for any $ \varphi \geq 0$ and $D_E \geq 0$.
Then we rewrite the consistency error as 
\begin{align*}
\myangle{e_{B}, \hvt, \psi} = \myangle{e_{B}, \hvt, \psi} + \int_{-\tau}^{\tau} \eqref{BalBal} \cdot \psi \dt,
\end{align*}
with the test function in \eqref{BalBal} as 
\begin{equation} \label{vt-test-1}
\phi_h =
\hvth(x) =  \begin{cases}
\vthB & \mbox{if } x \in K \subset Q, \facesK \cap \facesext \neq \emptyset, \\
\Pim \hvt & \mbox{otherwise},  \\
\end{cases}
\ \mbox{ and } \ \avs{\hvth}_{\sigma \in \facesext} = \vthB.
\end{equation}
Analogous to the decomposition in the consistency of entropy inequality, we have
\begin{align*}
\myangle{e_{B}, \hvt, \psi} & = \myangle{e_{B}^1, \hvt, \psi} + \myangle{e_{B}^2, \hvt, \psi},
\\
-\myangle{e_{B}^1, \hvt, \psi}  & = -E_t(E_{B,h}, \psi) + E_{B, \vt} + E_{B, res} + \int_{-\tau}^{\tau}  \psi \big(R_B(\hvth)- R_s(\hvth)\big) \dt,
\br
\myangle{e_{B}^2, \hvt, \psi}   & = \int_{-\tau}^{\tau} \psi\bigg(D_s(\hvth)+D_E\bigg)\, \dt   \geq 0,
\end{align*}
where $E_{B,h}= \frac{1}{2} \vrh |\vuh |^2 + c_v \vrh \vth - \vrh s_h \hvth$ and
\begin{align*}
E_{B, \vt} &= \intTauTauQ{  \vrh s_h (\hvt - \hvth) \pd_t \psi}  + \intTauTauQ{ \psi(\hvt-\hvth)\left[   \frac{ \difuh}{\vth} +  \kappa \frac{\chi_h}{\vth^2}\abs{\Gradd \vth}^2   \right]} ,
\end{align*}
and
\begin{align*}
E_{B, res} & = \intTauTauQ{\psi \left(  \vrh s_h \partial_t \hvt + \vrh s_h \vuh \cdot \Grad \hvt - \kappa \frac1{\vth} \; \Gradd \vth \cdot \Grad \hvt \right)}
\br
&
-\intTauTauQ{ \psi \vrh s_h  \bigg(  D_t \hvth + \vuh \cdot \Gradh \hvth  \bigg)}
+  \int_{-\tau}^{\tau}  \intfacesint{ \psi \frac{\kappa}{h} \jump{\vth } \jump{\hvth} \avs{ \frac{1}{\vth } } } \dt.
\end{align*}

 Following the analysis in \cite[Lemma A.7]{FeLMShYu:2024}, we have 
\begin{align*}
\abs{E_{B, \vt} + E_{B, res} + \int_{-\tau}^{\tau} \psi \big(R_B(\hvth)- R_s(\hvth)\big) \dt } \aleq h^{(1-\alpha)/2} .
\end{align*}
Thanks to
\begin{align*}
&\abs{E_t(E_{B,h},\varphi)} \aleq \TS, \quad \mbox{see \eqref{time},}
\end{align*}
we finish the proof of the consistency formulations.
\end{proof}

\begin{Remark}\label{rmk-b}
By Hypothesis {\bf (B)}  and the uniform bounds \eqref{ap},
we obtain, as in \cite{FeLMShYu:2025I}, that
\begin{align}
&
\intRQB{ \partial_t \widetilde{\vrh}  \phi - \vrh \vuh \cdot \Grad \phi } 
=\myangle{h_{\vr}; \phi},
 \\&
  \intRQB{ \partial_t \widetilde{\vrh \vuh} \cdot  \bfphi - \vrh \vuh \otimes \vuh : \Grad \bfphi } \br
&\hspace{2.5cm} + \intRQ{ ( \bS_h - p_h \I) : \Grad \bfphi } + \intRQ{\vrh \Grad G \cdot \bfphi}  = \myangle{h_{\vm};\bfphi},
\\&
 \intRQB{  c_v \pd_t \widetilde{\vrh \vth}  \psi  -   (c_v \vrh \vth  \vuh - \kappa  \Gradd \vth)\cdot \Grad \psi } \br
 &\hspace{6.cm}
 -\intRQB{ (\bS_h-p_h \I)  : \Gradh \vuh \psi } = \myangle{h_{\vt}; \psi} 
\end{align}
for $\phi \in C_c^1( \R \times \Ov{\Omega}) $, $\bfphi \in C_c^1(\R \times \Ov{\Omega};\R^d) $, $\psi \in C_c^1( \R \times \Ov{\Omega})$, where
%
\begin{align*}
 h_\vr \to 0 \mbox{ in } L^2_{\rm loc}(\R; [W^{1,2}(\Omega)]') \quad \mbox{ as } h \to 0,
\\h_{\vm} \mbox{ bounded in } L^{2}_{\rm loc}(\R;[W^{1,2}(\Omega;\R^d)]') \mbox{ uniformly for } h \to 0,
\\h_{\vt} \mbox{ bounded in } L^{2}_{\rm loc}(\R;[W^{1,2}(\Omega)]')  \mbox{ uniformly for } h \to 0.
\end{align*}
As a byproduct, we also have
\begin{align*}
& \pd_t \widetilde{\vrh} \in L^{2}_{\rm loc}(\R;[W^{1,2}(\Omega)]'), \ 
\pd_t \widetilde{\vrh \vuh} \in L^{2}_{\rm loc}(\R;[W^{1,2}(\Omega;\R^d)]'), \
\pd_t \widetilde{\vrh \vth} \in L^{2}_{\rm loc}(\R;[W^{1,2}(\Omega)]').
\end{align*}
\end{Remark}

\section{Proof of Theorem \ref{CT1}}
\label{P}	

Summarizing the results obtained in Section \ref{stacon}, we have 
\begin{itemize}
\item $\{\vU_{h}\}_{h > 0}$ admits the uniform bounds \eqref{ap}, satisfies the compatibility formulation \eqref{cf}, and the consistency formulation \eqref{cons} for any $\tau \in (0,\infty)$.

\item There is a sequence $\{\vU_{h_n}\}$ such that
\begin{align}\label{weak-con-1}
&\vU_{h_n}  \to \vU \ \mbox{ weakly in } L_{\rm loc}^q(\R\times \Omega; \R^{d+2}) \ \mbox{for any}\ 1 \leq q  < \infty
\end{align}
and 
\begin{align}\label{weak-con-3}
&\Gradh \vu_{h_n}  \to \Grad \vu \ \mbox{ weakly in } L^2_{\rm loc}(\R\times \Omega; \R^{d\times d}), \\\label{weak-con-4}
&\Gradd \vt_{h_n}  \to \Grad \vt \ \mbox{ weakly in } L^2_{\rm loc}(\R\times \Omega; \R^{d}).
\end{align}
Moreover, thanks to the inequalities 
\begin{align*}
\abs{\widetilde{\vU}_{h_n} - \vU_{h_n}} \leq \TS \abs{D_t  \vU_{h_n}}, \quad 
\norm{D_t  \vU_{h_n}}_{L^2((-\tau,\tau)\times \Omega;\R^{d+2})} \aleq \TS^{-1/2} \mbox{ for any } \tau >0,
\end{align*}
we have 
\begin{equation} \label{weak-con-2}
\widetilde{\vU}_{h_n}  \to \vU \ \mbox{ weakly in } L_{\rm loc}^q(\R\times \Omega; \R^{d+2}) \ \mbox{for any}\ 1 \leq q  < \infty.
\end{equation}
\end{itemize}
It remains to show that the above weak convergences \eqref{weak-con-1} and \eqref{weak-con-2} are strong and the limit $\vU = (\vr, \vu, \vt)$ is an entire weak solution. 
In what follows we write $\widetilde{\vU}_{h_n} $ (resp.\ $\vU_{h_n}$) as $\widetilde{\vU}_h$ (resp.\ $\vU_h$), for the sake of simplicity. 
 
\medskip

To begin, we recall a useful proposition in \cite{FeLMShYu:2025I}. 
\begin{Proposition}[{\cite[Proposition A.2]{FeLMShYu:2025I}}]\label{col-div-curl}
Let $\{r_n, v_n\}_{n=1}^{\infty}$ satisfy
\begin{align*}
r_n \to r \text{ weakly in } L^2((0,T)\times \Omega), \quad
v_n \to v \text{ weakly in } L^q((0,T)\times \Omega), 
\end{align*}
where $q>2$, 
and 
\begin{align} 
&\partial_t r_n =  h_n^1 + h_n^2,  \quad
 h_n^1\in_b L^1((0,T)\times \Omega), \quad h_n^2 \in L^2(0,T;W^{-1,2}(\Omega)), \\ 
&\Grad v_n = D_n^1 + D_n^2, \quad   
D_n^1 \in_b L^1((0,T)\times \Omega), \quad
{D_n^2 \to 0 \mbox{ in } W^{-1,2}((0,T) \times \Omega; \R^d)).  } 
\end{align}

Then it holds 
\[
r_h v_h \to r v \text{ weakly in } L^{\frac{2q}{2+q}}((0,T)\times \Omega).
\]
\end{Proposition}

\subsection{Step 1}
{\bf We show the strong convergence of the velocity 
\begin{align}\label{con-u}
\vuh \to \vu \mbox{ in } \ L^q_{\rm loc}(\R\times\Omega; \R^d) \ \mbox{for any}\ 1 \leq q  < \infty,
\end{align}
together with the continuity equation 
\begin{align}\label{wsol-d}
 \pd_t \vr + \Div (\vr \vu) = 0 
\end{align}
satisfied by the limits $\vr, \vu$ in the sense of Definition \ref{Dw1}.}

On the one hand, there hold the weak convergences of $\vrh$ and $\vuh$, see \eqref{weak-con-1}. 
On the other hand, there holds the consistency equation of continuity (see Remark~\ref{rmk-b}) 
\begin{equation} \label{divcurl-d}
 \pd_t \widetilde{\vrh} + \Div (\vrh \vuh)  = h_\vr \  \mbox{ with } \  h_\vr \to 0 \mbox{ in } L^2_{\rm loc}(\R; [W^{1,2}(\Omega)]'),
\end{equation}
together with the estimate $\norm{\Gradh\vuh}_{L^2((-\tau,\tau)\times \Omega; \R^{d\times d})} \aleq 1$ and the compatibility \eqref{cf1}. 
We apply \cite[Lemma 8.1]{AbFeNo:2021} or Proposition~\ref{col-div-curl} with $r_h = \widetilde{\vrh}, \, v_n = \vuh^{(j)}$ to conclude
\begin{align}\label{68}
\widetilde{\vrh} \vuh  \to \vr \vu  \ \mbox{ weakly in } L_{\rm loc}^q(\R\times \Omega; \R^d) \ \mbox{for any}\ 1 \leq q  < \infty.
\end{align}
Further, combining \eqref{68} with
\begin{align*}
\abs{\widetilde{\vrh} \vuh - \vrh \vuh} \aleq \TS \abs{D_t  \vrh}, \quad 
\norm{D_t  \vrh}_{L^2((-\tau,\tau)\times \Omega)} \aleq \TS^{-1/2},\\
\abs{\widetilde{\vrh \vuh} - \vrh \vuh} \aleq \TS (\abs{D_t  \vrh}+\abs{D_t  \vuh}), \quad 
\norm{D_t  \vuh}_{L^2((-\tau,\tau)\times \Omega;\R^{d})} \aleq \TS^{-1/2},
\end{align*}
we obtain
\begin{align}\label{con-m}
\vrh \vuh  \to \vr \vu  \ \mbox{ weakly in } L^q_{\rm loc}(\R\times \Omega; \R^d) \ \mbox{for any}\ 1 \leq q  < \infty,\\ \label{con-m-1}
\widetilde{\vrh \vuh}  \to \vr \vu  \ \mbox{ weakly in } L^q_{\rm loc}(\R\times \Omega; \R^d) \ \mbox{for any}\ 1 \leq q  < \infty.
\end{align}
Passing to the limit $h \to 0$ in the consistency equation of continuity \eqref{cons-1}, we conclude that the limits $\vr, \vu$ satisfy the continuity equation \eqref{wsol-d}.  
Moreover, as $\vr$ and $\vu$ are bounded, we can apply the DiPerna-Lions theory~\cite{DiPena-Lions} to deduce the renormalized formulation \eqref{w4}.

\medskip
Next, as stated in Remark~\ref{rmk-b},  $\widetilde{\vrh \vuh}$ satisfies  
\begin{align*}
& \pd_t \widetilde{\vrh \vuh} + \Div (\vrh \vuh \otimes \vuh - \bS_h + p_h \bI)  = \vrh \Grad G + h_{\vm}, \\
& \mbox{ with } \vrh \Grad G \mbox{ bounded in } L^\infty(\R\times\Omega), \quad
\abs{\myangle{ h_{\vm},\bfphi}} \aleq \norm{\bfphi}_{L^2(-\tau,\tau;W^{1,2}(\Omega;\R^d))}.
\end{align*}
Analogously to the proof of \eqref{con-m}, 
we apply Proposition~\ref{col-div-curl} with $r_h = \widetilde{\vrh \vuh^{(i)}}, \, v_h = \vuh^{(j)}$ to obtain 
\begin{align}
&\widetilde{\vrh \vuh} \cdot \vuh \to \vr \abs{\vu}^2  \ \mbox{ weakly in } L^q_{\rm loc}(\R\times \Omega) \ \mbox{for any}\ 1 \leq q  < \infty,\\ \label{con-mmd}
& \widetilde{\vrh \vuh} \otimes \vuh \to \vr \vu \otimes \vu  \ \mbox{ weakly in } L^q_{\rm loc}(\R\times \Omega; \R^{d\times d}) \ \mbox{for any}\ 1 \leq q  < \infty,
\end{align}
which gives
\begin{align}\label{con-mmd}
&\vrh \vuh \cdot \vuh \to \vr \abs{\vu}^2  \ \mbox{ weakly in } L^q_{\rm loc}(\R\times \Omega) \ \mbox{for any}\ 1 \leq q  < \infty,\\ 
& \vrh \vuh \otimes \vuh \to \vr \vu \otimes \vu  \ \mbox{ weakly in } L^q_{\rm loc}(\R\times \Omega; \R^{d\times d}) \ \mbox{for any}\ 1 \leq q  < \infty.
\end{align}

Further, we know from Lemma \ref{compatibility-vel} that $\Gradh (\abs{\vuh}^2)$ satisfies the compatibility equation \eqref{cf2} and belongs to the regularity class  $L^2_{\rm loc}(\R\times \Omega; \R^d)$.
Hence, we apply Proposition~\ref{col-div-curl} again with $r_n = \vr - \widetilde{\vrh}, \, v_n = \abs{\vuh}^2$ and obtain 
\begin{align}
(\vr - \widetilde{\vrh}) \abs{\vuh}^2 \to 0 \mbox{ weakly in } \ L^q_{\rm loc}(\R\times\Omega) \ \mbox{for any}\ 1 \leq q  < \infty,\\ \label{con-mmd-1}
(\vr - \vrh) \abs{\vuh}^2 \to 0 \mbox{ weakly in } \ L^q_{\rm loc}(\R\times\Omega) \ \mbox{for any}\ 1 \leq q  < \infty. 
\end{align}

Consequently, we conclude from \eqref{con-mmd} and \eqref{con-mmd-1} that 
\begin{align*}
\vr  \abs{\vuh}^2 \to \vr  \abs{\vu}^2 \mbox{ weakly in } \ L^q_{\rm loc}(\R\times\Omega) \ \mbox{for any}\ 1 \leq q  < \infty. 
\end{align*}
As $\vr$ is bounded and strictly positive, and $\vuh$ converges weakly to $\vu$, we establish the strong convergence of the velocity stated in \eqref{con-u}.

\subsection{Step 2}
{\bf We show the strong convergence of the temperature 
\begin{align}\label{con-T}
\vth \to \vt \mbox{ in } \ L^q_{\rm loc}(\R\times\Omega) \ \mbox{for any}\ 1 \leq q  < \infty
\end{align}
together with the momentum equation
\begin{align} \label{wsol-m}
 \pd_t(\vr \vu) + \Div(\vr\vu \otimes \vu) + \Grad p = \Div \bS(\Grad \vu) + \vr \Grad G 
\end{align}
satisfied by the  limits $\vr, \vu,\vt$ in the sense of Definition \ref{Dw1}.
}


Thanks to \eqref{weak-con-1}, \eqref{divcurl-d}, \eqref{ap1} and \eqref{cf3}, we apply Proposition~\ref{col-div-curl}  with $r_n = \widetilde{\vrh}$ and $v_n = \vth$ to conclude 
\begin{align*}
\widetilde{\vrh} \vth  \to \vr \vt \ \mbox{ weakly in } L^q_{\rm loc}(\R\times \Omega) \ \mbox{for any}\ 1 \leq q  < \infty.	
\end{align*}
Together with
\begin{align*}
\abs{\widetilde{\vrh} \vth - \vrh \vth} \aleq \TS \abs{D_t  \vrh}, \quad 
\norm{D_t  \vrh}_{L^2((-\tau,\tau)\times \Omega)} \aleq \TS^{-1/2}
\end{align*}
we have
\begin{align}\label{con-p}
\vrh \vth  \to \vr \vt \ \mbox{ weakly in } L^q_{\rm loc}(\R\times \Omega) \ \mbox{for any}\ 1 \leq q  < \infty.	
\end{align}
Passing to the limit $h\to 0$ in the consistency equation of momentum \eqref{cons-2}, we obtain from \eqref{con-m}, \eqref{con-mmd} and \eqref{con-p} that the limits $\vr, \vu, \vt$ satisfy the momentum equation \eqref{wsol-m}.

\medskip
Next, we know from Remark \ref{rmk-b} that 
\begin{align*}
&  \pd_t \widetilde{\vrh \vth} + \Div (c_v \vrh \vth  \vuh - \kappa  \Gradd \vth)  = (\bS_h-p_h \I)  : \Gradh \vuh  + h_{\vt},
\end{align*}
with
\begin{align*}
(\bS_h-p_h \I)  : \Gradh \vuh  \mbox{ bounded in } L^1_{\rm loc}(\R\times\Omega),\quad
\abs{\myangle{ h_{\vt},\phi}} \aleq \norm{\phi}_{L^2((-\tau,\tau);W^{1,2}(\Omega))}.
\end{align*}
Applying Proposition~\ref{col-div-curl}  with $r_n = \widetilde{\vrh \vth}$ and $v_n = \vth$, we obtain
\begin{align*}
&\widetilde{\vrh \vth} \cdot \vth \to \vr \vt^2  \ \mbox{ weakly in } L^q_{\rm loc}(\R\times \Omega) \ \mbox{for any}\ 1 \leq q  < \infty, \\
&\vrh \vth \cdot \vth \to \vr \vt^2 \ \mbox{ weakly in } L^q_{\rm loc}(\R\times \Omega) \ \mbox{for any}\ 1 \leq q  < \infty.
\end{align*}

Further, thanks to the compatibility results  \eqref{regularity-sq} and \eqref{cf4},  
we apply Proposition~\ref{col-div-curl} once again with $r_n = \vr - \widetilde{\vrh}, \, v_n = \vth^2$ and obtain 
\begin{align*}
&(\vr - \widetilde{\vrh}) \vth^2 \to 0 \mbox{ weakly in } \ L^q_{\rm loc}(\R\times\Omega) \ \mbox{for any}\ 1 \leq q  < \infty, \\
&(\vr - \vrh) \vth^2 \to 0 \mbox{ weakly in } \ L^q_{\rm loc}(\R\times\Omega) \ \mbox{for any}\ 1 \leq q  < \infty.
\end{align*}

Therefore, 
\begin{align*}
\vr  \vth^2 \to \vr  \vt^2 \mbox{ weakly in } \ L^q_{\rm loc}(\R\times\Omega)\ \mbox{for any}\ 1 \leq q  < \infty.
\end{align*}
Similarly to Step 1, we obtain the strong convergence claimed in \eqref{con-T}. 

\subsection{Step 3}
{\bf We show the strong convergence of the density 
\begin{align}\label{con-d}
\vrh \to \vr \mbox{ in } \ L^q_{\rm loc}(\R\times\Omega) \ \mbox{for any}\ 1 \leq q  < \infty
\end{align}
together with  the entropy and ballistic energy inequalities
\begin{align}\label{wsol-S}
& \pd_t (\vr s) +  \Grad (\vr s \vu ) - \Div \left(\frac{\kappa \Grad \vt}{\vt} \right)
\geq \frac{1}{\vt} \left(\bS: \Grad \vu + \kappa \frac{|\Grad \vt|^2}{\vt}  \right), 
\\ \label{wsol-BE}
&  \pd_t \intOB{\frac{1}{2} \vr |\vu|^2 + \vr e-\vr s \hvt}+ \intO{  \frac{1}{\vt} \left(\bS: \Grad \vu + \kappa \frac{ \hvt |\Grad \vt|^2}{\vt}  \right)}  - \intO{\vr \vu \Grad G }\br
& \hspace{1cm} +  \intOB{ \vr s  \pd_t \hvt + \vr s \vu \Grad \hvt- \frac{\kappa}{\vt} \Grad \vt \cdot \Grad \hvt}\leq 0  
\end{align}
satisfied by the  limits $\vr, \vu,\vt$ in the sense of Definition \ref{Dw1}.
}

The strong convergence claimed in \eqref{con-d} is the most delicate part of the proof. 
The key ingredient is the discrete version of Lions identity stated below.

\begin{Lemma}\label{lem}
Let $\vU_h = (\vrh ,\vuh, \vth)$ be a numerical solution obtained by the FV method \eqref{VFV} with $\alpha \in (-1,1)$.
Let Hypothesis {\bf (B)} hold (with the same bounds $\underline{\vr}$, $\Ov{\vr}$, $\Ov{u}$, $\underline{\vt}$, $\Ov{\vt}$) uniformly for $h \to 0$.
Let $\vU = (\vr ,\vu, \vt)$ be the limit obtained in \eqref{weak-con-1}. 
 
Then it holds 
\begin{align}\label{lions-tool}
\lim_{h\to 0} \int_{-\infty}^{\infty}\intO{\phi \psi \vrh \Big( (2\mu+\lambda) \Divh \vuh - \vrh \vth\Big)} \dt =  \int_{-\infty}^{\infty}\intO{\phi \psi \vr \Big( (2\mu+\lambda)  \Div \vu - \vr  \vt\Big)} \dt
\end{align}
for any $ \psi \in C_c^{\infty}(\R\times \Omega)$ and $ \phi \in C_c^{\infty}(\Omega)$. 
\end{Lemma}
In the context of numerical analysis, this result was first proved by Karper \cite{Karper} for a mixed finite volume -- finite element method applied to the barotropic Navier--Stokes system. The proof for the present finite volume scheme requires a different technique, and was carried out in  \cite{FeLMShYu:2025I}.

The second tool is a discrete version of the renormalized continuity equation proved in \cite[Lemma 8.3]{FeLMMiSh} or \cite[Lemma A.1]{LMShYu:2025}. 
\begin{Lemma}[Renormalized continuity equation
]\label{lem_r2}
Let $(\vrh ,\vuh)$ satisfy \eqref{VFV_D}. Then for any $\phi_h \in Q_h$ and any function $b\in C^1(\R)$
we have
\begin{align} \label{renormalized_density}
& \intO{ {\rm D}_t b(\vrh) \phi_h } - \intfacesint{ \Up[b(\vrh), \vuh] \cdot \jump{\phi_h} } + \intO{ \phi_h \; (\vrh b'(\vrh) - b(\vrh) ) \; \Divh \vuh }
\nonumber \\
& = - \frac{1}{\Delta t}\intO{\phi_h \Eb{\vrh^\triangleleft |\vrh} } - h^\alpha \intfacesint{ \jump{\vrh} \jump{b'(\vrh) \phi_h}}
-\intfacesint{ |\avs{\vuh } \cdot \vn | \phi_h^{\rm down} \Eb{\vrh^{\rm up}|\vrh^{\rm down}} },
\end{align}
where $E_f(v_1|v_2) = f(v_1) - f'(v_2)(v_1-v_2) - f(v_2),  f \in C^1(\R)$.
\end{Lemma}

\medskip
We are now ready to show the strong convergence of density. 
Taking $b = \vr \log(\vr)$ and $\phi_h \equiv 1$ in \eqref{renormalized_density} we obtain 
\begin{align*}
& 
-\intRQB{\widetilde{\vrh \log( \vrh)} \pd_t \psi - \vrh \Divh \vuh \psi} 
\\&=  \intRQB{\pd_t \widetilde{\vrh \log( \vrh)}  \psi + \vrh \Divh \vuh \psi} 
\\& 
= 
\intRQ{ \left( {\rm D}_t (\vrh \log(\vrh))   + \vrh \Divh \vuh  \right) \psi } \leq 0
\end{align*}
for any $\psi = \psi(t) \in C_c^\infty(\R),  \psi \geq 0$.

Passing to the limit for $h\to 0$ we obtain
\begin{align*}
& 
-\intRQB{\Ov{\vr \log( \vr)} \pd_t \psi - \Ov{\vr \Div \vu} \psi} 
\leq 0 \mbox{ for any } \psi \in C_c^\infty(\R),  \psi \geq 0,
\end{align*}
where the weak limit  $\Ov{\vr \log( \vr)}$ of $\vrh \log(\vrh)$ coincides with the weak limit of $\widetilde{\vrh \log(\vrh)}$ because of
\begin{align*}
\abs{\widetilde{\vrh \log(\vrh)} - \vrh \log(\vrh)} \aleq \TS \abs{D_t  \vrh}, \quad 
\norm{D_t  \vrh}_{L^2((-\tau,\tau)\times \Omega)} \aleq \TS^{-1/2}.
\end{align*}
On the other hand, as the limit satisfy the renormalized equation, we have 
\begin{align*}
& 
-\intRQB{ {\vr \log( \vr)} \pd_t \psi - {\vr \Div \vu} \psi} 
= 0 \mbox{ for any } \psi \in C_c^\infty(\R),
\end{align*}
whence, 
\begin{align*}
& 
\intRQB{\left(\vr \log( \vr)-\Ov{\vr \log( \vr) } \right) \pd_t \psi - \left(\vr \Div \vu - \Ov{\vr \Div \vu}\right) \psi} 
\leq 0 \mbox{ for any } \psi \in C_c^\infty(\R),  \psi \geq 0.
\end{align*}

Next, using the discrete Lions identity \eqref{lions-tool}, we obtain
\begin{align}
\vr \Div \vu - \Ov{\vr \Div \vu} = \frac{ \vt  }{2\mu+\lambda} \left(\vr^2  - \Ov{\vr^2} \right).
\end{align}
Thanks to Hypothesis {\bf (B)} and strict convexity of $\vr \log(\vr)$ we get
\[
0 \leq \Ov{\vr \log(\vr)} - \vr \log\vr \aleq \Ov{\vr^2}  - \vr^2.
\]
Consequently, there is a constant $C>0$ such that
\begin{align*}
& 
\intRQ{-\left(\Ov{\vr \log( \vr) }-\vr \log( \vr)  \right) \pd_t \psi + C \left(  \Ov{\vr \log(\vr)} - \vr \log\vr\right) \psi} 
\leq 0 \mbox{ for any } \psi \in C_c^\infty(\R),  \psi \geq 0.
\end{align*}
This implies 
\begin{align}
0 \leq \intO{ (\Ov{\vr \log(\vr)} - \vr \log\vr) (T)} \leq \exp(C(\tau-T)) \intO{ (\Ov{\vr \log(\vr)} - \vr \log\vr) (\tau)} 
\end{align}
for a.a.\ $\tau, T$ with $\tau < T$.
Using Hypothesis {\bf (B)} and letting $\tau \to -\infty$, we conclude 
\[
\intO{ (\Ov{\vr \log(\vr)} - \vr \log\vr) (T)} \equiv 0 \quad \mbox{for a.a.} \ T \in \R.
\]
This yields the strong convergence of density claimed in \eqref{con-d}.

Finally, combining with the strong convergence of velocity \eqref{con-u} and temperature \eqref{con-T}, we derive from the consistency formulation of entropy \eqref{cons-3} and ballistic energy \eqref{cons-4} that $\vU$ satisfy the entropy inequality \eqref{wsol-S} and ballistic energy inequality \eqref{wsol-BE}, which finishes the proof of Theorem \ref{CT1}.

\section{Long--time behavior}
\label{L}

Our final objective is to illustrate the available analytical results and to suggest possible conjectures in the situations where analysis has failed.

\subsection{Long--time behavior: Analytical results}

The central issue is validity of the so--called \emph{ergodic hypothesis}:
\begin{equation} \label{L1}
\lim_{T \to \infty} \frac{1}{T} \int_0^T  F(\vU(t, \cdot)) \dt 	\ \mbox{exists}
\end{equation} 
for any entire solution $\vU$
of the Navier-Stokes-Fourier system and any bounded continuous function $F$ defined on a suitable \emph{phase space}, cf. \cite{FanFeiHof}. 
As the quantities $\vr$, $\vm$ are only weakly continuous in time 
with respect to the $L^q-$topology for some $q > 1$, while 
$S$ is weakly c\` agl\` ad (see \cite{FeiSwGw}) in $L^q$, it is convenient to consider the phase space 
\[
X = W^{-\ell,2}(\Omega) \times W^{-\ell,2}(\Omega; \R^d) \times 
W^{-\ell,2}(\Omega),\ \ell > d/2
\] 
with the associated Hilbert topology.

For any time shift invariant set  $\mathcal{U} \subset \mathcal{A}$, where $\mathcal{A}$ is the global attractor consisting of entire solutions, see   \eqref{i12}, there exists a \emph{stationary statistical solution} supported by $\Ov{\mathcal{U}}$. Specifically, there exists a Borel probability measure 
$\mathcal{V}$, 
\begin{align} 
	{\rm supp}[\mathcal{V}] &\subset \Ov{\mathcal{U}}, \quad
	\mathcal{V}[\mathfrak{B}] = \mathcal{V}[\mathfrak{B}(\cdot + T)] 
	\ \mbox{for any}\ T \in \R,
	\nonumber	
\end{align}	
for any Borel set $\mathfrak{B} \subset \mathcal{A}$, see \cite[Theorem 5.3]{FeiSwGw}.	
In addition, there exists an \emph{ergodic stationary statistical solution} $\mathcal{V}$ enjoying the property 
\[
\mathfrak{B} \ \mbox{a Borel time shift invariant set}\ 
\Rightarrow\ \ \mbox{either} \ \mathcal{V}[\mathfrak{B}] = 1 
\ \mbox{or} \ \mathcal{V}[\mathfrak{B}] = 0, 
\]
cf.\ \cite[Theorem 7.3]{FanFeiHof}.

Finally, 
for any ergodic stationary statistical solution $\mathcal{V}$ and any Borel measurable function $F:X \to R$ such that 
\[
\int_X F \Big(\vr(0, \cdot), \vm(0, \cdot), S(0, \cdot) \Big) \D \mathcal{V} < \infty,
\]
the ergodic limit 
\[
 \frac{1}{T} \int_0^T  F \Big(\vr (t, \cdot), \vm (t, \cdot), S (t ,\cdot) \Big) \dt \to 
\int_X F \Big(\vr(0, \cdot), \vm(0, \cdot), S(0, \cdot) \Big) \D \mathcal{V} \ \mbox{as}\ T \to \infty
\]
exists $\mathcal{V}-$a.s.\ in $\mathcal{A}$, see \cite[Theorem 5.4]{FeiSwGw} and \cite[Theorem 7.2]{FanFeiHof}.

\subsection{Long time behavior: Numerics}
In order to illustrate the theoretical results, we present numerical  Rayleigh--B\'enard simulations for two regions: i) a stable region which admits a stationary state; ii) a weak turbulent region with Rayleigh number $Ra \approx 8\cdot 10^4$. 
Further, we dig into the simulations and generate several conjectures. 

\medskip

Let us consider $\Omega = [-2,2]|_{\{-2,2\}} \times [-1,1]$ with following initial and boundary data 
\begin{align*}
& \vr_D(x) = 1.2 +   \sin \left( \frac{\pi x_2}2\right), \quad \vu_D(x) = (0, \,c \sin(2\pi x_2))^t \quad \mbox{implying} \quad  \vu_B|_{\partial \Omega} = 0,\\
& \vt_D(x) = \vt_M + S_\vt  x_2 +  c P(x_1)\sin(\pi x_2) +   \widetilde{P}(x_1)\sin\left(\frac{\pi (x_2+1)}{4}\right) + \widehat{P}(x_2), \\
&\vt_M =\frac{\vt_L +\vt_H}2, \quad S_\vt =\frac{\vt_L -\vt_H}2, \quad  \widehat{P}(x_2)|_{x=(\cdot,\pm1)} = 0, \\
& \mbox{implying} \quad
\vt_B|_{x=(\cdot,-1)} =  \vt_H,\quad \vt_B|_{x=(\cdot,1)} =  \vt_L + \widetilde{P}(x_1),
\end{align*}
where
\begin{align*}
&
P(x_1) = \sum_{j=1}^{10} a_j \cos( b_j + 2j\pi x_1), \quad  c = 0.01
\end{align*}
and $a_j \in [0,1], b_j \in [-\pi, \pi], j = 1, \dots, 10$ are arbitrary fixed numbers. The coefficients $a_j$ have been normalized so that $\sum_{j=1}^{10} a_j = 1$ to guarantee that the perturbation is small.
The parameters appearing in the Navier--Stokes--Fourier system are taken as
\begin{align*}
\mu =  \lambda = 0.1, \quad  \kappa = 0.01,\quad \gamma = 1.4.
\end{align*}
We point out that $\widetilde{P}$ is used to make perturbation upon the bottom boundary temperature, whereas $\widehat{P}$ is used to produce more general initial data with larger energy.   

\medskip
This is the basic setting. Due to the Rayleigh number $Ra$ given by
\begin{align*}
Ra = \frac{g \beta L^d \Delta \vt}{\kappa \nu},
\end{align*}
we adjust $\vt_H, \vt_L, \widetilde{P}(x_1)$ as well as $\Grad G=(0,g)$ so that the solutions live in different regions and generate different structures.  
Note that, $L = 2,\, d= 2$, $\beta =1/\vt_M$ is the thermal expansion coefficient, $\nu = \mu/\vr_M$ is the kinematic viscosity with $\vr_M= \int_{\Omega} \vr \dx /|\Omega|=1.2$.  

\subsubsection{Experiment 1: Stable region} 
\label{sr}
This section is devoted to verify the theoretical results \eqref{conv2stat} that the attractor reduces to be a single point. 
Additionally, we investigate here experimentally how $Ra$ influences the flow structure.  
Specifically, we take 
\[
\widetilde{P}(x_1) = \widehat{P}(x_2) \equiv 0, \quad \vt_L \equiv 1, \quad g=S_\vt, \quad S_\vt \in\{ -100, \ -10, \ -2, \ -1.1, \ -1\}. 
\] 
Figure \ref{fig:GEN} presents temperature $\vt_h$ and streamline $\vu_h$  obtained on the uniform mesh with $h = 2/80$ for different values of  $S_\vt$. 
Note that, the simulations for $\vth, \vuh$ with $S_\vt = -10$ are only shown at $T=50$. This is before chaotic behaviour starts to develop. Further details will be discussed in Section~\ref{tr}.
These numerical simulations indicate that $Ra$ does influent the flow structure:  the solution becomes more chaotic with an increasing $Ra$. 

\begin{figure}[htbp]
	\setlength{\abovecaptionskip}{0.cm}
\begin{minipage}[b]{.9\linewidth}
		\centering
	\includegraphics[width=0.45\textwidth]{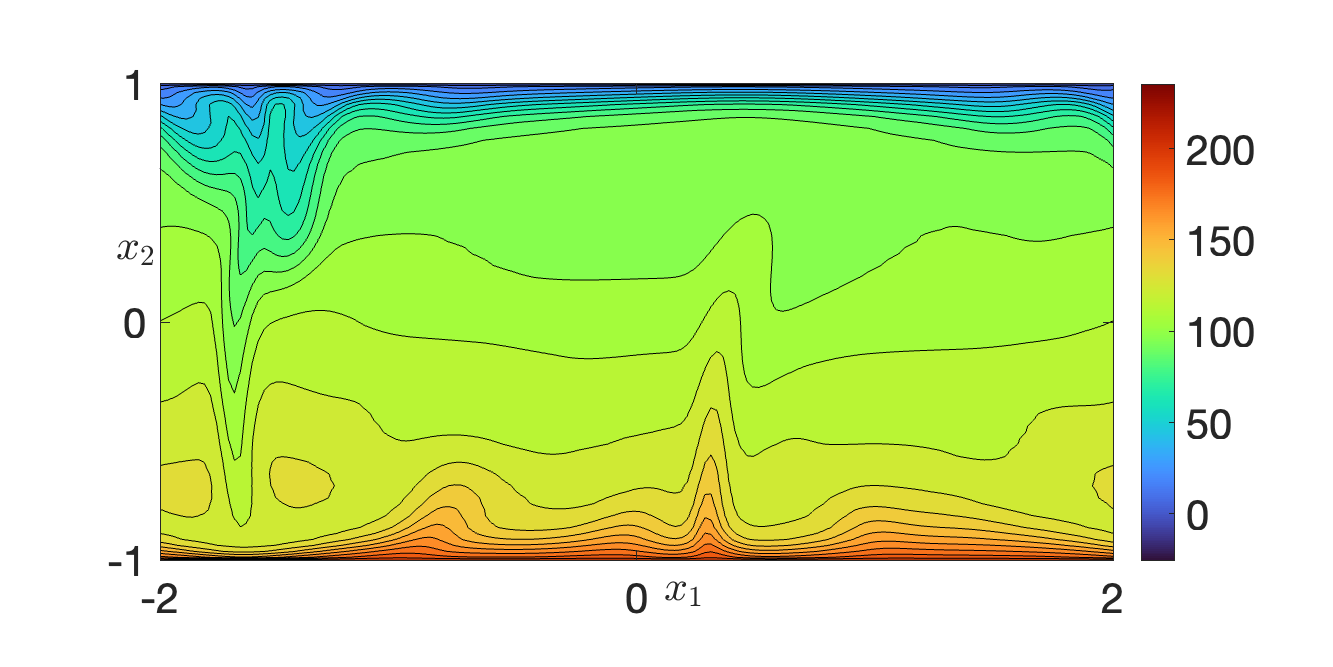}
	\includegraphics[width=0.45\textwidth]{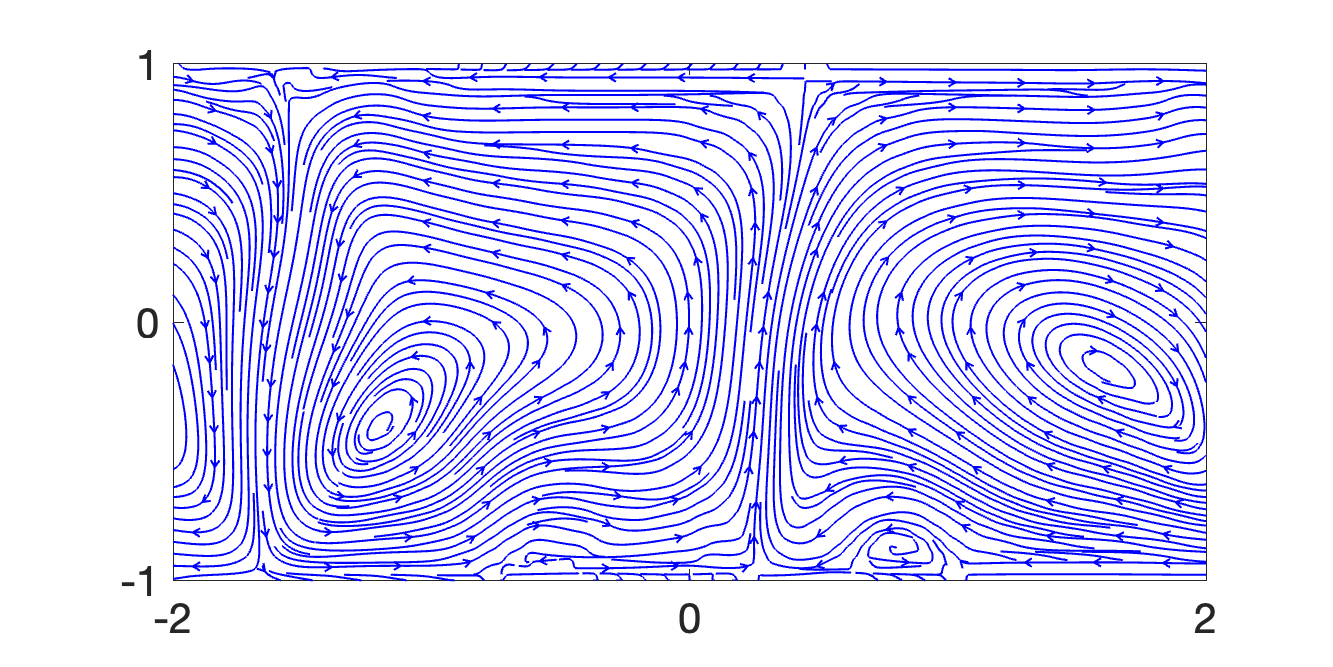}
	\caption*{ $S_\vt = -100,\  T=250,\ Ra =  9.5\cdot 10^5$ }
	\end{minipage}\vspace{0.1cm}
\begin{minipage}[b]{.9\linewidth}
		\centering
	\includegraphics[width=0.45\textwidth]{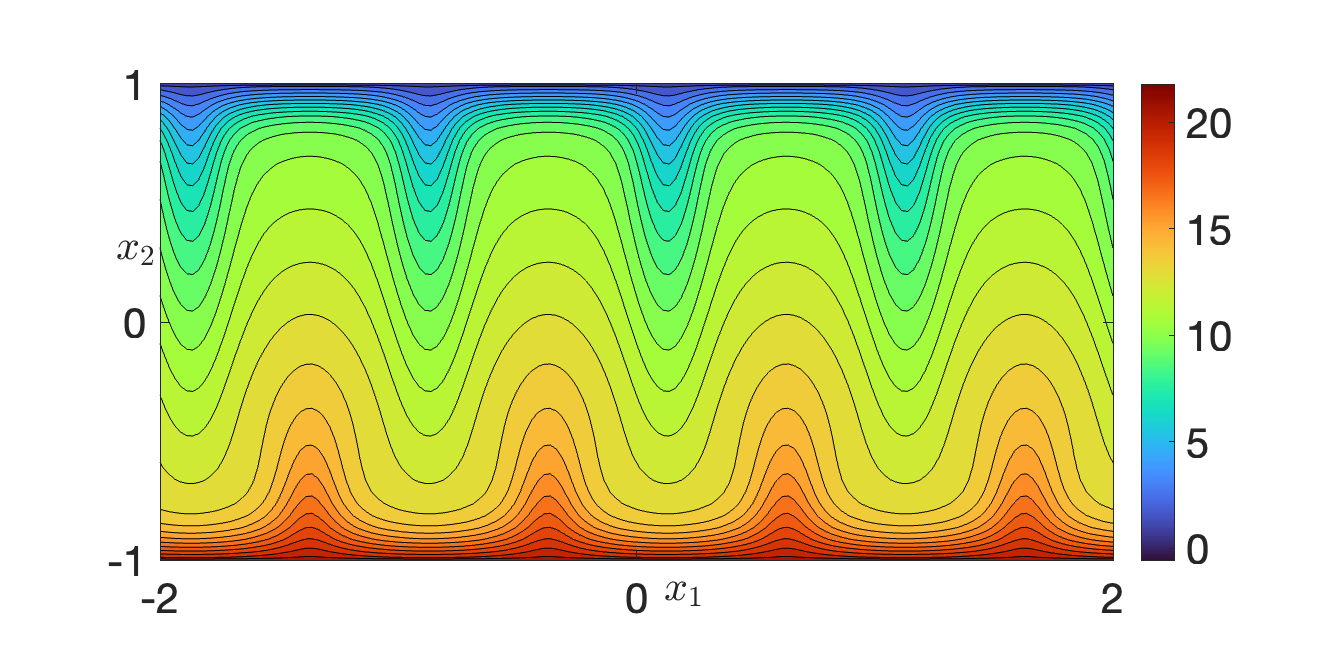}
	\includegraphics[width=0.45\textwidth]{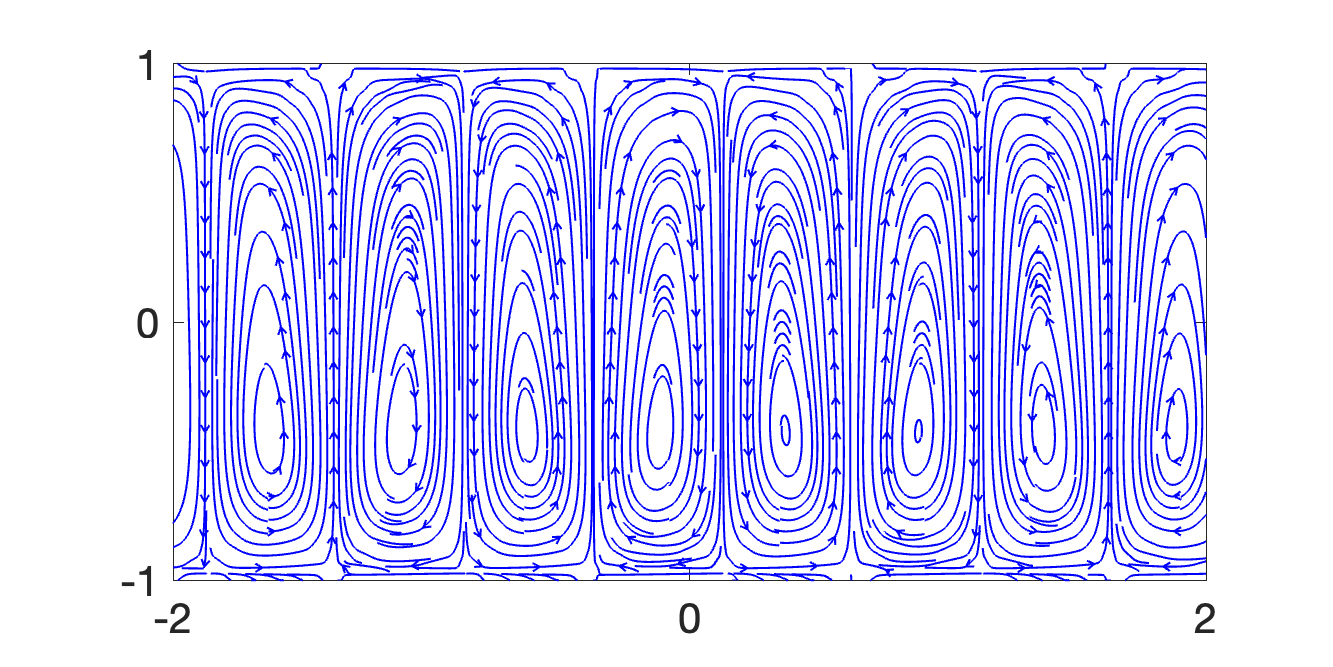}
	\caption*{ $S_\vt = -10,\  T=50,\ Ra =  8.7\cdot 10^4$ }
	\end{minipage}\vspace{0.1cm}
\begin{minipage}[b]{.9\linewidth}
		\centering
	\includegraphics[width=0.45\textwidth]{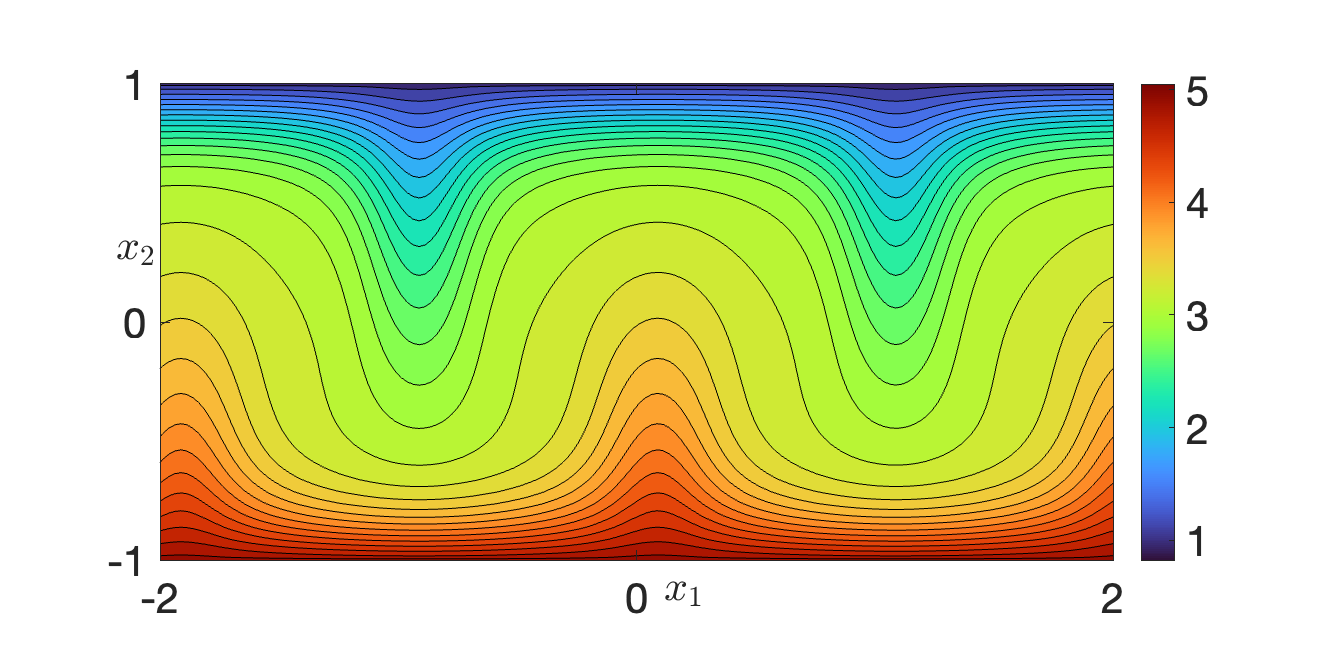}
	\includegraphics[width=0.45\textwidth]{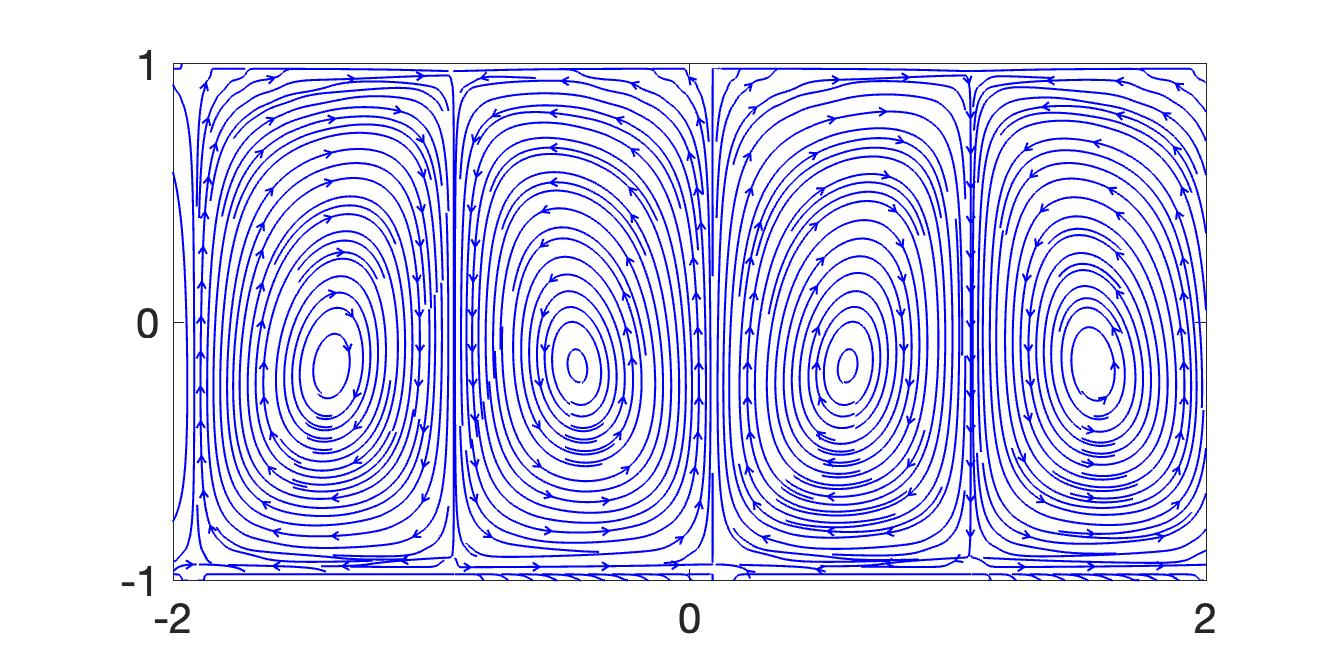}
	\caption*{ $S_\vt = -2,\  T=250,\ Ra =  1.2\cdot 10^4$ }
	\end{minipage}\vspace{0.1cm}
\begin{minipage}[b]{.9\linewidth}
		\centering
	\includegraphics[width=0.45\textwidth]{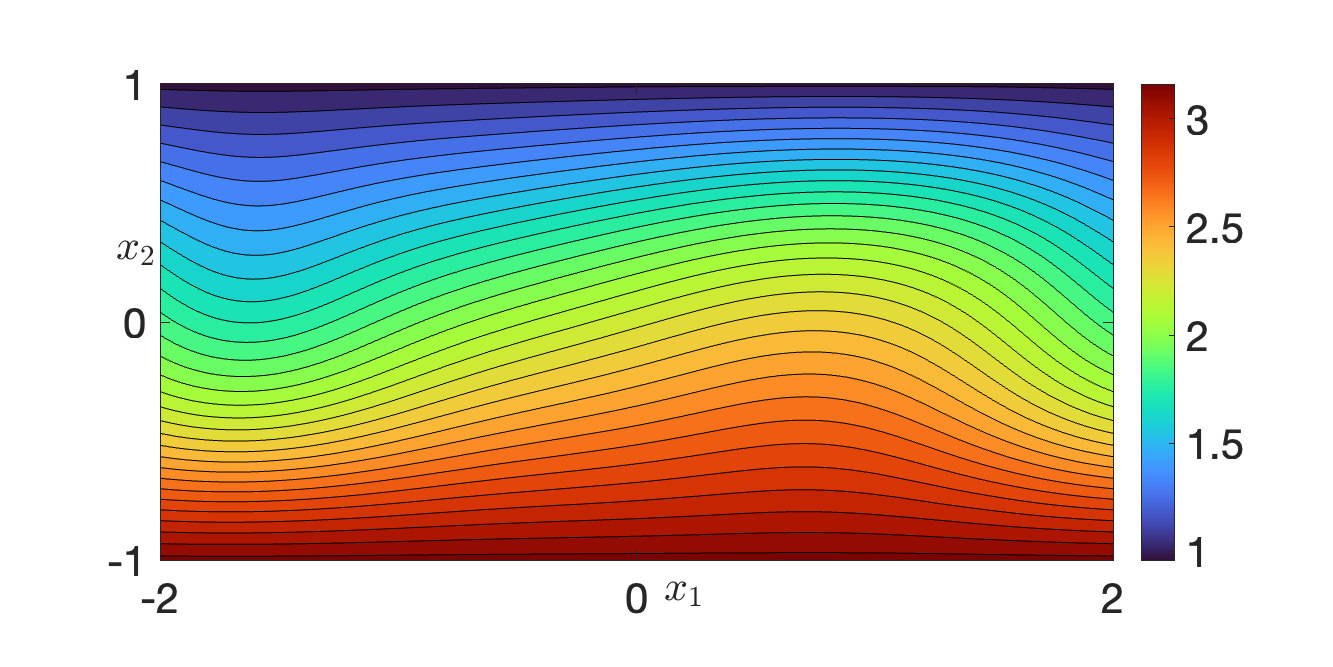}
	\includegraphics[width=0.45\textwidth]{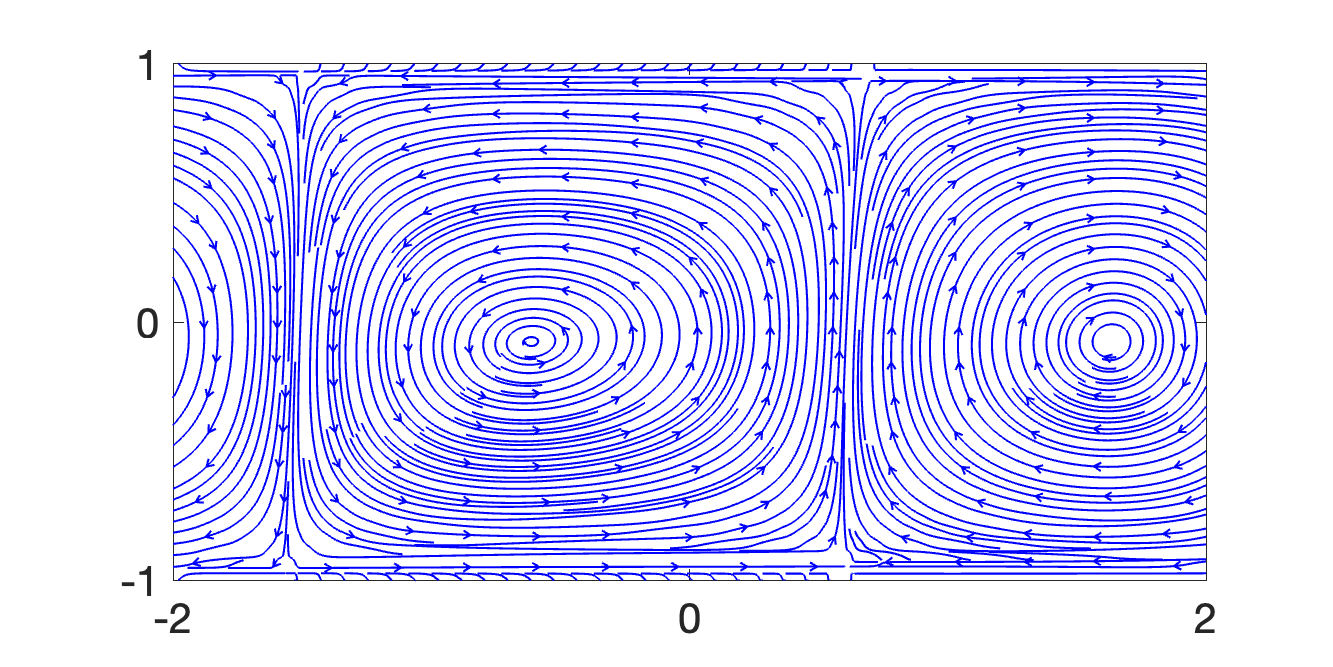}
	\caption*{ $S_\vt = -1.1,\  T=250,\ Ra =  5.5\cdot 10^3$ }
	\end{minipage}\vspace{0.1cm}
\begin{minipage}[b]{.9\linewidth}
		\centering
	\includegraphics[width=0.45\textwidth]{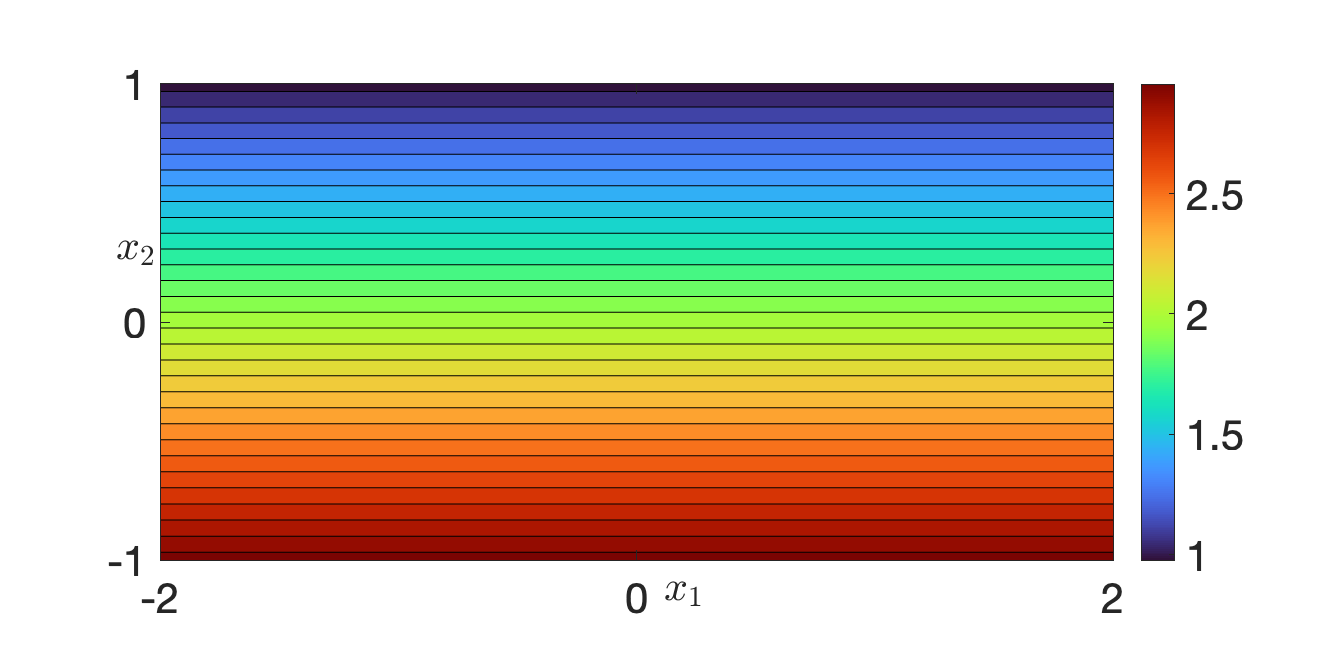}
	\includegraphics[width=0.45\textwidth]{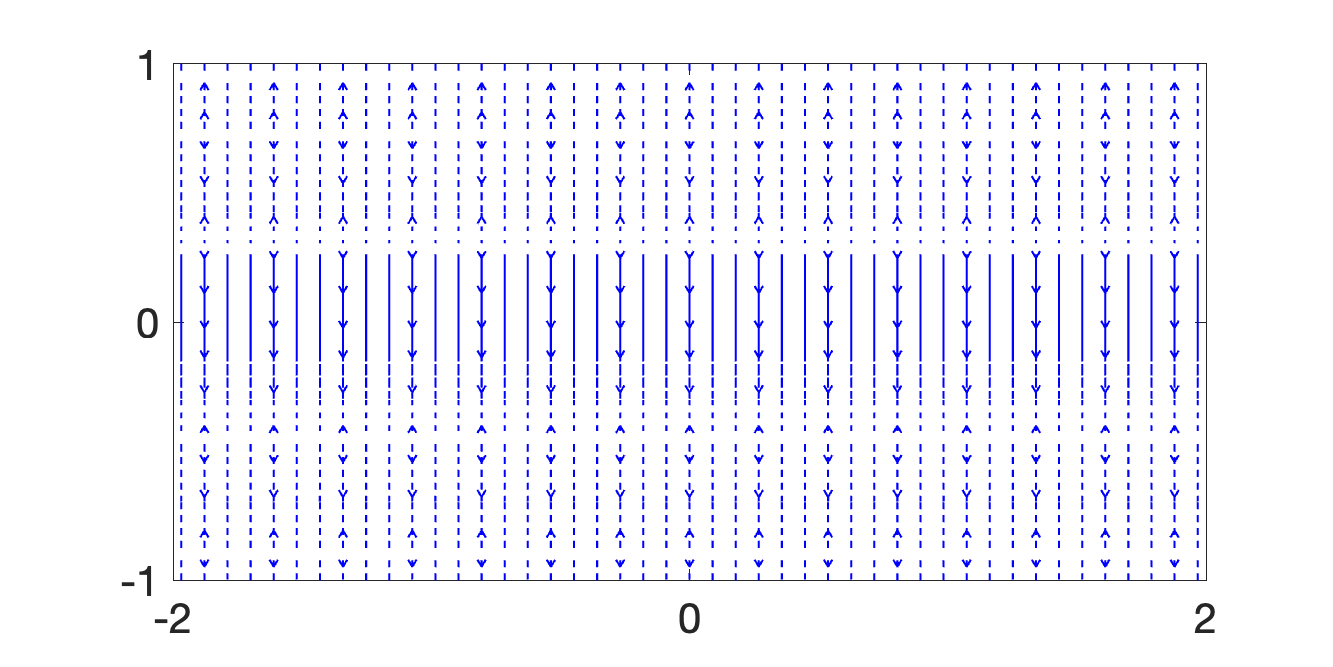}
	\caption*{ $S_\vt = -1,\  T=250,\ Ra = 4.8 \cdot 10^3$ }
	\end{minipage}\vspace{0.1cm}
\caption{Rayleigh--B\' enard Experiment 1: Temperature $\vt$ (left)  and streamlines $\vu$ (right) obtained with different $S_\vt$ and $T$. }\label{fig:GEN}
\end{figure}

\medskip
Now let us focus on the case that the attractor might reduce to be a single point.  The parameters are taken as 
\[
\widetilde{P}(x_1) = 0, \quad \widehat{P}(x_2) =  \begin{cases}
100 \cos^2(\pi x_2), & \mbox{ if } x_2 \in [-1/2, \ 1/2],\\
0, & \mbox{ otherwise},
\end{cases} \quad \vt_L \equiv 1, \quad g=S_\vt=-0.3, 
\] 
which gives $Ra \approx 554$. 
Figure \ref{fig-Evo-Ex1} presents the evolutions of  $L^1$-norms of the solution $\norm{U_h(T_M,\cdot)}_{L^1(\Omega)}$ as well as  means(-in-time) of norms 
\[
\Ov{\norm{U_h(T_M,\cdot)}_{L^1(\Omega)}} := \frac{1}{M}\sum_{m=1}^{M}\norm{U_h(T_m,\cdot)}_{L^1(\Omega)}
\]
obtained on the uniform mesh with $h = 2/320$. Here, $U\in \{m_1,m_2, E, \rho e\}, T_M = 2 M$ and $M = 1, \dots, 400$. 
Numerical results hint that this specific problem might admit a stationary solution 
\begin{align*}
\vr_{s} = 1.2, \quad \vu_s = \vc{0}, \quad \vt_s = \vt_M + S_\vt x_2.
\end{align*}

Numerical solutions $(\vrh, \vuh, \vt_h)$ at $T=800$ are shown in Figure~\ref{fig-Ex1-sol}.  
Furthermore, Figure~\ref{fig-Err-ex1} presents errors between a single solution (resp.\ its temporal-average) and the exact solution $\vU_s=(\vr_s, \vu_s, \vt_s)$, defined by
\begin{align*}
&\widetilde{E_1}(U_h, T_M) = \norm{ U_h(T_M,\cdot ) - U_s(\cdot )}_{L^1(\Omega)}, \\
&\widetilde{E_2}(U_h, T_M) = \norm{\Ov{U_h}(T_{M},\cdot) - U_s(\cdot)}_{L^1(\Omega)},\quad  \Ov{U_h}(T_{M},\cdot) =\frac{1}{M}\sum_{m=1}^{M} U_h(T_{M},\cdot) 
\end{align*}
with  $U\in \{ \vr, \vm, \vr s, \vu, \vt, E, BE\}$. Numerical results show that the solution as well as its temporal-average do converge. 

\begin{figure}[htbp]
	\setlength{\abovecaptionskip}{0.cm}
	\setlength{\belowcaptionskip}{-0.cm}
	\centering
	\includegraphics[width=\textwidth]{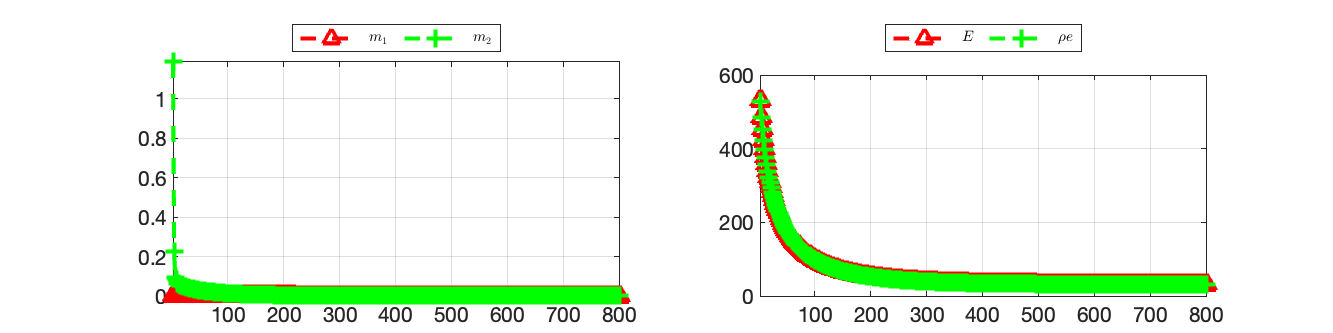}
	\includegraphics[width=\textwidth]{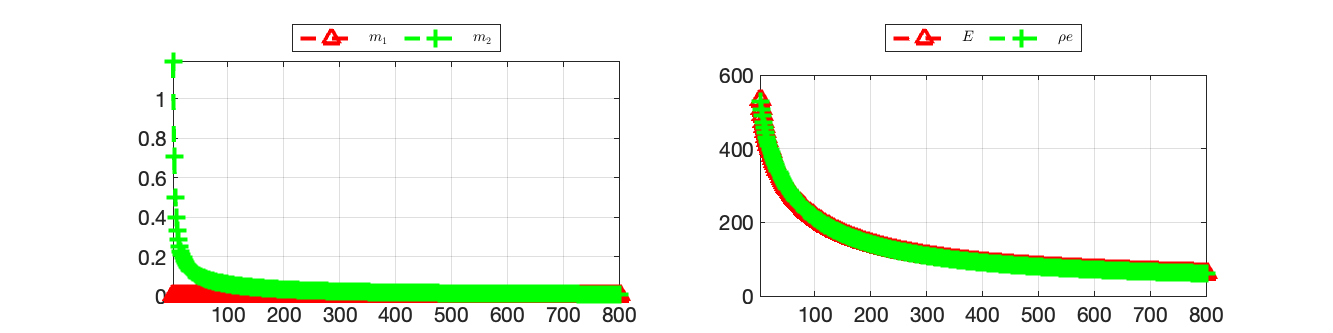}
	\caption{  \small{Rayleigh--B\' enard Experiment 1: $\norm{U_h(T_M,\cdot)}_{L^1(\Omega)}$ (top) and $\Ov{\norm{U_h(T_M,\cdot)}_{L^1(\Omega)}}$ (bottom).}}\label{fig-Evo-Ex1}
\end{figure}

\begin{figure}[htbp]
	\setlength{\abovecaptionskip}{0.cm}
	\setlength{\belowcaptionskip}{-0.cm}
	\centering
	\includegraphics[width=0.48\textwidth]{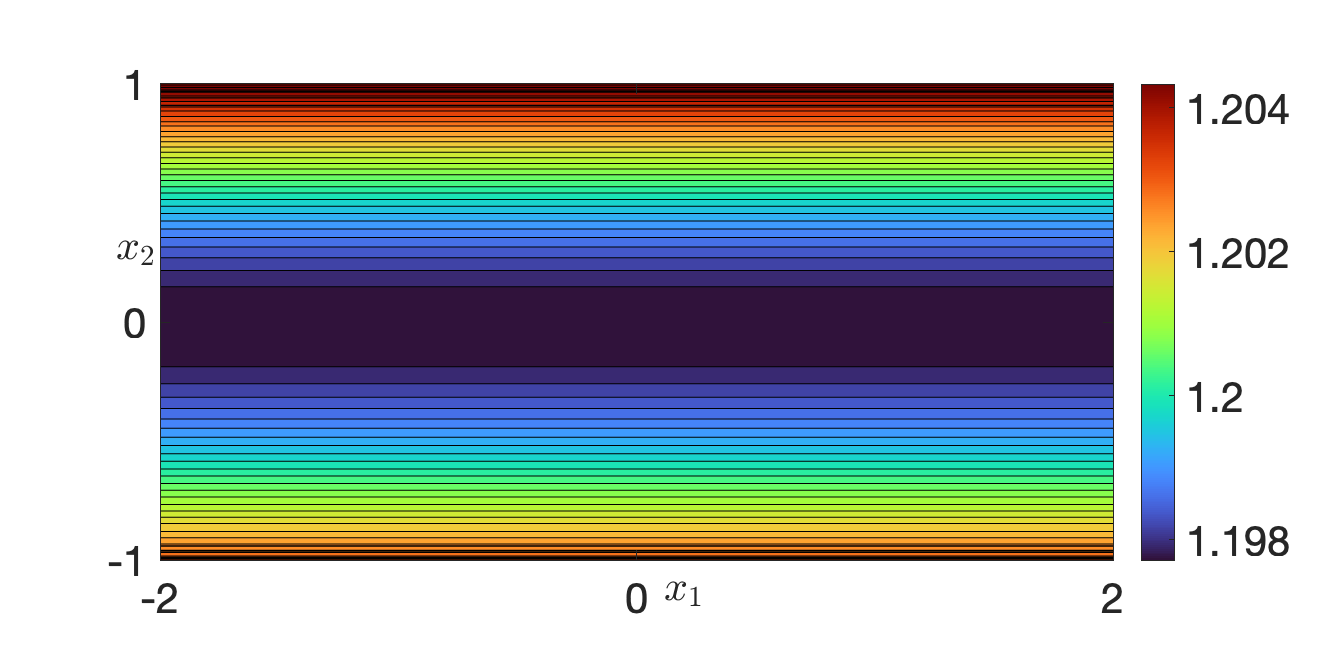}
	\includegraphics[width=0.48\textwidth]{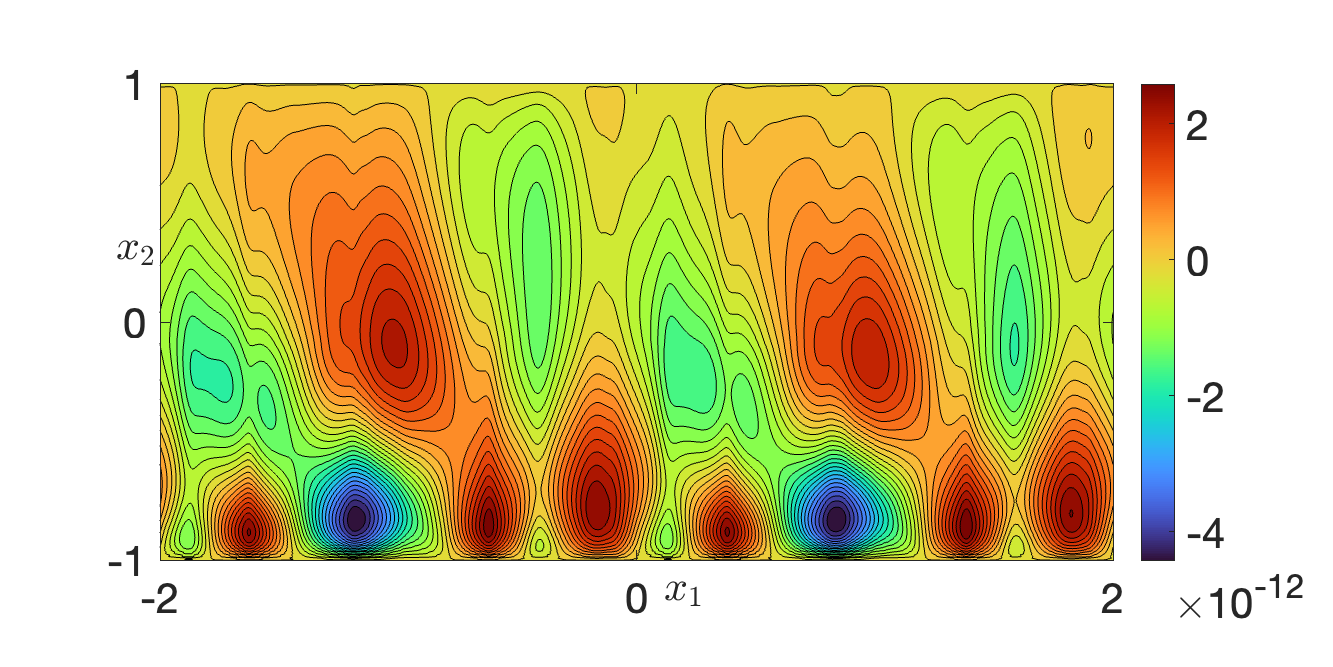}\\
	\includegraphics[width=0.48\textwidth]{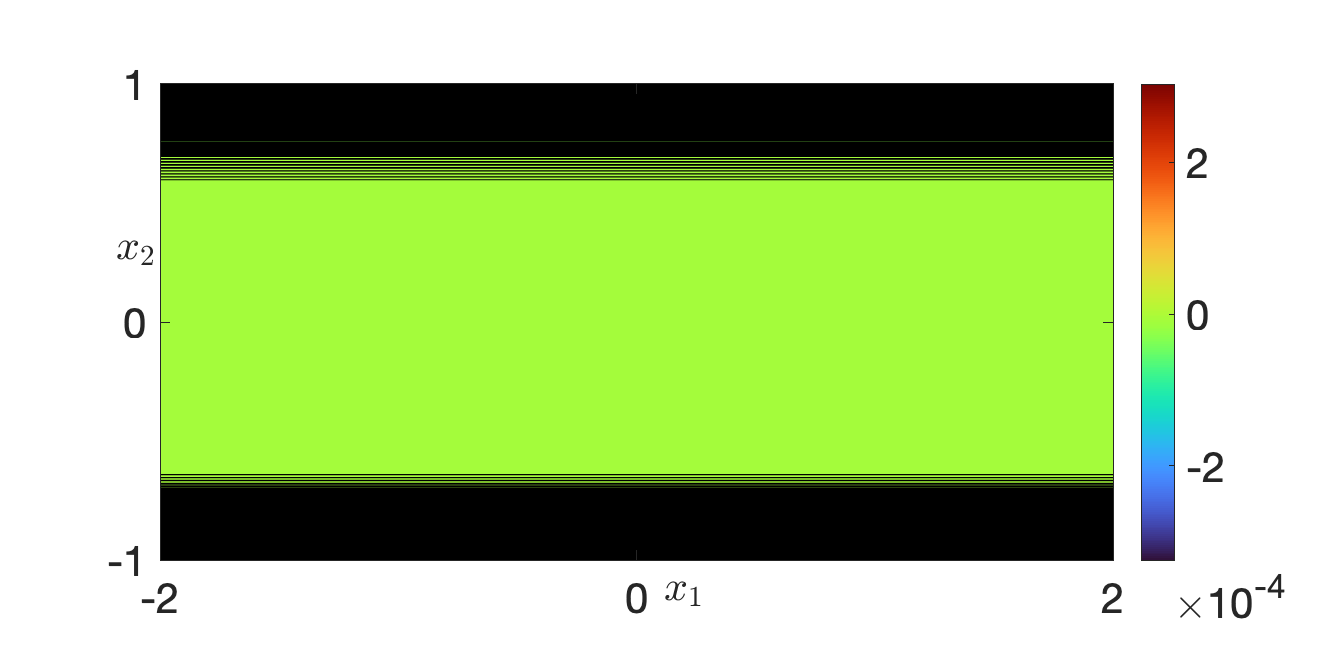}
	\includegraphics[width=0.48\textwidth]{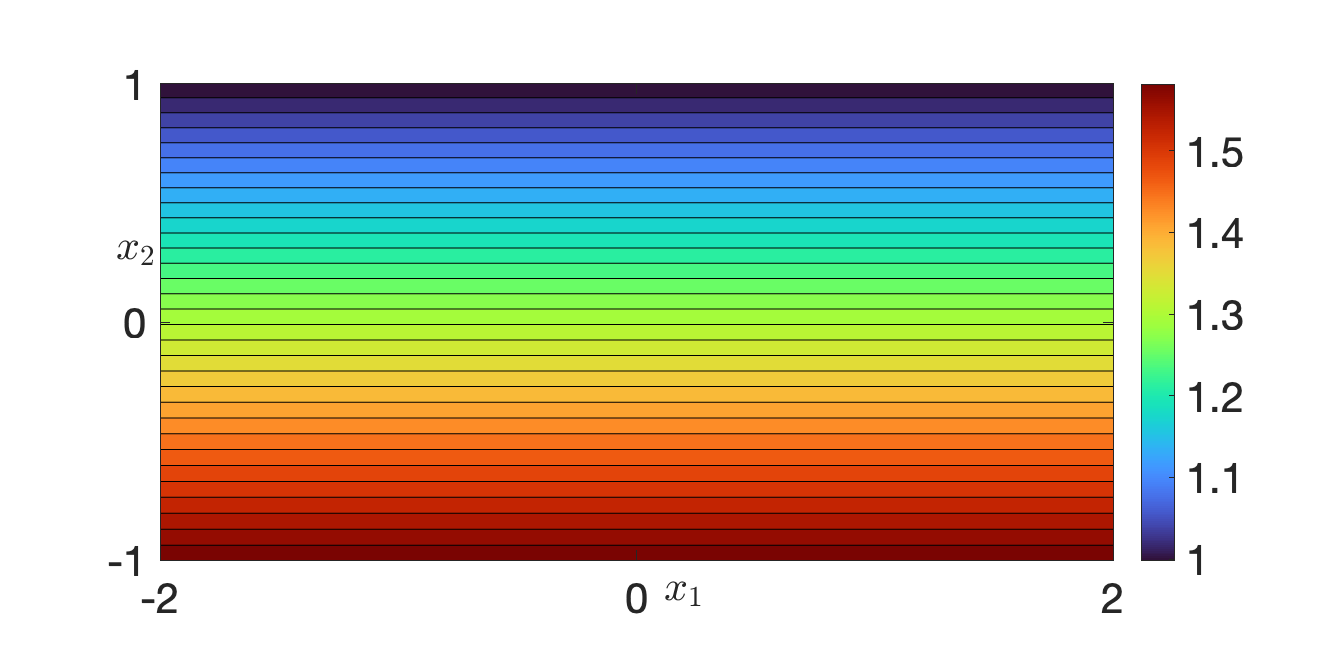}
	\caption{  \small{Rayleigh--B\' enard Experiment 1: Numerical solutions $(\vr_h, u_{1,h}, u_{2,h}, \vt_h)$ (from left to right, from top to bottom) at $T = 800$.}}\label{fig-Ex1-sol}
\end{figure}

\begin{figure}[htbp]
	\setlength{\abovecaptionskip}{0.cm}
	\setlength{\belowcaptionskip}{-0.cm}
	\centering
	\includegraphics[width=0.45\textwidth]{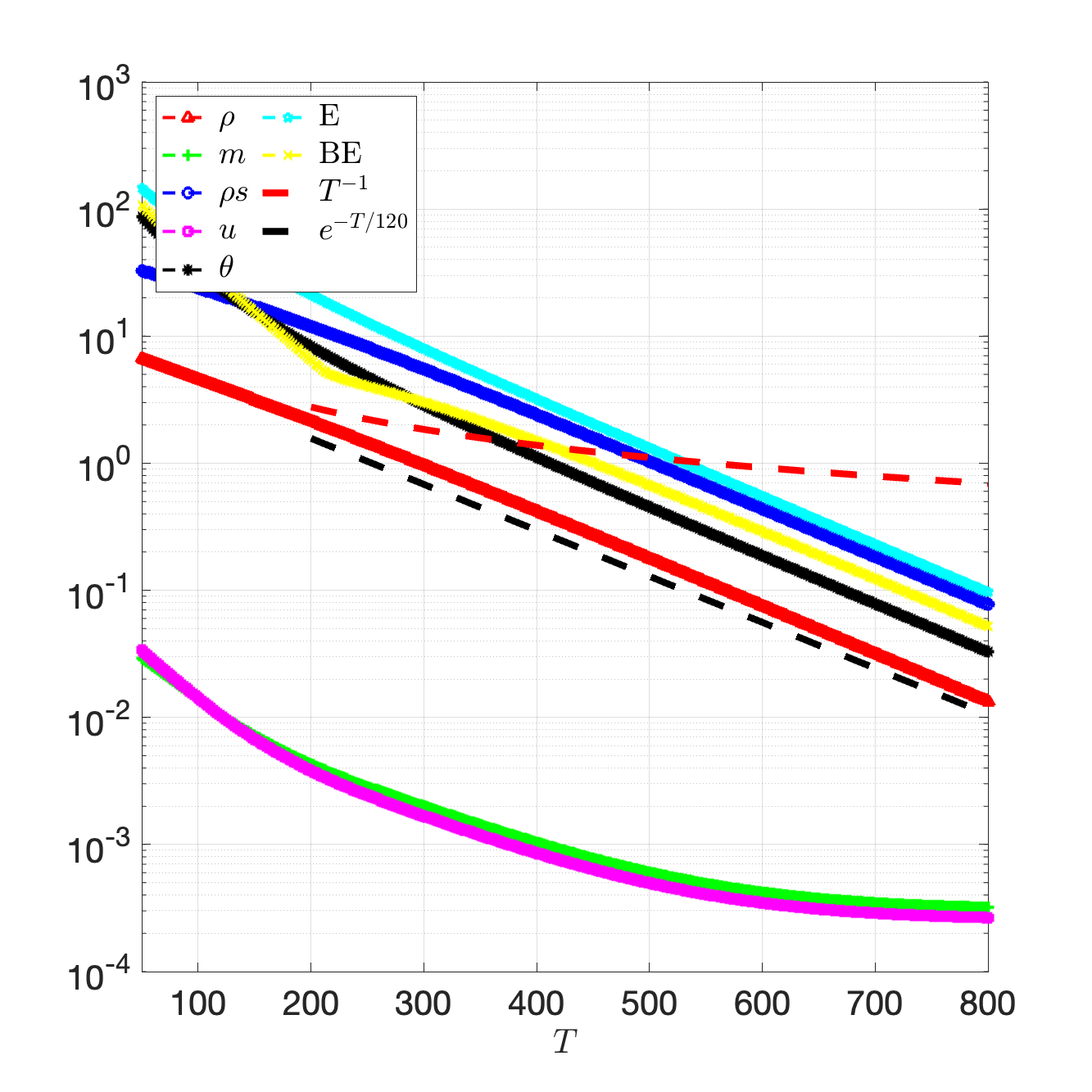} 
	\includegraphics[width=0.45\textwidth]{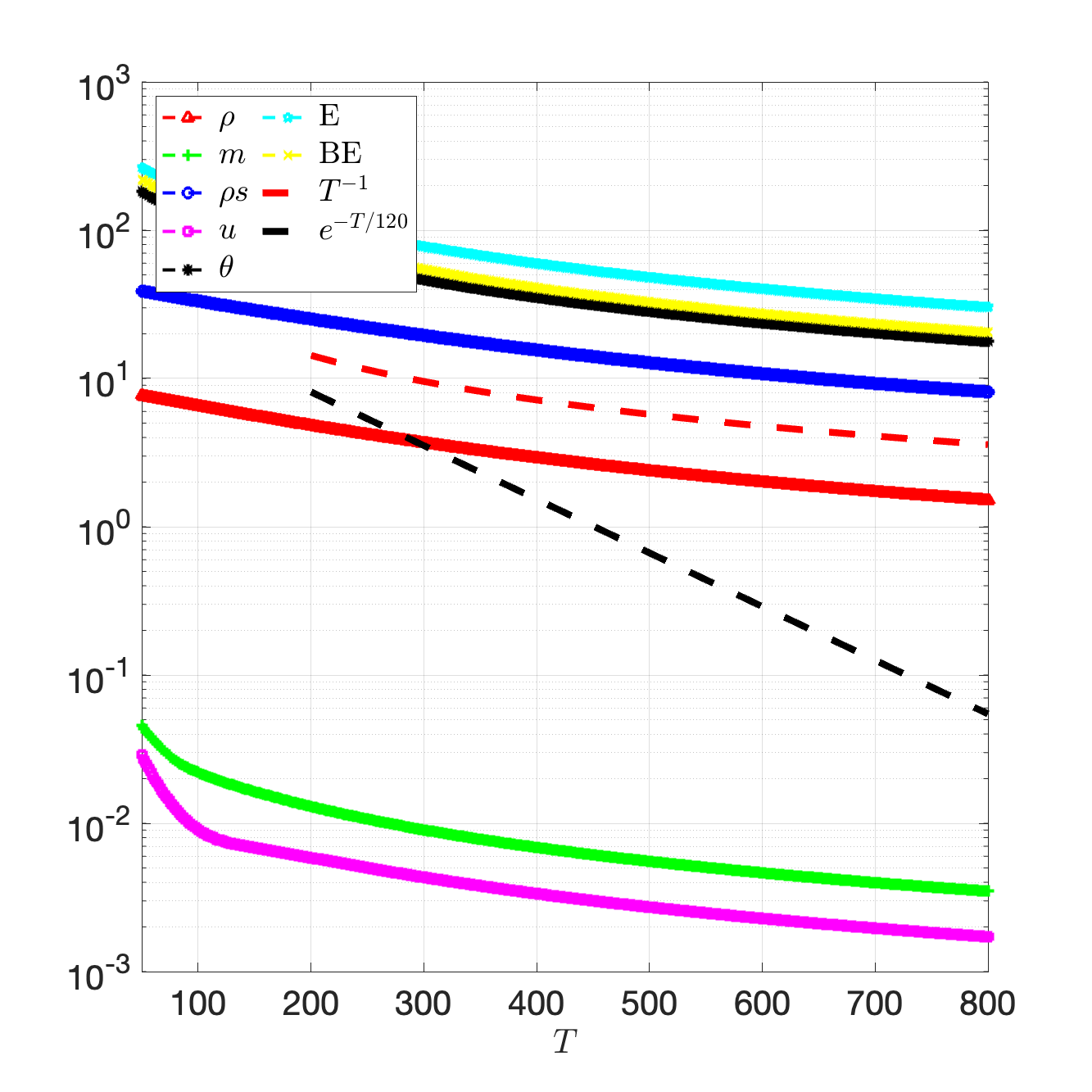}
	\caption{  \small{Rayleigh--B\' enard Experiment 1: Errors $\widetilde{E_1}, \widetilde{E_2}$ over $[0,800]$.}}\label{fig-Err-ex1}
\end{figure}


\newpage
\subsubsection{Experiment 2: Turbulent region}
\label{tr}
In this section, we verify our main theoretical results as well as the theoretical results \eqref{i12} within weak turbulent regions. 

To this end, we take 
\[
\widetilde{P}(x_1) = \widehat{P}(x_2) \equiv 0, \quad \vt_L \equiv 1, \quad \vt_H \equiv 15, \quad g\equiv-10,
\] 
which yields $Ra \approx 8\cdot 10^4$. 
This is the experiment first created in our previous work \cite{FeLMShYu:2024}, where some interesting phenomena about random effects are demonstrated.

\medskip

We simulate this experiment till $T=1600$ on a fixed fine mesh with $h = 2/320$.
Here we calculate
\begin{itemize}
\item[(1)] {\bf Evolution}.  
Figure \ref{fig-Evo-1} presents the evolutions of temperature $\vt_h(T_M,\cdot)$ at different $T_M$ with $M=20, 160, 195, 415, 575, 615$. As shown in Figure \ref{fig-Evo-1} we can see the convection structures evolving over time.

Figure \ref{fig-Evo} presents the evolutions of  $L^1$-norms $\norm{U_h(T_M,\cdot)}_{L^1(\Omega)}$ as well as  means(-in-time) of norms 
\[
\Ov{\norm{U_h(T_M,\cdot)}_{L^1(\Omega)}} := \frac{1}{M}\sum_{m=1}^{M}\norm{U_h(T_m)}_{L^1(\Omega)}
\]
with $U\in \{m_1,m_2, E, \rho e\}, T_M = 2 M$ and $M = 1, \dots, 800$.

Figure \ref{fig-Evo} indicates that the solution reaches the turbulent region only after a considerable duration. Therefore, in what follows we present the temporal-averages starting from $T_{M_0} = 400$ with $M_0 = 200$.

\begin{figure}[htbp]
	\setlength{\abovecaptionskip}{0.cm}
	\setlength{\belowcaptionskip}{-0.cm}
	\centering
	\includegraphics[width=0.3\textwidth]{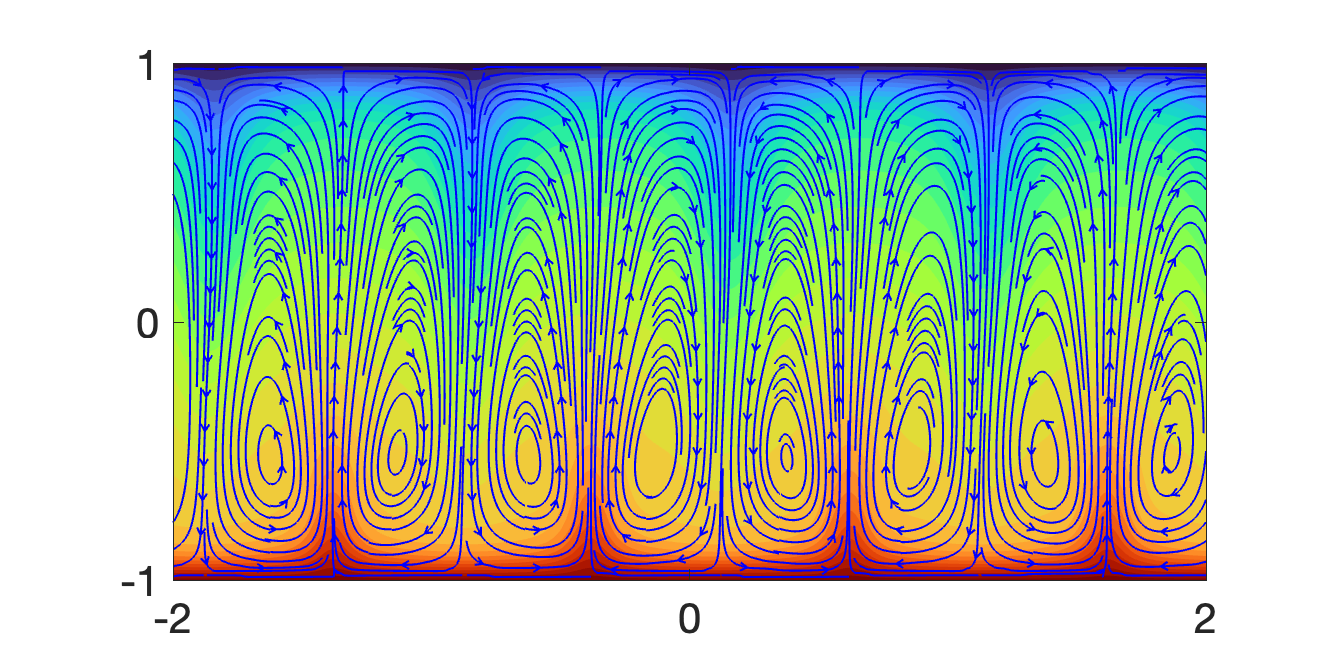}
	\includegraphics[width=0.3\textwidth]{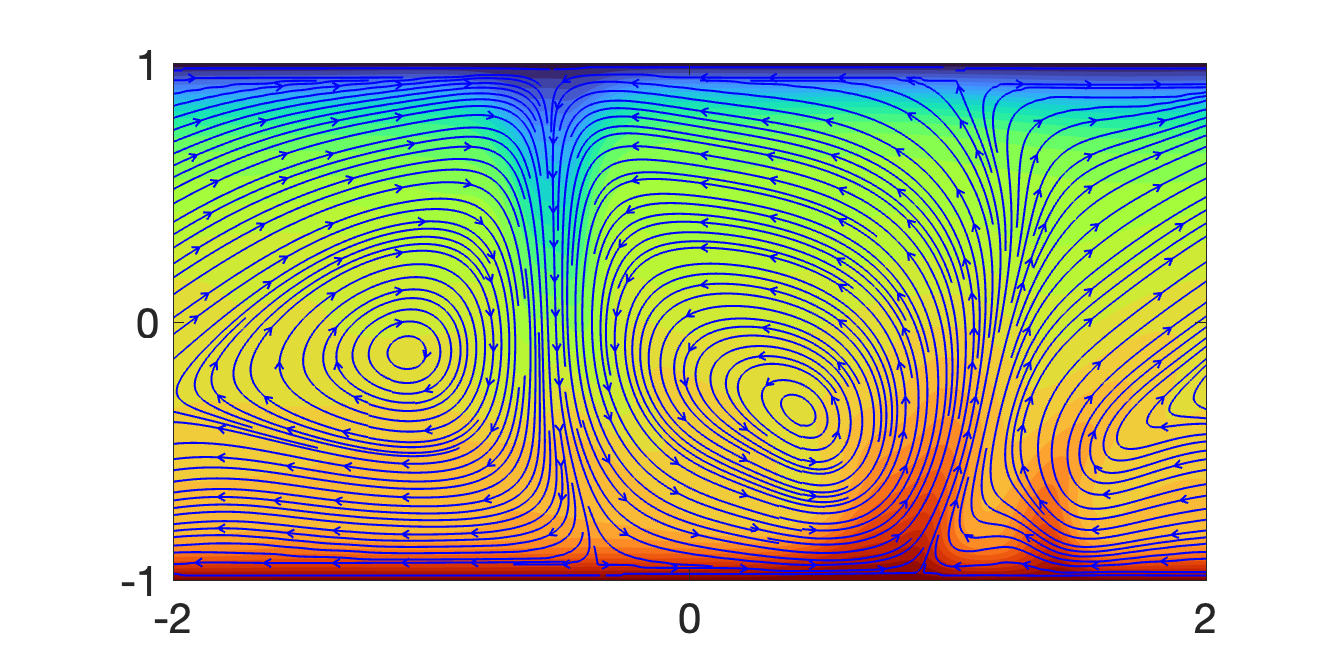}
	\includegraphics[width=0.3\textwidth]{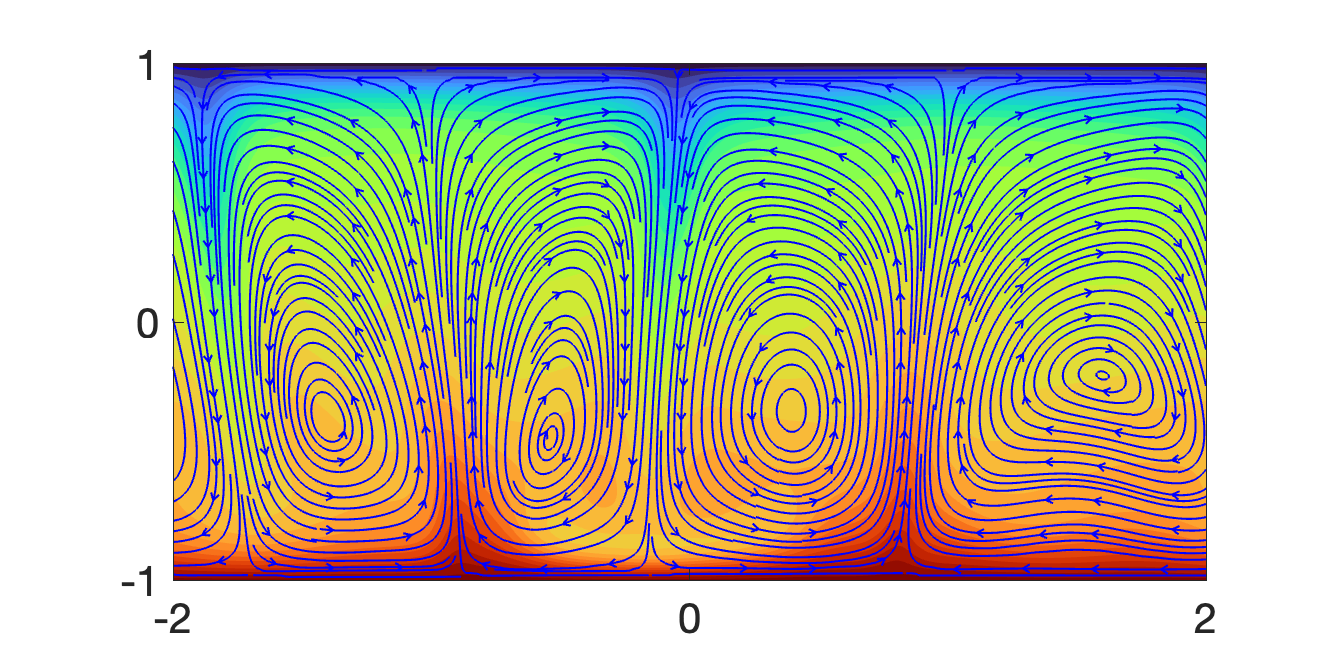}\\
	\includegraphics[width=0.3\textwidth]{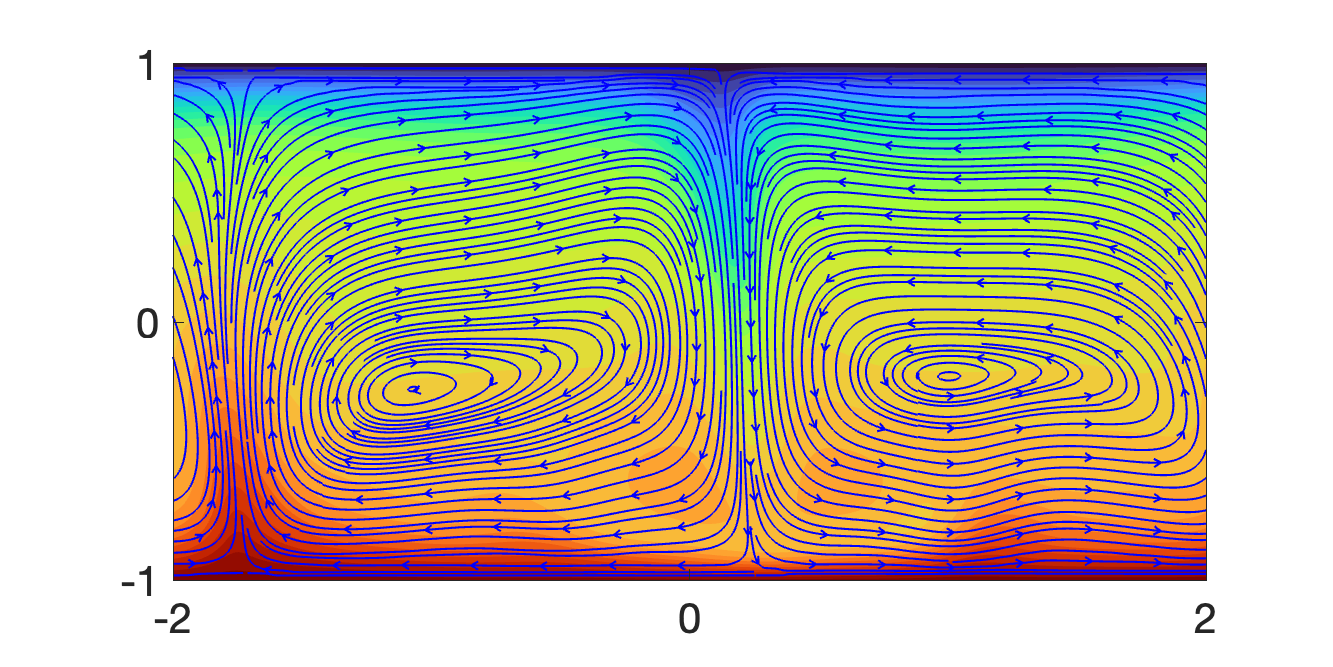}
	\includegraphics[width=0.3\textwidth]{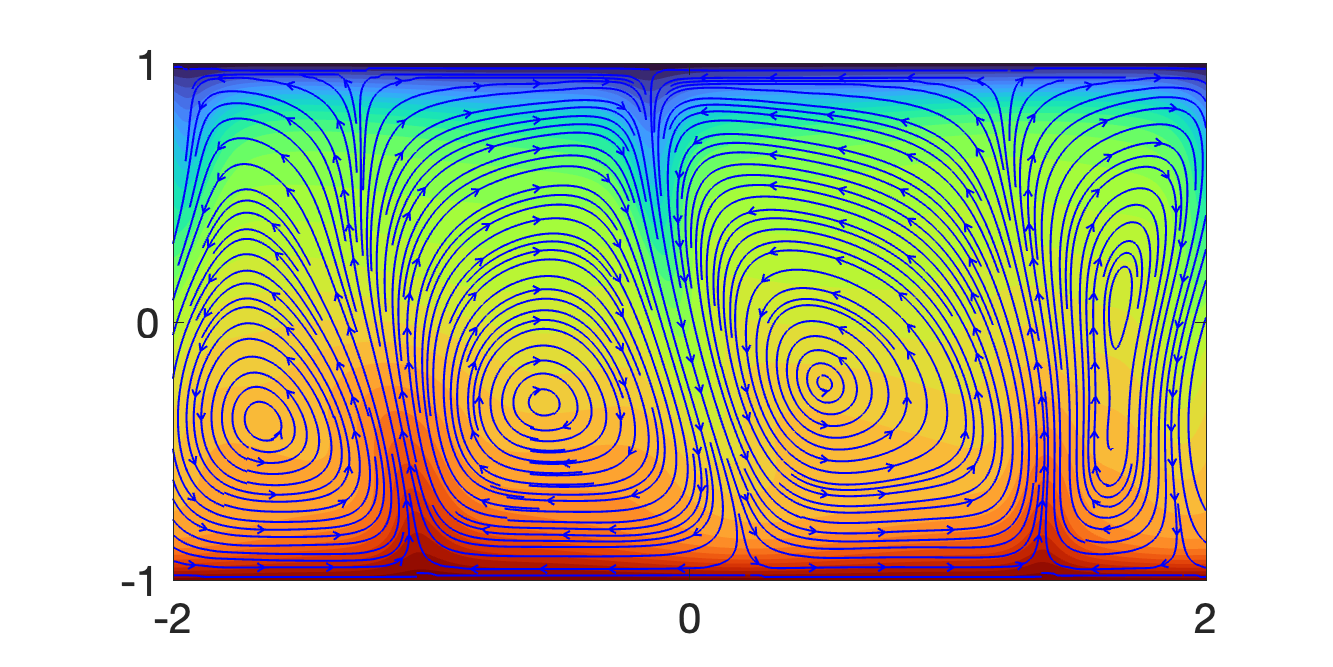}
	\includegraphics[width=0.3\textwidth]{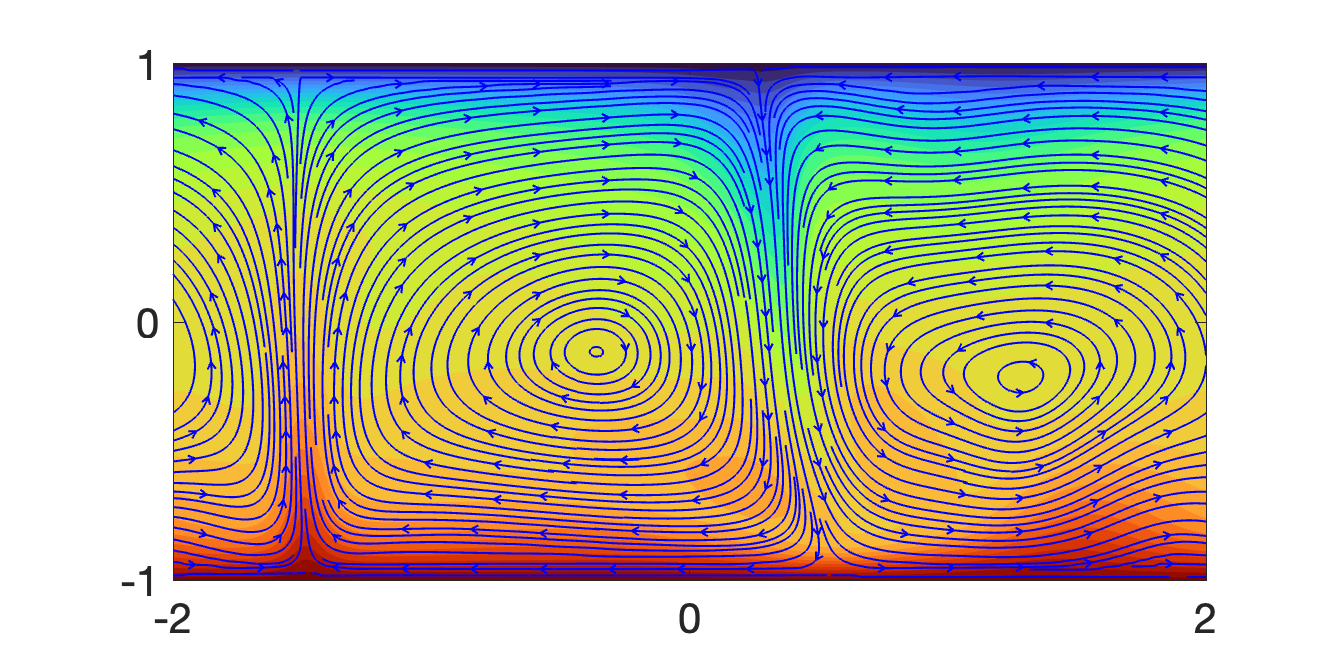}
	\caption{  \small{Rayleigh--B\' enard Experiment 2: $\vt_h(T_M,\cdot)$ together with the streamline $\vuh$ at different $T_M$ with $M=20, 160, 195, 415, 575$, $615$ (from left to right, from top to bottom).}}\label{fig-Evo-1}
\end{figure}

\begin{figure}[htbp]
	\setlength{\abovecaptionskip}{0.cm}
	\setlength{\belowcaptionskip}{-0.cm}
	\centering
	\includegraphics[width=\textwidth]{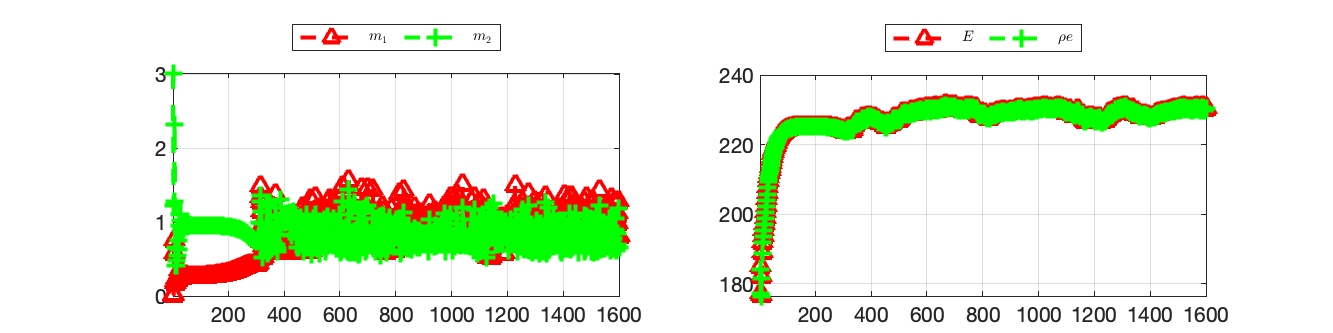}
	\includegraphics[width=\textwidth]{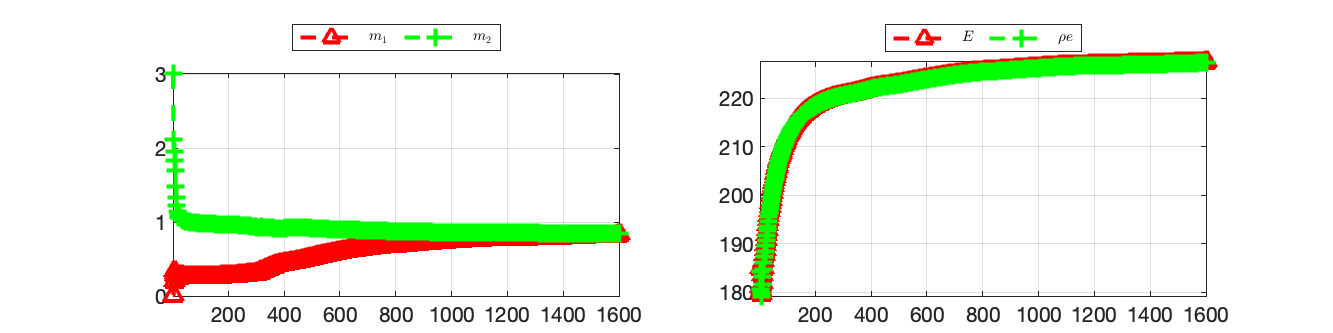}
	\caption{  \small{Rayleigh--B\' enard Experiment 2: $\norm{U_h(T_M,\cdot)}_{L^1(\Omega)}$ (top) and $\Ov{\norm{U_h(T_M,\cdot)}_{L^1(\Omega)}}$ (bottom).}}\label{fig-Evo}
\end{figure}

\item[(2)] {\bf Mean and deviation}. Let us introduce the notations for mean(-in-time) and deviation(-in-time)
\begin{align}
& \Ov{U_h}(T_{M_{ref}}) = \frac1{M_{ref}-M_0}\sum_{m=M_0+1}^{M_{ref}} U_h(T_m,\cdot ), \ T_m = 2 m, \ M_0 = 200, \ M_{ref} = 800, \br
& \mbox{Dev}(U_h,T_{M_{ref}}) = \frac1{M_{ref}-M_0}\sum_{m=M_0+1}^{M_{ref}} \abs{  U_h(T_m,\cdot )  - \Ov{U_h}(T_{M_{ref}})},
\end{align}
approximating  
\begin{align*}
& \frac{1}{T-T_{M_0}} \int_{T_{M_0}}^{T}  U(t,\cdot ) \dt, \quad 
\frac{1}{T-T_{M_0}} \int_{T_{M_0}}^{T} \left| U(t,\cdot ) - \frac{1}{T-T_{M_0}} \int_{T_{M_0}}^{T}  U(t,\cdot ) \dt \right| \dt. 
\end{align*}

Figure~\ref{fig-Mean} shows the means $\Ov{U_h}(T_{M_{ref}})$  with  $\ U \in \{ \vr, \vt, E \}$ and its streamline $\Ov{\vuh}$.
The details of the mean and derivation of numerical temperature $\vt_h$ are shown in Figure~\ref{fig-Mean-Var-vt}.

Numerical simulations reveal that the temporal-averaged convection structure exhibits a well-defined form despite the instability of a single solution. The observed single convection pair structure may be intrinsically linked to the length of the fluid domain.

\begin{figure}[htbp]
	\setlength{\abovecaptionskip}{0.cm}
		\centering
	\includegraphics[width=0.49\textwidth]{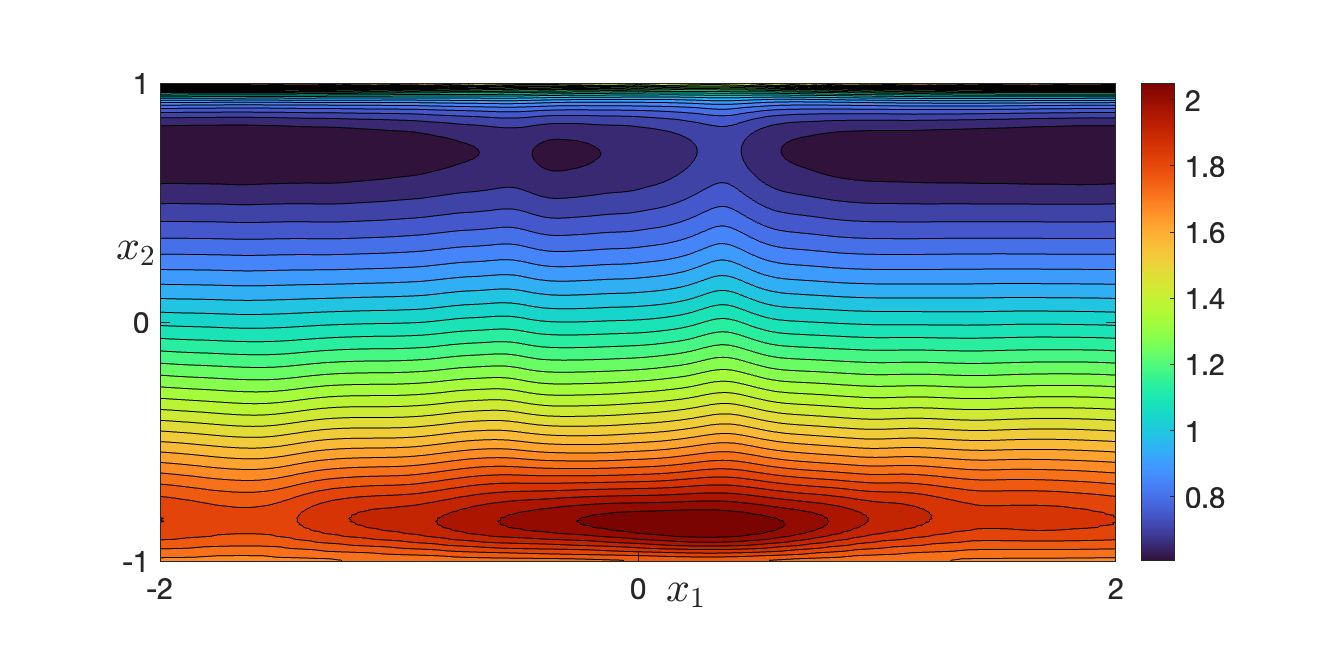}
	\includegraphics[width=0.49\textwidth]{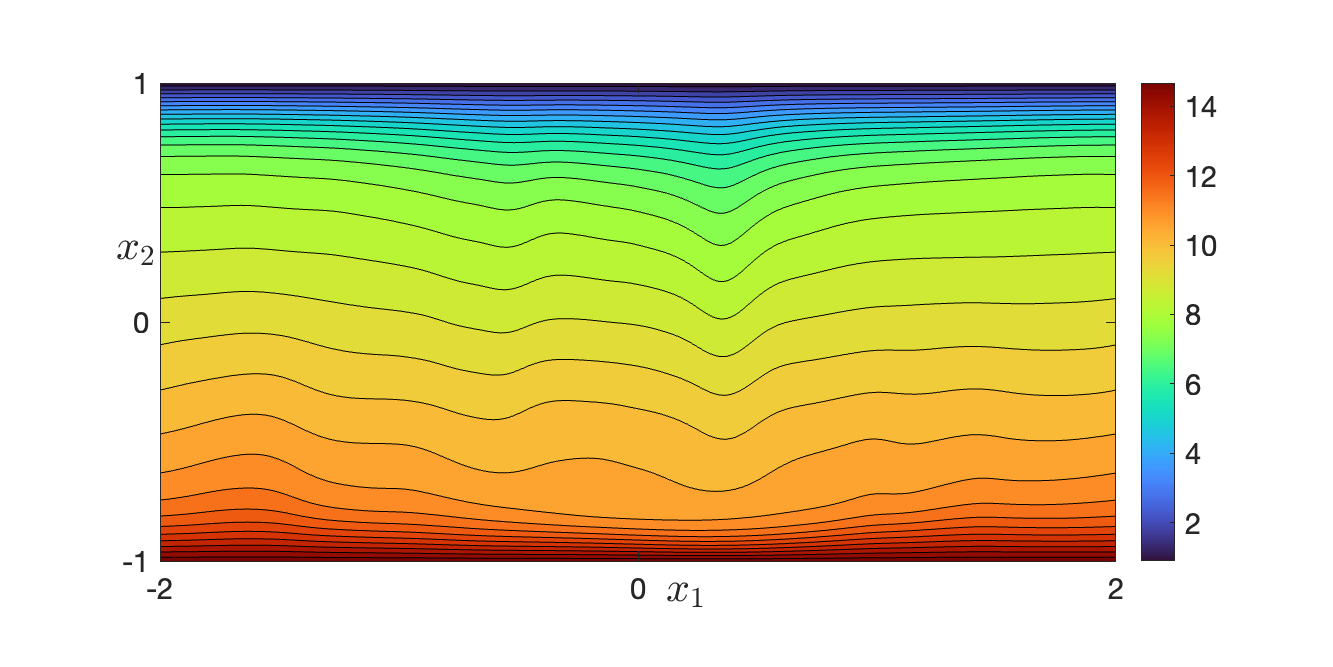}\\
	\includegraphics[width=0.49\textwidth]{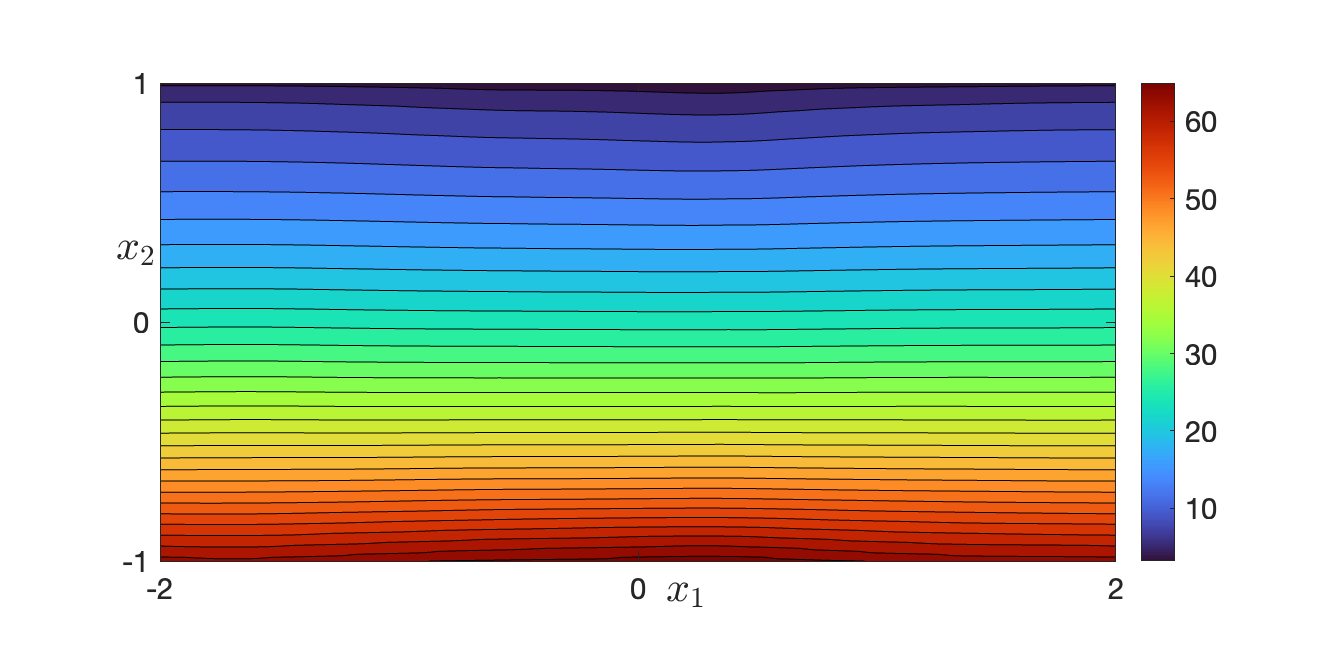}
	\includegraphics[width=0.49\textwidth]{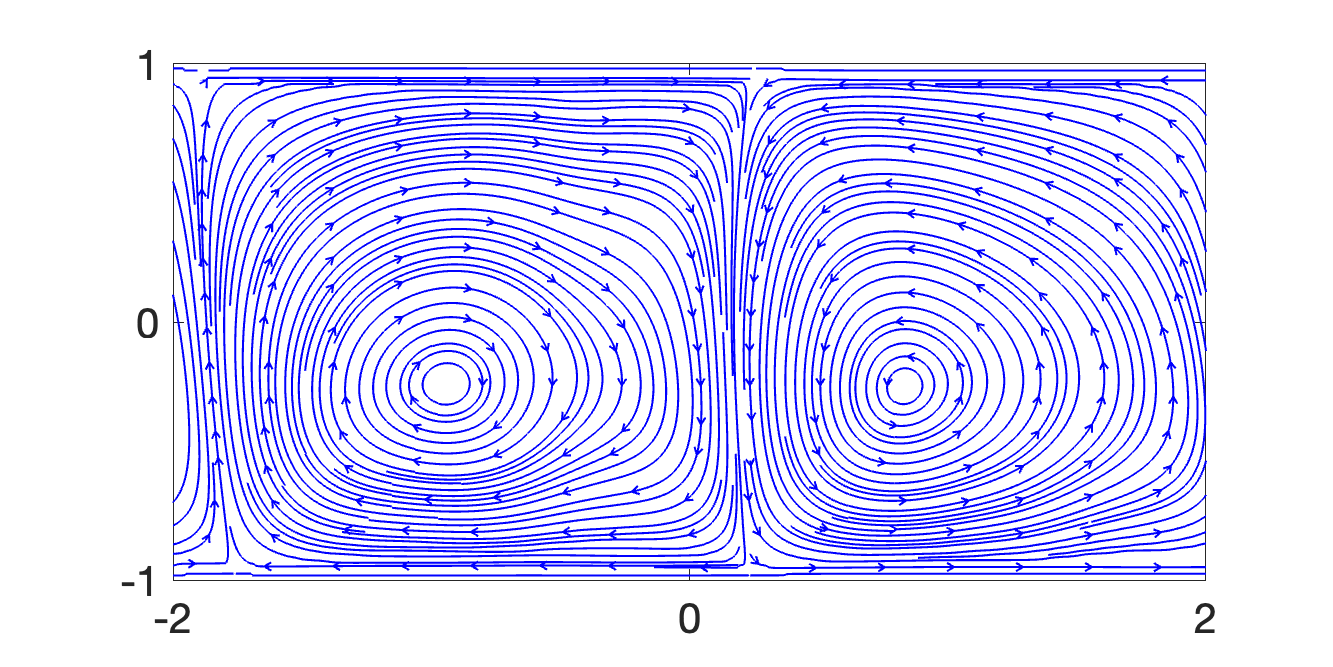}
	\caption{  \small{Rayleigh--B\' enard Experiment 2: $\Ov{U_h}(T_{M_{ref}})$. From left to right, from top to bottom: $\Ov{\vrh}, \, \Ov{\vt_h}$,  $\Ov{E_h}$ and  streamlines $\Ov{\vuh}$.}   }\label{fig-Mean}
\end{figure}

\begin{figure}[htbp]
	\setlength{\abovecaptionskip}{0.cm}
	\setlength{\belowcaptionskip}{-0.cm}
	\centering
	\includegraphics[width=0.49\textwidth]{./gif/case0/temperaturePRONUM52MX640MY320IS201IE800_Mean.png}
	\includegraphics[width=0.49\textwidth]{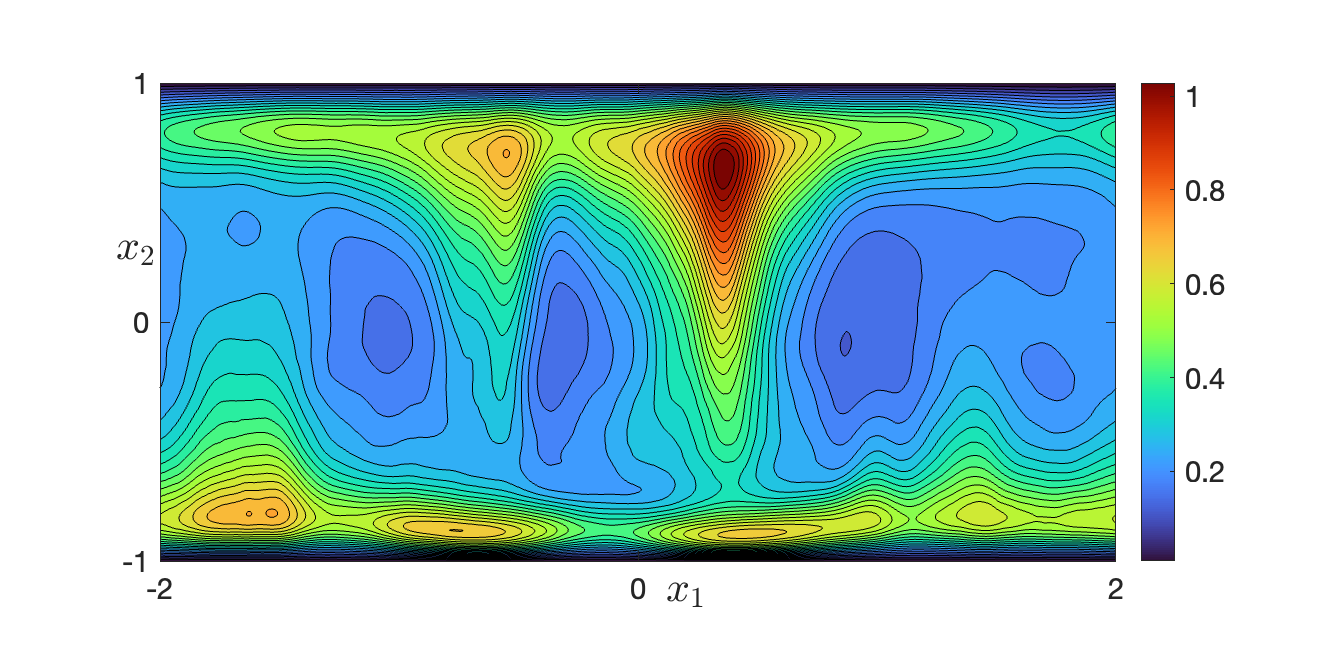}\\
	\includegraphics[width=0.3\textwidth]{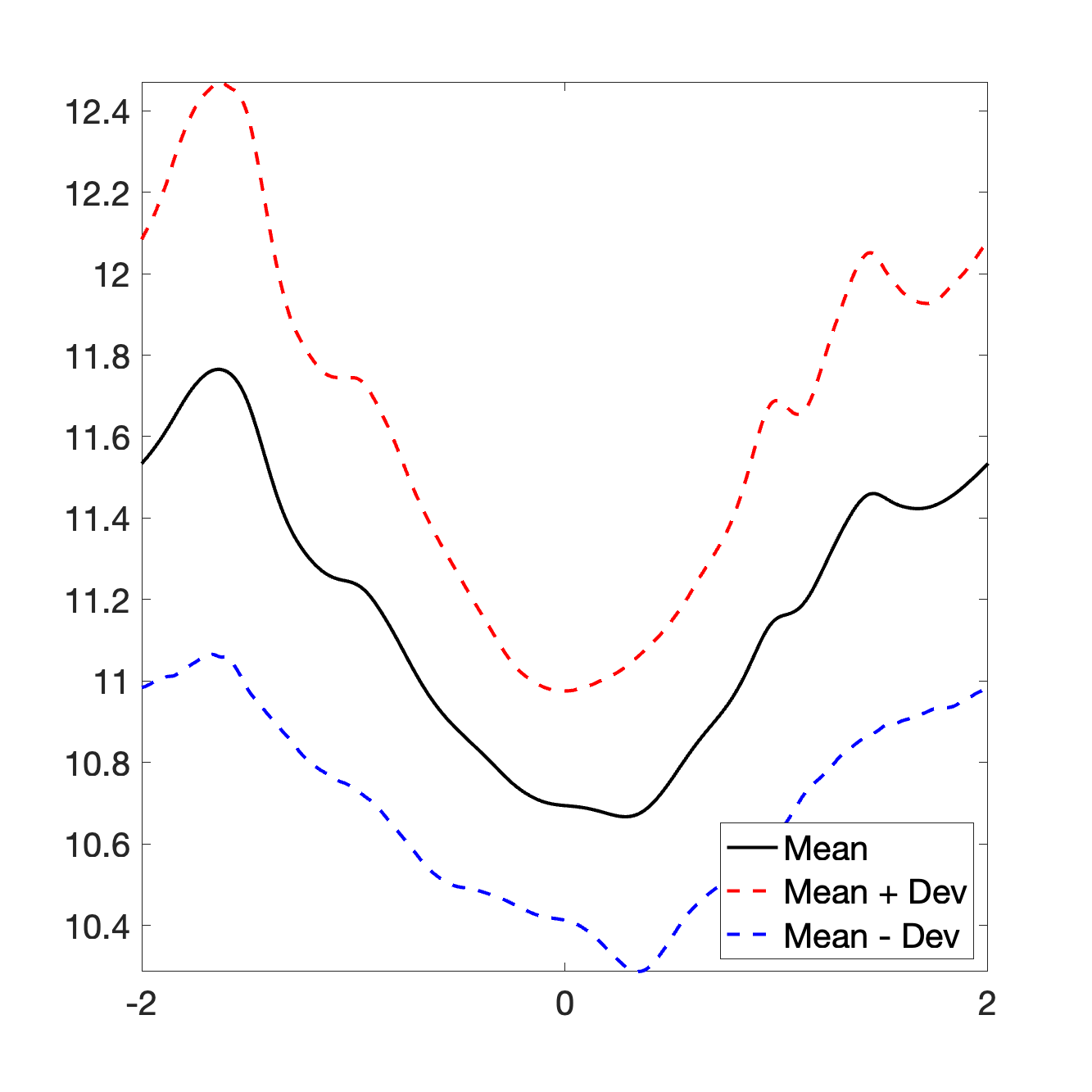}
	\includegraphics[width=0.3\textwidth]{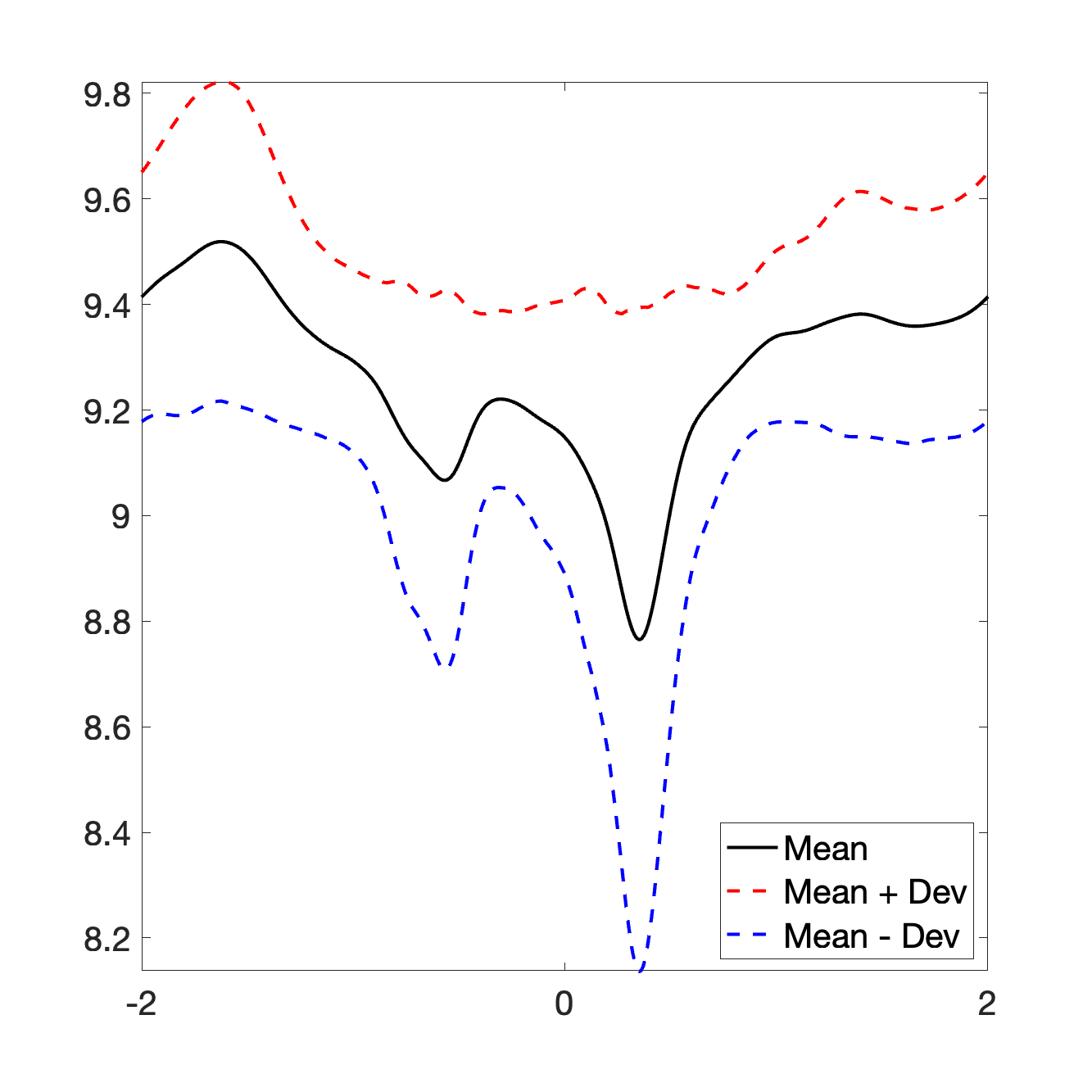}
	\includegraphics[width=0.3\textwidth]{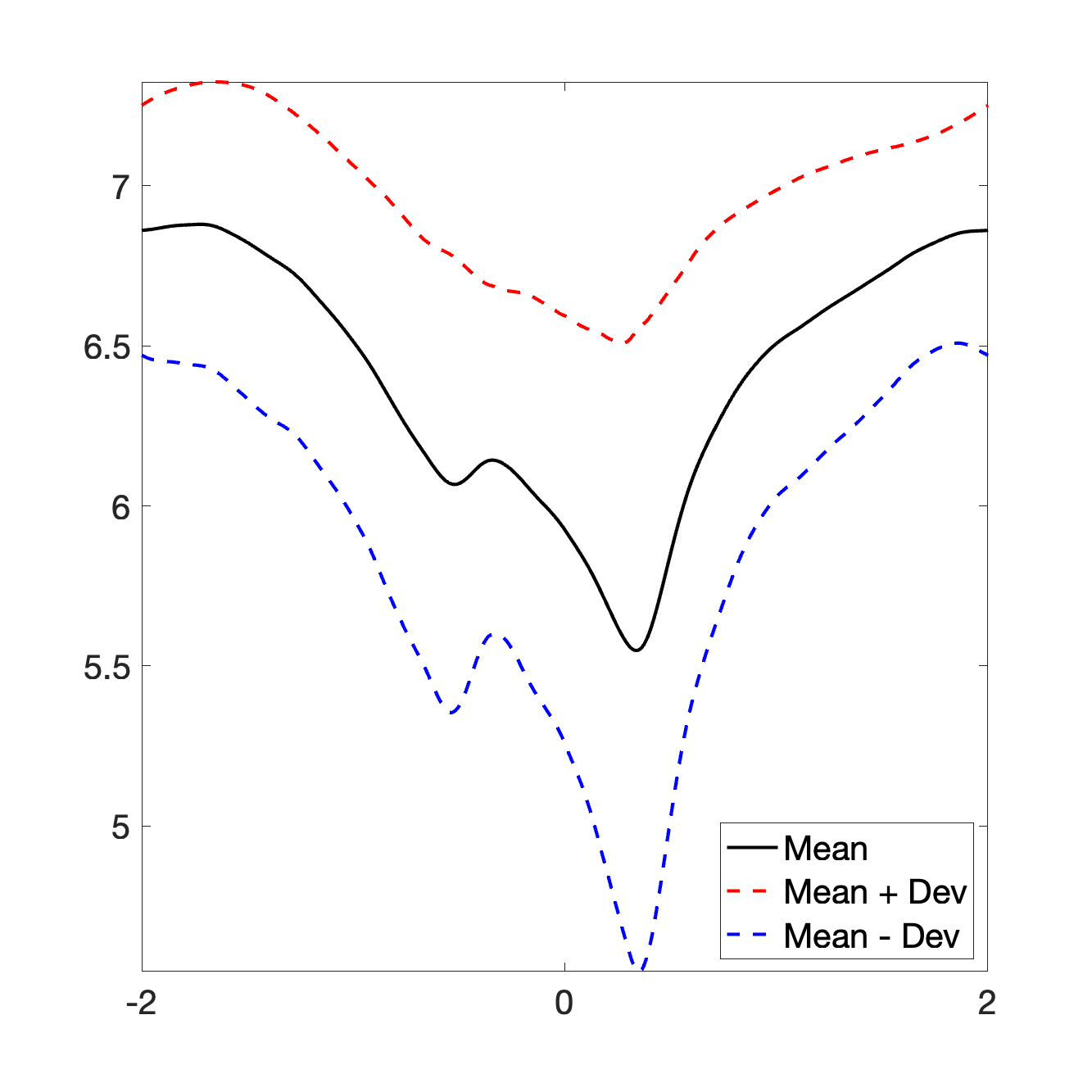}
	\caption{  \small{Rayleigh--B\' enard Experiment 2: $\Ov{\vt_h}(T_{M_{ref}})$ and $\mbox{Dev}(\vt_h,T_{M_{ref}})$. Top: mean (left), deviation (right); Bottom: mean and deviation along lines $y=-3/4$ (left), $0$ (middle), $3/4$ (right).}}\label{fig-Mean-Var-vt}
\end{figure}

\item[(4)] {\bf Errors}. Let us introduce the following error definitions for a single solution and its corresponding mean(-in-time) and deviation(-in-time)
\begin{align*}
&E_1(U, T_M) = \norm{ U_h(T_M,\cdot ) - U_h(T_{M_{ref}},\cdot )}_{L^1(\Omega)}, \\
&E_2(U, T_M) = \norm{\Ov{U_h}(T_{M}) - \Ov{U_h}(T_{M_{ref}})}_{L^1(\Omega)}, \\
&E_3(U, T_M) = \norm{\mbox{Dev}(U_h,T_{M})  - \mbox{Dev}(U_h,T_{M_{ref}}) }_{L^1(\Omega)}. 
\end{align*}
Figure~\ref{fig-Err} presents errors $E_i(U, T_M), i =1,2,3$ with  $U\in \{ \vr, \vm, \vr s, \vu, \vt, E, BE\}$. The numerical results show that the solution does not converge with time increasing, while its corresponding mean(-in-time) and deviation(-in-time) do converge. 

\begin{figure}[htbp]
	\setlength{\abovecaptionskip}{0.cm}
	\setlength{\belowcaptionskip}{-0.cm}
	\centering
	\includegraphics[width=0.32\textwidth]{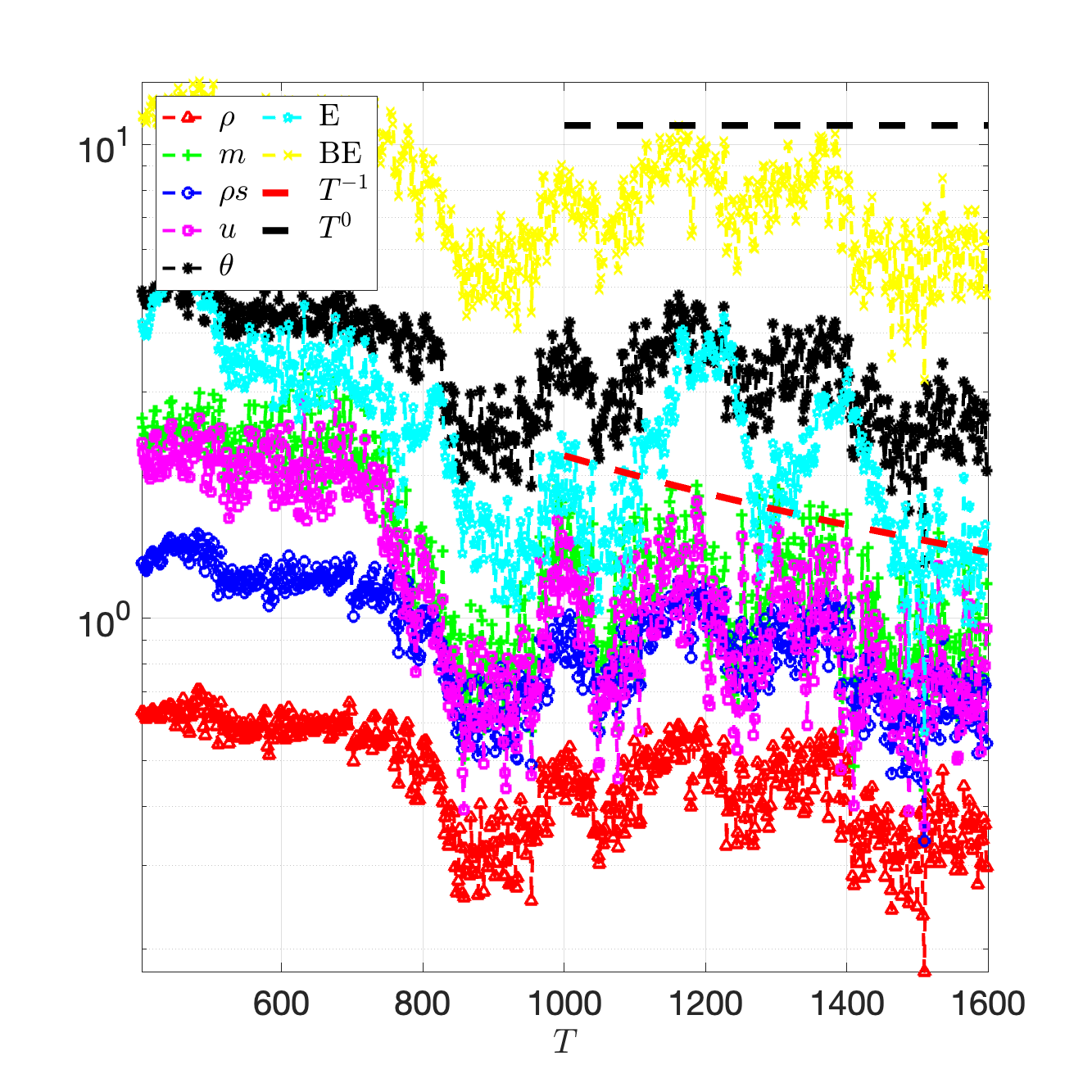} 
	\includegraphics[width=0.32\textwidth]{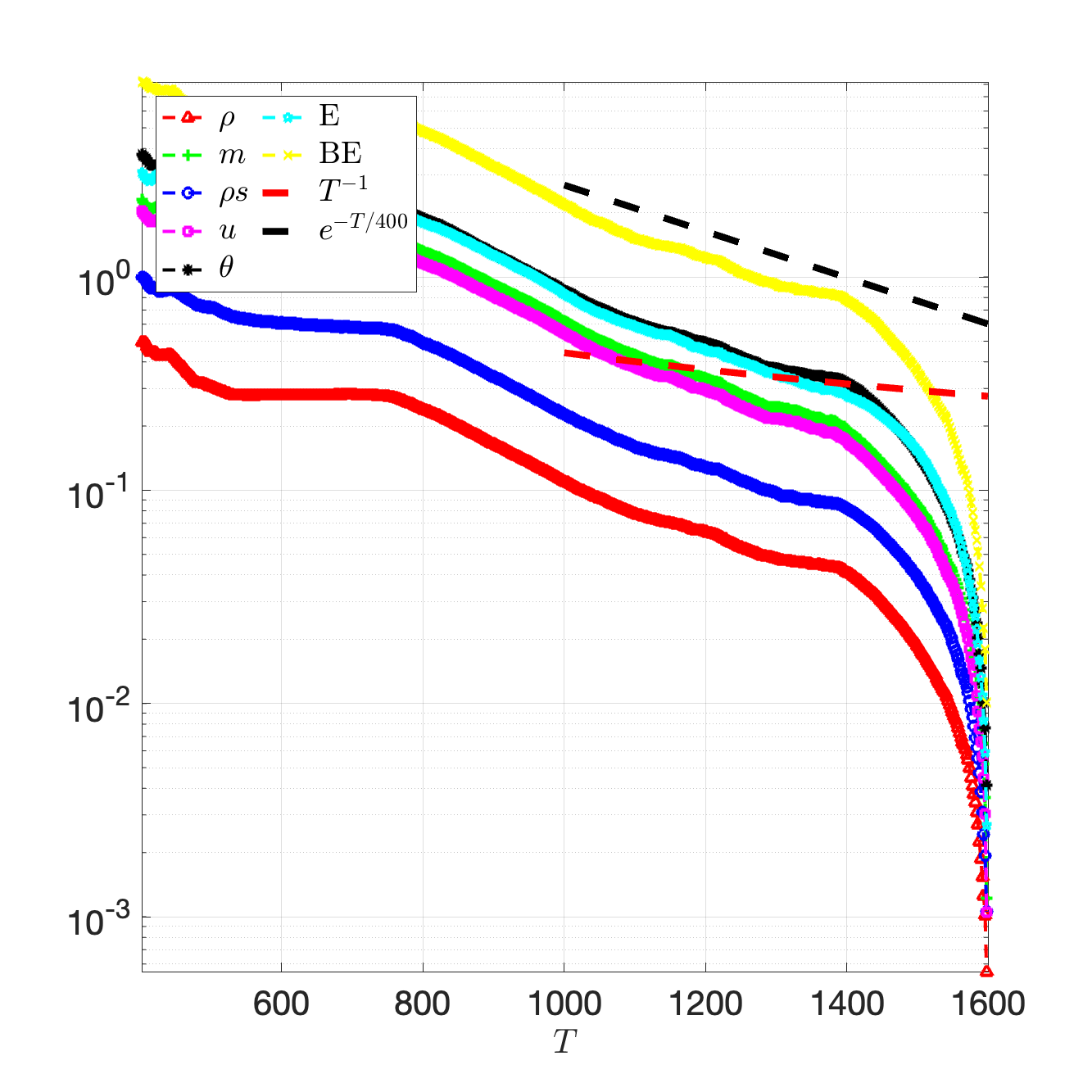}
	\includegraphics[width=0.32\textwidth]{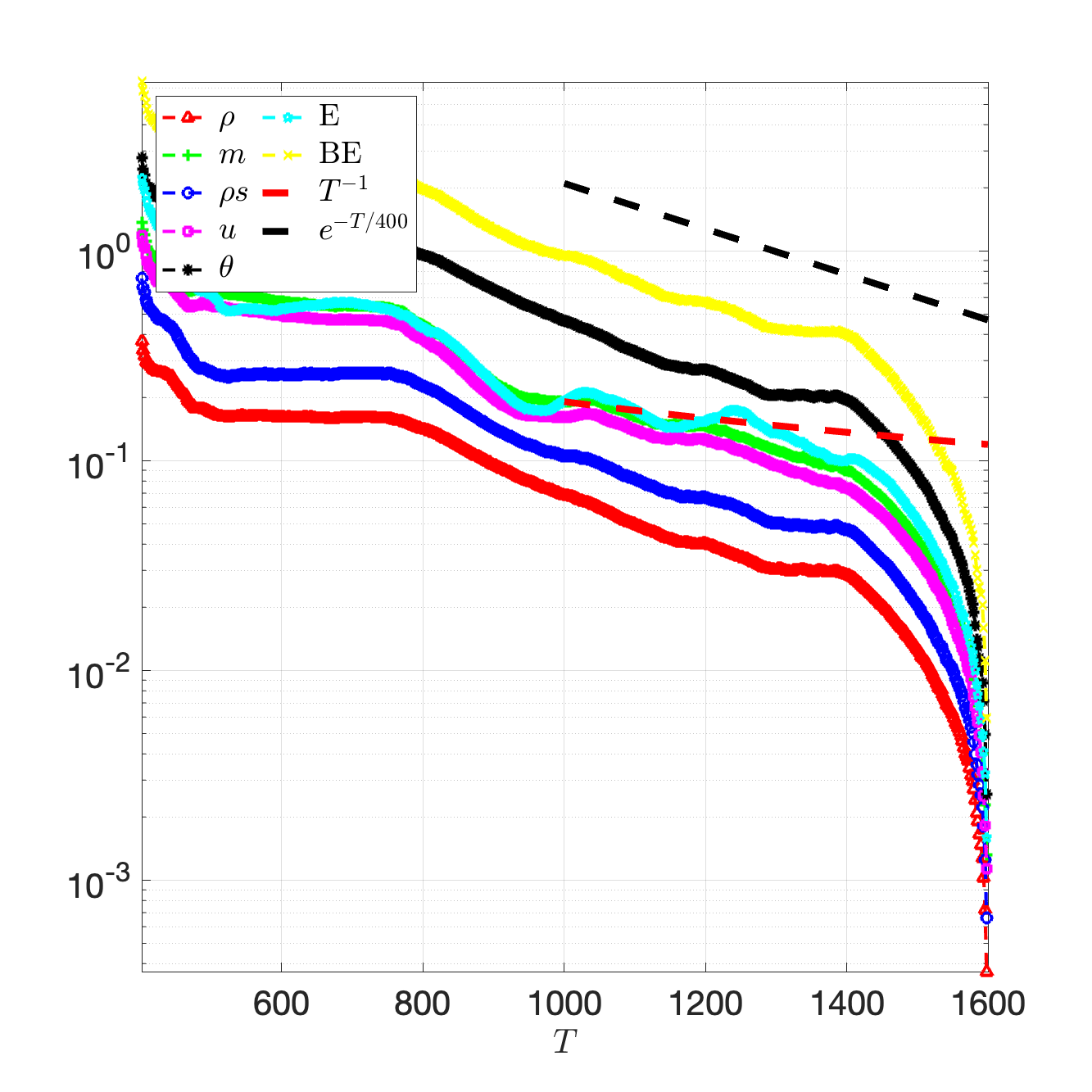}
	\caption{  \small{Rayleigh--B\' enard Experiment 2: errors $E_1, E_2$ and $E_3$ over $[400,1600]$.}}\label{fig-Err}
\end{figure}

\item[(5)] {\bf Reynolds stress and energy fluctuation}. Let us introduce the Reynolds stress and energy fluctuation as follows
\begin{align*}
&\mathfrak{R} \Big(T_{M} \Big) = \Ov{\frac{\vmh \otimes \vmh}{\vrh} + p_h \mathbb{I}}\Big(T_{M} \Big) - \left( \frac{\Ov{\vmh} \otimes \Ov{\vmh}}{\Ov{\vrh}} + p(\Ov{\vrh}, \Ov{S_h}) \mathbb{I} \right) \Big(T_{M} \Big), \\
&\mathfrak{E} \Big(T_{M} \Big) = \Ov{E_h}\Big(T_{M} \Big) - E\Big(\Ov{\vrh}(T_{M}),\, \Ov{\vmh}(T_{M}),\, \Ov{S_h}(T_{M}) \Big).
\end{align*}
Figure~\ref{fig-Defect-1} shows the evolution of $L^1$-norm and $L^{\infty}$-norm of Reynolds stress and energy fluctuation $\mathfrak{R}_{11}, \mathfrak{R}_{12}$, $\mathfrak{R}_{22}$, $\mathfrak{E}$, $\mbox{tr}(\mathfrak{R})$, $\lambda_1(\mathfrak{R})$, $\lambda_2(\mathfrak{R})$.
The details in a long run, i.e.\ $\mathfrak{R} (T_{M_{ref}}), \mathfrak{E} (T_{M_{ref}} )$, are demonstrated in Figure~\ref{fig-Defect}. 

Figure~\ref{fig-Defect-2} shows the $L^1$- and $L^{\infty}$-errors of Reynolds stress and energy fluctuation, defined by
\[
E_4(D, T_M) = \norm{D_h(T_{M}) - D_h(T_{M_{ref}})}_{L^p(\Omega)}, \ 
D\in\{ \mathfrak{R}_{11}, \mathfrak{R}_{12}, \mathfrak{R}_{22}, \mathfrak{E},\mbox{tr}(\mathfrak{R}),\lambda_1(\mathfrak{R}), \lambda_2(\mathfrak{R}) \}.
\]
Numerical results show that Reynolds stress and energy fluctuation do converge with time increasing. 
 
\begin{figure}[htbp]
	\setlength{\abovecaptionskip}{0.cm}
	\setlength{\belowcaptionskip}{-0.cm}
	\centering
	\includegraphics[width=0.49\textwidth]{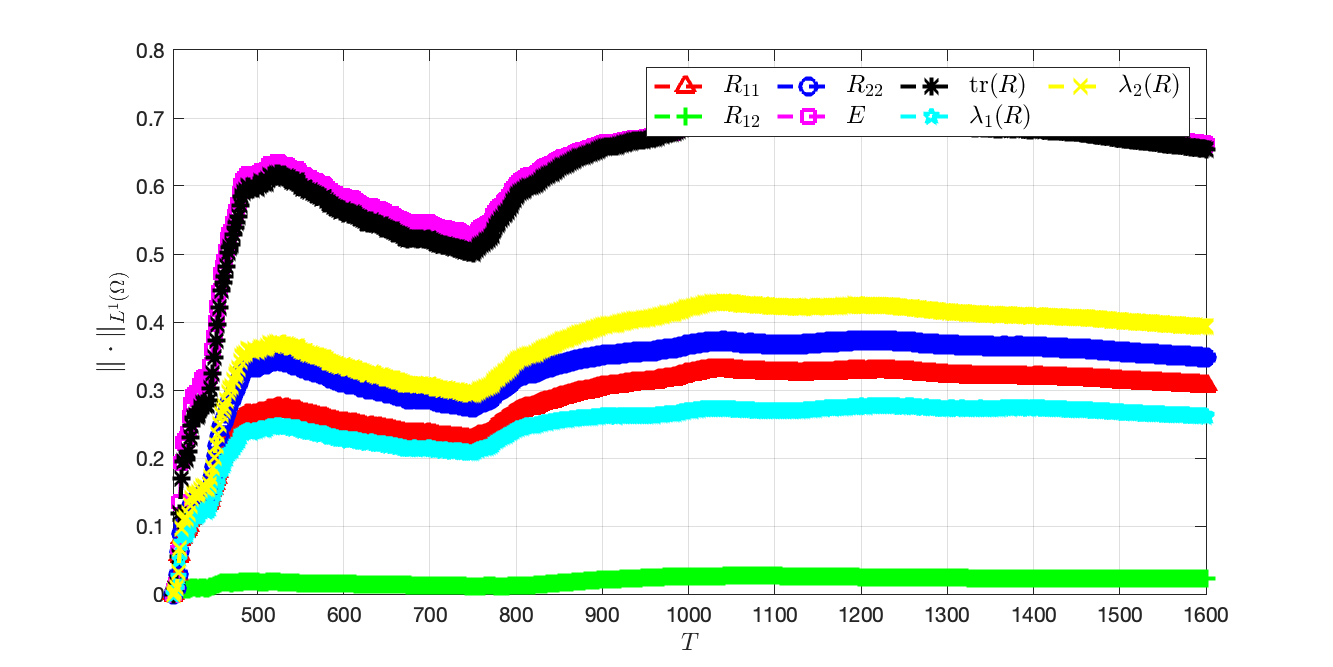} 
	\includegraphics[width=0.49\textwidth]{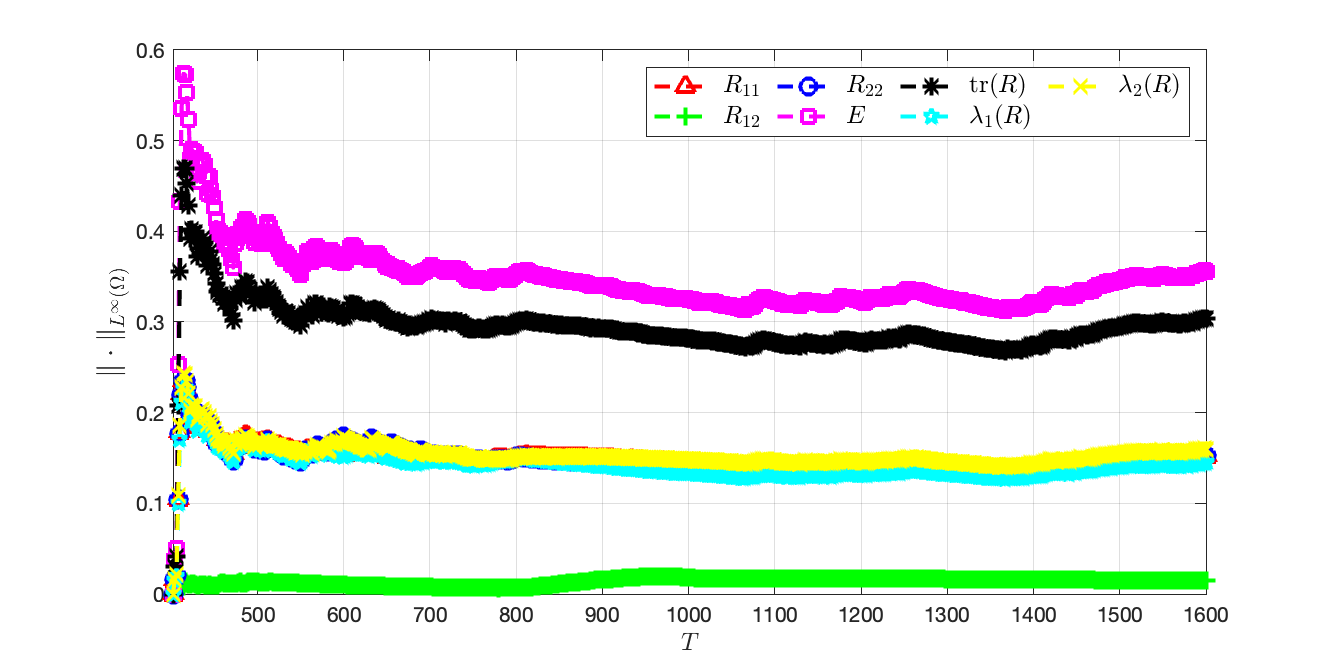}
	\caption{  \small{Rayleigh--B\' enard Experiment 2: Evolution of $L^1$-norm (left) and $L^{\infty}$-norm (right) of Reynolds stress and energy fluctuation.}}\label{fig-Defect-1}
\end{figure}

\begin{figure}[htbp]
	\setlength{\abovecaptionskip}{0.cm}
	\setlength{\belowcaptionskip}{-0.cm}
	\centering
	\includegraphics[width=0.49\textwidth]{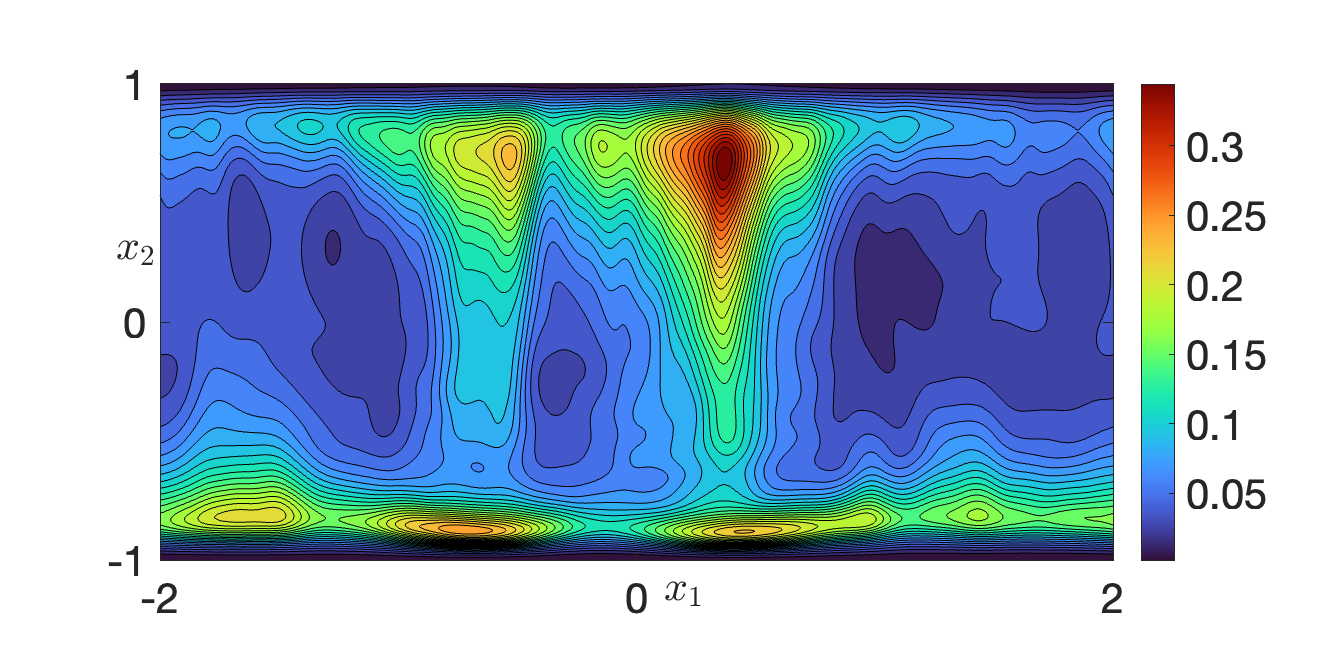} 
	\includegraphics[width=0.49\textwidth]{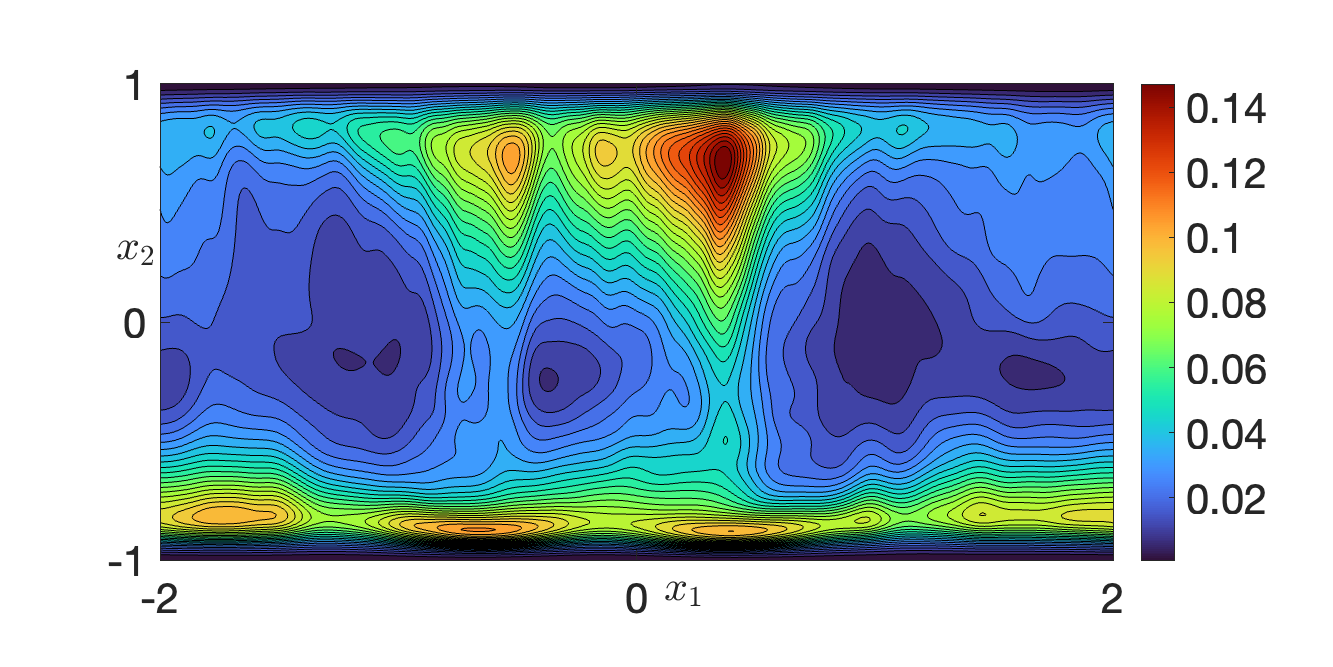}\\
	\includegraphics[width=0.49\textwidth]{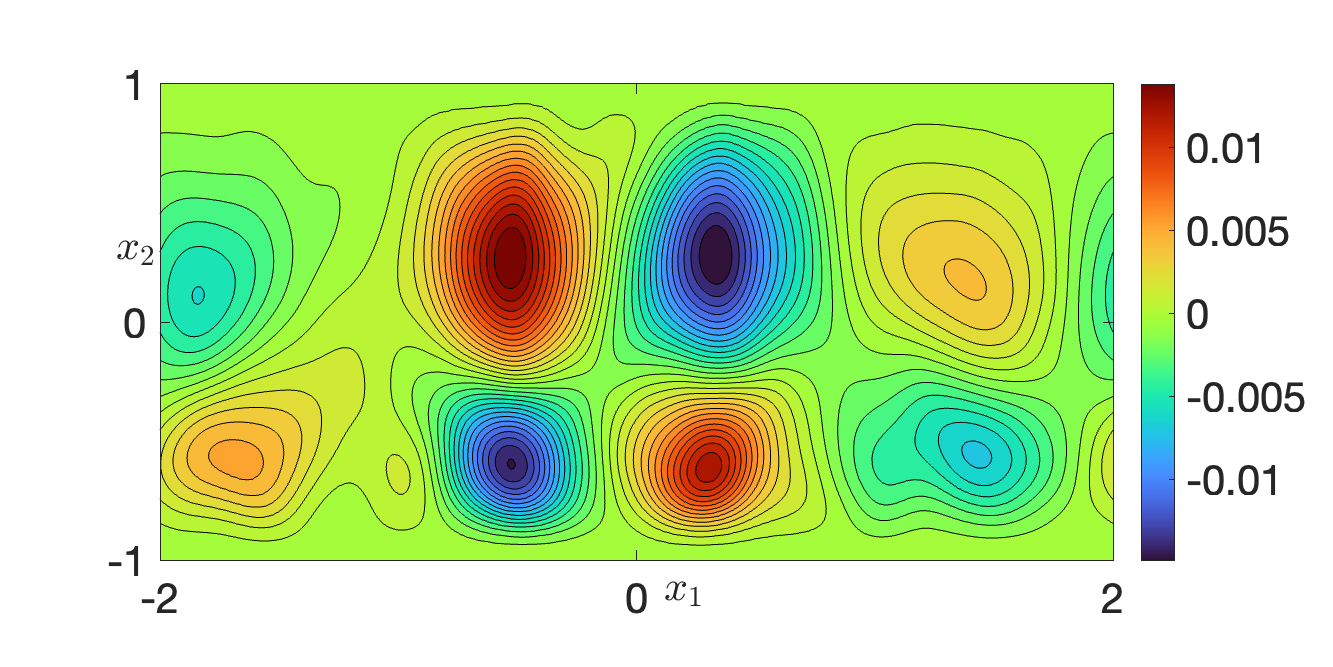}
	\includegraphics[width=0.49\textwidth]{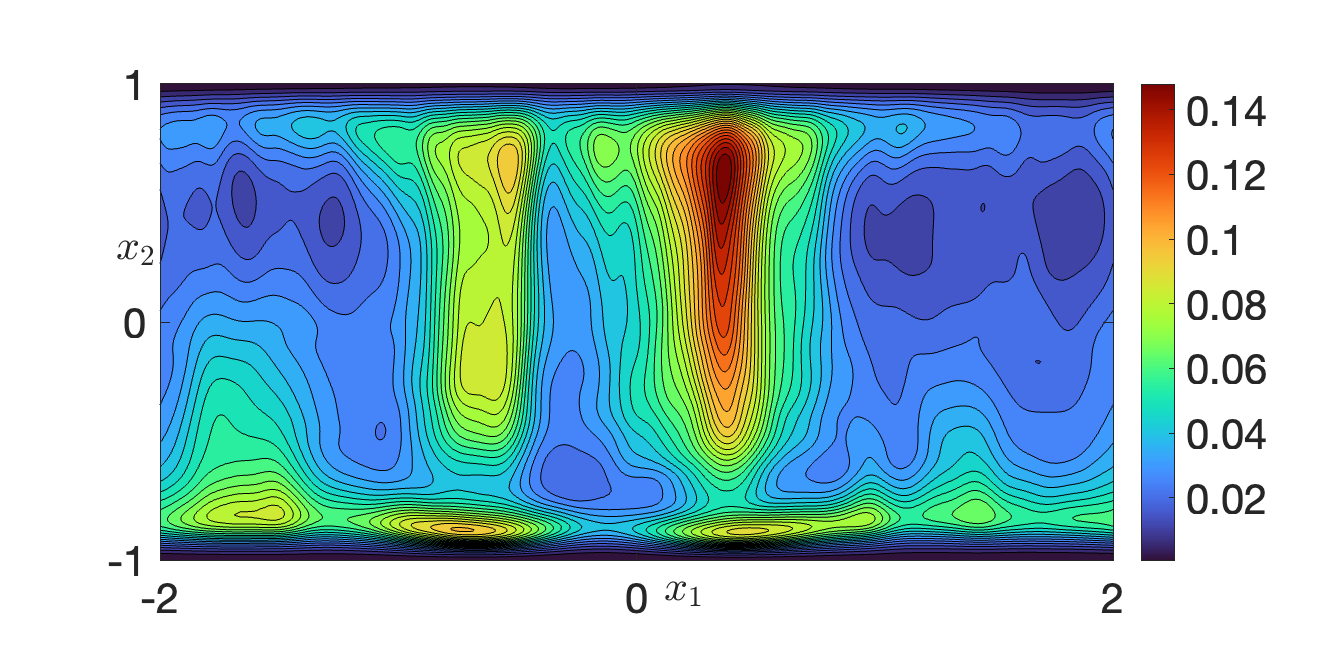}\\
	\includegraphics[width=0.49\textwidth]{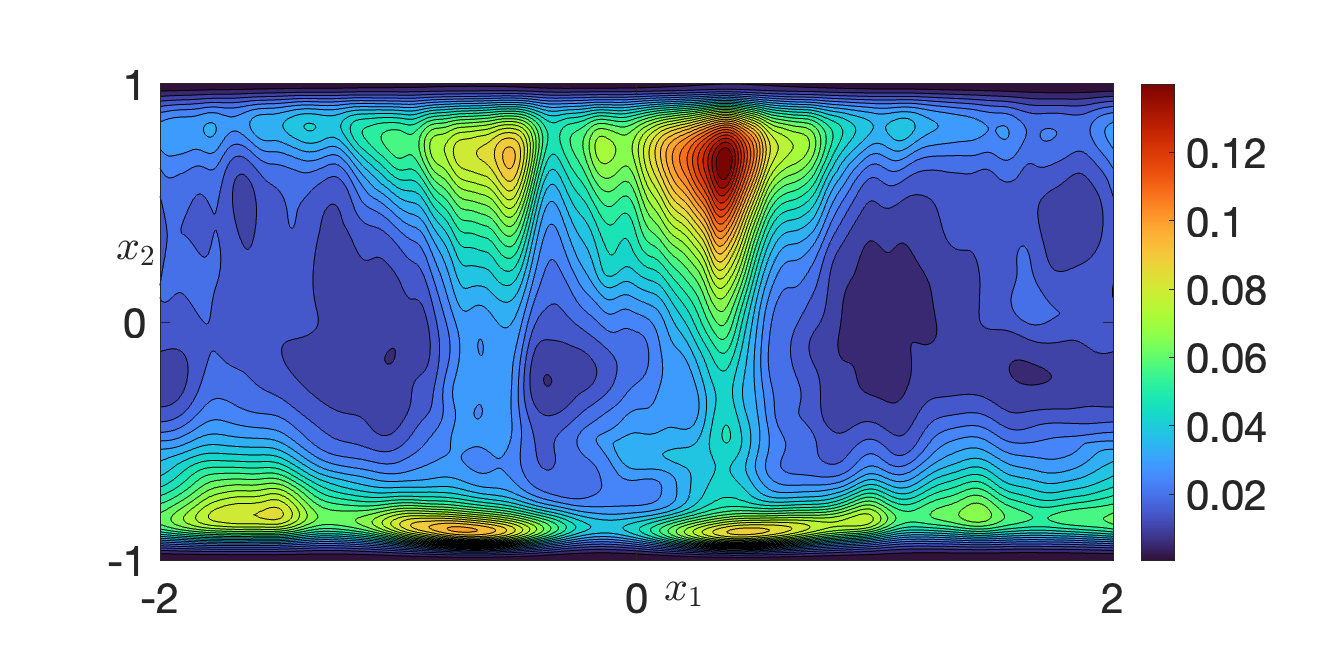}
	\includegraphics[width=0.49\textwidth]{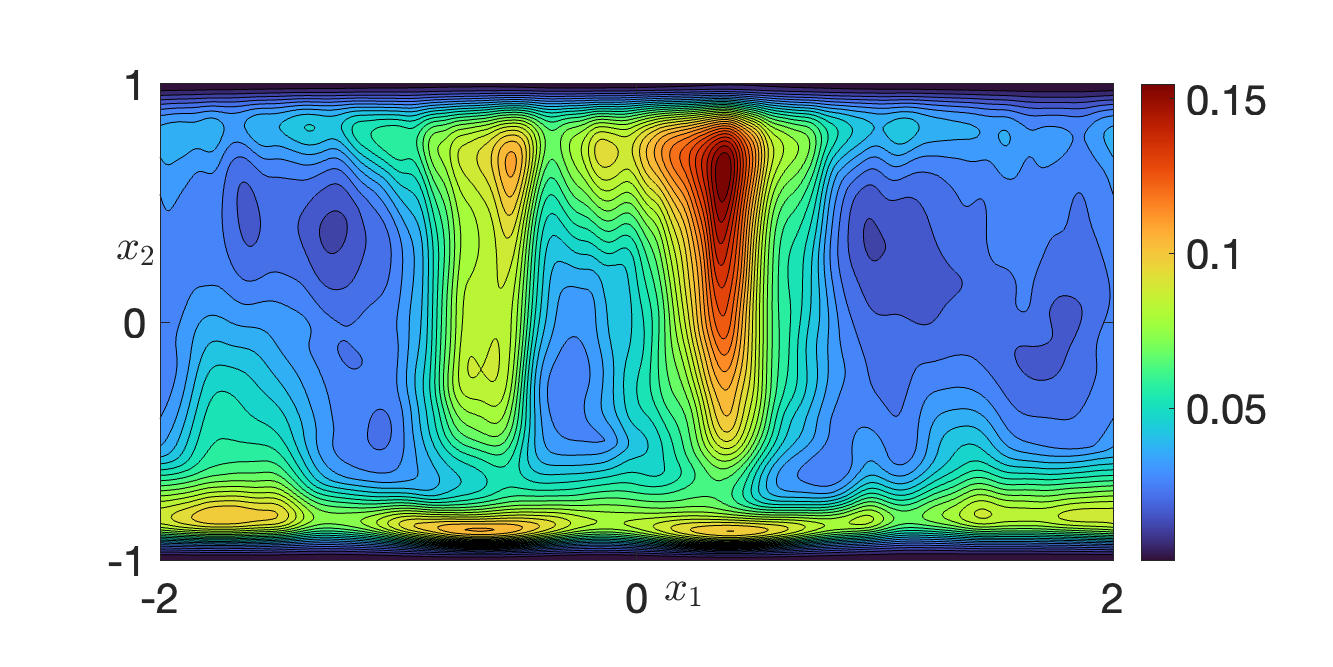}
	\caption{  \small{Rayleigh--B\' enard Experiment 2: Reynolds stress and energy fluctuation. From left to right, from top to bottom: $\mathfrak{E}(T_{M_{ref}}), \mathfrak{R}_{11}(T_{M_{ref}}), \mathfrak{R}_{12}(T_{M_{ref}}), \mathfrak{R}_{22}(T_{M_{ref}})$ and the eigenvalues  $\lambda_1, \lambda_2, (\lambda_1 \leq \lambda_2)$ of Reynolds stress $\mathfrak{R}(T_{M_{ref}})$.}}\label{fig-Defect}
\end{figure}

\begin{figure}[htbp]
	\setlength{\abovecaptionskip}{0.cm}
	\setlength{\belowcaptionskip}{-0.cm}
	\centering
	\includegraphics[width=0.45\textwidth]{./gif/case0/PRONUM52MX640MY320IS201IE800_Defects_Err_norm1} 
	\includegraphics[width=0.45\textwidth]{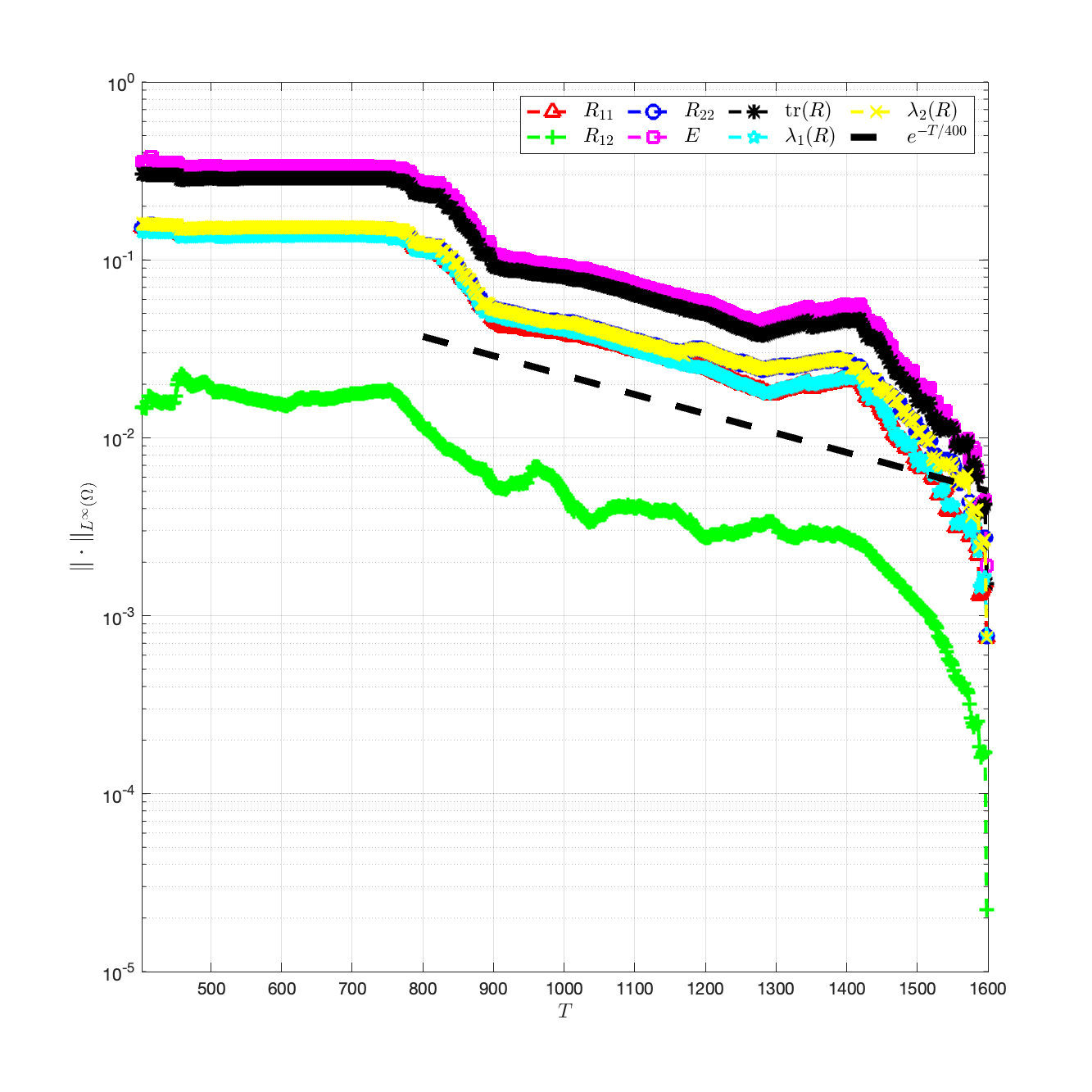}
	\caption{  \small{Rayleigh--B\' enard Experiment 2:  $L^1$- (left) and $L^{\infty}$-errors (right) of Reynolds stress and energy fluctuation.}}\label{fig-Defect-2}
\end{figure}

\newpage
\item[(6)] {\bf Measures}. Let us introduce the following definitions of measures (with respect to time)
\begin{align*}
& \mathcal{M}\left(F(U)\right) = \frac{\#\{ m \, | \, F(U_h(T_m,\cdot)) \in [a,b], \  m = M_0+1,\dots, M_{ref} \}}{M_{ref} - M_0} \quad   
\end{align*}
with three choices of $F$
\begin{align}
F_1(U_h(t,\cdot)) = \norm{U_h(t,\cdot)}_{L^1(\Omega)}, \quad 
F_2(U_h(t,\cdot)) = \int_{\Omega} U_h(t,\cdot) \dx, \quad 
F_3(U) = U_h(t,x).
\end{align}
Figures~\ref{fig-Measure-Ex2-1} and \ref{fig-Measure-Ex2-2} show measures of $L^1$-norms $\mathcal{M}\left(\norm{U_h(t,\cdot)}_{L^1(\Omega)} \right)$ and spatial-averages $\mathcal{M}\left(\int_{\Omega} U_h(t,\cdot) \dx \right)$ of solutions $(\vrh,\vuh,\vt_h)$. 
Further, the measures at six fixed ``spatial-points" 
\begin{align*}
& P_1 = (-1.4, -0.8),  \quad   P_2 =  (-1.4,  0), \quad P_3 =  (-1.4,  0.8),\\ 
& P_4 = (-0.8, -0.8),  \quad   P_5 =  (-0.8,  0), \quad P_6 =  (-0.8,  0.8)
\end{align*}
are shown in Figures~\ref{fig-Measure-Ex2-3}. 
Here we view the point $P = (x_1,x_2)$ to be a square $[x_1-h,x_1+h]\times [x_2-h,x_2+h]$ that contains 4 cells.

\begin{figure}[htbp]
	\setlength{\abovecaptionskip}{0.cm}
	\setlength{\belowcaptionskip}{-0.cm}
	\centering
	\includegraphics[width=0.48\textwidth]{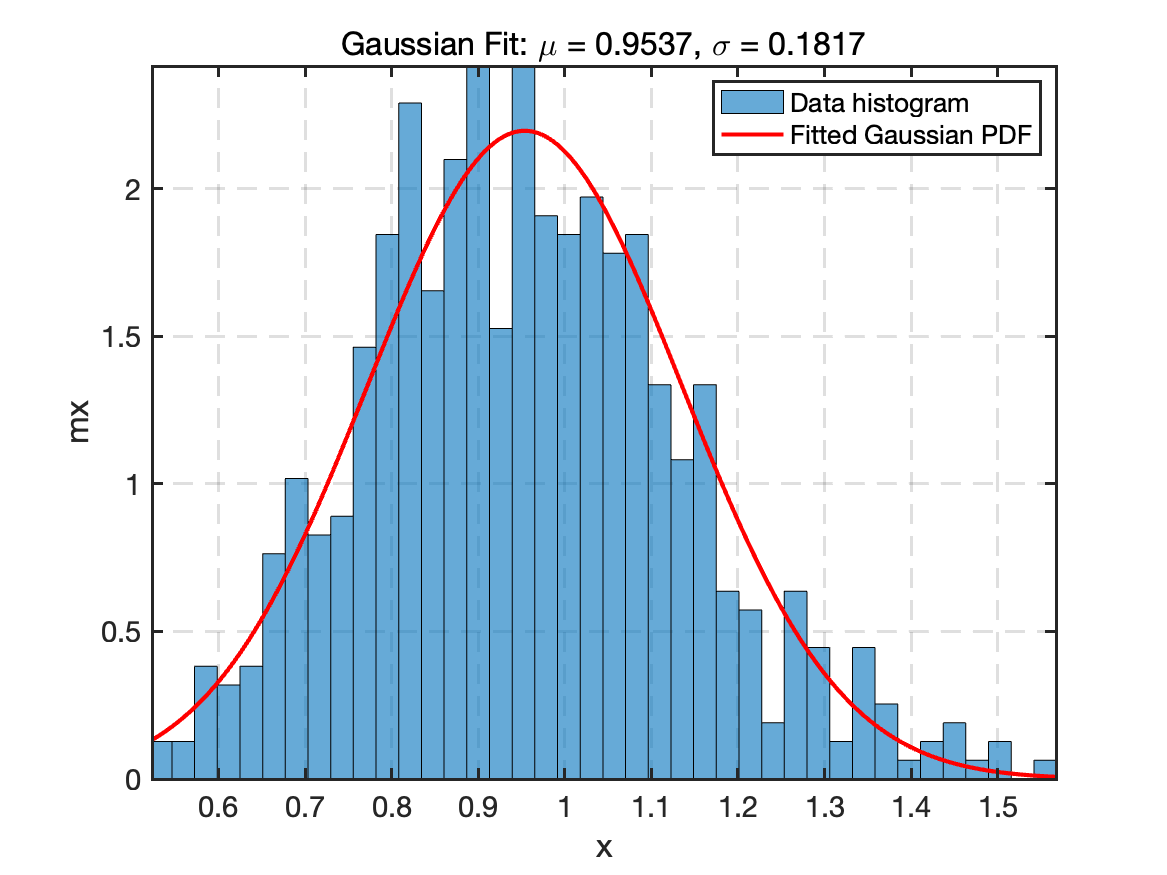}
	\includegraphics[width=0.48\textwidth]{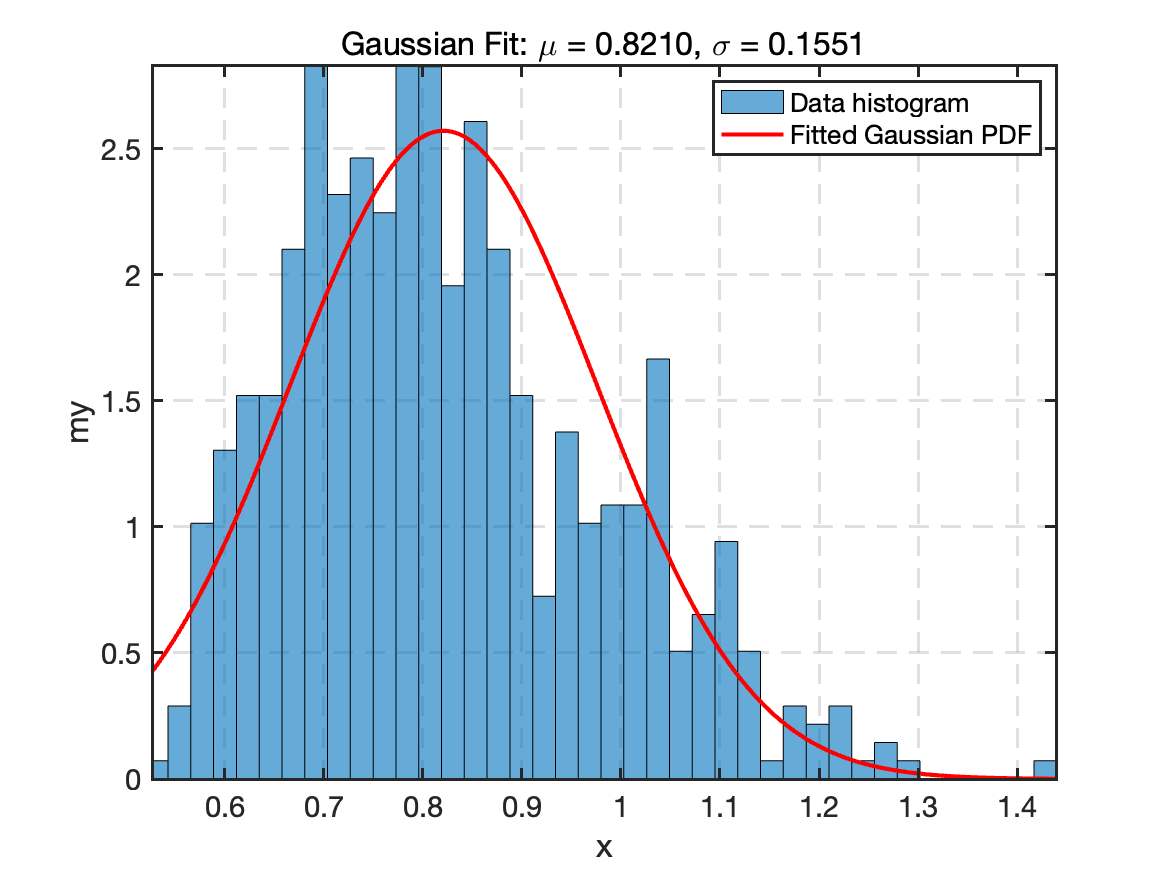} 
	\includegraphics[width=0.48\textwidth]{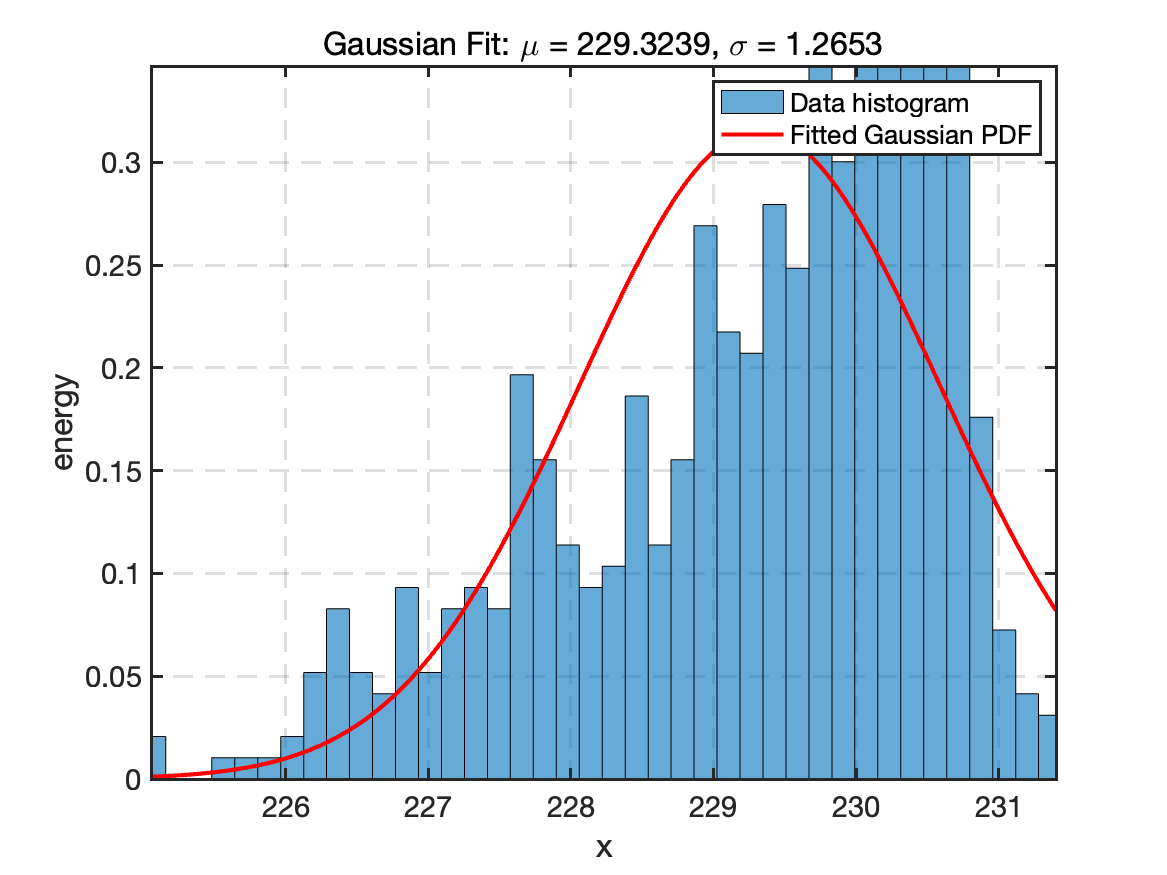}	
	\includegraphics[width=0.48\textwidth]{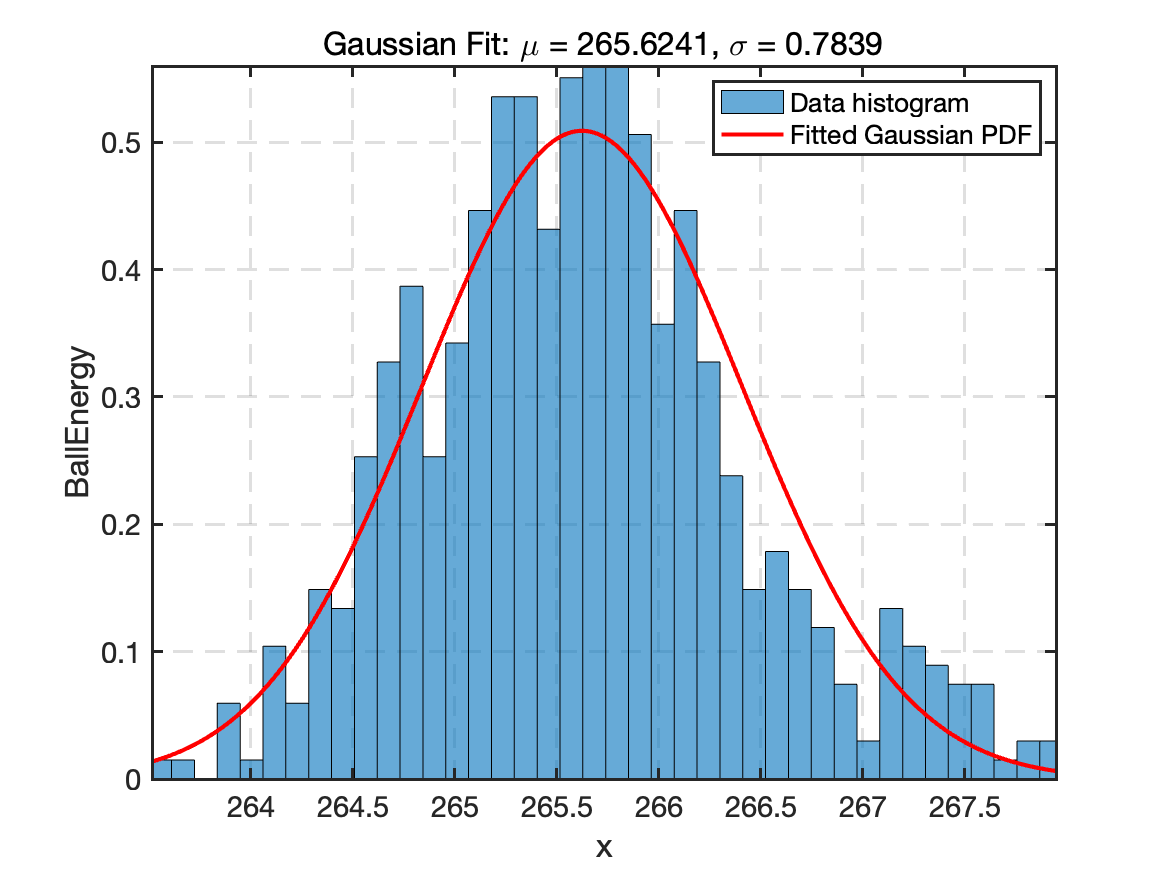}
	\includegraphics[width=0.48\textwidth]{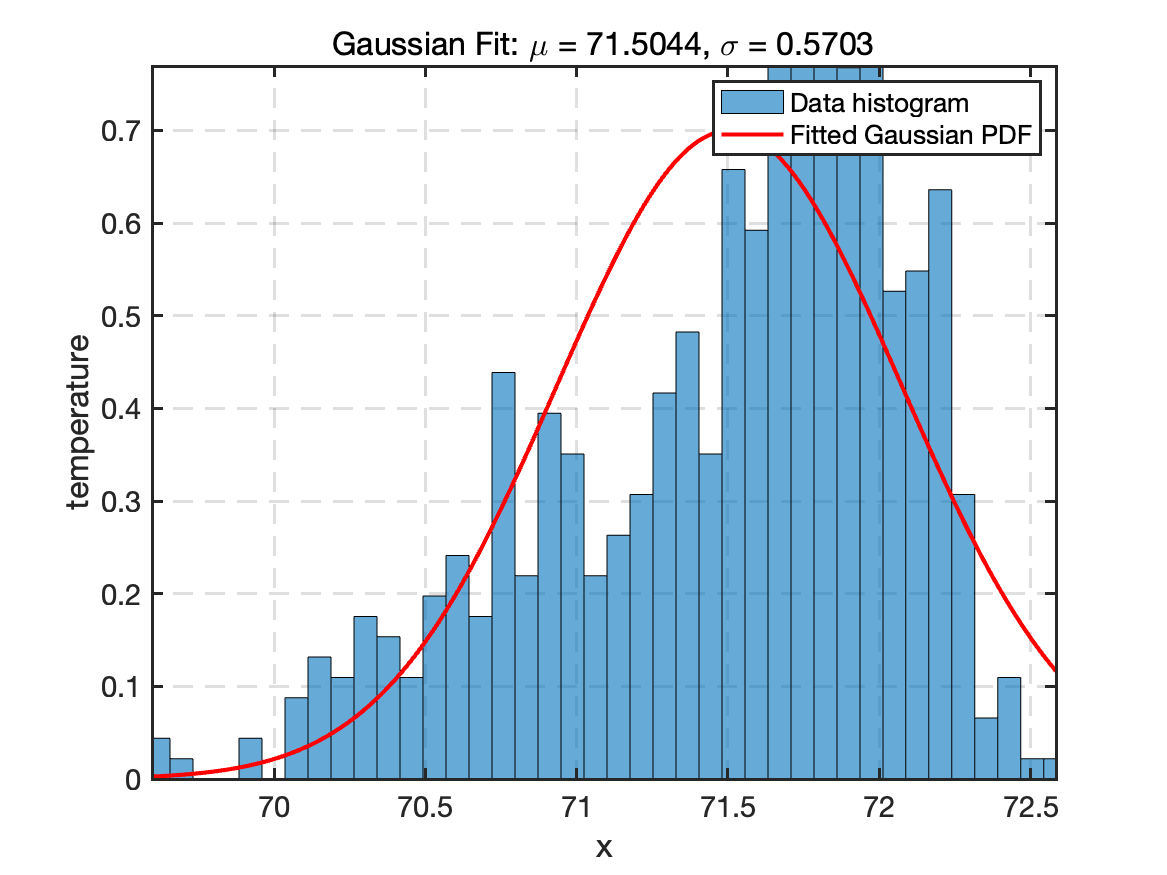}
	\includegraphics[width=0.48\textwidth]{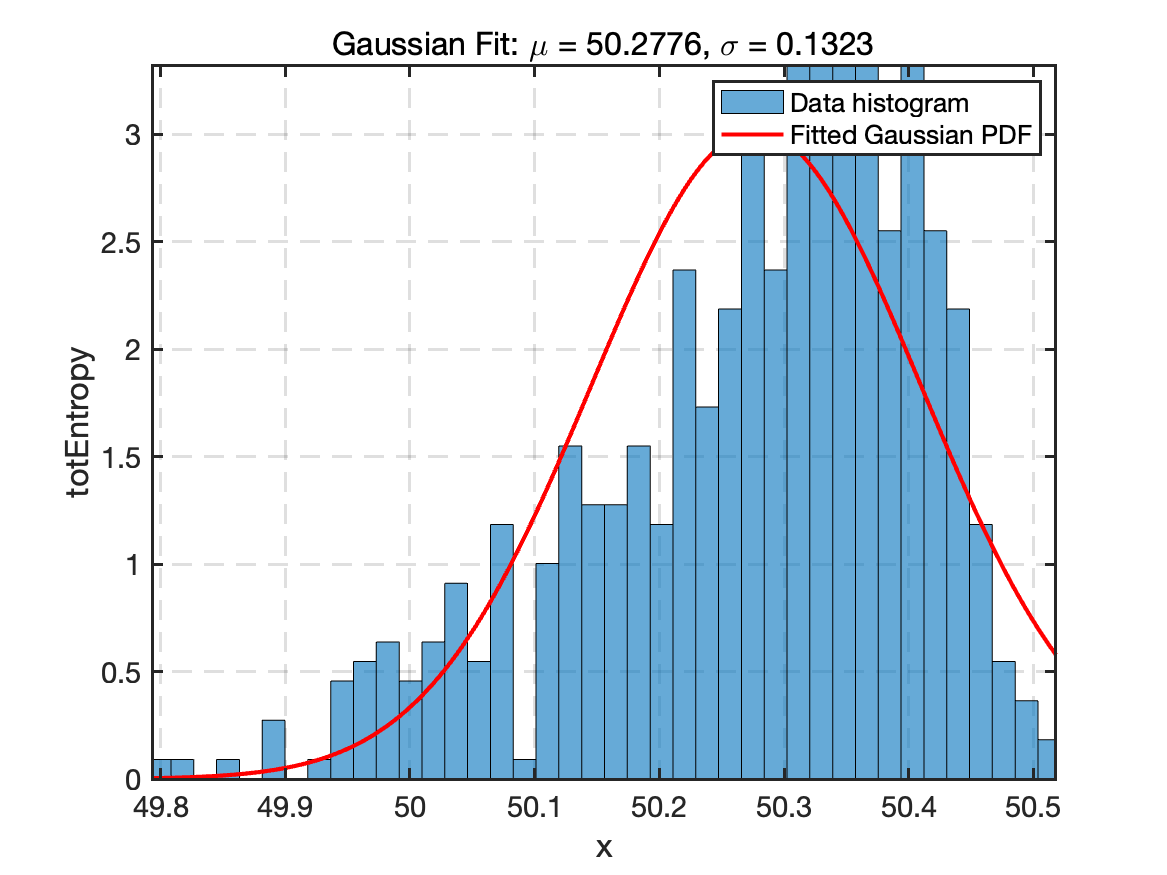} 
	\caption{ \small{Rayleigh--B\' enard Experiment 2: Measures $\mathcal{M}\left(\norm{U_h(t,\cdot)}_{L^1(\Omega)} \right)$ with $U \in \{m_x,m_y,E,BE,\vt, S\}$.}}\label{fig-Measure-Ex2-1}
\end{figure}

\begin{figure}[htbp]
	\setlength{\abovecaptionskip}{0.cm}
	\setlength{\belowcaptionskip}{-0.cm}
	\centering
	\includegraphics[width=0.48\textwidth]{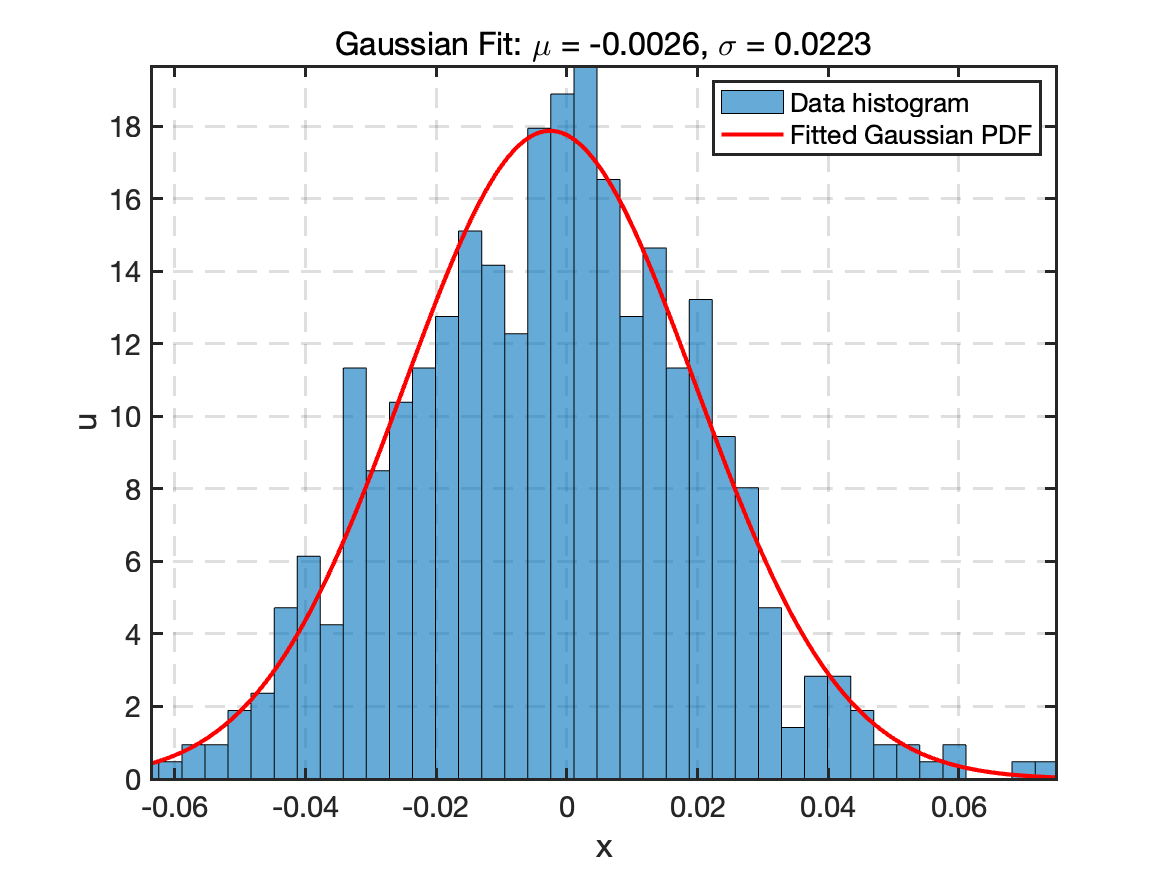}
	\includegraphics[width=0.48\textwidth]{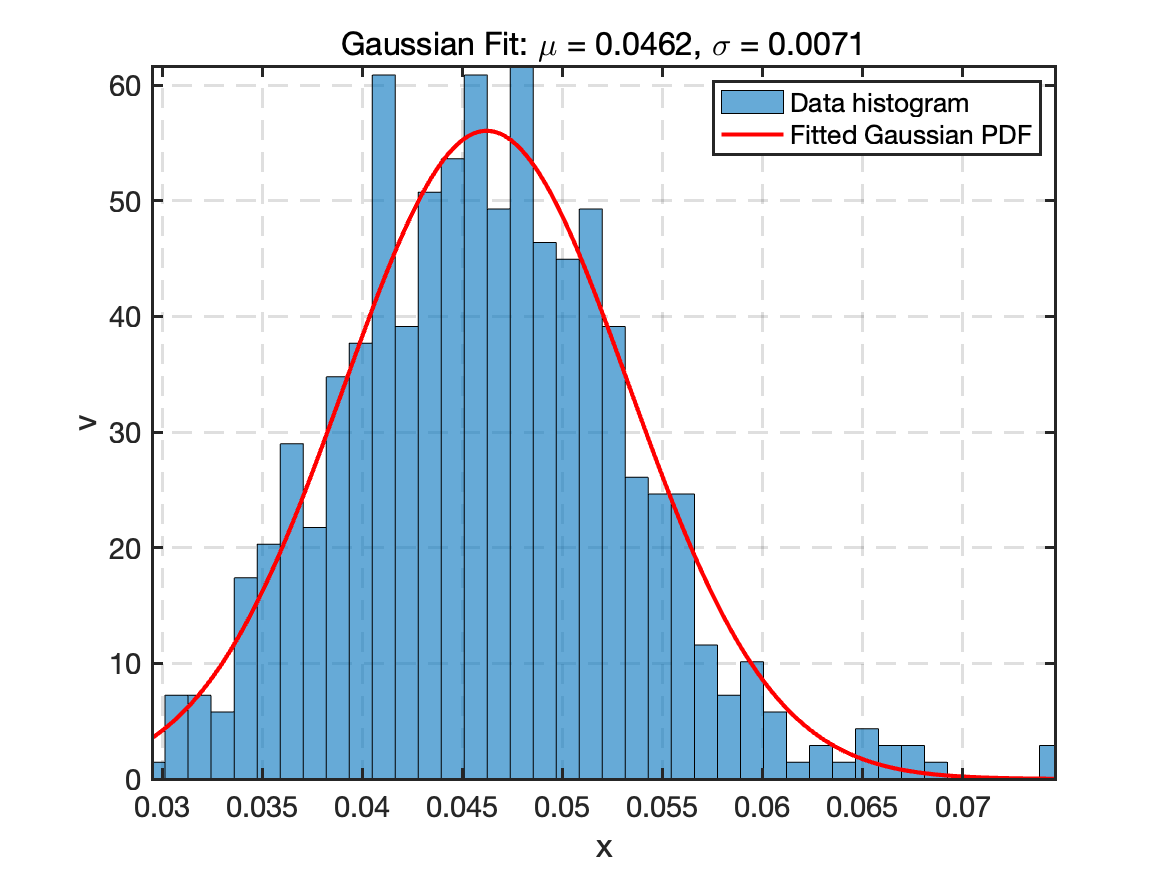} 	
	\includegraphics[width=0.48\textwidth]{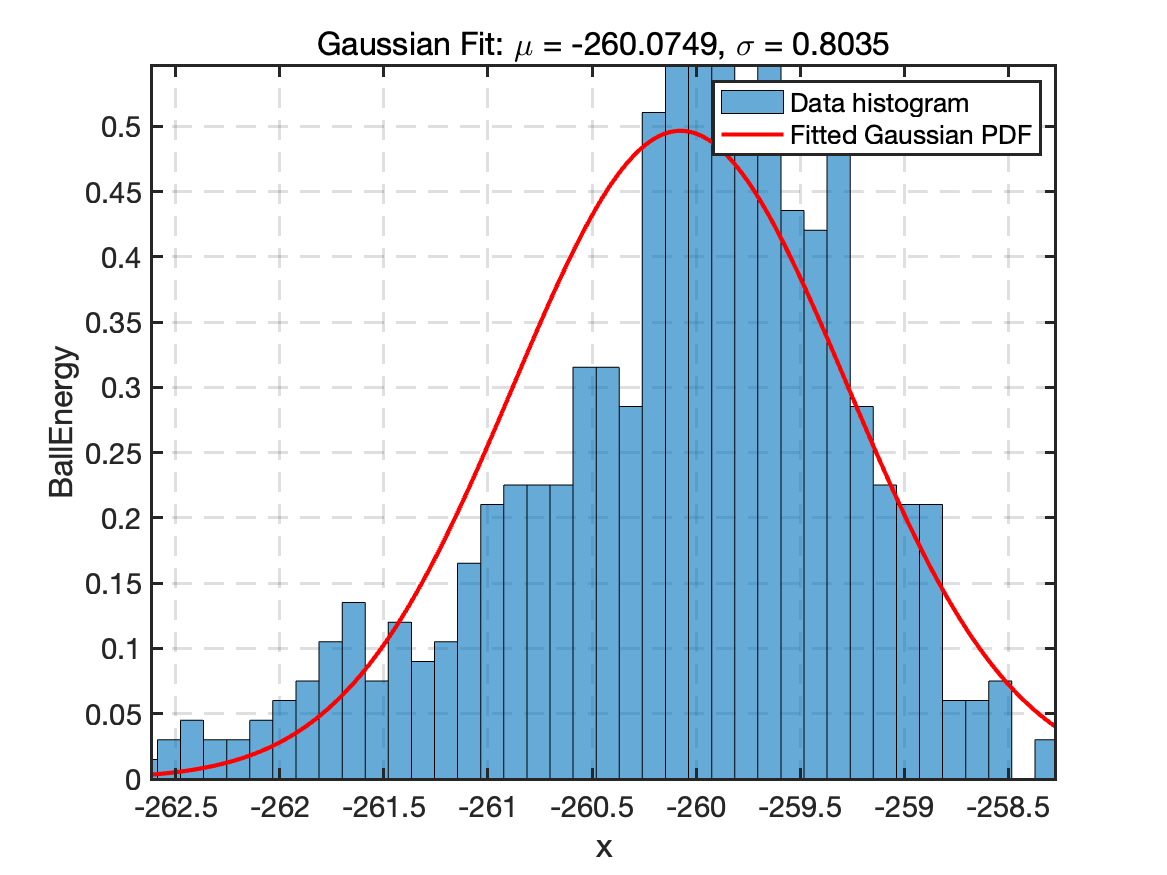}
	\includegraphics[width=0.48\textwidth]{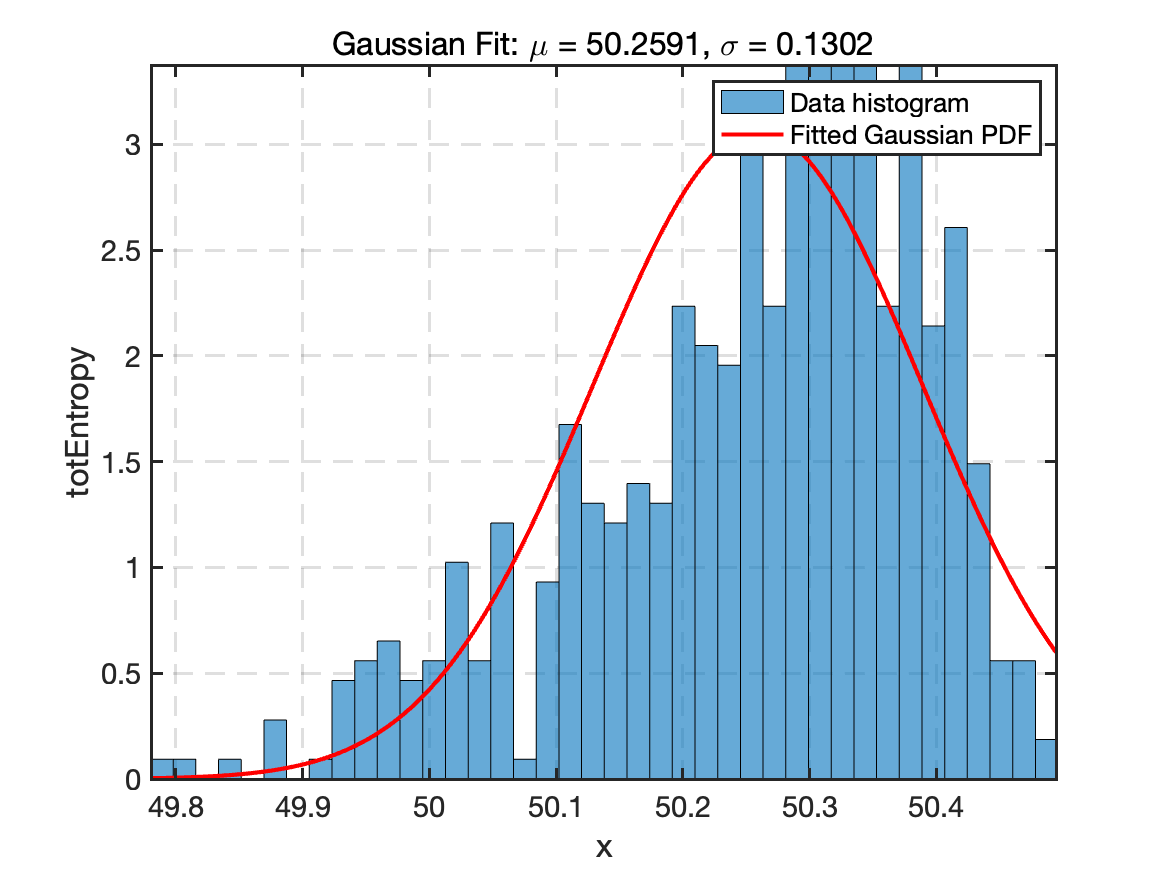} 
	\caption{ \small{Rayleigh--B\' enard Experiment 2: Measures $\mathcal{M}\left(\int_{\Omega} U_h(t,\cdot) \dx \right)$ with $U \in \{u_1,u_2,BE,S\}$.}}\label{fig-Measure-Ex2-2}
\end{figure}

\begin{figure}[htbp]
	\setlength{\abovecaptionskip}{0.cm}
	\setlength{\belowcaptionskip}{-0.cm}
	\centering
	\includegraphics[width=0.24\textwidth]{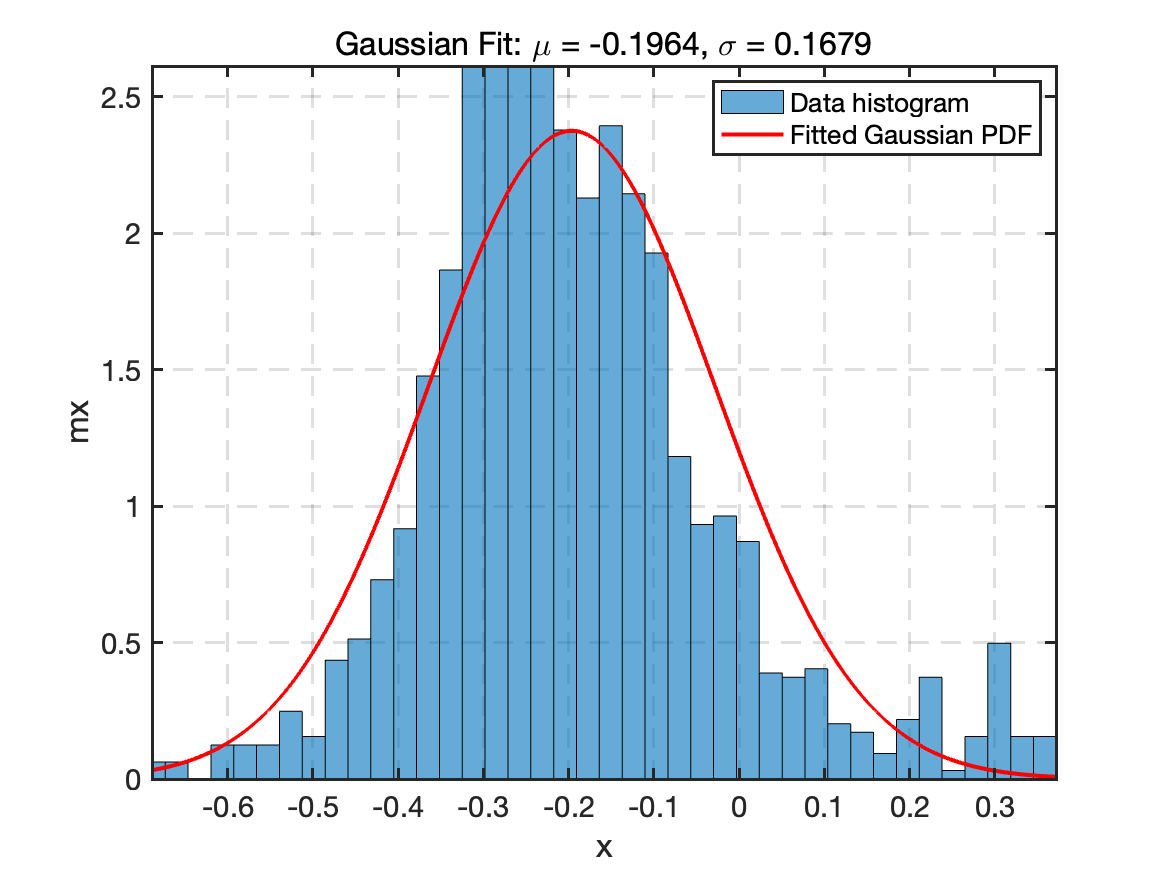}
	\includegraphics[width=0.24\textwidth]{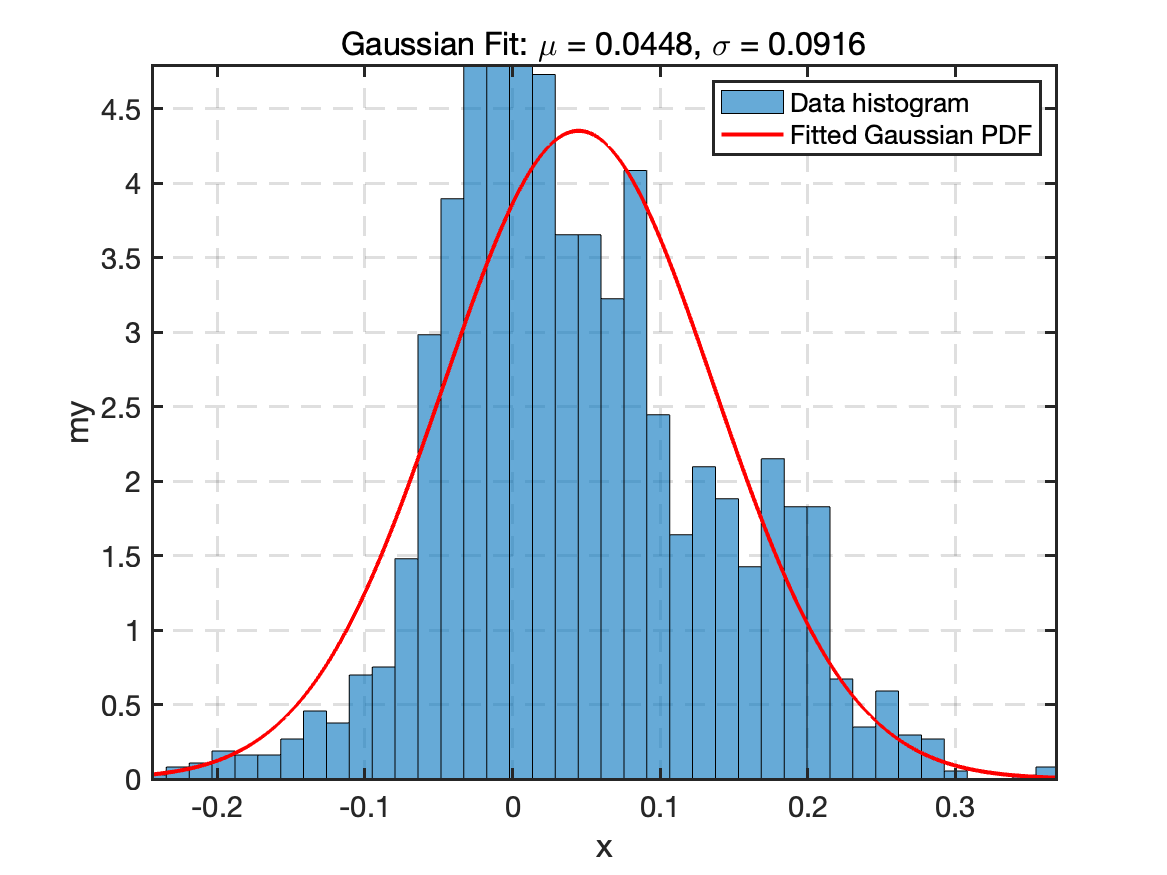} 	
	\includegraphics[width=0.24\textwidth]{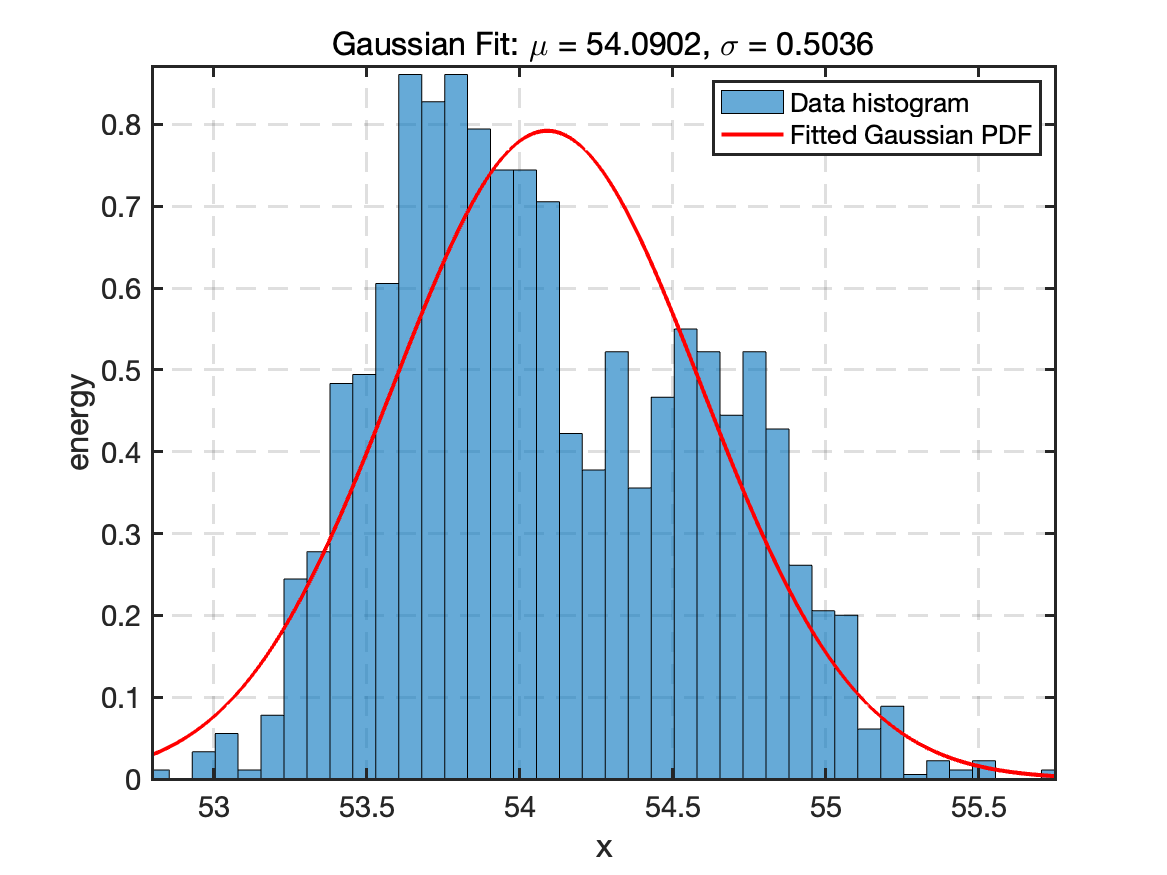}
	\includegraphics[width=0.24\textwidth]{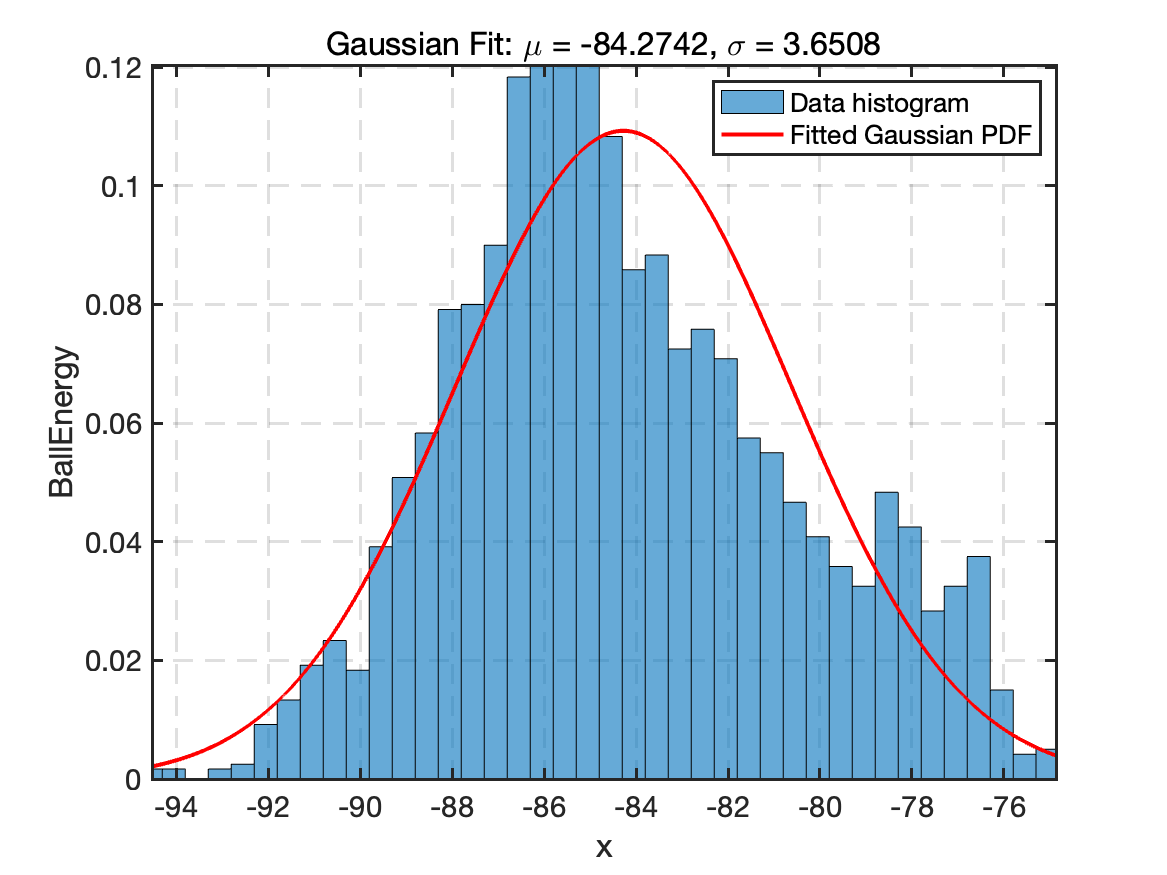} 
	\\
	\includegraphics[width=0.24\textwidth]{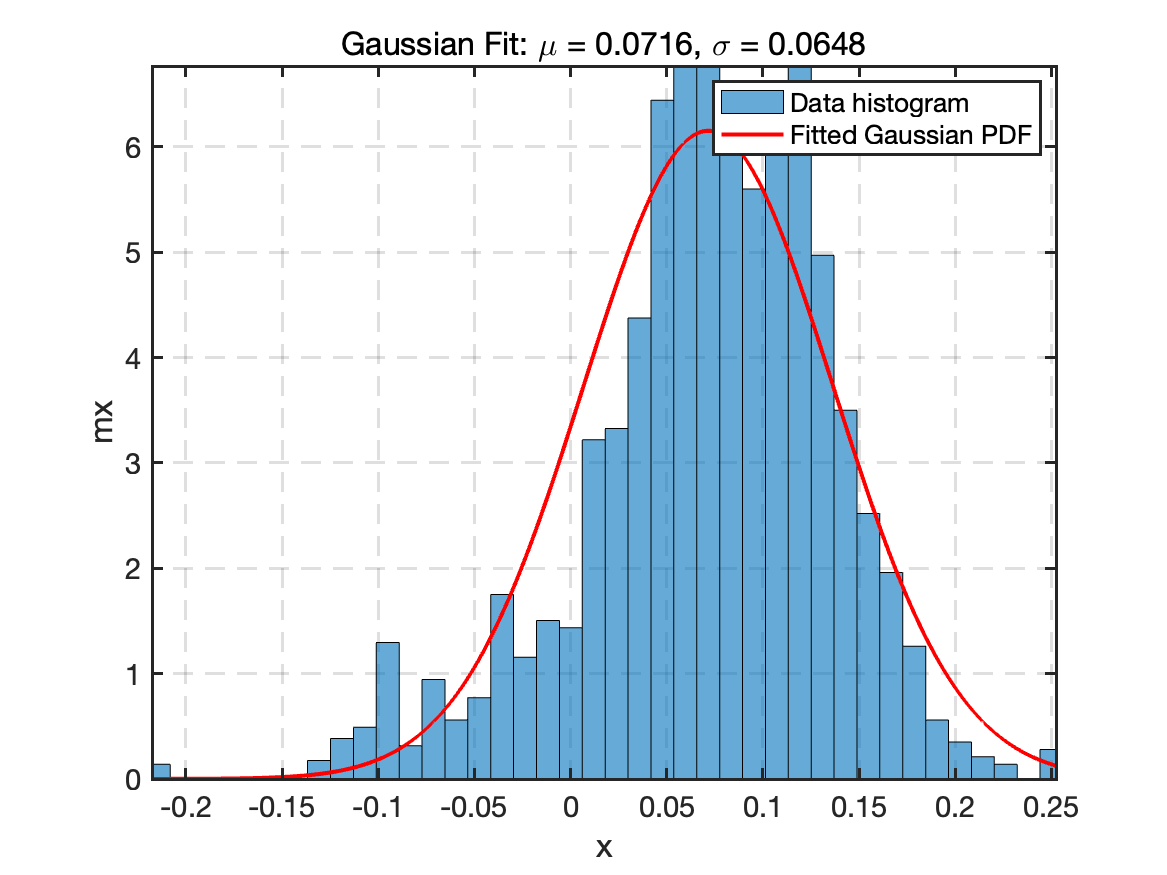}
	\includegraphics[width=0.24\textwidth]{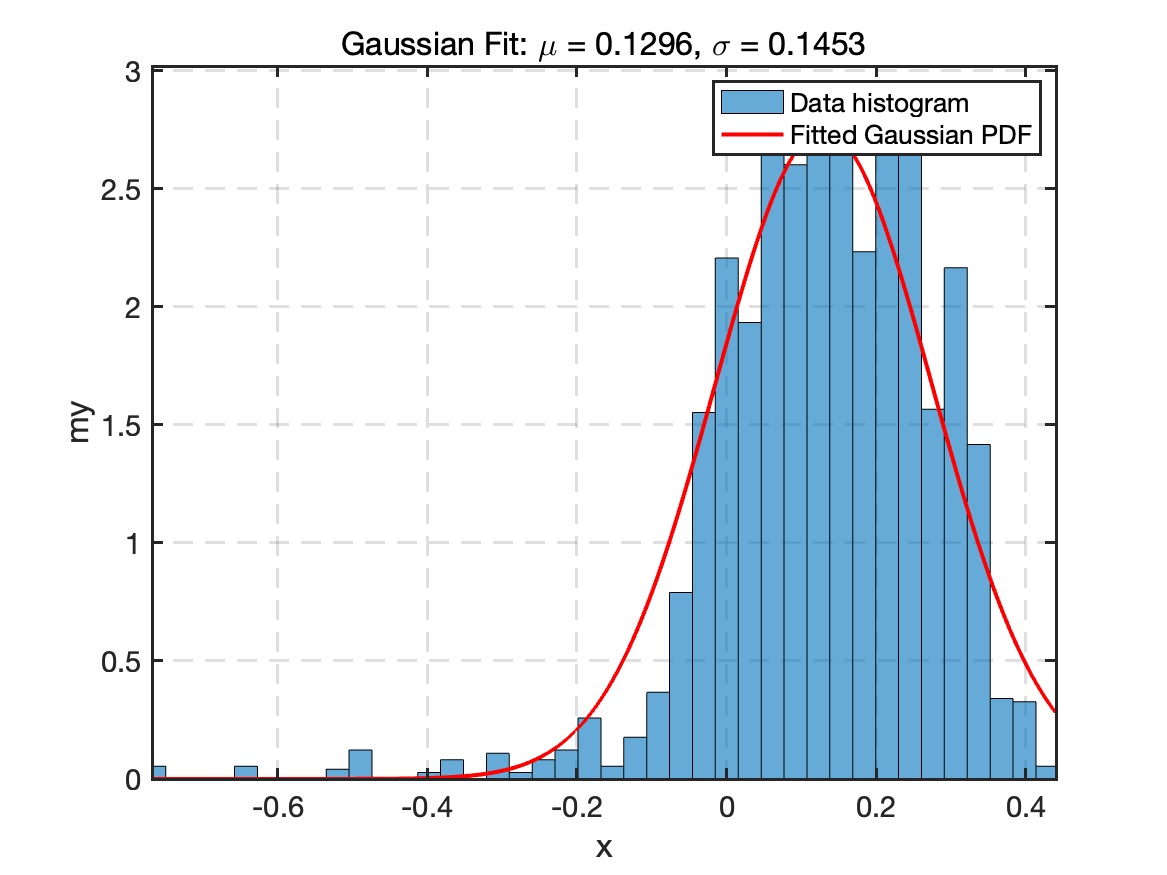} 	
	\includegraphics[width=0.24\textwidth]{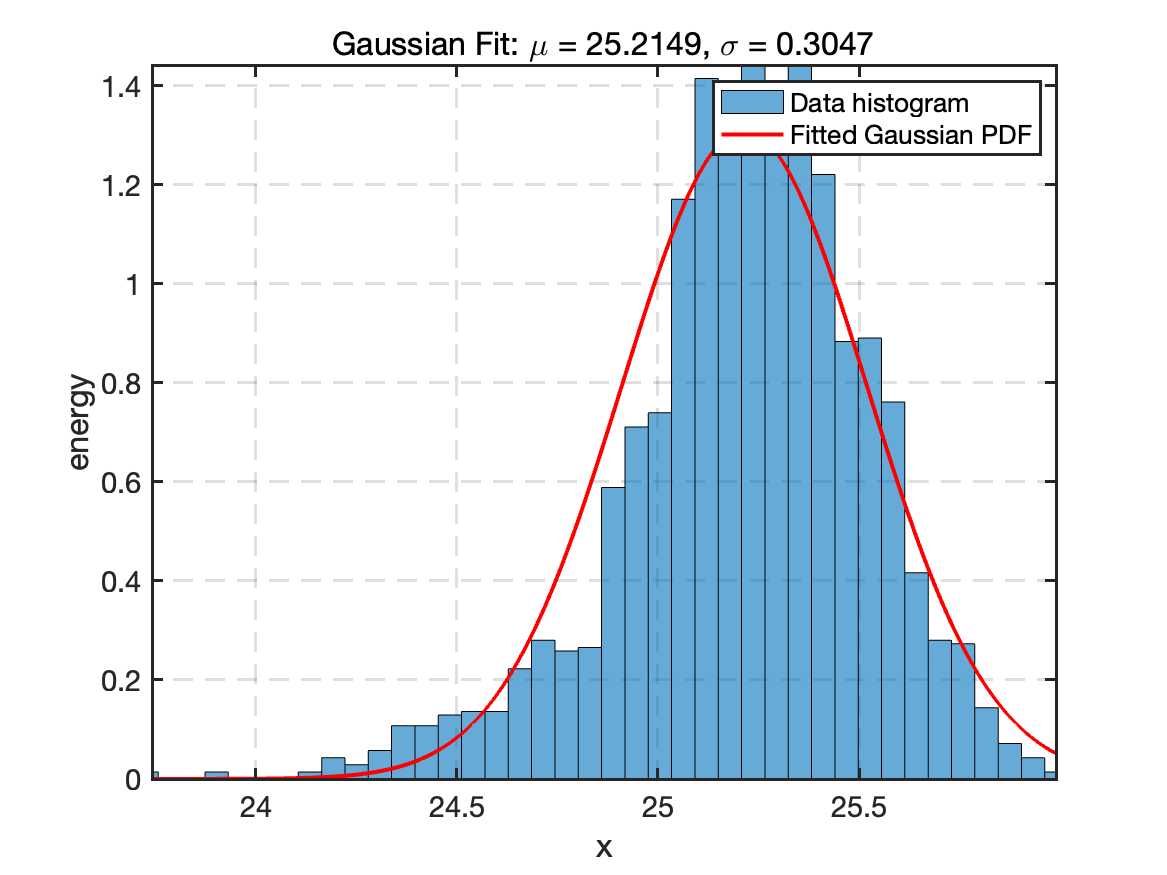}
	\includegraphics[width=0.24\textwidth]{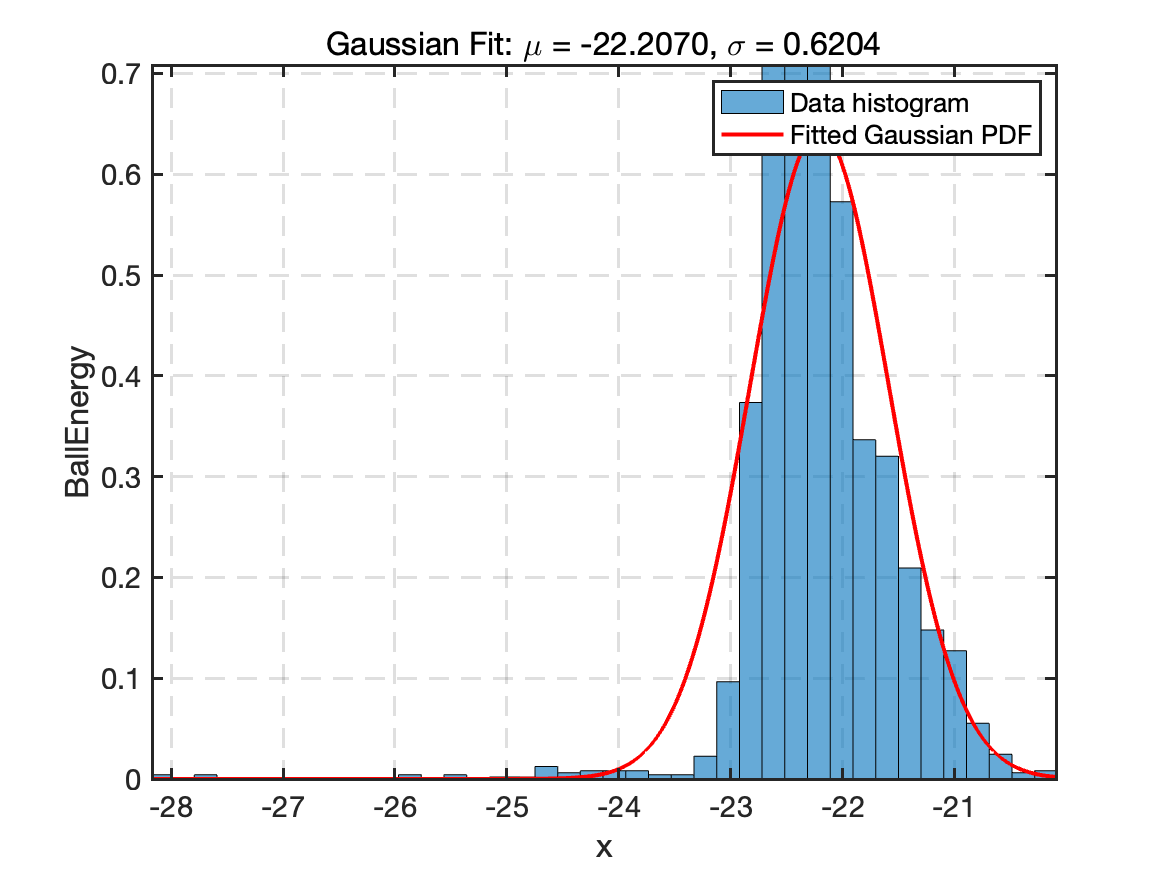} 
	\\
	\includegraphics[width=0.24\textwidth]{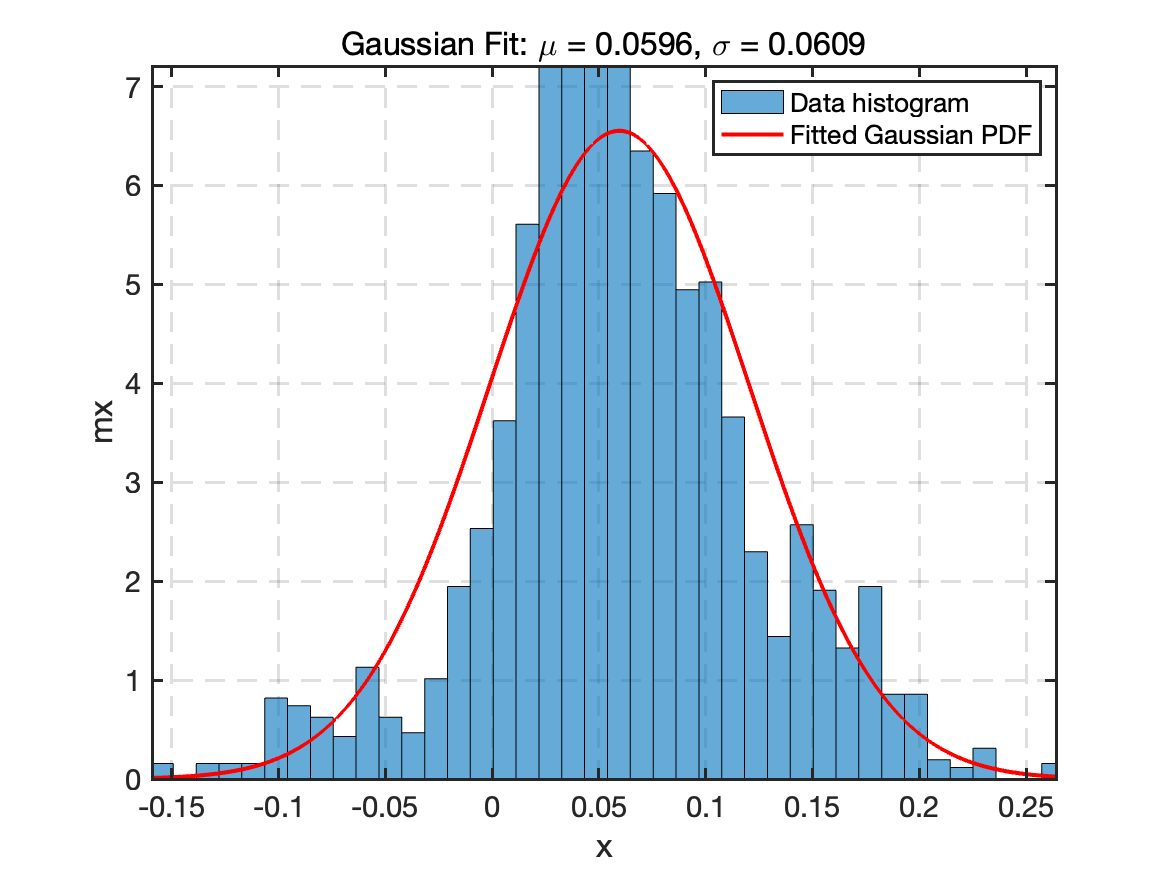}
	\includegraphics[width=0.24\textwidth]{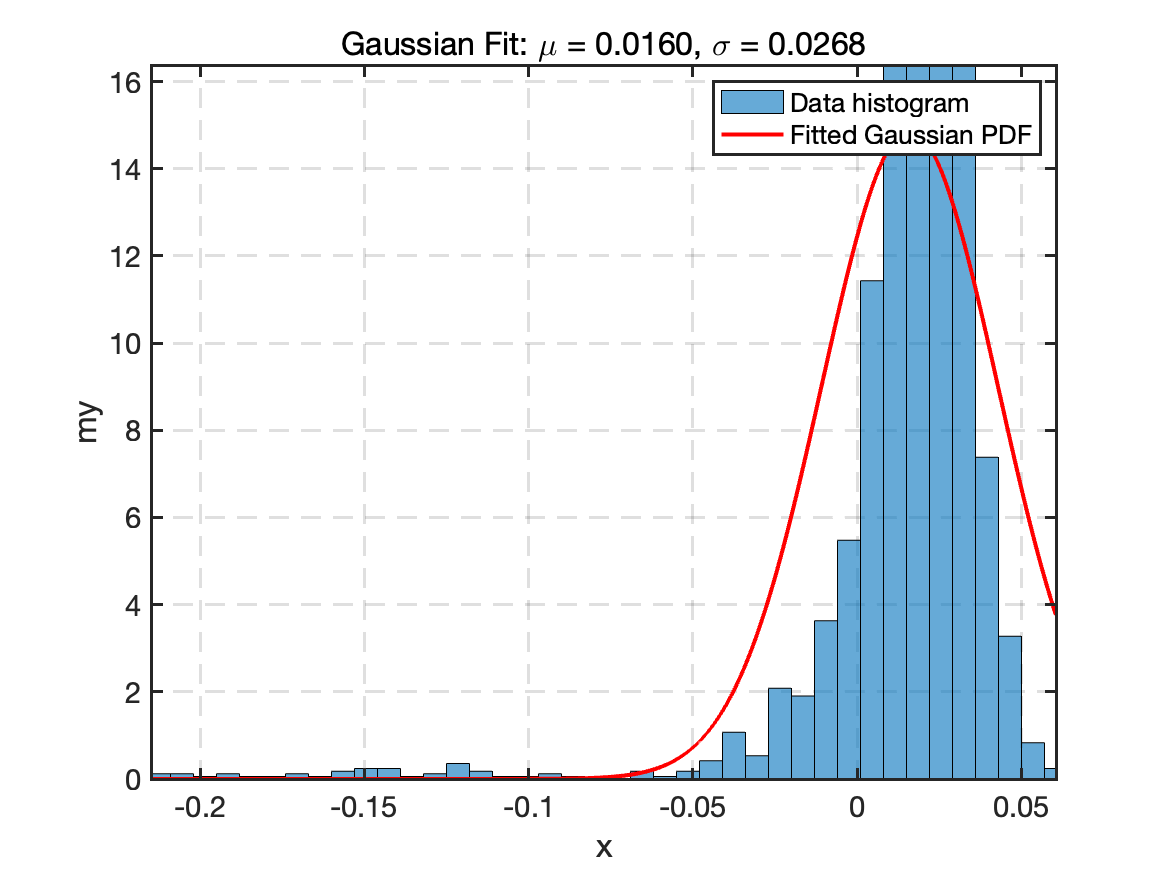} 	
	\includegraphics[width=0.24\textwidth]{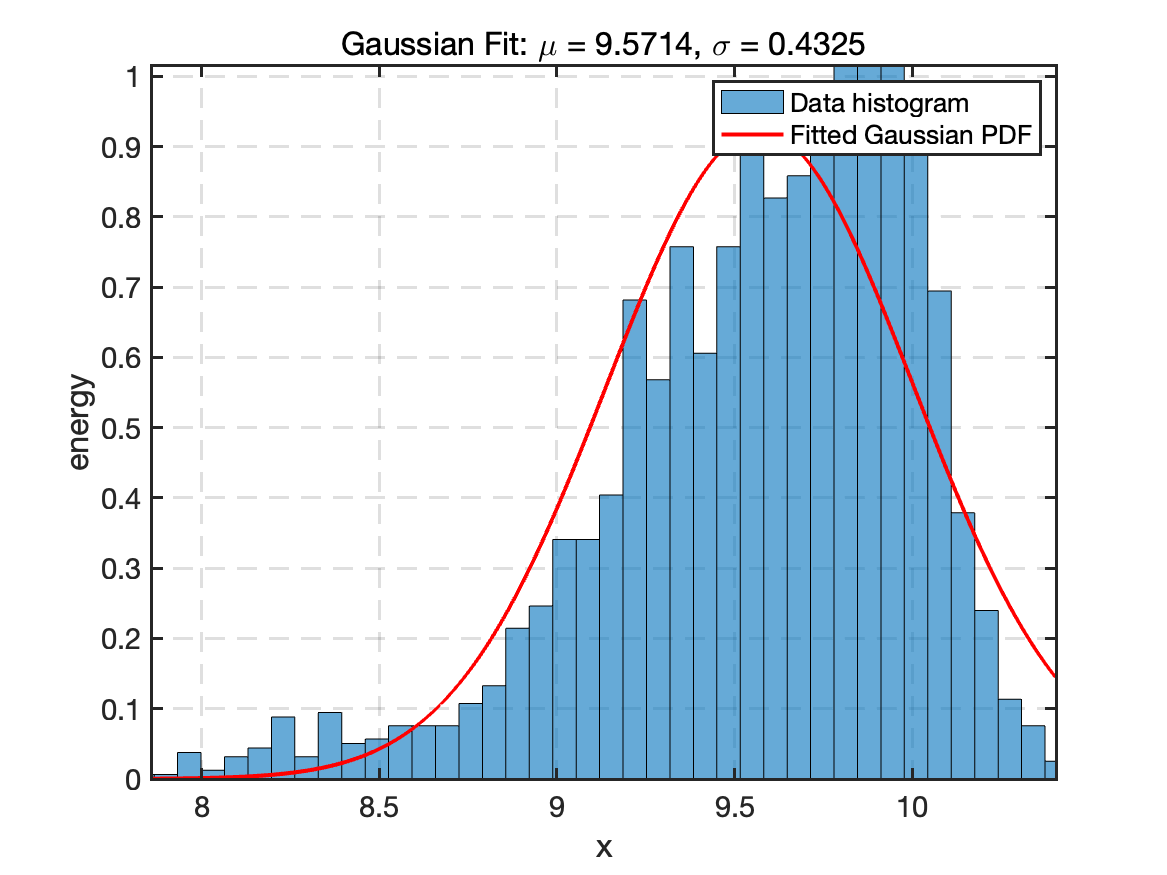}
	\includegraphics[width=0.24\textwidth]{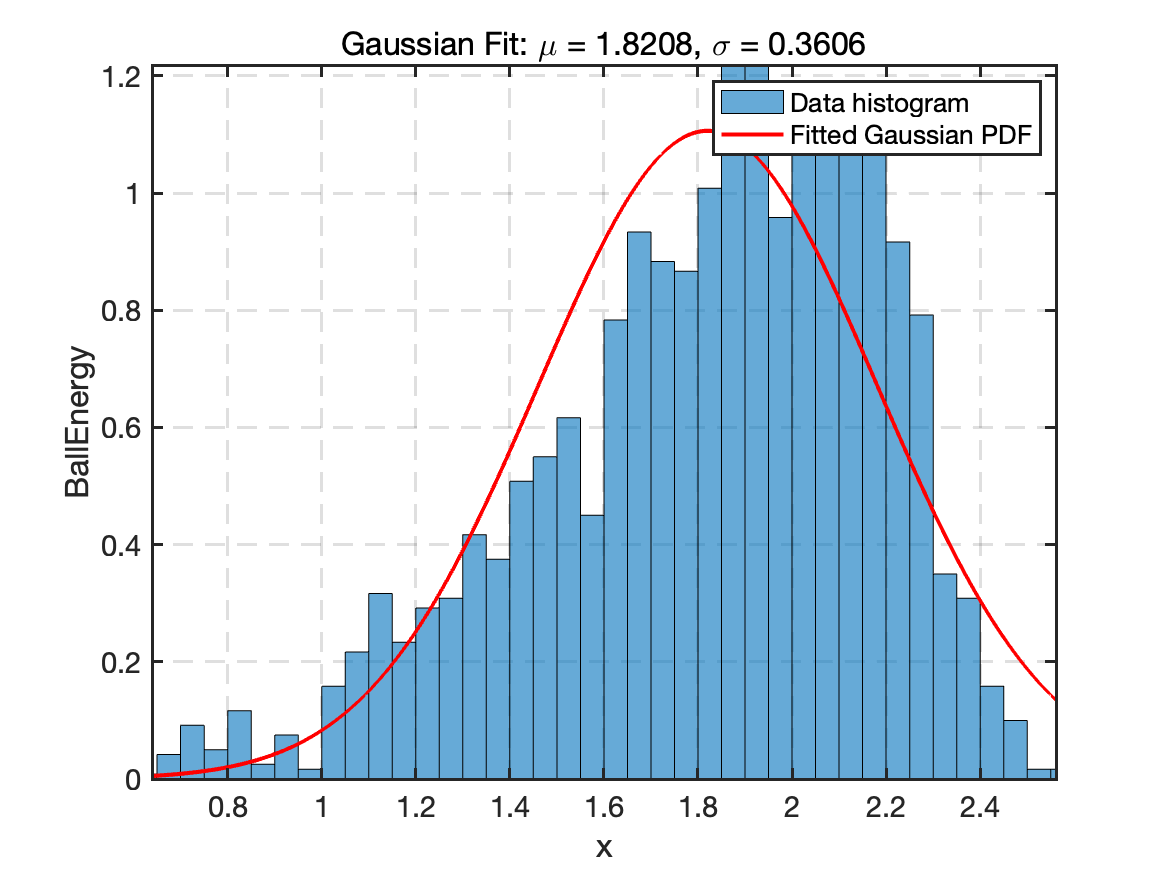} 
	\\
	\includegraphics[width=0.24\textwidth]{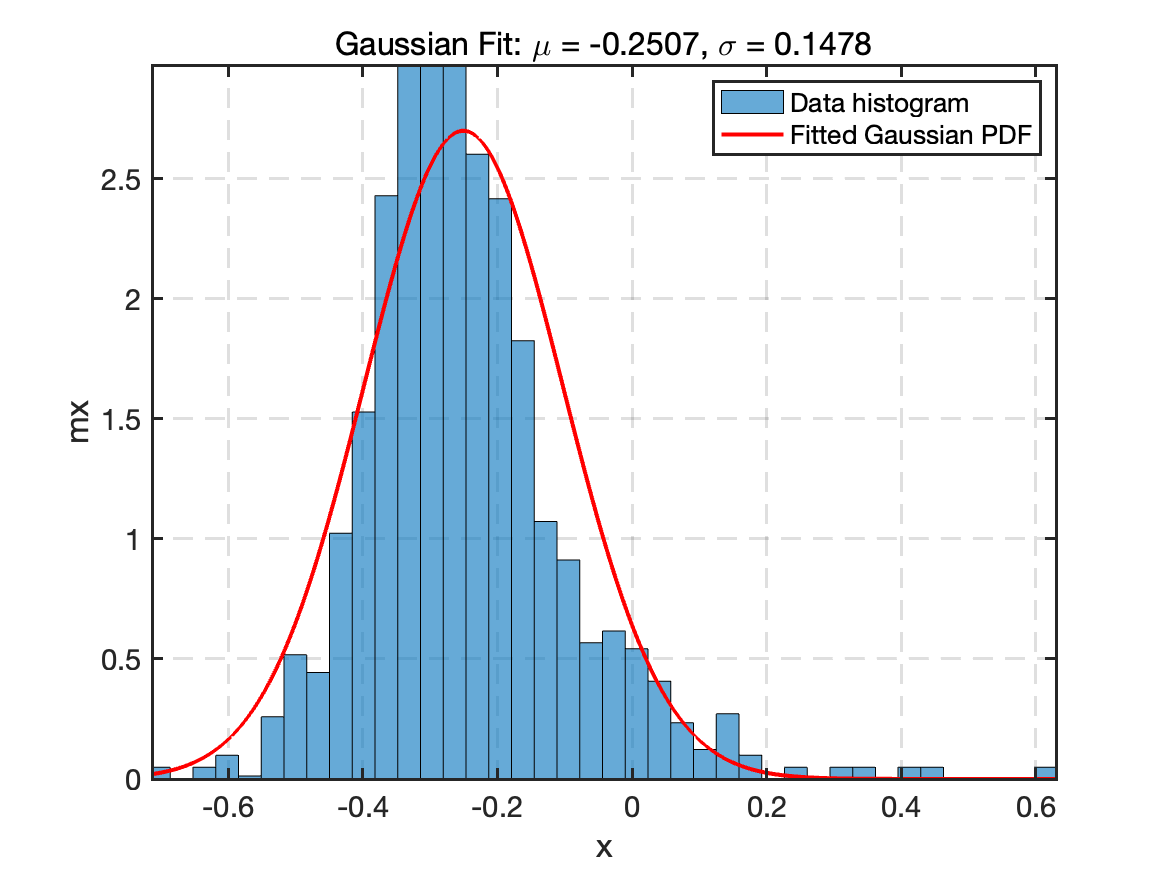}
	\includegraphics[width=0.24\textwidth]{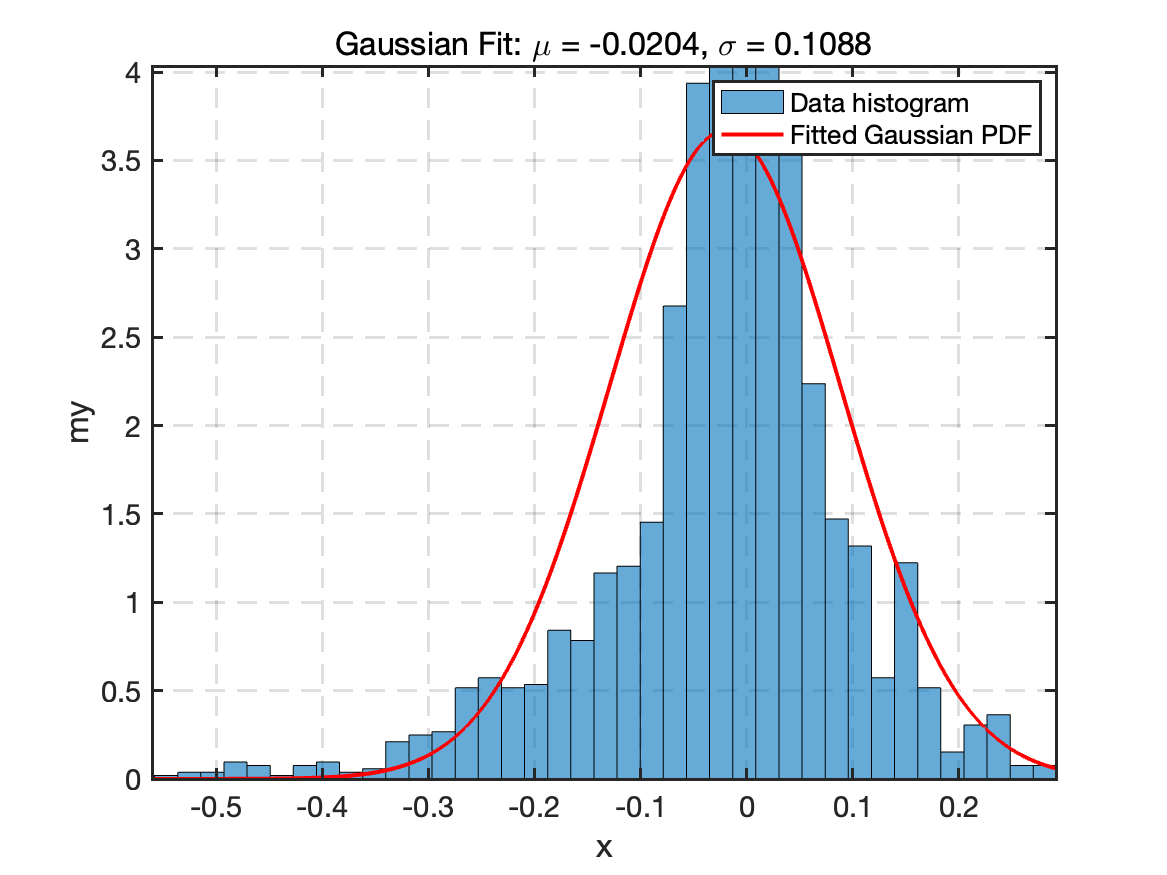} 	
	\includegraphics[width=0.24\textwidth]{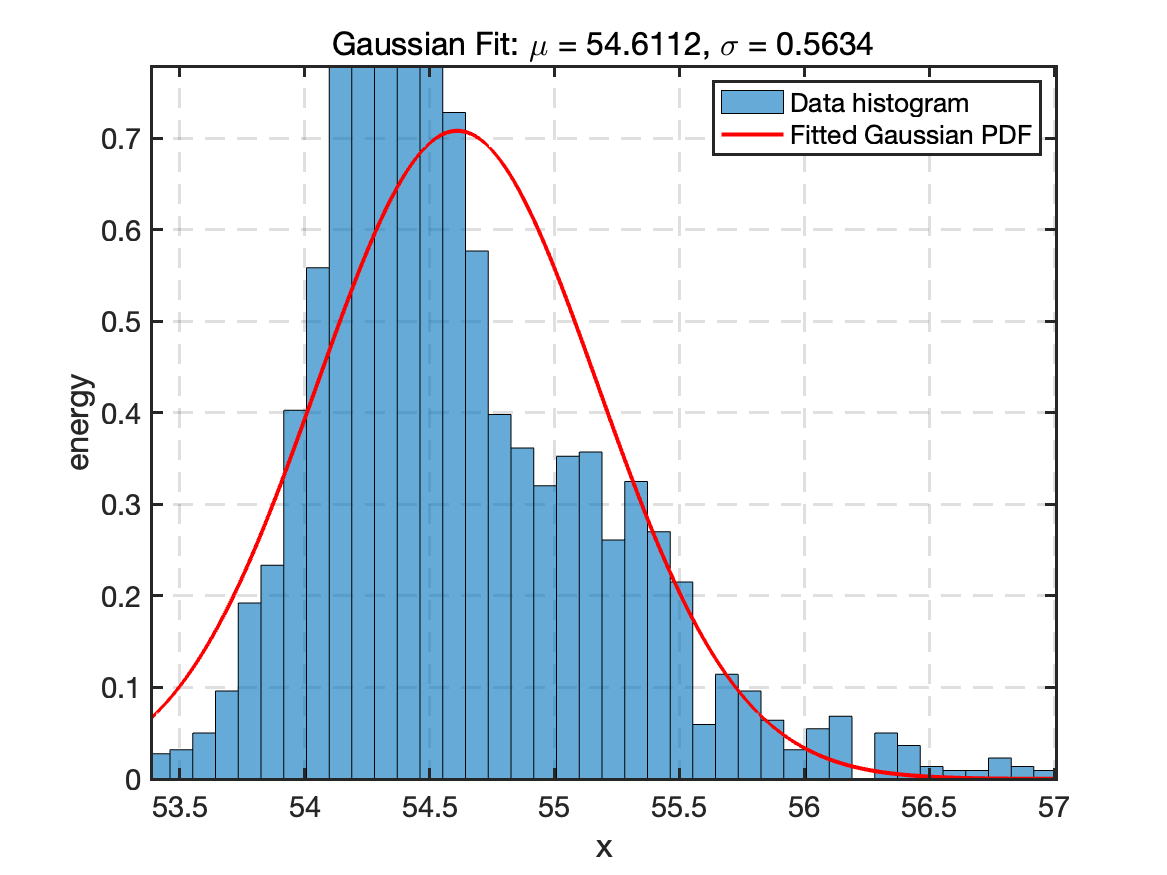}
	\includegraphics[width=0.24\textwidth]{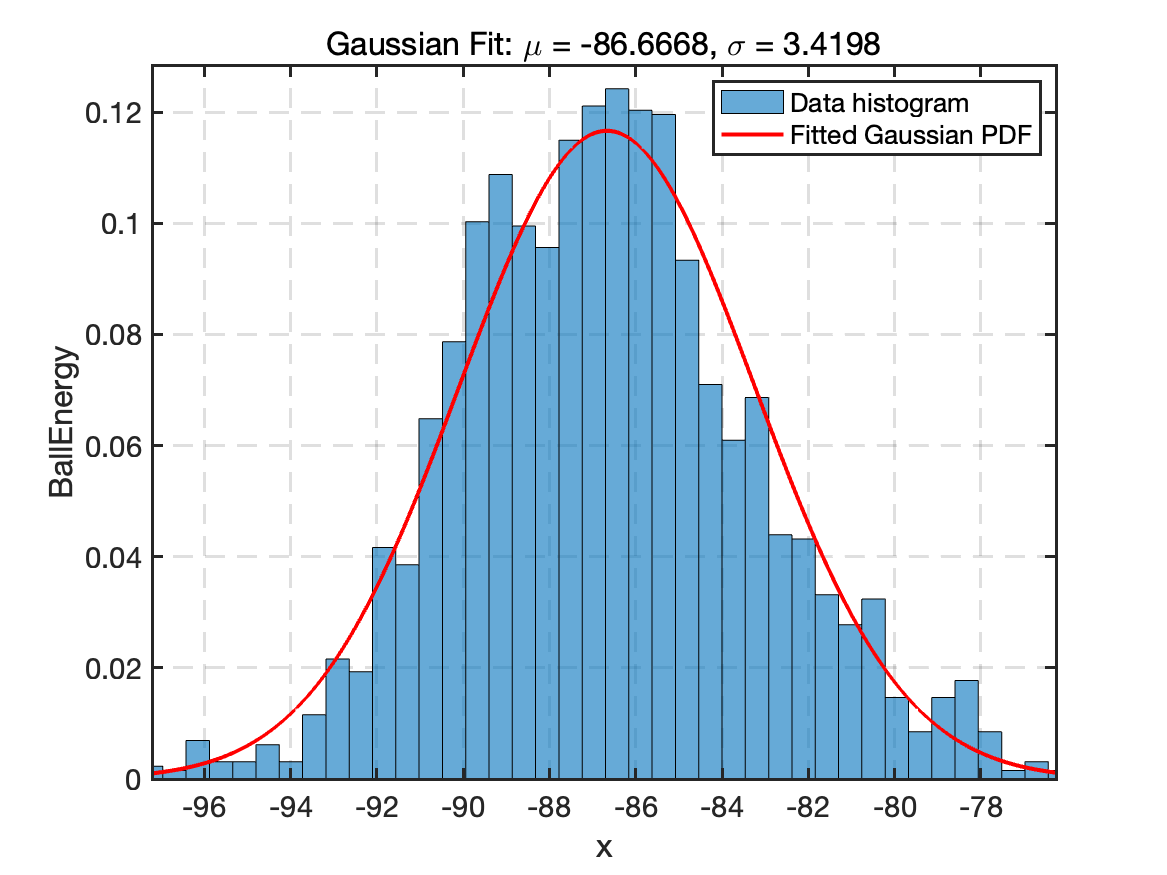} 
	\\
	\includegraphics[width=0.24\textwidth]{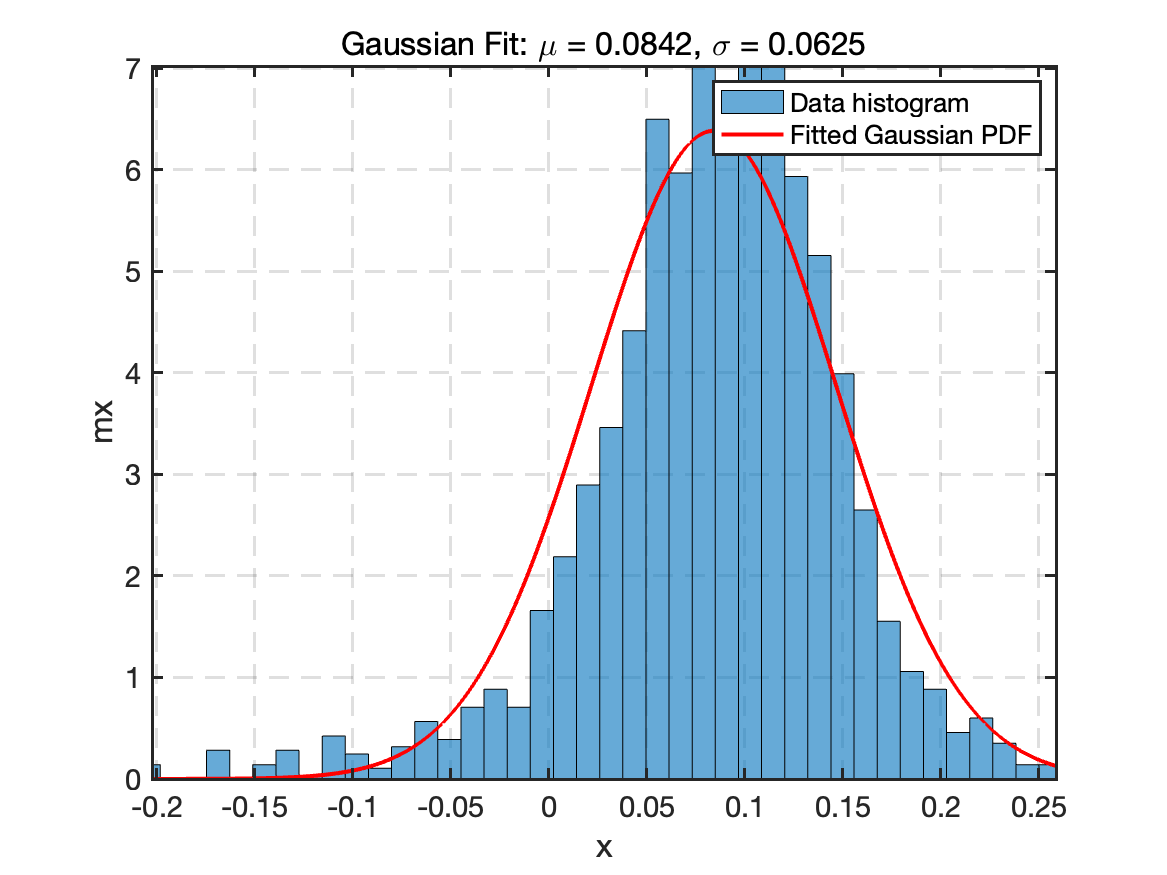}
	\includegraphics[width=0.24\textwidth]{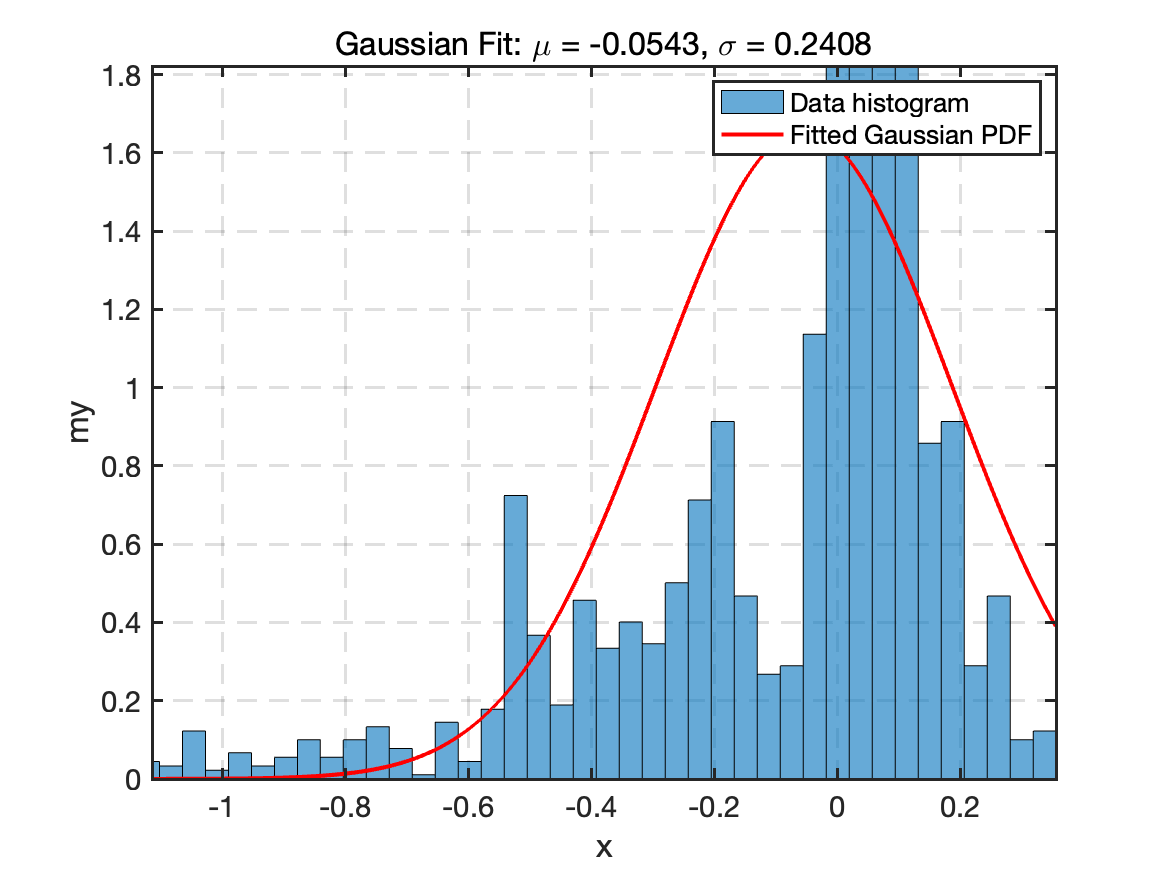} 	
	\includegraphics[width=0.24\textwidth]{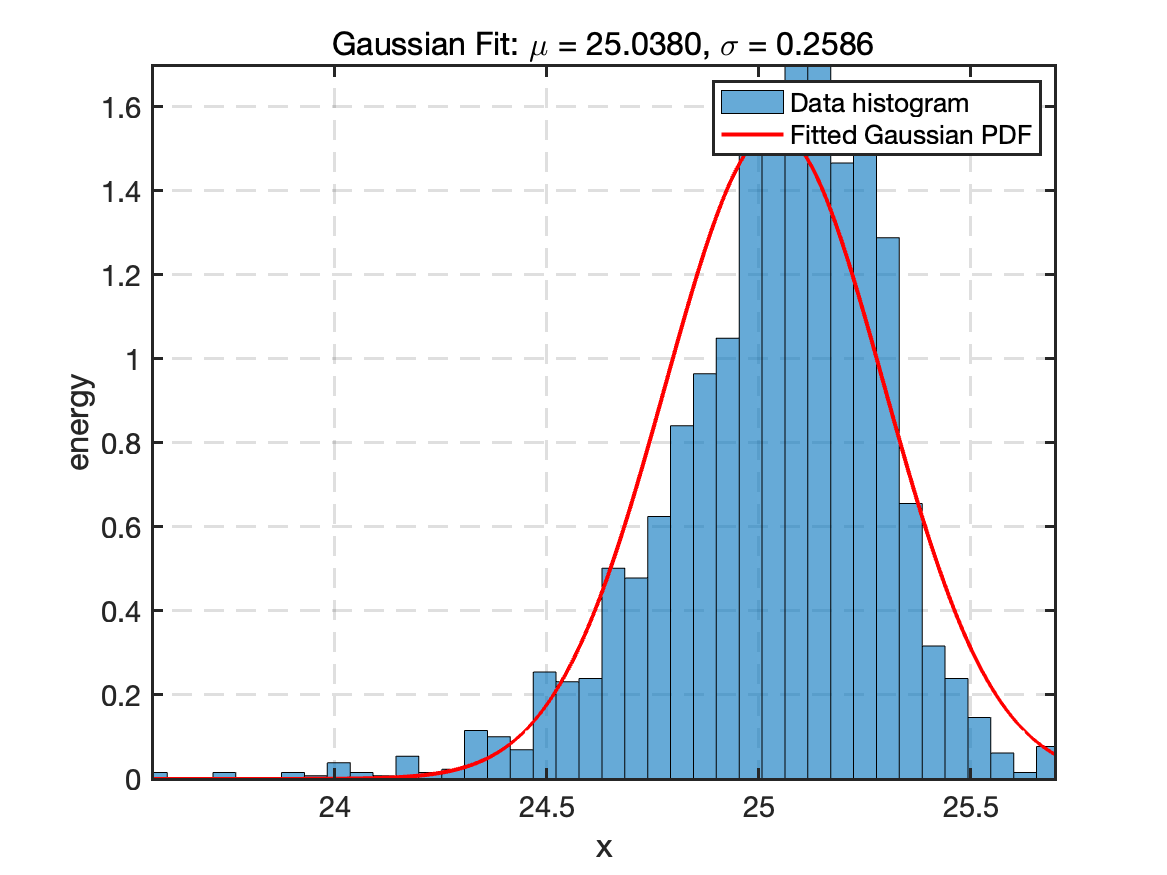}
	\includegraphics[width=0.24\textwidth]{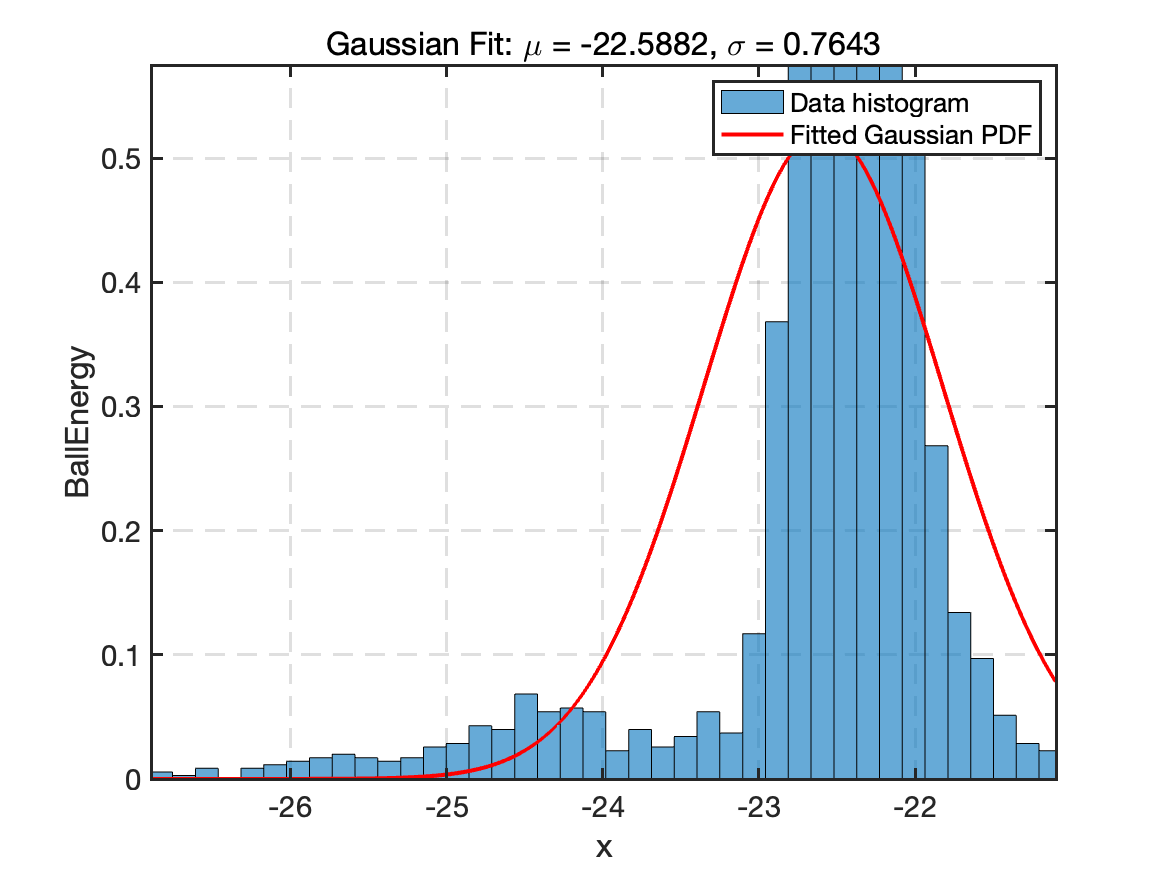} 
	\\
	\includegraphics[width=0.24\textwidth]{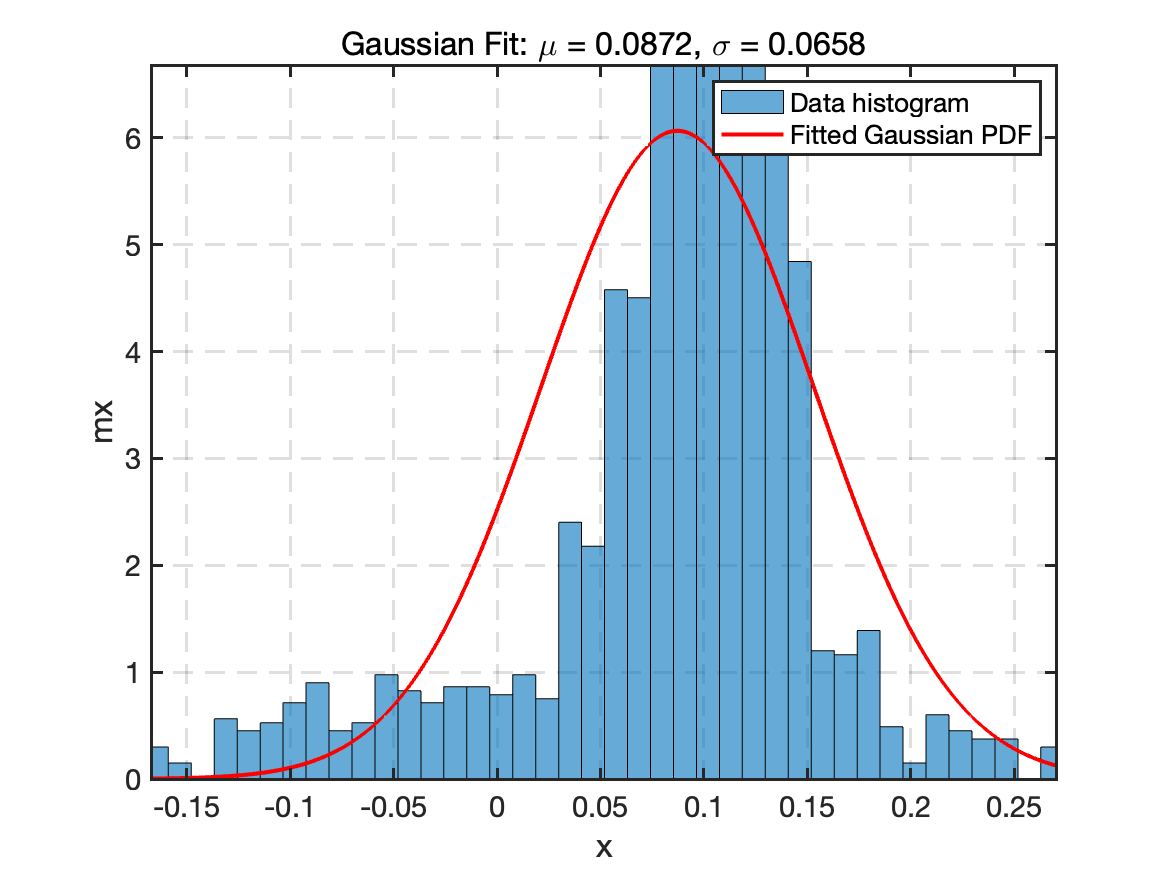}
	\includegraphics[width=0.24\textwidth]{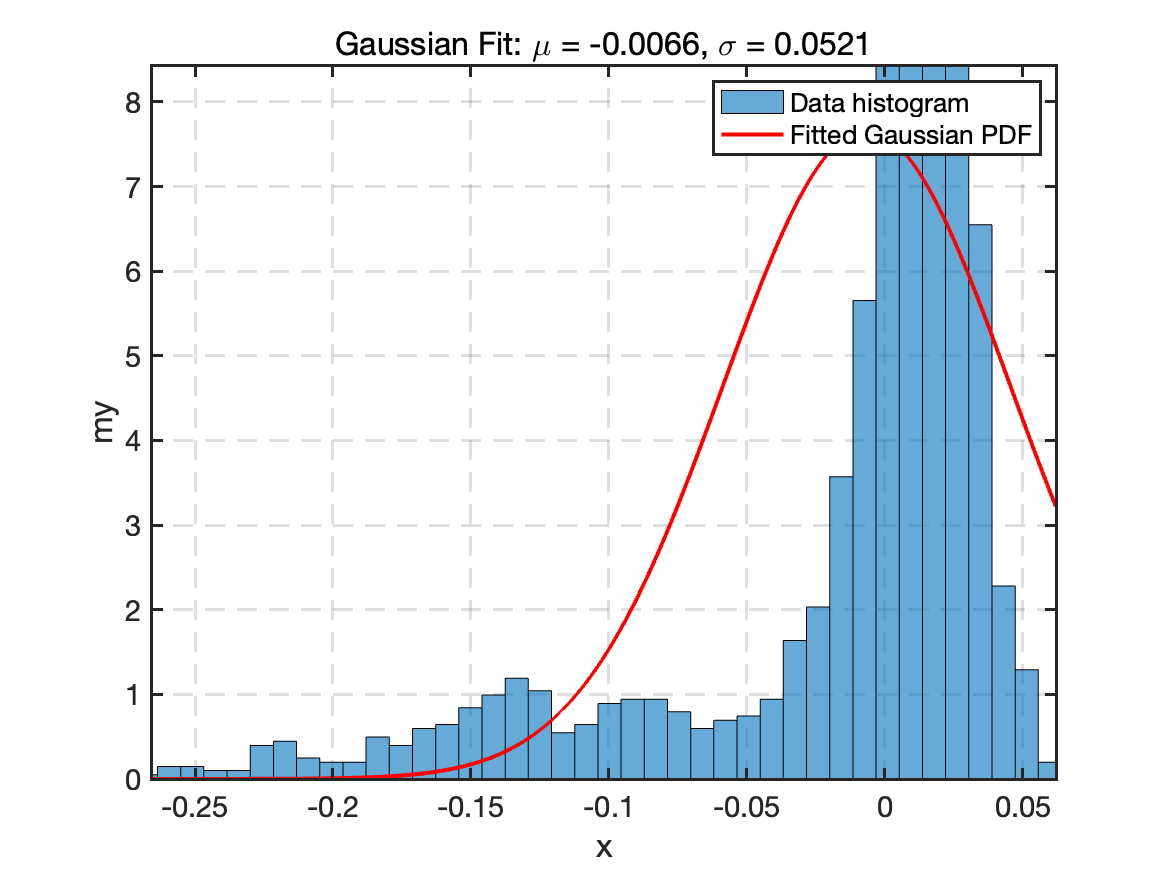} 	
	\includegraphics[width=0.24\textwidth]{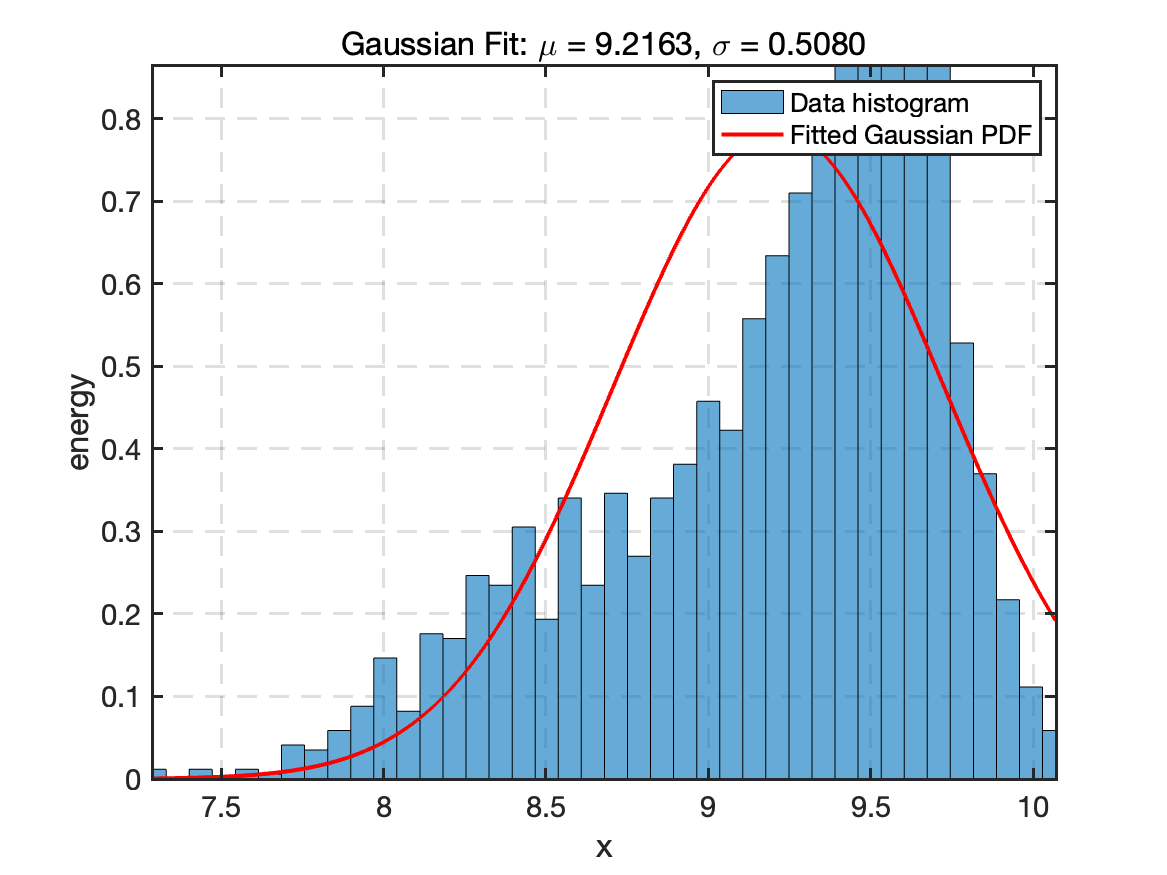}
	\includegraphics[width=0.24\textwidth]{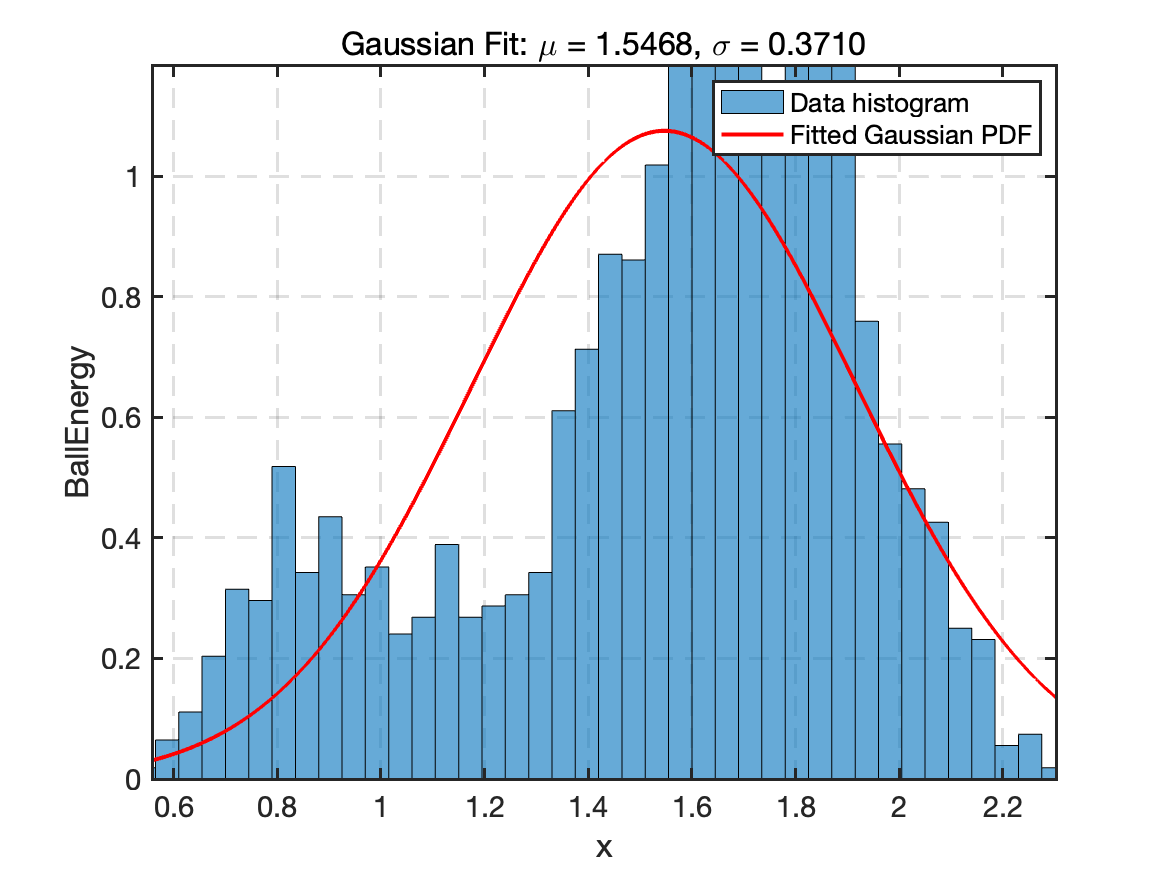} 
	\\
	\caption{ \small{Rayleigh--B\' enard Experiment 2: Measures $\mathcal{M}\left(U_h(t,P_i) \right), i =1,\dots,6$ (from top to bottom) with $U \in \{m_1,m_2,
E, BE\}$ (from left to right).}}\label{fig-Measure-Ex2-3}
\end{figure}

\end{itemize}

\newpage 
We have performed further extensive numerical testings and present in what follows some of them. In particular, we show 1) temporal averages of the $L^1$-norms and temporal averages of spatial averages; 2) time evolution of $L^1$-norms of solutions; 3) errors of solutions, their corresponding temporal average, deviation and Reynolds stress and energy fluctuation; and 4) measures of $L^1$-norms of solutions.  The results will be presented for the following experiments:

\begin{itemize}
\item Experiment 3 with {\bf large initial energy} 
\begin{align*}
\widetilde{P}(x_1) = 0, \quad \widehat{P}(x_2) =  \begin{cases}
100 \cos^2(\pi x_2), & \mbox{ if } x_2 \in [-1/2, \ 1/2],\\
0, & \mbox{ otherwise},
\end{cases}  \quad \vt_L \equiv 1, \quad \vt_H \equiv 15, \quad g\equiv-10.
\end{align*}

\item Experiment 4 with {\bf small initial energy}
\begin{align*}
&\widetilde{P}(x_1) = 0, \quad \vt_L \equiv 1, \quad \vt_H \equiv 15, \quad g\equiv-10,\\
&\widehat{P}(x_2) = -(\vt_M + S_\vt  x_2) + \begin{cases}
\vt_H, & \mbox{ if } x_2 \in [-1, \ -9/10],\\
0.5 + 14.5 \cos^2(5\pi (x_2+9/10)), & \mbox{ if } x_2 \in [-9/10, \ -8/10],\\
0.5 + 0.5 \cos^2(5\pi (x_2-8/10)/16), & \mbox{ if } x_2 \in [-8/10, \ 8/10],\\
\vt_L, & \mbox{ if } x_2 \in [8/10, \ 1].
\end{cases}
\end{align*}

\item Experiment 5 with {\bf boundary perturbations} 
\begin{align*}
\widetilde{P}(x_1) =  P(x_1)/2, \quad \widehat{P}(x_2) = 0,\quad \vt_L \equiv 1, \quad \vt_H \equiv 15, \quad g\equiv-10.
\end{align*}
\end{itemize}

The numerical results for the above problems are summarized in Tables \ref{tabel1}-\ref{tabel2} and Figures~\ref{fig-Evo-fur}-\ref{fig-Measure-fur}.
\begin{table}[htbp]
	\centering
	\caption{Temporal-averages of $L^1$-norms of solutions $\Ov{\norm{U_h}_{L^1(\Omega)}}$ for Experiments 2-5.  } \label{tabel1}
	\begin{tabular}{|c|ccccc|}
		\hline
		\multirow{2}{*}{Experiment} & \multicolumn{5}{c|}{ $\Ov{\norm{U_h}_{L^1(\Omega)}}$  } \\
		\cline{2-6}
		& $E$ & $BE$ & $m_1$ & $m_2$ & $S$ \\
		\hline
		\hline
Ex2	&229.3239		&265.6241		&0.9537		&0.821		&50.2776\\
Ex3	&229.4342		&265.6716		&0.9374		&0.8274		&50.2927\\
Ex4	&230.4742		&265.3276		&0.996		&0.8129		&50.3982\\
Ex5	&229.8286		&265.4391		&0.9831		&0.805		&50.3332\\
		\hline
	\end{tabular}
\end{table}

\begin{table}[htbp]
	\centering
	\caption{Temporal-averages of spatial averages of solutions $\Ov{\int_{\Omega} U_h \dx }$ for Experiments 2-5.  } \label{tabel2}
	\begin{tabular}{|c|cccc|}
		\hline
		\multirow{2}{*}{Experiment} & \multicolumn{4}{c|}{ $\Ov{\int_{\Omega} U_h \dx }$  } \\
		\cline{2-5}
		& $BE$ & $m_1$ & $m_2$ & $S$ \\
		\hline
		\hline
Ex2 & 	-260.0749		&-0.0008		&-0.0019		&50.2591\\
Ex3 & 	-260.0827		&-0.0006		&-0.002		&50.2739\\
Ex4 & 	-259.5645		&0		&-0.0023		&50.3779\\
Ex5 & 	-259.7957		&0.0015		&-0.0021		&50.3038\\
		\hline
	\end{tabular}
\end{table}

\begin{figure}[htbp]
	\setlength{\abovecaptionskip}{0.cm}
	\setlength{\belowcaptionskip}{-0.cm}
	\centering
	\includegraphics[width=\textwidth]{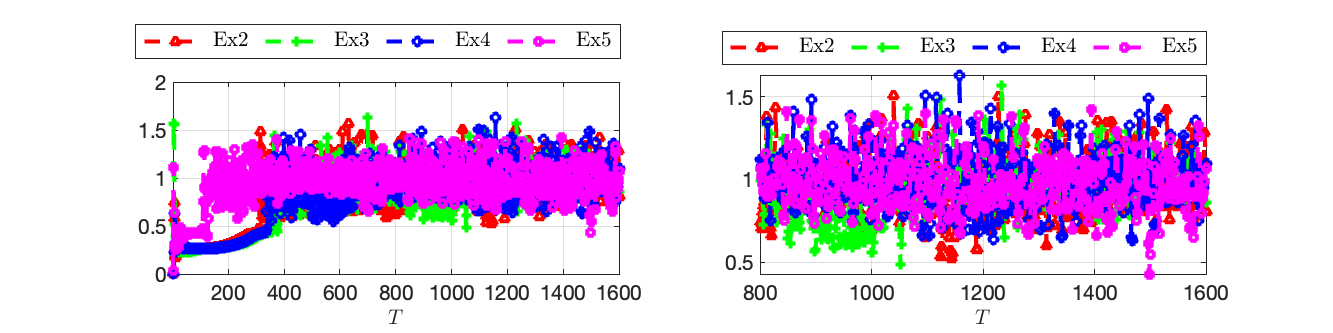}
	\includegraphics[width=\textwidth]{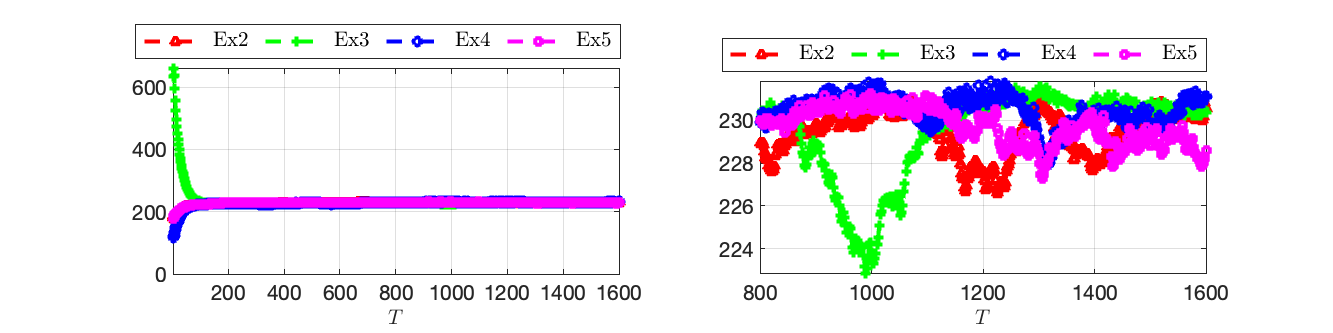}
	\caption{  \small{Rayleigh--B\' enard Experiments: evolutions of $\norm{m_{1,h}(t,\cdot)}_{L^1(\Omega)}$ (top) and $\norm{E_h(t,\cdot)}_{L^1(\Omega)}$ (bottom) for Experiments 2-5 (from left to right).}}\label{fig-Evo-fur}
\end{figure}


\begin{figure}[htbp]
	\setlength{\abovecaptionskip}{0.cm}
	\setlength{\belowcaptionskip}{-0.cm}
	\centering
	\includegraphics[width=0.24\textwidth]{./gif/case0/PRONUM52MX640MY320IS201IE800_Err1.png} 
	\includegraphics[width=0.24\textwidth]{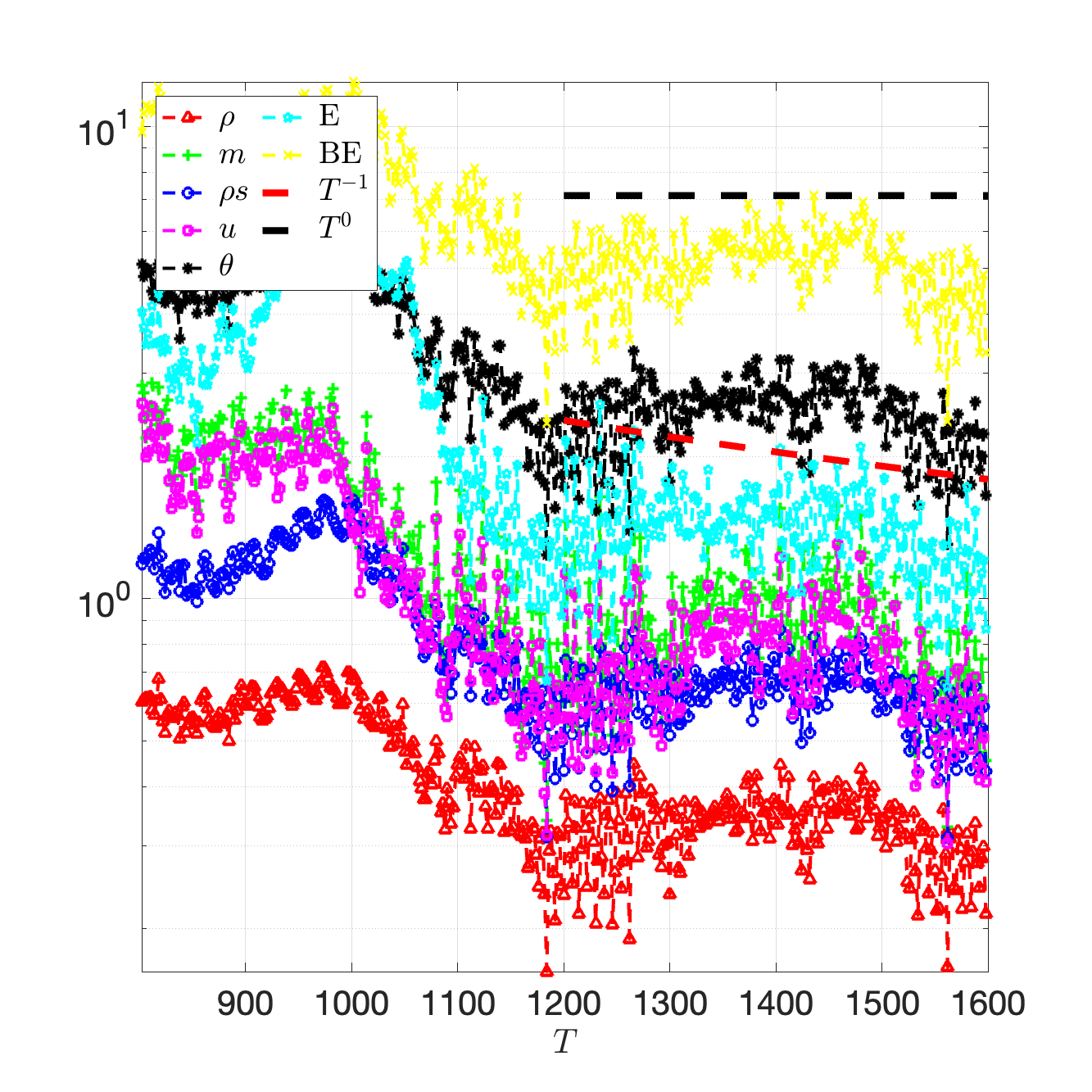}
	\includegraphics[width=0.24\textwidth]{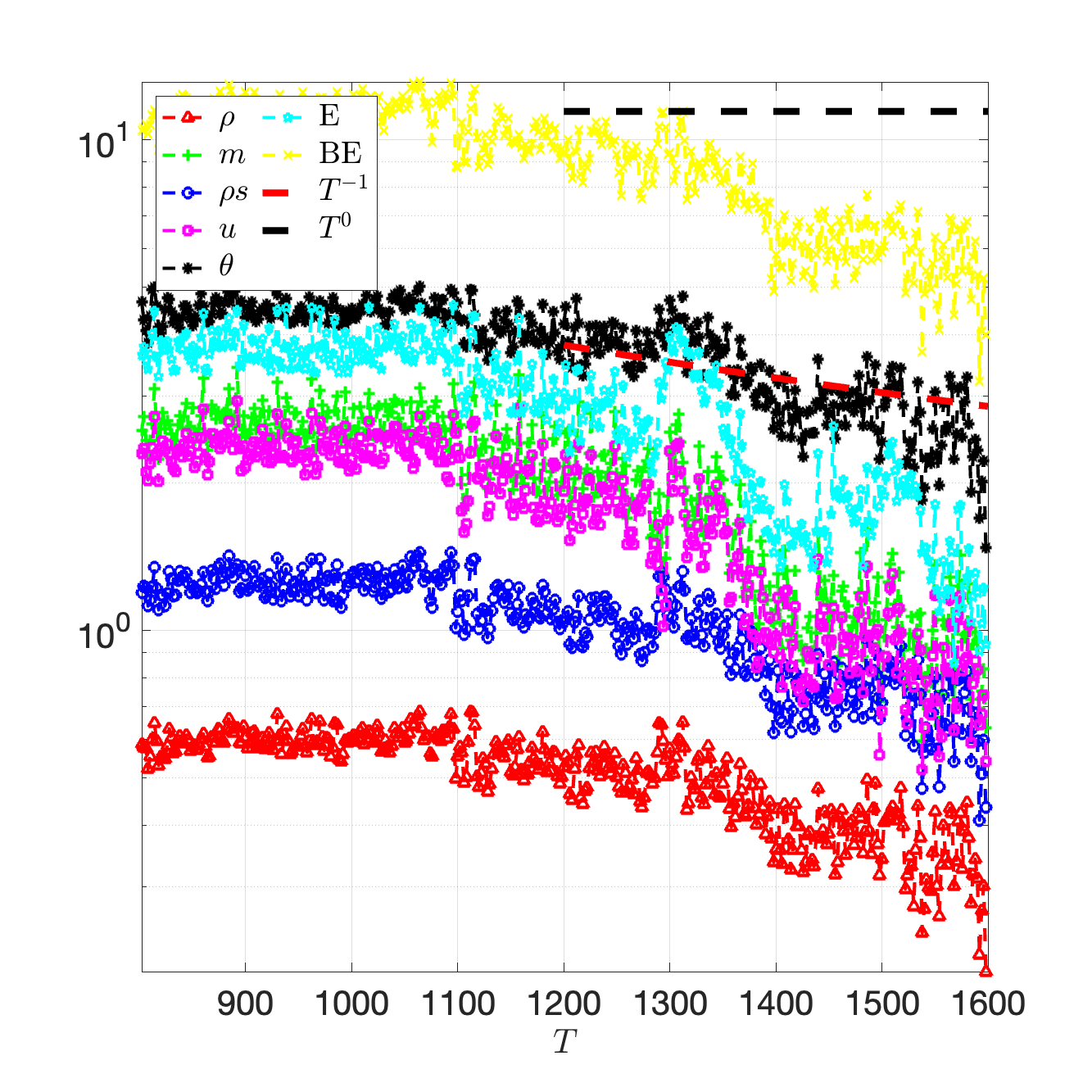} 
	\includegraphics[width=0.24\textwidth]{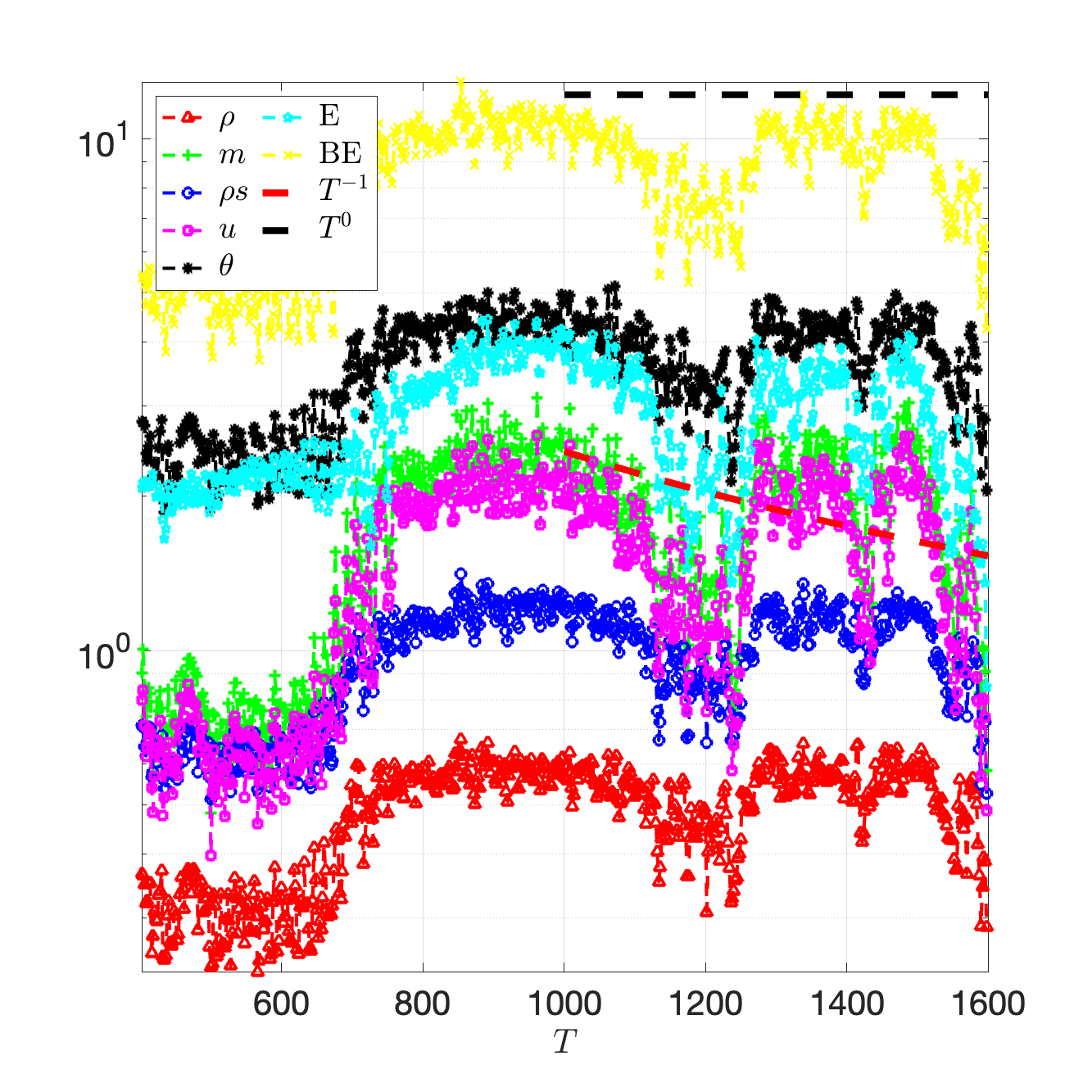}
	\\
	\includegraphics[width=0.24\textwidth]{./gif/case0/PRONUM52MX640MY320IS201IE800_Err2.png} 
	\includegraphics[width=0.24\textwidth]{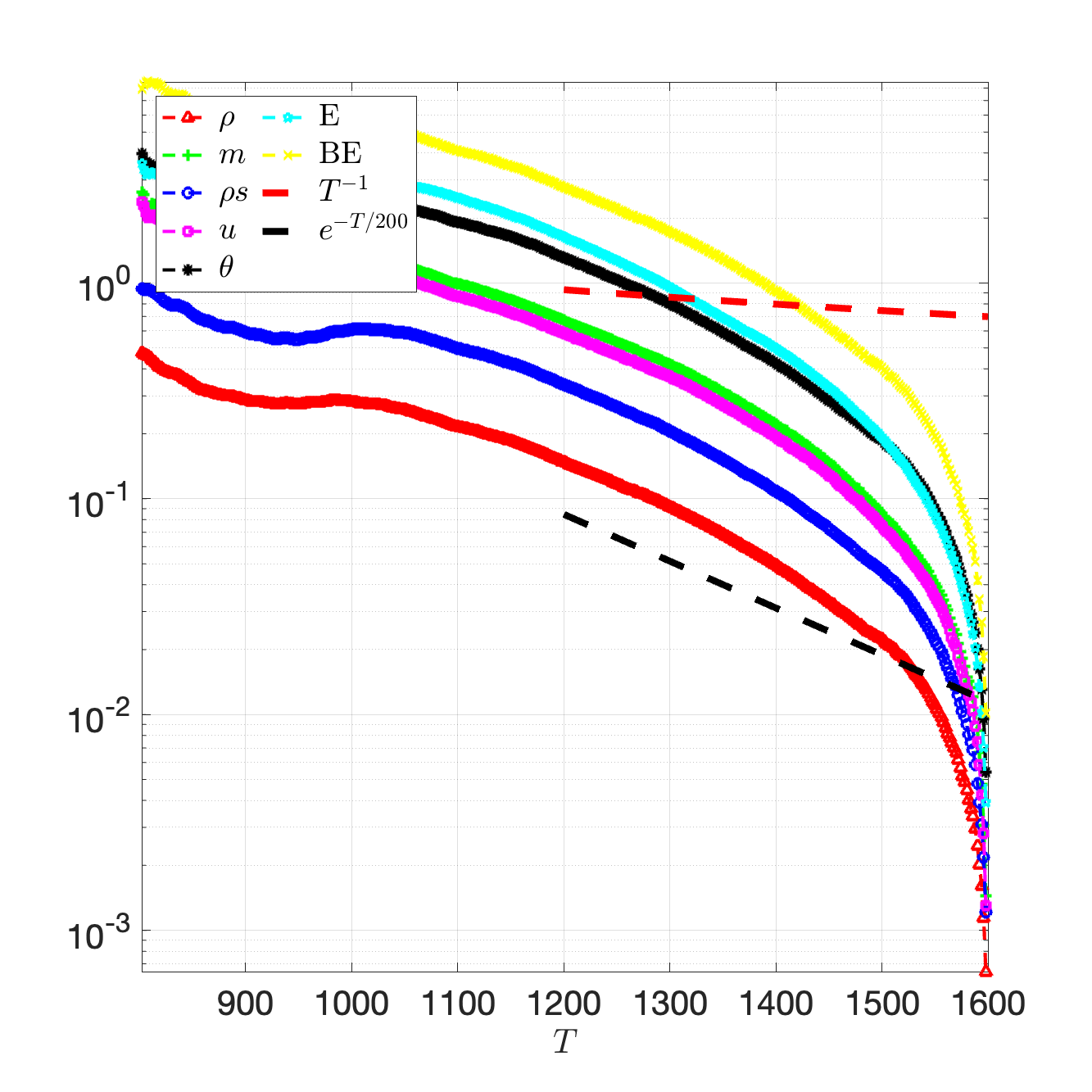}
	\includegraphics[width=0.24\textwidth]{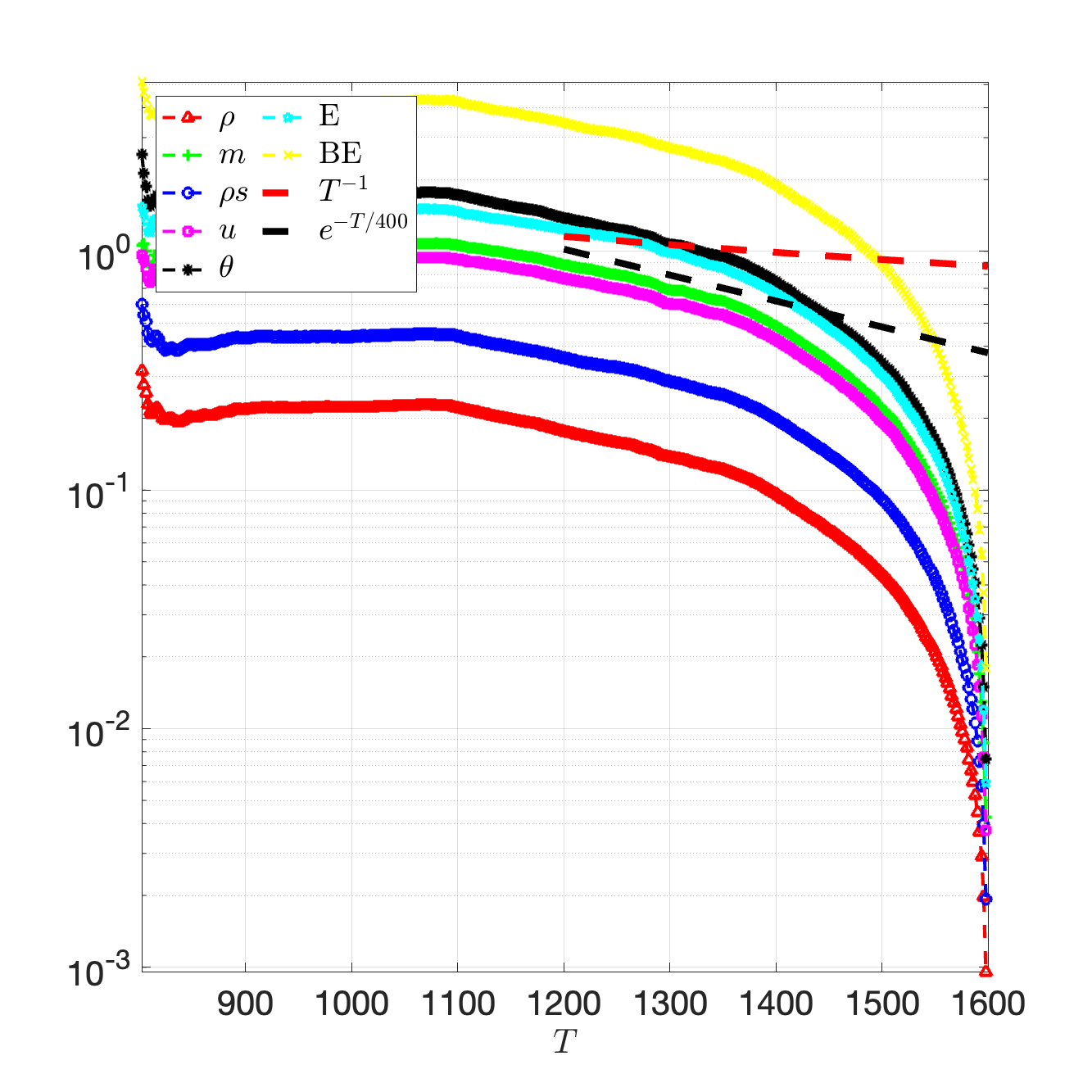} 
	\includegraphics[width=0.24\textwidth]{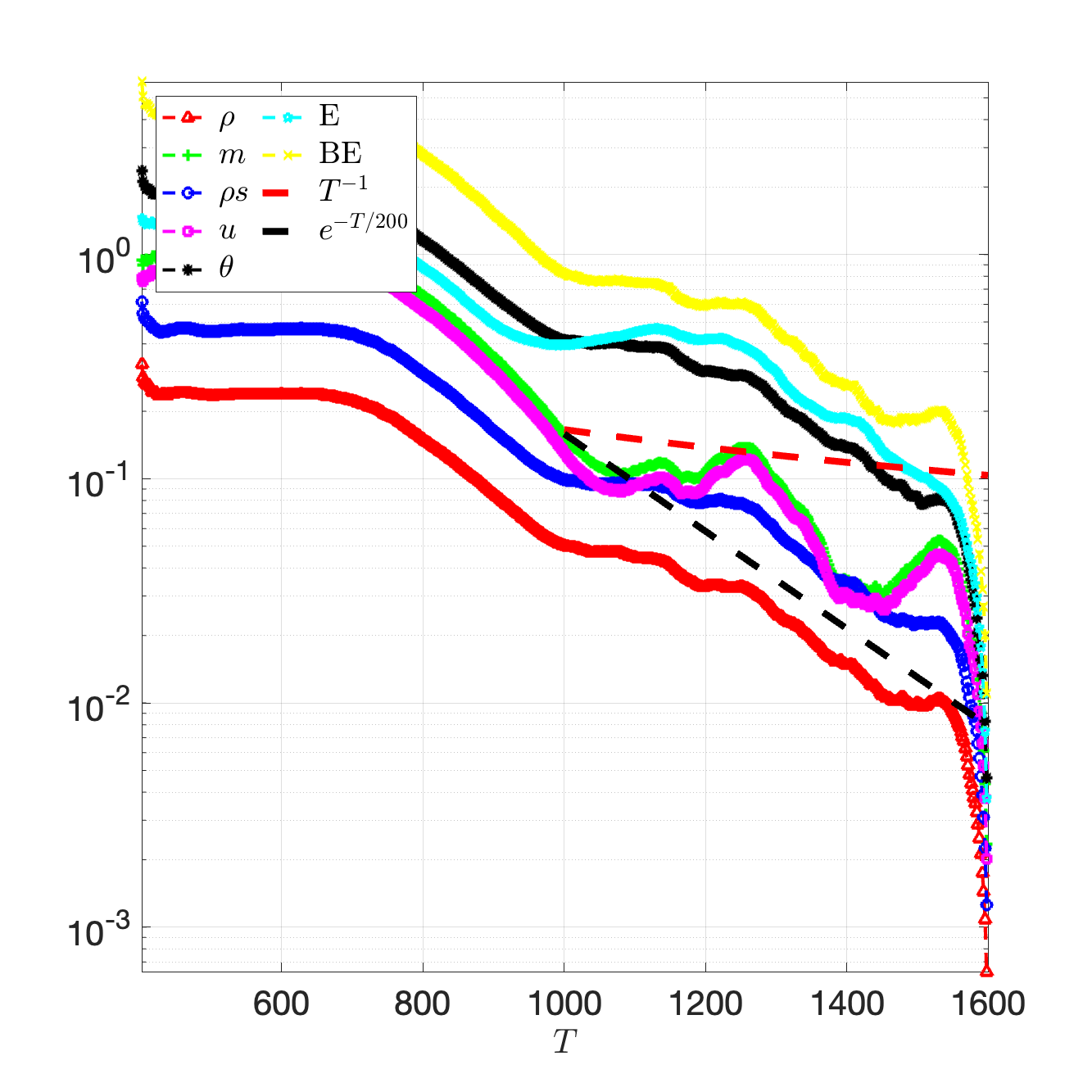}
	\\
	\includegraphics[width=0.24\textwidth]{./gif/case0/PRONUM52MX640MY320IS201IE800_Err3.png} 
	\includegraphics[width=0.24\textwidth]{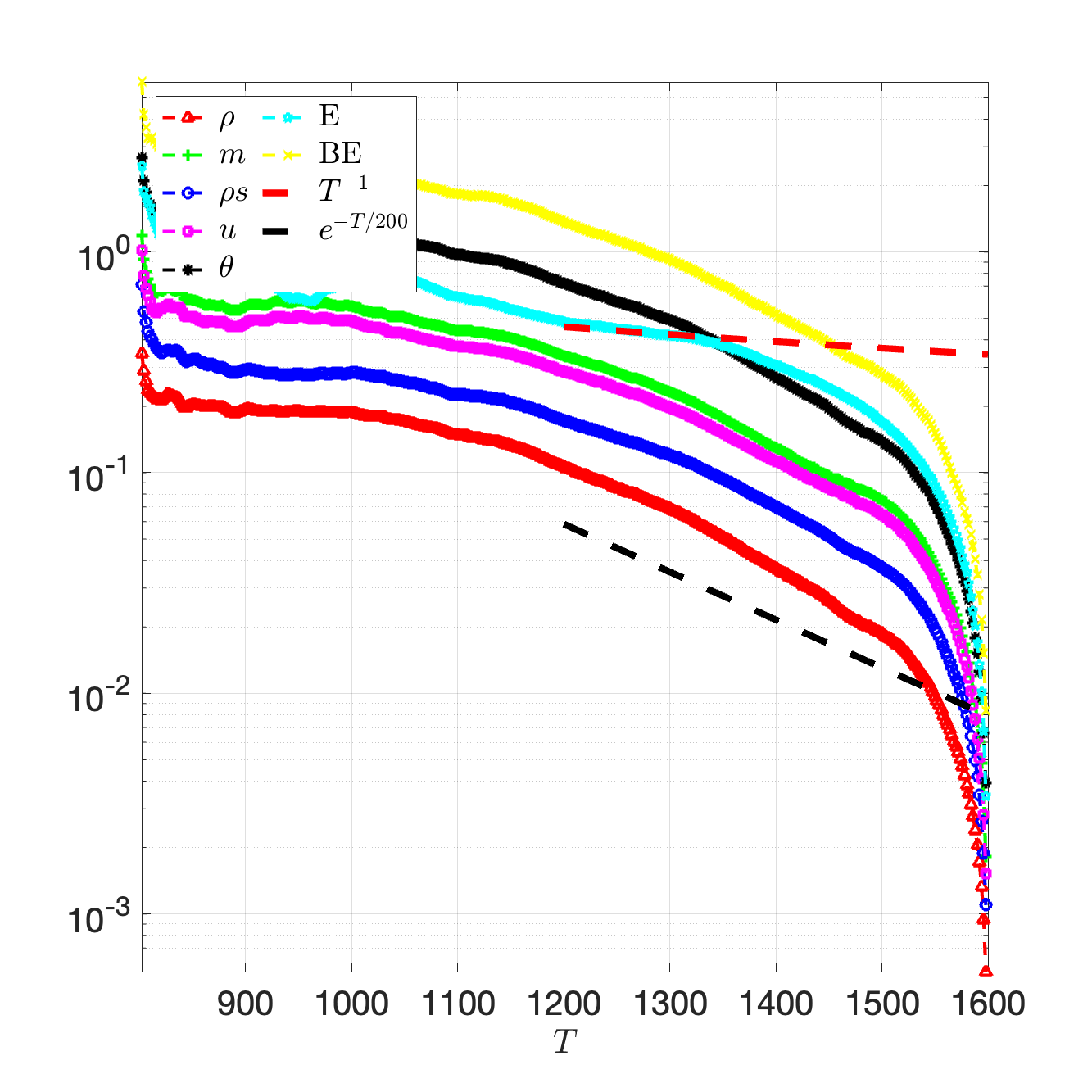}
	\includegraphics[width=0.24\textwidth]{./gif/case13/PRONUM52MX640MY320IS401IE800_Err2.png} 
	\includegraphics[width=0.24\textwidth]{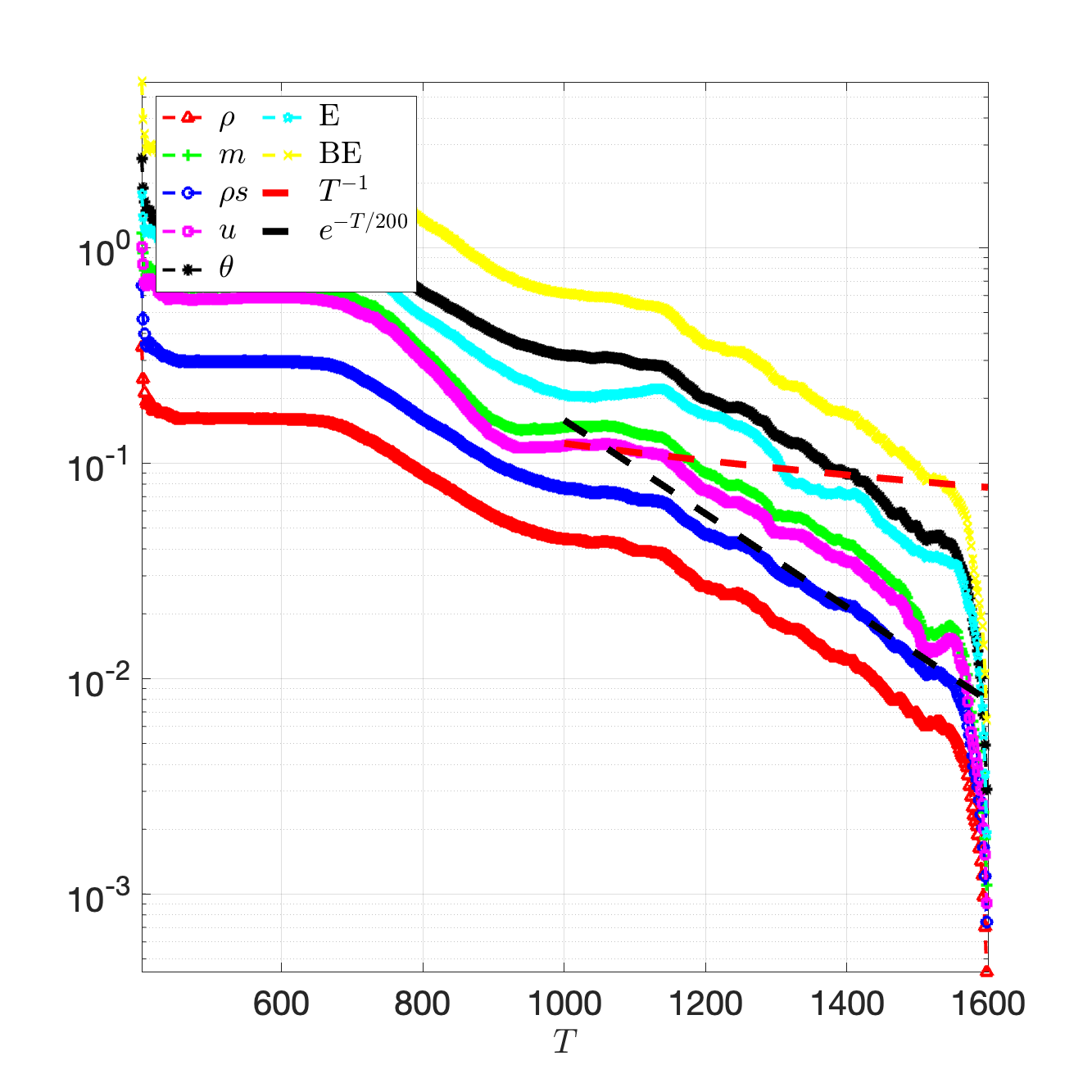}
	\\
	\includegraphics[width=0.24\textwidth]{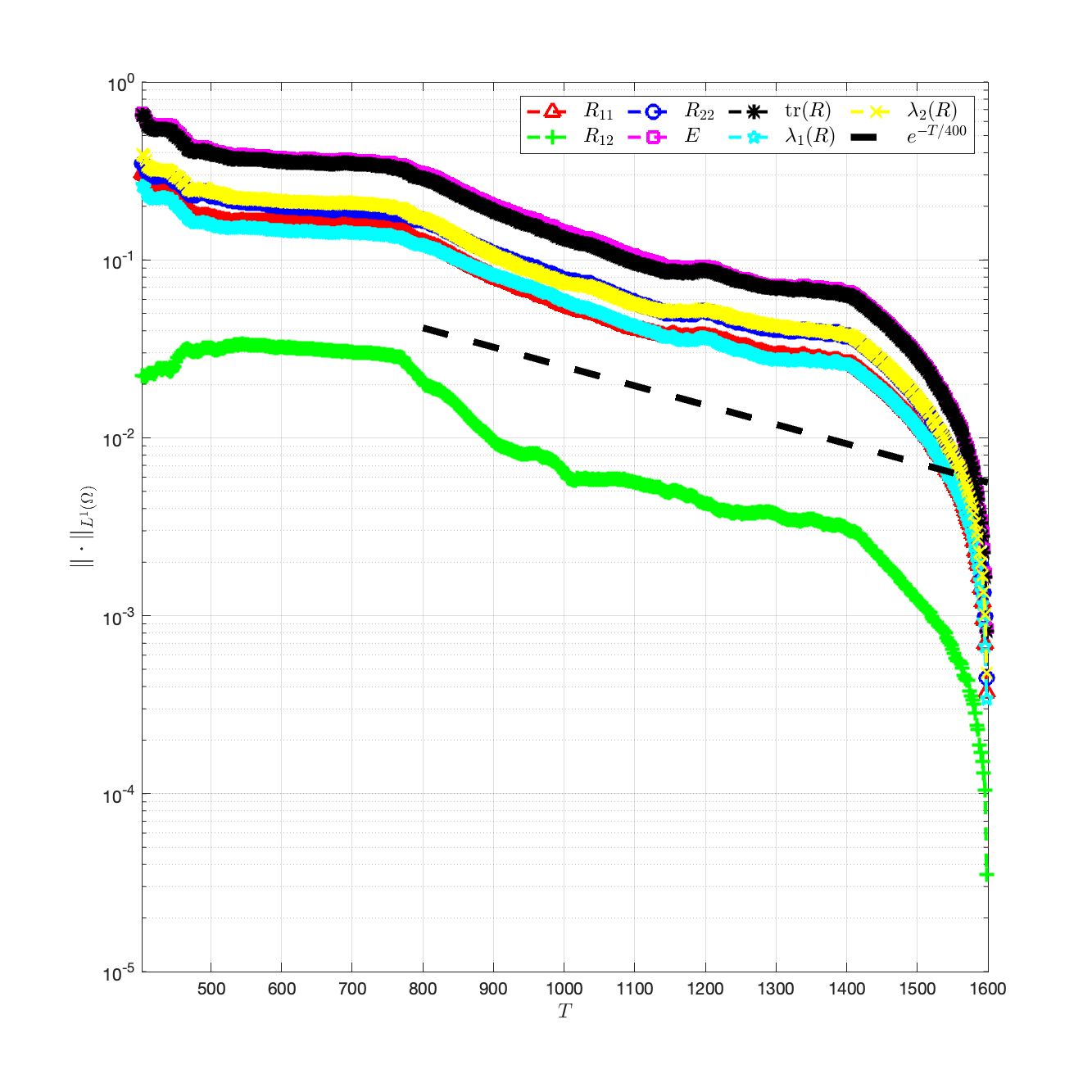} 
	\includegraphics[width=0.24\textwidth]{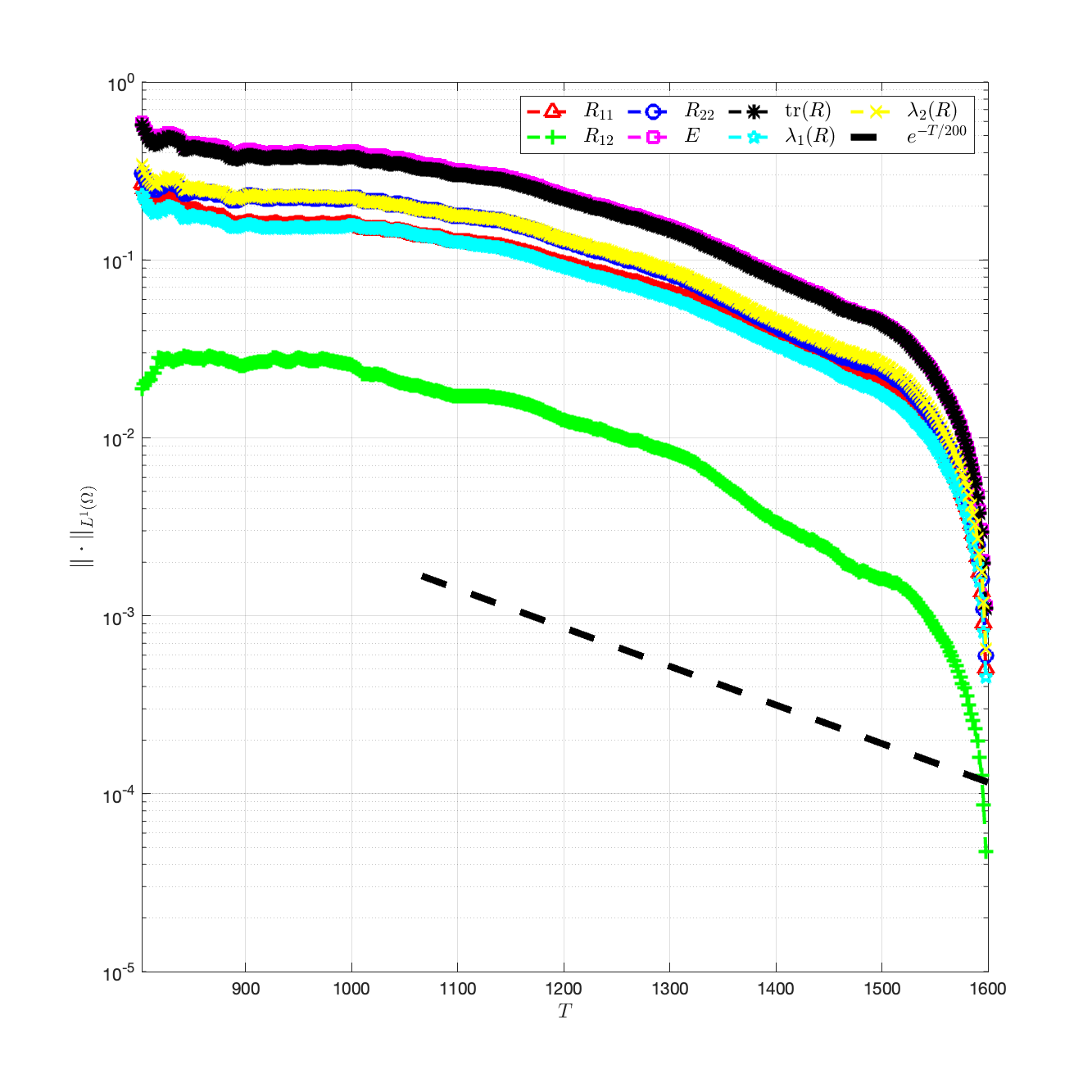}
	\includegraphics[width=0.24\textwidth]{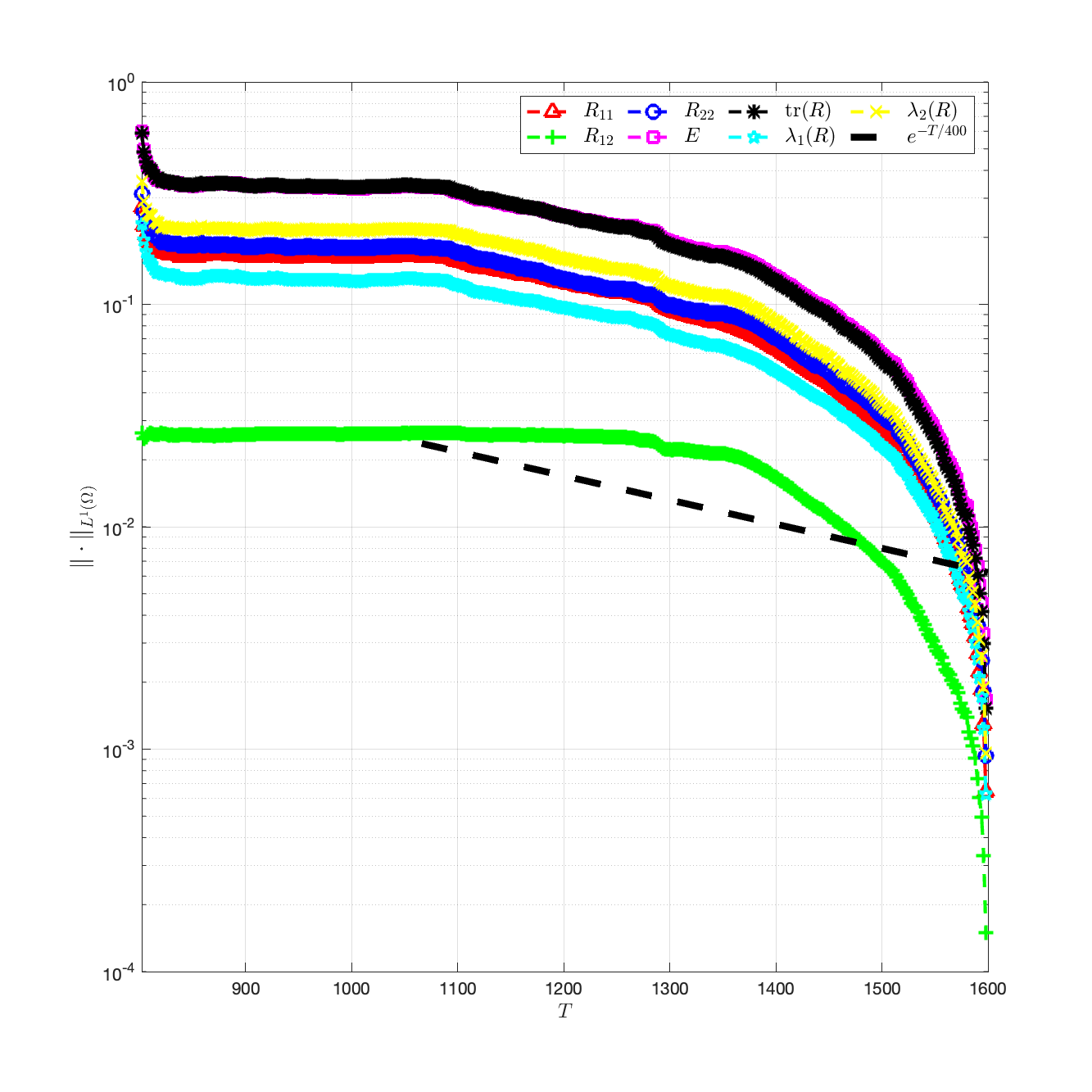} 
	\includegraphics[width=0.24\textwidth]{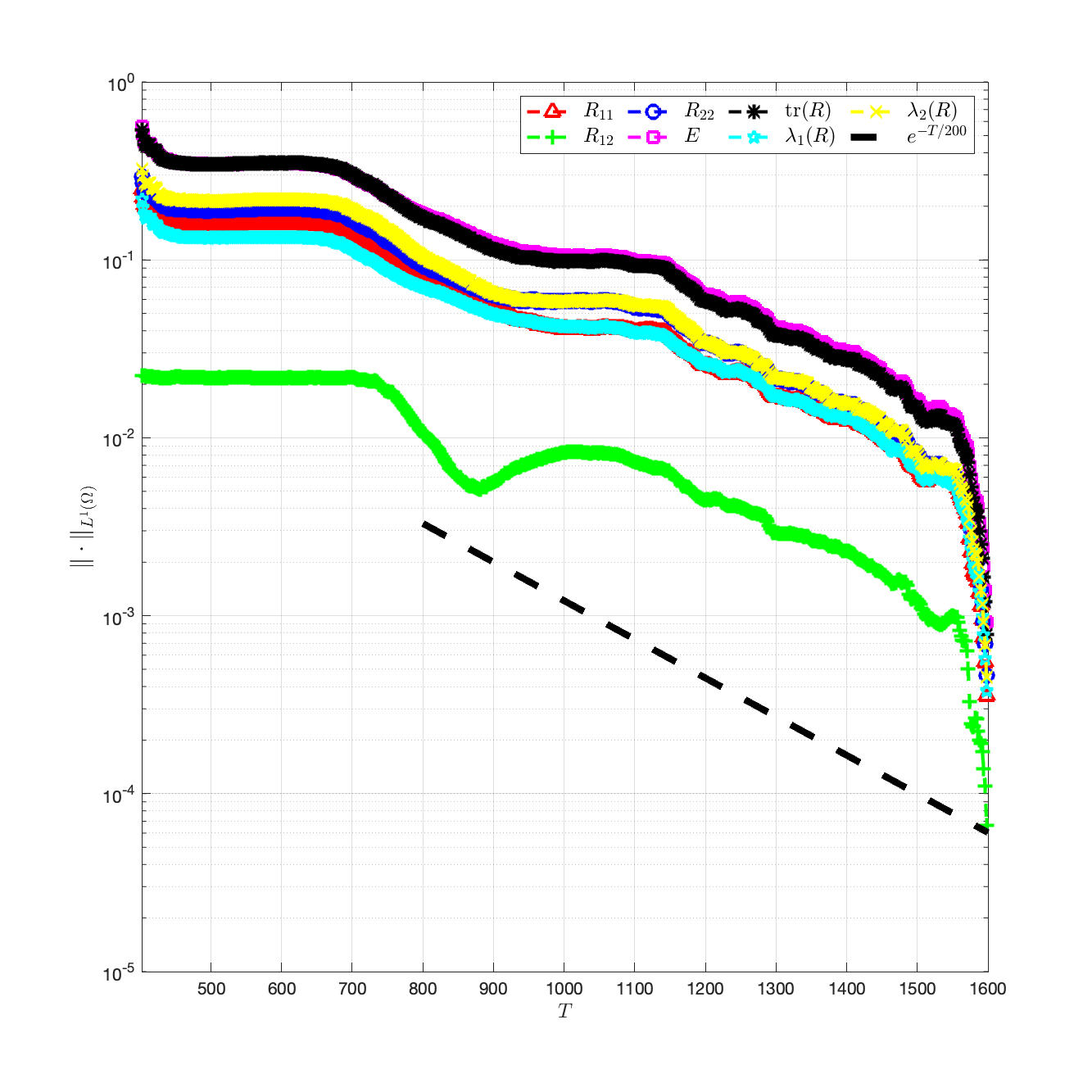}
	\caption{  \small{Rayleigh--B\' enard Experiments: errors $E_i, \, i = 1,2,3,4$ (from top to bottom) for Experiments 2-5 (from left to right).}}\label{fig-Err-fur}
\end{figure}

 \begin{figure}[htbp]
	\setlength{\abovecaptionskip}{0.cm}
	\setlength{\belowcaptionskip}{-0.cm}
	\centering
	\includegraphics[width=0.24\textwidth]{./gif/case0/mxprobabilityPRONUM52MX640MY320IS201IE800_Norm} 
	\includegraphics[width=0.24\textwidth]{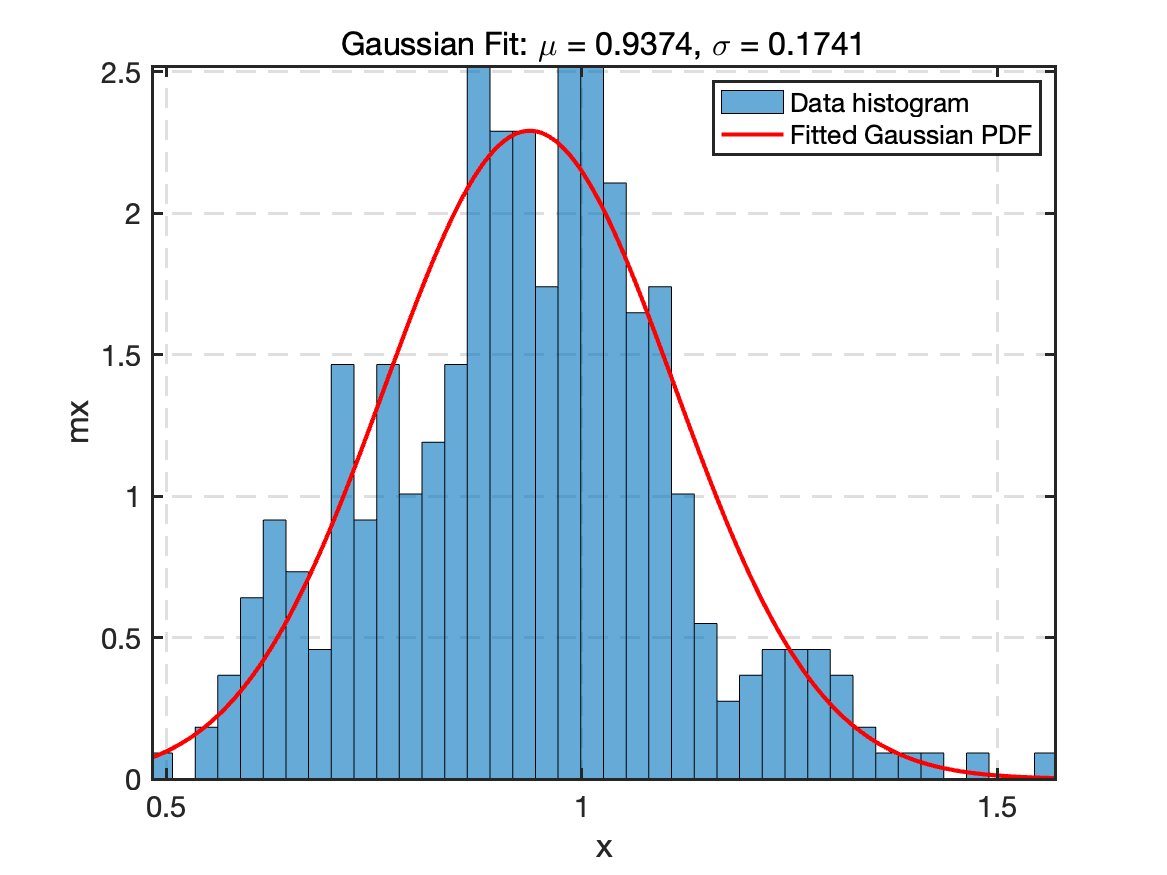}
	\includegraphics[width=0.24\textwidth]{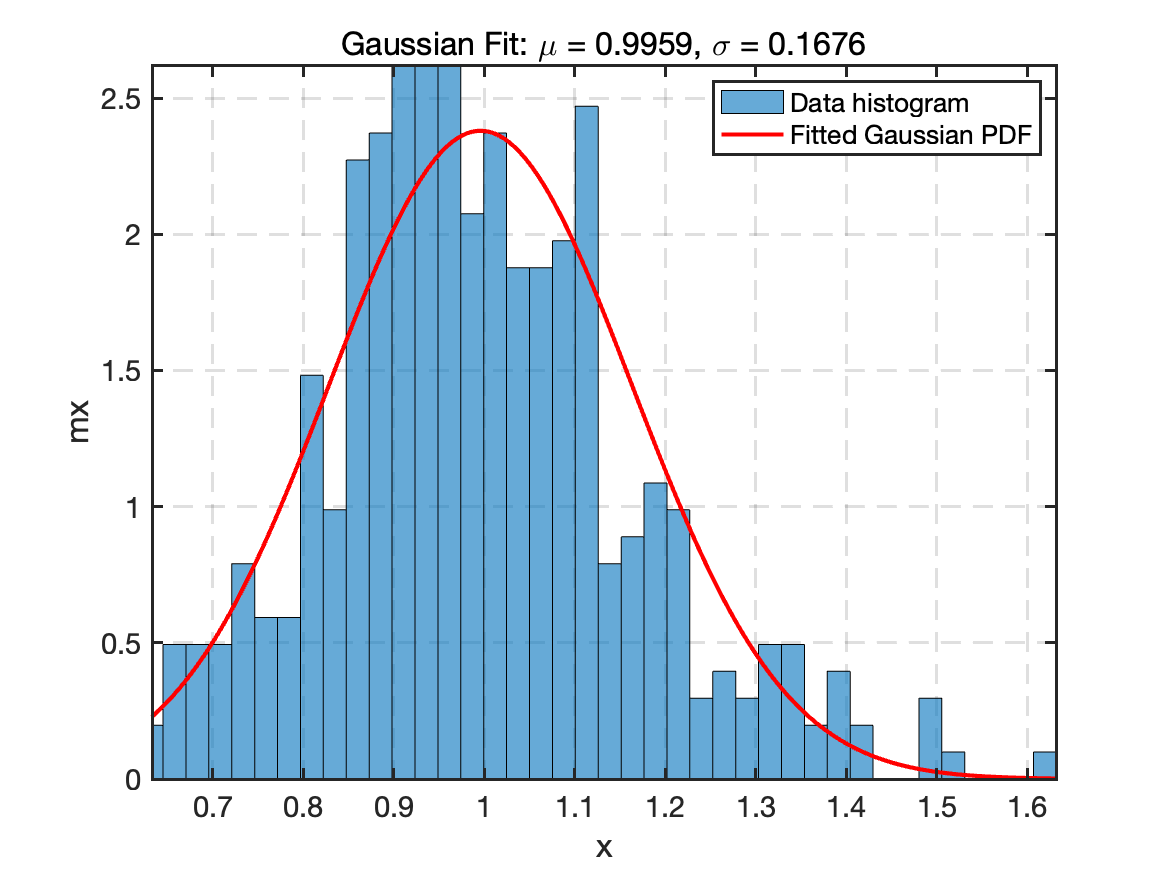} 
	\includegraphics[width=0.24\textwidth]{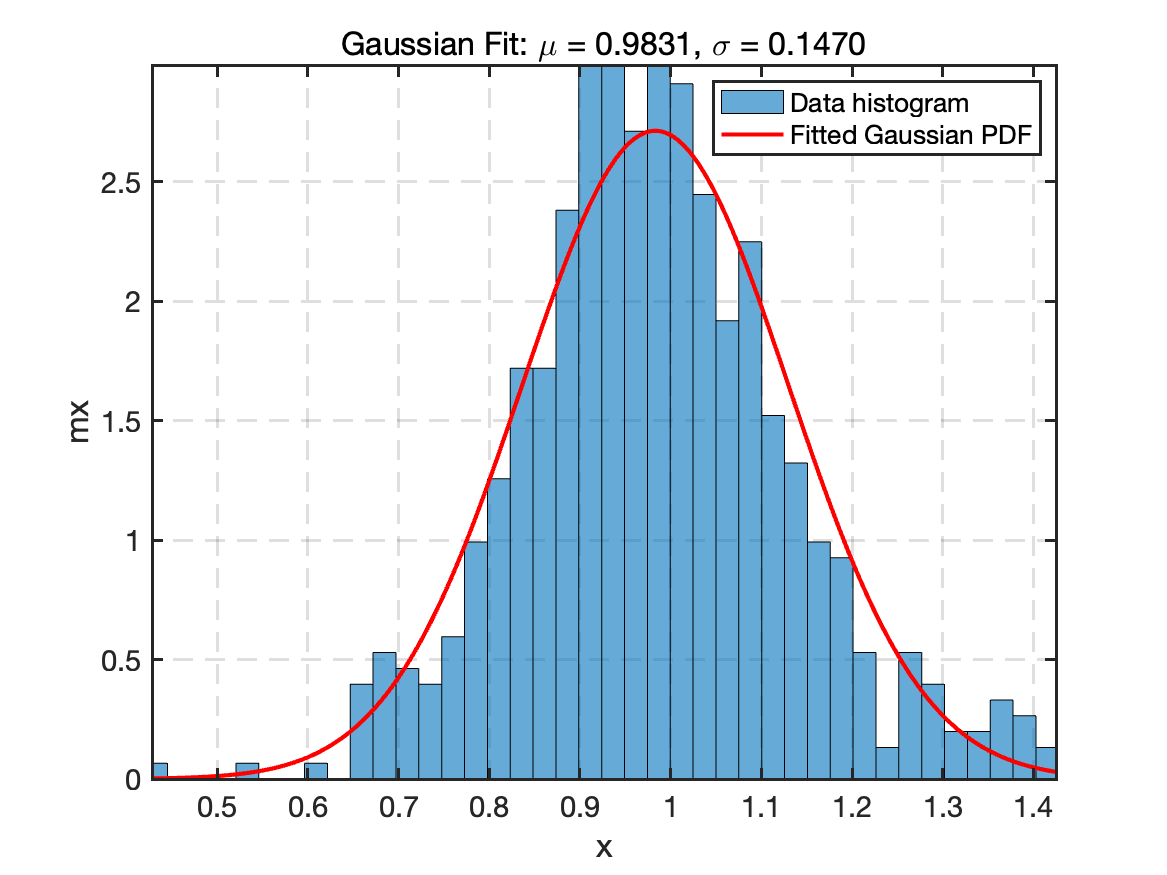}
	\\
	\includegraphics[width=0.24\textwidth]{./gif/case0/myprobabilityPRONUM52MX640MY320IS201IE800_Norm} 
	\includegraphics[width=0.24\textwidth]{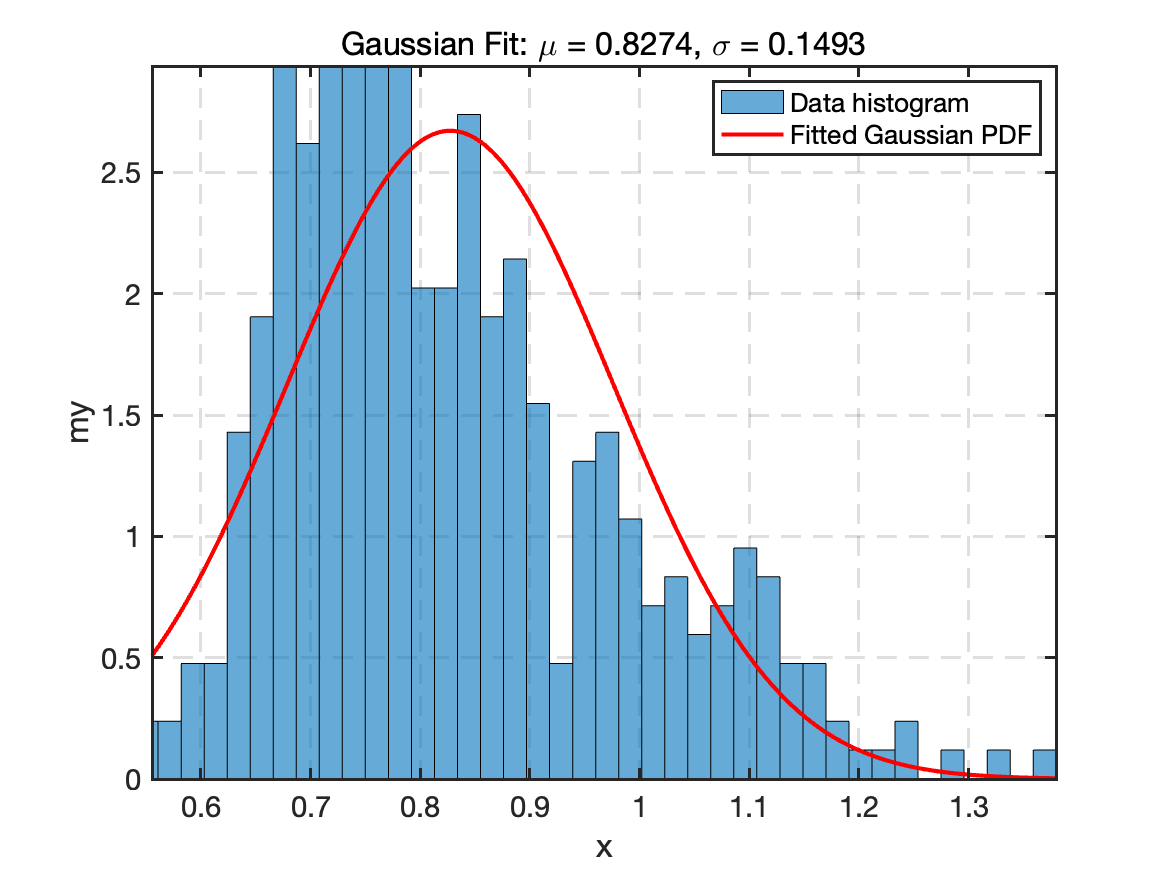}
	\includegraphics[width=0.24\textwidth]{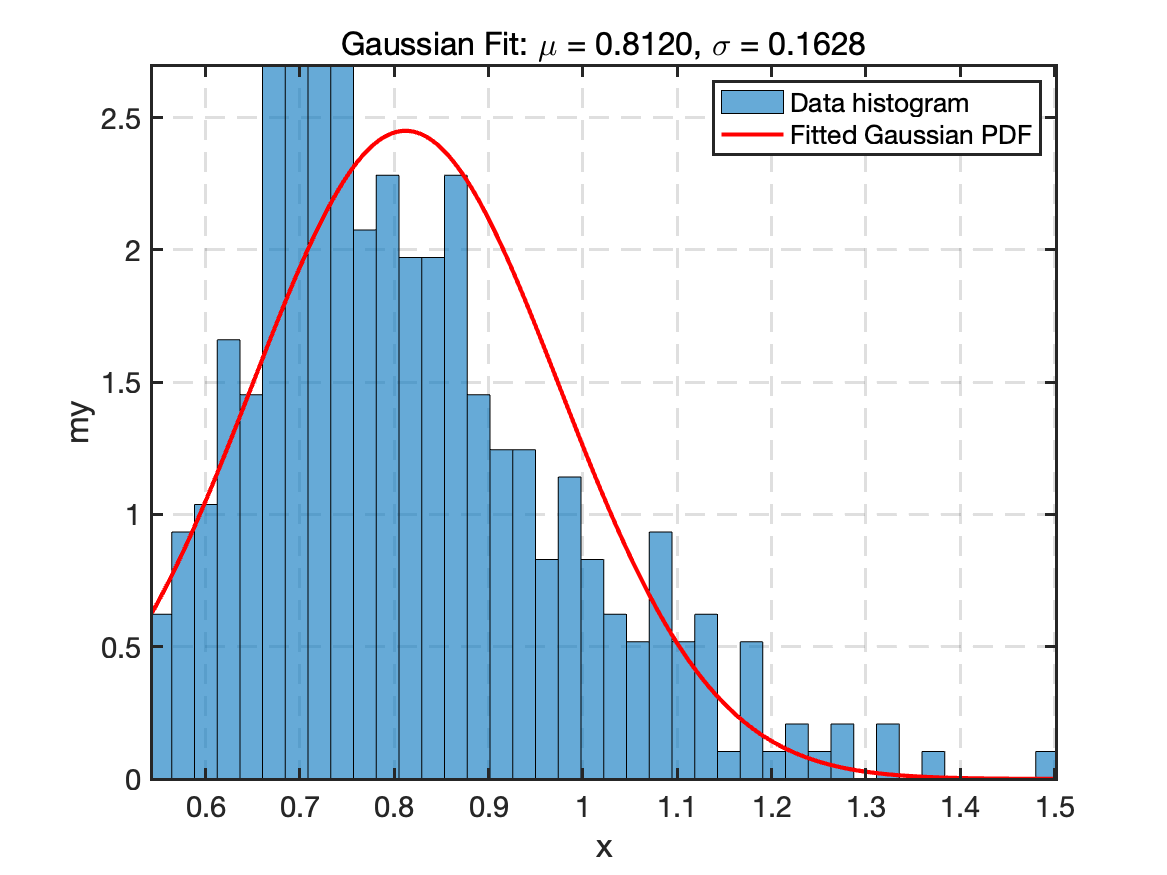} 
	\includegraphics[width=0.24\textwidth]{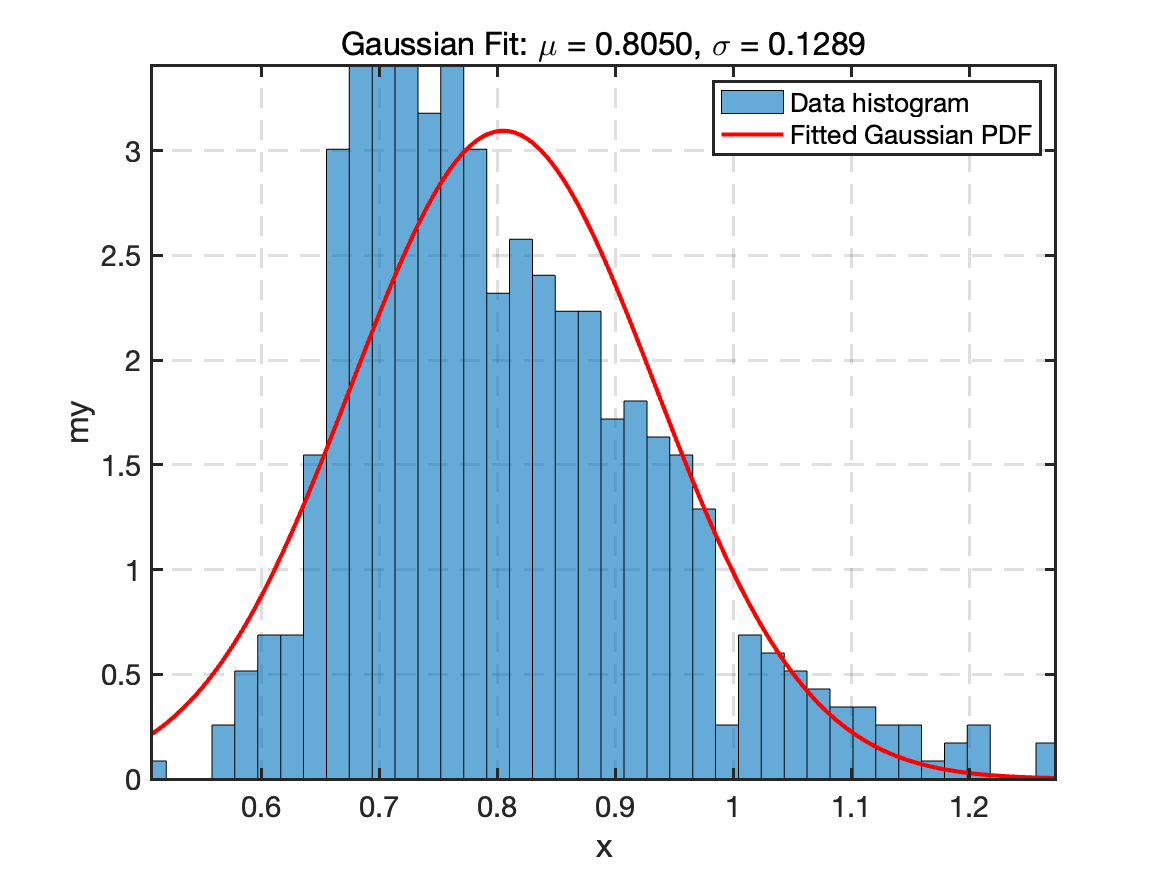}
	\\
	\includegraphics[width=0.24\textwidth]{./gif/case0/energyprobabilityPRONUM52MX640MY320IS201IE800_Norm} 
	\includegraphics[width=0.24\textwidth]{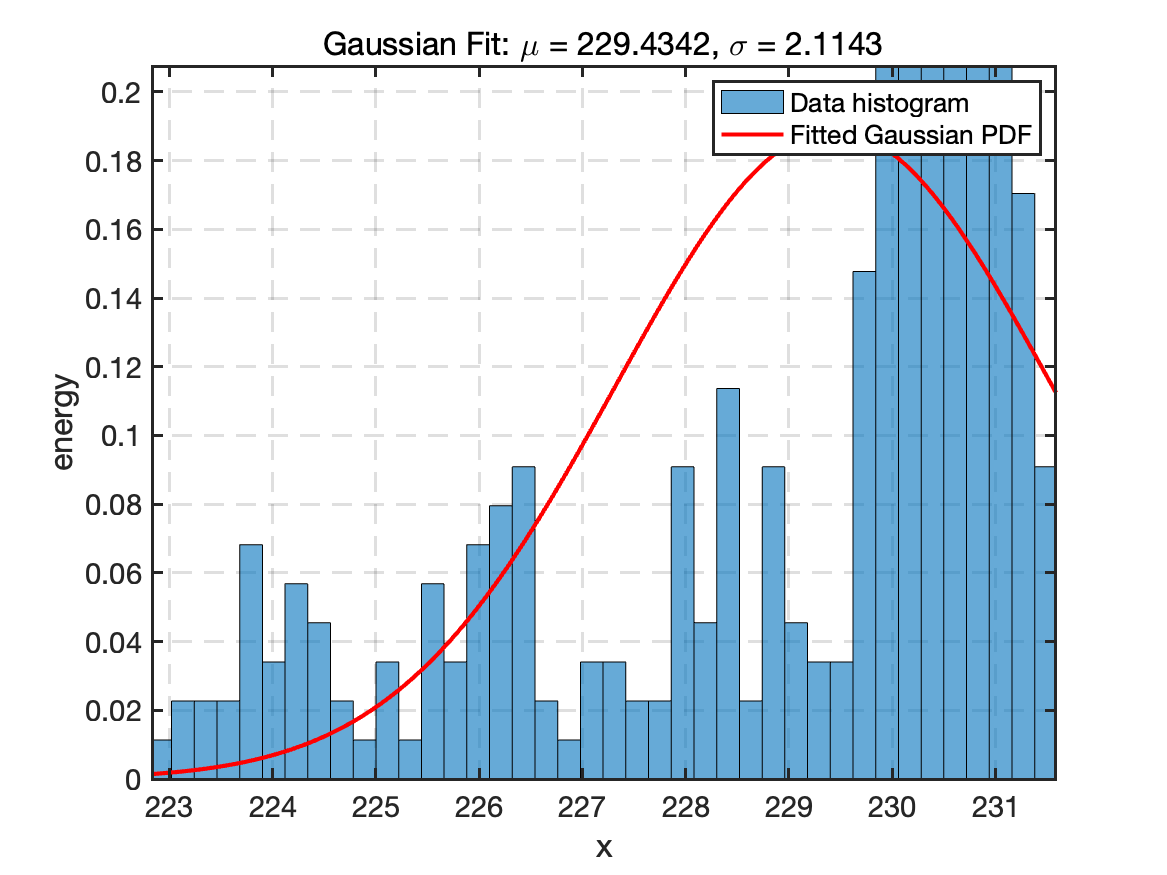}
	\includegraphics[width=0.24\textwidth]{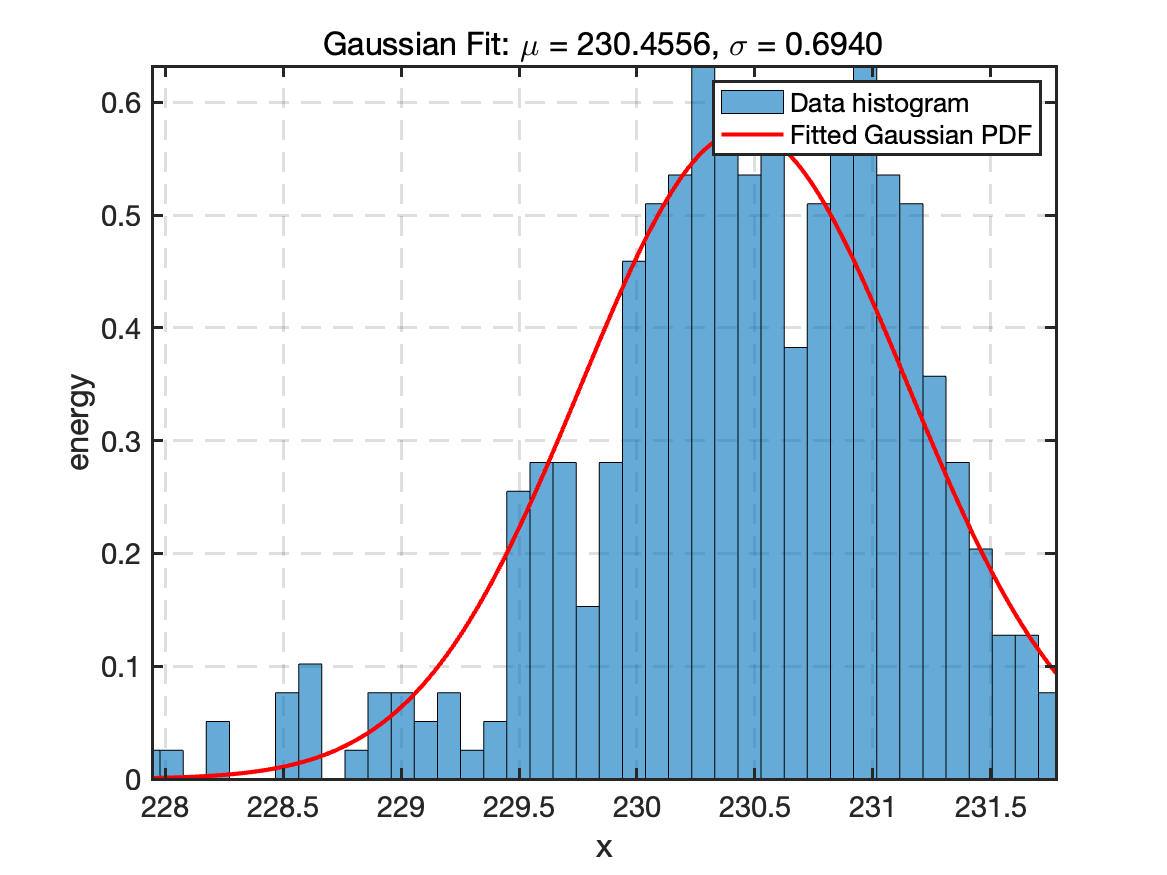} 
	\includegraphics[width=0.24\textwidth]{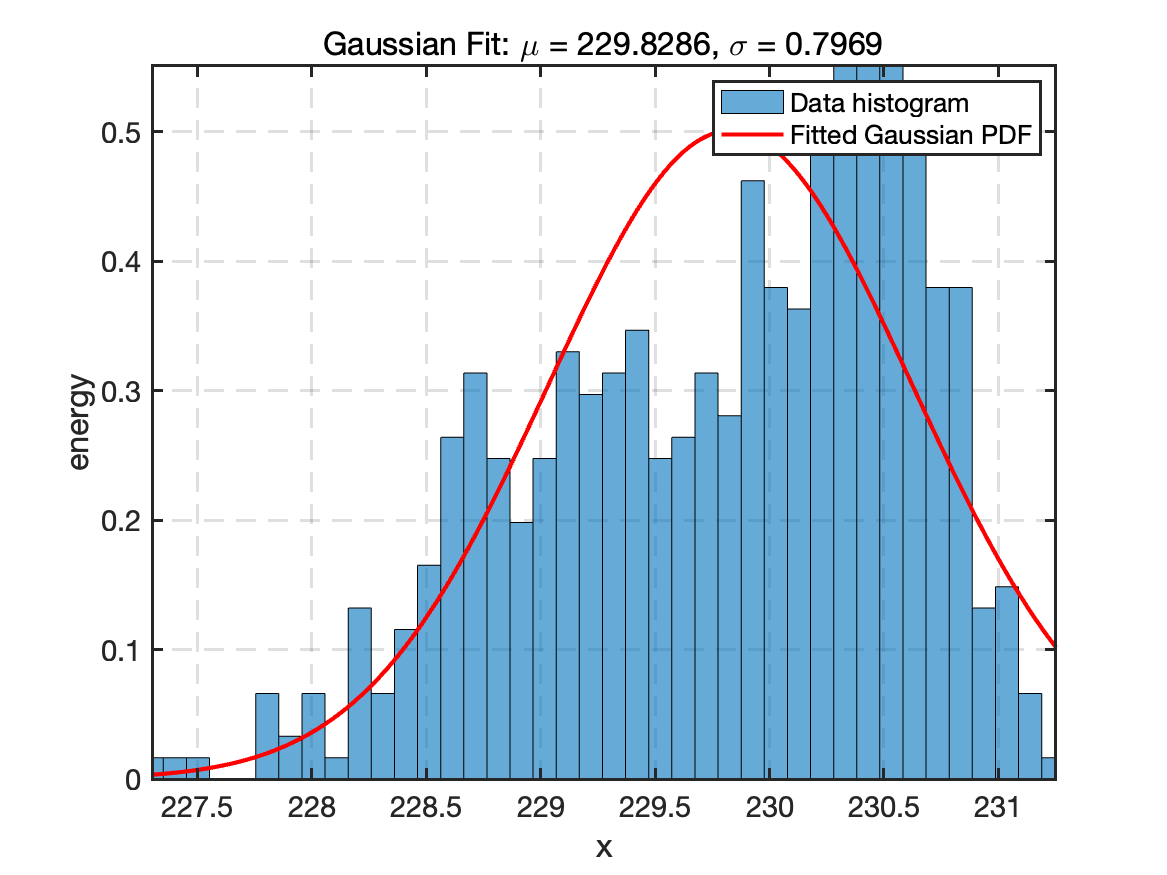}
	\\
	\includegraphics[width=0.24\textwidth]{./gif/case0/BallEnergyprobabilityPRONUM52MX640MY320IS201IE800_Norm} 
	\includegraphics[width=0.24\textwidth]{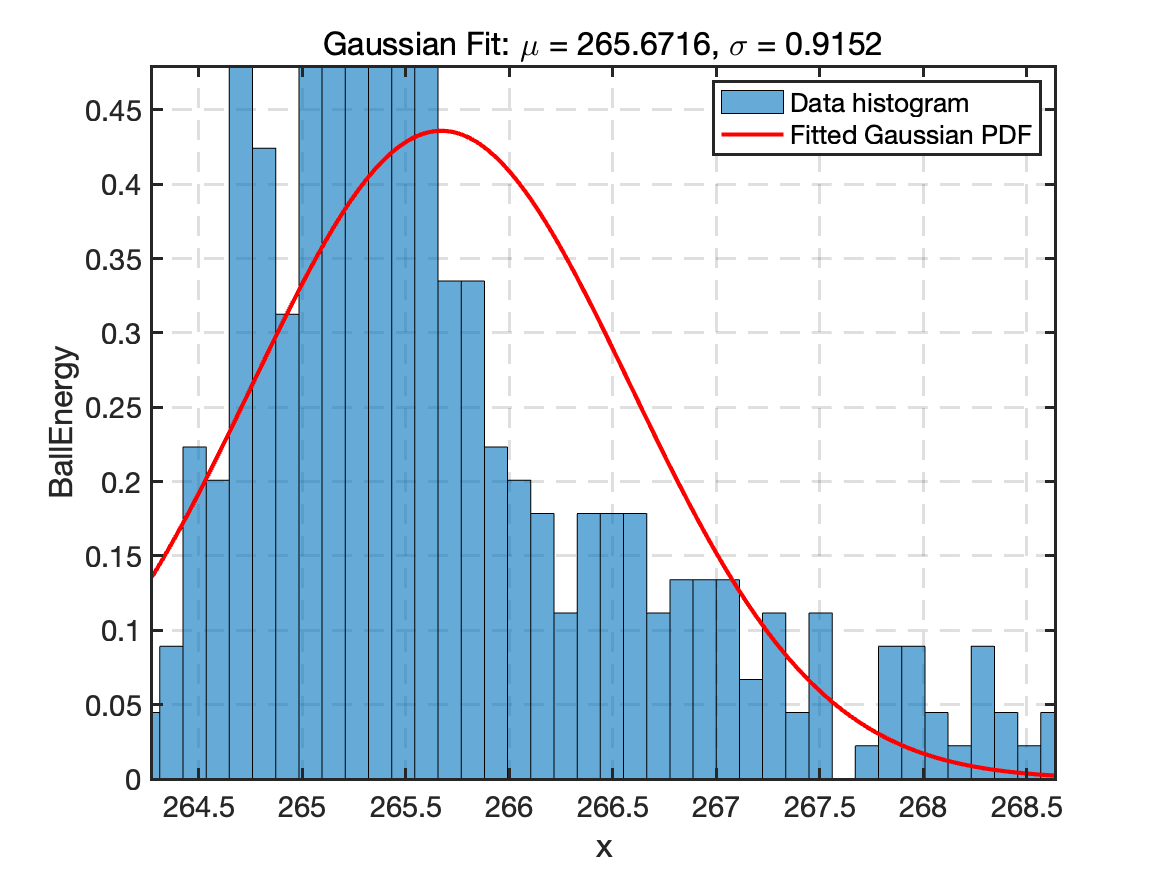}
	\includegraphics[width=0.24\textwidth]{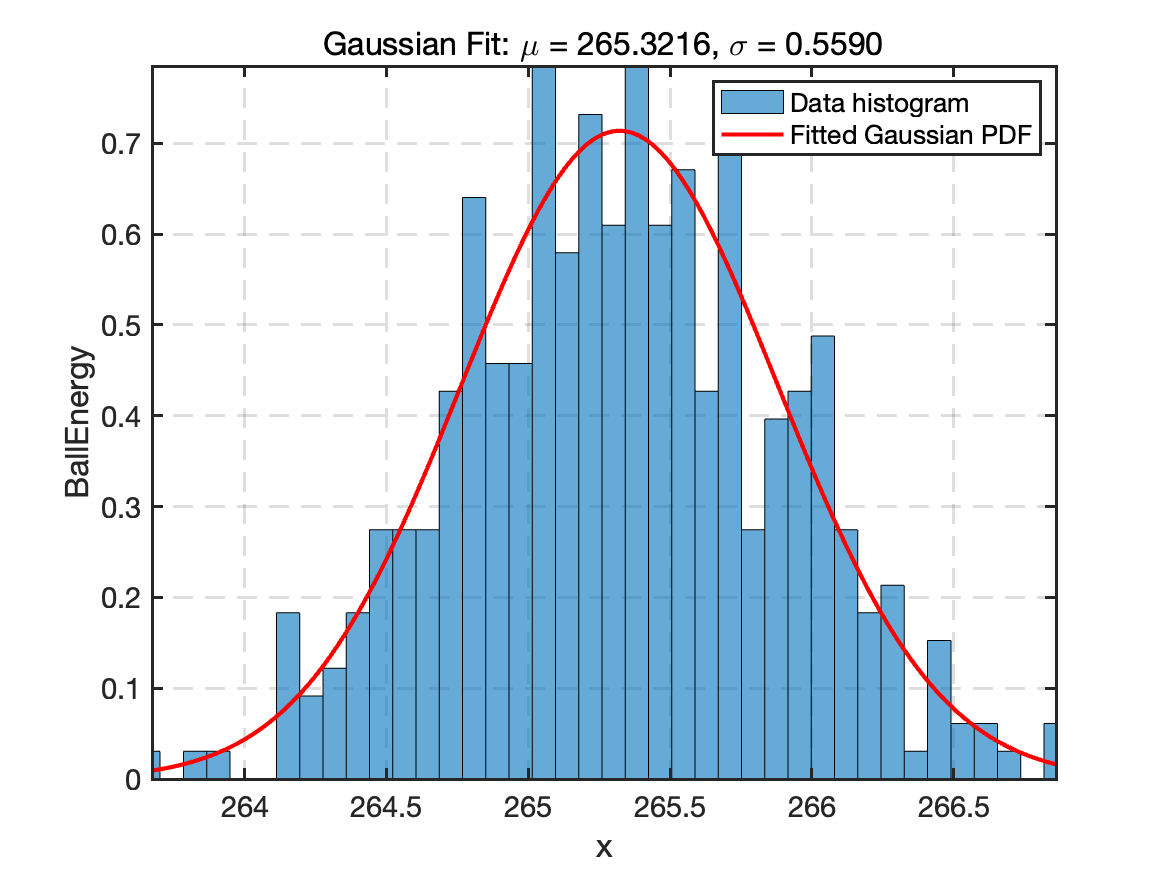} 
	\includegraphics[width=0.24\textwidth]{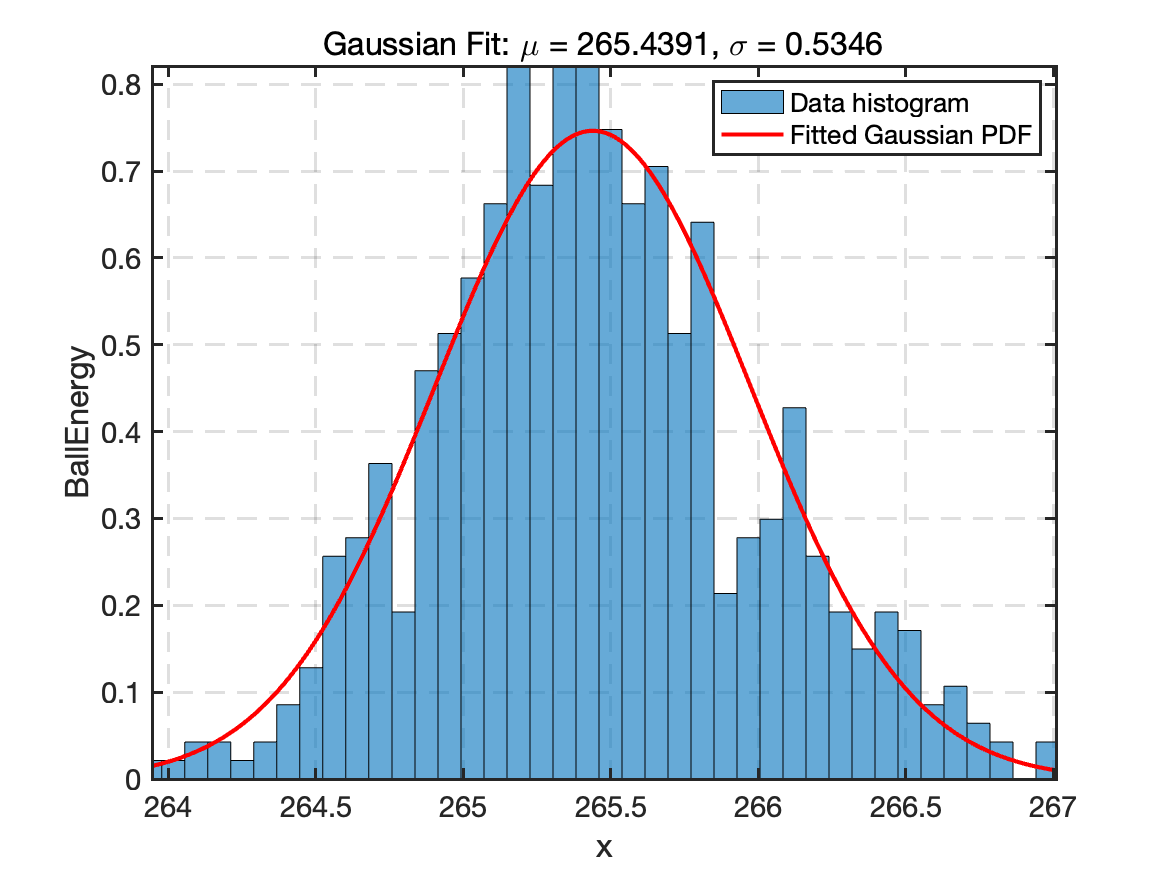}
	\caption{  \small{Rayleigh--B\' enard Experiments:  $\mathcal{M}\left(\norm{U_h(t,\cdot)}_{L^1(\Omega)} \right)$ with $U \in \{m_x,m_y,E,BE\}$ (from top to bottom) for Experiments 2-5 (from left to right).}}\label{fig-Measure-fur}
\end{figure}

\newpage
\subsubsection{Summary of numerical simulations}

We conclude with summarizing the results of numerical experiments and formulate conclusions.  

\begin{enumerate}
\item {\bf Attractor}: Numerical experiments confirm the existence of an attractor, cf.~Figure~\ref{fig-Evo-fur}.

\item {\bf  Ergodic hypothesis}: 
 The numerical simulations in Tables~\ref{tabel1}, \ref{tabel2} and Figures~\ref{fig-Err-ex1}, \ref{fig-Err}, \ref{fig-Err-fur} 
 are in agreement with the ergodic hypothesis \eqref{L1}.


\item {\bf Invariant measure}: We conjecture that any invariant measure -- a stationary statistical solution sitting on the attractor -- is of   Gaussian type, cf.~Figures~\ref{fig-Measure-Ex2-1}, \ref{fig-Measure-Ex2-2}, \ref{fig-Measure-Ex2-3}, \ref{fig-Measure-fur}.

\item {\bf Reynolds stress}: The Reynolds stress tensor and energy fluctuation converge to a constant state for large time, but do not vanish, cf.~Figures~\ref{fig-Defect-1}, \ref{fig-Defect-2}, \ref{fig-Err-fur}.

\end{enumerate}


%


\section*{Acknowledgments}
This work was partially supported by the Mathematical Research Institute Oberwolfach via the Oberwolfach Research Fellows project in 2025. The authors gratefully acknowledge the hospitality of the institute and its stimulating working atmosphere.

\bibliographystyle{plain}

\end{document}